\pdfoutput=1
\documentclass[12pt,a4paper,reqno]{amsart}
\usepackage{amssymb}

\usepackage{amscd}
\usepackage{enumerate}
\usepackage{graphicx}
\usepackage{siunitx}
\usepackage{tikz-cd}
\usepackage{color}
\usetikzlibrary{arrows}
\numberwithin{equation}{section}

\usepackage{mathabx}

\usepackage{mathtools}
\usepackage[tableposition=top]{caption}
\usepackage{booktabs,dcolumn}

\DeclareMathOperator*\MPZ{MPZ}

\DeclareMathOperator\swan{swan}
\DeclareMathOperator*\cond{cond}
\DeclareMathOperator*\Gal{Gal}

\newcommand{\FT}{\mathrm{FT}} 
\DeclareMathOperator{\hypk}{Kl}
\newcommand{\Kl}{\mathcal{K}\ell}
\DeclareCaptionLabelSeparator{separation}{:\quad}
\DeclareFontFamily{OT1}{rsfs}{}
\DeclareFontShape{OT1}{rsfs}{n}{it}{<-> rsfs10}{}
\DeclareMathAlphabet{\mathscr}{OT1}{rsfs}{n}{it}

\theoremstyle{plain}

\newtheorem{theorem}{Theorem}[section]
\newtheorem{proposition}[theorem]{Proposition}
\newtheorem{lemma}[theorem]{Lemma}
\newtheorem{corollary}[theorem]{Corollary}

\newtheorem{claim}[theorem]{Claim}

\theoremstyle{definition}

\newtheorem{definition}[theorem]{Definition}
\newtheorem{remark}[theorem]{Remark}

\newtheorem{example}[theorem]{Example}

\newcommand\F{\mathbb{F}}

\newcommand\R{\mathbb{R}}
\newcommand\Z{\mathbb{Z}}
\newcommand\N{\mathbb{N}}
\newcommand\C{\mathbb{C}}
\newcommand\Q{\mathbb{Q}}
\newcommand\eps{\varepsilon}

\newcommand\Scal{\mathcal{S}}
\newcommand\Dcal[2]{\mathcal{D}^{({#1})}({#2})}
\newcommand\Dcone[1]{\mathcal{D}({#1})}
\newcommand{\DI}[3]{\mathcal{D}_{{#1}}^{({#2})}({#3})}
\newcommand{\DIone}[2]{\mathcal{D}_{{#1}}({#2})}
\newcommand{\emt}{\mathrm{EMT}}
\newcommand{\uple}[1]{\text{\boldmath${#1}$}}

\newcommand\TypeI{\operatorname{Type}_{\operatorname{I}}}
\newcommand\TypeII{\operatorname{Type}_{\operatorname{II}}}
\newcommand\TypeIII{\operatorname{Type}_{\operatorname{III}}}
\renewcommand\P{\mathbb{P}}

\newcommand\rk{\mathrm{rk}}
\newcommand\Fr{\mathrm{Fr}}
\newcommand\tr{\mathrm{tr}}

\newcommand\ov{\overline}
\newcommand\wcheck{\widecheck}
\newcommand\mcF{\mathcal{F}}
\newcommand\mcG{\mathcal{G}}
\newcommand\mcK{\mathcal{K}}

\newcommand\mcL{\mathcal{L}}
\newcommand\Gm{\mathbb{G}_m}
\newcommand\Af{\mathbb{A}^1}
\newcommand\Pl{\mathbb{P}^1}
\newcommand\Aa{\mathbb{A}}
\DeclareMathOperator{\Tr}{Tr}
\newcommand\Fp{\F_{p}}
\newcommand\Fpt{\F_{p}^{\times}}

\def\stacksum#1#2{{\stackrel{{\scriptstyle #1}}
{{\scriptstyle #2}}}}

\newcommand{\piar}{\pi_1}
\newcommand{\pige}{\pi_1^{g}}
\newcommand\sumsum{\mathop{\sum\sum}\limits}
\newcommand\multsum{\mathop{\sum\cdots\sum}\limits}
\newcommand\trpsum{\mathop{\sum\sum\sum}\limits}
\newcommand\Rr{\mathbb{R}}
\newcommand\Zz{\mathbb{Z}}
\newcommand\Nn{\mathbb{N}}
\newcommand\Cc{\mathbb{C}}

\newcommand{\ra}{\rightarrow}

\newcommand{\GL}{\mathrm{GL}}

\newcommand{\ftd}{\mathcal{I}}
\renewcommand{\sim}{\asymp} 

\renewcommand{\lessapprox}{\llcurly}
\renewcommand{\gtrapprox}{\ggcurly}
\renewcommand{\phi}{\varphi}
\newcommand{\onef}{\mathbf{1}}

\usepackage{hyperref}

\begin{document}

\title[Equidistribution of primes in arithmetic progressions]{New equidistribution estimates of Zhang type}

\author{D.H.J. Polymath}
\address{http://michaelnielsen.org/polymath1/index.php}


\begin{abstract} We prove distribution estimates for primes in
  arithmetic progressions to large smooth squarefree moduli, with
  respect to congruence classes obeying Chinese Remainder Theorem
  conditions, obtaining an exponent of distribution $\frac{1}{2} +
  \frac{7}{300}$.
\end{abstract}


\maketitle

\setcounter{tocdepth}{1}
\tableofcontents


\section{Introduction}\label{sec:intro}

In May 2013, Y. Zhang~\cite{zhang} proved the existence of infinitely
many pairs of primes with bounded gaps. In particular, he showed that
there exists at least one $h\geq 2$ such that the set
$$
\{p\text{ prime }\,\mid\, p+h\text{ is prime}\}
$$
is infinite.  (In fact, he showed this for some even $h$ between $2$ and $7 \times 10^7$, although the precise value of $h$ could not be extracted from his method.)
\par
Zhang's work started from the method of Goldston, Pintz and Y\i ld\i
r\i m~\cite{gpy}, who had earlier proved the bounded gap property,
conditionally on distribution estimates concerning primes in
arithmetic progressions to \emph{large moduli}, i.e., beyond the reach
of the Bombieri--Vinogradov theorem.
\par
Based on work of Fouvry and Iwaniec \cite{ft, fi,fi-2,fik-3} and Bombieri, Friedlander and
Iwaniec \cite{bfi,bfi-2,bfi-3}, distribution estimates going beyond the Bombieri--Vinogradov range for arithmetic functions such as the von Mangoldt function were already
known. However, they involved restrictions concerning the residue
classes which were incompatible with the method of Goldston, Pintz and
Y\i ld\i r\i m.
\par
Zhang's resolution of this difficulty proceeded in two stages.  First, he isolated a weaker distribution estimate
that sufficed to obtain the bounded gap property (still involving the
crucial feature of going beyond the range accessible to the
Bombieri--Vinogradov technique), where
(roughly speaking) only smooth\footnote{\ I.e, friable.}  moduli were
involved, and the residue classes had to obey strong multiplicative
constraints (the possibility of such a weakening
had been already noticed by Motohashi and Pintz~\cite{mp}). Secondly, and more significantly, Zhang then proved such a distribution estimate.
\par
This revolutionary achievement led to a flurry of activity. In
particular, the \textsc{Polymath8} project was initiated by T. Tao
with the goal first of understanding, and then of improving and
streamlining, where possible, the argument of Zhang. This was highly
successful, and through the efforts of a number of people, reached a
state in October 2013, when the first version of this paper \cite{poly8-original}
established the bounded gap property in the form
$$
\liminf (p_{n+1}-p_n)\leq 4680,
$$
where $p_n$ denotes the $n$-th prime number.
\par
However, at that time, J. Maynard~\cite{maynard-new} obtained another
conceptual breakthrough, by showing how a modification of the
structure and of the main term analysis of the method of Goldston,
Pintz and Y\i ld\i r\i m was able to establish not only the bounded
gap property using only the Bombieri-Vinogradov theorem (in fact the
bound
$$
\liminf (p_{n+1}-p_n)\leq 600
$$
obtained was significantly better than the one obtained by \textsc{Polymath8}), but
also the bounds
$$
\liminf (p_{n+k}-p_n)<+\infty
$$
for any fixed $k\geq 1$ (in a quantitative way), something which was
out of reach of the earlier methods, even for $k=2$. (Similar results
were obtained independently in unpublished work of T. Tao.)
\par
Because of this development, a part of the \textsc{Polymath8} paper
became essentially obsolete.  Nevertheless, the distribution estimate
for primes in arithmetic progressions are not superceded by the new
method, and it has considerable interest for analytic number
theory. Indeed, it is the best known result concerning primes in
arithmetic progressions to large moduli without fixing the residue
class. (When the class is fixed, the best results remain those of
Bombieri, Friedlander and Iwaniec, improving on those of Fouvry and
Iwaniec~\cite{fi-2,bfi}.)  The results here are also needed to obtain the best known bounds on $\liminf (p_{n+k}-p_n)$ for large values of $k$; see \cite{poly8b}.
\par
The present version of the work of \textsc{Polymath8} therefore
contains only the statement and proof of these estimates.  We note
however that some of the earlier version is incorporated in our subsequent
paper~\cite{poly8b}, which builds on Maynard's method to further
improve many bounds concerning gaps between primes, both conditional
and unconditional.  Furthermore, the original version of this paper, and the history
of its elaboration, remain available online~\cite{poly8-original}.
\par
\medskip
\par
Our main theorem is:

\begin{theorem}\label{th-main}
  Let $\theta=1/2+7/300$.  Let $\eps>0$ and $A\geq 1$ be fixed real
  numbers. For all primes $p$, let $a_p$ be a fixed invertible
  residue class modulo $p$, and for $q\geq 1$ squarefree, denote by
  $a_q$ the unique invertible  residue class modulo $q$ such that
  $a_q\equiv a_p$ modulo all primes $p$ dividing $q$.
\par
There exists $\delta>0$, depending only on $\eps$, such that for
$x\geq 1$, we have
$$
\sum_{\substack{q\leq x^{\theta-\eps}\\q\text{ $x^{\delta}$-smooth,
      squarefree}}} \Bigl|\psi(x;q,a_q)-\frac{x}{\varphi(q)}\Bigr| \ll
\frac{x}{(\log x)^A},
$$
where the implied constant depends only on $A$, $\eps$ and $\delta$,
and in particular is independent of the residue classes $(a_p)$.
\end{theorem}

In this statement, we have, as usual, defined
$$
\psi(x;q,a)=\sum_{\substack{n\leq x\\n=a\ (q)}}\Lambda(n),
$$
where $\Lambda$ is the von Mangoldt function.  In \cite{zhang}, Zhang establishes a weaker form of Theorem \ref{th-main}, with
  $\theta=1/2+1/584$, and with the $a_q$ required to be roots of a polynomial $P$ of the form $P(n) := \prod_{1 \leq j \leq k; j \neq i} (n+h_j-h_i)$ for a fixed admissible tuple $(h_1,\dots,h_k)$ and $i=1,\dots,k$.





In fact, we will prove a number of variants of this bound. These
involve either weaker restrictions on the moduli
(``dense-divisibility'', instead of smoothness, which may be useful in
some applications), or smaller values of $\theta>1/2$, but with
significantly simpler proofs.  In particular, although the full strength of
Theorem~\ref{th-main} depends in crucial ways on applications of
Deligne's deepest form of the Riemann Hypothesis over finite fields,
we show that, for a smaller value of $\theta>1/2$, it is possible to
obtain the same estimate by means of Weil's theory of exponential sums
in one variable over finite fields.
\par
\medskip
\par
The outline of this paper is as follows: in the next section, we
briefly outline the strategy, 
starting from the work of Bombieri, Fouvry, Friedlander, and Iwaniec
(in chronological order, \cite{fi, fi-2, fi-3, bfi, bfi-2, bfi-3,
  fik-3}), and explaining Zhang's innovations. These involve different 
types of estimates of bilinear or trilinear nature, which we present
in turn. All involve estimates for exponential sums over finite
fields.  We therefore survey the relevant theory, separating that part
depending only on one-variable character sums of Weil type
(Section~\ref{exp-sec}), and the much deeper one which depends on Deligne's
form of the Riemann Hypothesis (Section~\ref{deligne-sec}).  In both
cases, we present the formalism in sometimes greater generality than
strictly needed, as these results are of independent interest and
might be useful for other applications.

\subsection{Overview of proof}\label{ssec-outline}

We begin with a brief and informal overview of the methods used in
this paper.
\par
Important work of Fouvry and Iwaniec~\cite{fi, fi-2} and of Bombieri,
Friedlander and Iwaniec~\cite{bfi,bfi-2, bfi-3} had succeeded, in some
cases, in establishing distribution results similar to
Theorem~\ref{th-main}, in fact with $\theta$ as large as $1/2+1/14$,
but with the restriction that the residue classes $a_p$ are obtained
by reduction modulo $p$ of a fixed integer $a\geq 1$.
\par
Following the techniques of Bombieri, Fouvry, Friedlander and Iwaniec,
Zhang used the Heath-Brown identity \cite{hb-ident} to reduce the
proof of (his version of) Theorem~\ref{th-main} to the verification of
three families of estimates, which he called ``Type I'', ``Type II'',
and ``Type III''.  These estimates were then reduced to exponential
sum estimates, using techniques such as Linnik's dispersion method,
completion of sums, and Weyl differencing. Ultimately, the exponential
sum estimates were established by applications of the Riemann
Hypothesis over finite fields, in analogy with all previous works of
this type.  The final part of Zhang's argument is closely related to
the study of the distribution of the ternary divisor function in
arithmetic progressions by Friedlander and Iwaniec~\cite{fi-3}, and
indeed the final exponential sum estimate that Zhang uses already
appears in their work (this estimate was proved by Birch and Bombieri
in the Appendix to~\cite{fi-3}).  An important point is that, by using
techniques that are closer to some older work of Fouvry and Iwaniec
\cite{fi}, Zhang avoids using the spectral theory of automorphic
forms, which is a key ingredient in~\cite{fi-2} and~\cite{bfi}, and
one of the sources of the limitation to a fixed residue in these
works.

Our proof of Theorem \ref{th-main} follows the same general strategy
as Zhang's, with improvements and refinements.
\par
First, we apply the Heath-Brown identity \cite{hb-ident} in Section
\ref{heath-brown-sec}, with little change compared with Zhang's
argument, reducing to the ``bilinear'' (Types I/II) and ``trilinear''
(Type III) estimates. 

For the Type I and Type II estimates, we follow the arguments of Zhang
to reduce to the task of bounding incomplete exponential sums similar
to
$$ \sum_{N<n\leq 2N}
e\Bigl(\frac{c_1\bar{n}+c_2\overline{n+l}}{q}\Bigr),
$$
(where $e(z)=e^{2i\pi z}$ and $\bar{x}$ denotes the inverse of $x$
modulo $q$) for various parameters $N, c_1, c_2, l, q$.  We obtain
significant improvements of Zhang's numerology at this stage, by
exploiting the smooth (or at least densely divisible) nature of $q$,
using the $q$-van der Corput $A$-process of Heath-Brown
\cite{heath-hybrid} and Graham-Ringrose \cite{graham}, combined with
the Riemann Hypothesis for curves over finite fields. Additional gains
are obtained by optimizing the parameterizations of sums prior to
application of the Cauchy-Schwarz inequality. In our strongest Type I
estimate, we also exploit additional averaging over the modulus by
means of higher-dimensional exponential sum estimates, which now do
depend on the deep results of Deligne.  We refer to Sections
\ref{exp-sec}, \ref{typei-ii-sec} and \ref{typei-advanced-sec} for
details of these parts of the arguments.

Finally, for the Type III sums, Zhang's delicate argument \cite{zhang}
adapts and improves the work of Friedlander and Iwaniec~\cite{fi-3} on
the ternary divisor function in arithmetic progressions.  As we said,
it ultimately relies on a three-variable exponential sum estimate that
was proved by Birch and Bombieri in the Appendix to \cite{fi-3}.
Here, we proceed slightly differently, inspired by the streamlined
approach of Fouvry, Kowalski, and Michel \cite{FKM3}.  Namely, in
Section~\ref{typeiii-sec} we show how our task can be reduced to
obtaining certain correlation bounds on hyper-Kloosterman sums.  These
bounds are established in Section~\ref{deligne-sec}, by fully
exploiting the formalism of ``trace functions'' over finite fields
(which relies on Deligne's second, more general proof of the Riemann
Hypothesis over finite fields~\cite{WeilII}).  The very general
techniques presented in Section~\ref{deligne-sec} are also used in the
proof of the strongest Type I estimate in Section
\ref{typei-advanced-sec}, and we present them in considerable detail
in order to make them more accessible to analytic number theorists.

\subsection{About this project}

This paper is part of the \emph{Polymath project}, which was launched
by Timothy Gowers in February 2009 as an experiment to see if research
mathematics could be conducted by a massive online collaboration.
The current project (which was administered by Terence Tao) is the eighth
project in this series.  Further information on the Polymath project can be
found on the web site \url{michaelnielsen.org/polymath1}.  Information
about this specific project may be found at
\begin{center}
\small{\url{michaelnielsen.org/polymath1/index.php?title=Bounded\_gaps\_between\_primes}}
\end{center}
and a full list of participants and their grant acknowledgments may be
found at
\begin{center}
\small{\url{michaelnielsen.org/polymath1/index.php?title=Polymath8\_grant\_acknowledgments}}
\end{center}
\par
We thank John Friedlander for help with the references.  We are
indebted to the multiple referees of the first version of this paper
for many cogent suggestions and corrections.

\subsection{Basic notation}\label{notation-sec}

We use $|E|$ to denote the cardinality of a finite set $E$, and $\onef_E$ to denote the indicator function of a set $E$, thus $\onef_E(n)=1$ when $n \in E$ and $\onef_E(n)=0$ otherwise.  

All sums and products will be over the natural numbers $\N := \{1,2,3,\ldots\}$ unless otherwise specified, with the exceptions of sums and products over the variable $p$, which will be understood to be over primes.

The following important asymptotic notation will be in use throughout
most of the paper; when it is not (as in
Section~\ref{deligne-sec}), we will mention this explicitly.

\begin{definition}[Asymptotic notation]\label{asym}  We use $x$ to denote a large real parameter, which one should think of as going off to infinity; in particular, we will implicitly assume that it is larger than any specified fixed constant. Some mathematical objects will be independent of $x$ and referred to as \emph{fixed}; but unless otherwise specified we allow all mathematical objects under consideration to depend on $x$ (or to vary within a range that depends on $x$, e.g. the summation parameter $n$ in the sum $\sum_{x \leq n \leq 2x} f(n)$).  If $X$ and $Y$ are two quantities depending on $x$, we say that $X = O(Y)$ or $X \ll Y$ if one has $|X| \leq CY$ for some fixed $C$ (which we refer to as the \emph{implied constant}), and $X = o(Y)$ if one has $|X| \leq c(x) Y$ for some function $c(x)$ of $x$ (and of any fixed parameters present) that goes to zero as $x \to \infty$ (for each choice of fixed parameters).  We use $X \lessapprox Y$ to denote the estimate $X \leq x^{o(1)} Y$, $X \sim Y$ to denote the estimate $Y \ll X \ll Y$, and $X \approx Y$ to denote the estimate $Y \lessapprox X \lessapprox Y$.  Finally, we say that a quantity $n$ is of \emph{polynomial size} if one has $n = O(x^{O(1)})$.

If asymptotic notation such as $O()$ or $\lessapprox$ appears on the left-hand side of a statement, this means that the assertion holds true for any specific interpretation of that notation.  For instance, the assertion $\sum_{n=O(N)} |\alpha(n)| \lessapprox N$ means that for each fixed constant $C>0$, one has $\sum_{|n| \leq CN} |\alpha(n)|\lessapprox N$.
\end{definition}

If $q$ and $a$ are integers, we write $a|q$ if $a$ divides $q$. 
\par
If $q$ is a natural number and $a \in \Z$, we use $a\ (q)$ to denote
the congruence class
$$ a\ (q) := \{ a+nq: n \in \Z \}$$
and let $\Z/q\Z$ denote the ring of all such congruence classes $a\
(q)$. The notation $b=a\ (q)$ is synonymous to $b \in \, a \ (q)$. We
use $(a,q)$ to denote the greatest common divisor of $a$ and $q$, and
$[a,q]$ to denote the least common
multiple.\footnote{When $a,b$ are real numbers, we will also
  need to use $(a,b)$ and $[a,b]$ to denote the open and closed
  intervals respectively with endpoints $a,b$.  Unfortunately, this
  notation conflicts with the notation given above, but it should be
  clear from the context which notation is in use.  Similarly for the notation $\overline{a}$ for $a \in \Z/q\Z$, and the notation $\overline{z}$ to denote the complex conjugate of a complex number $z$.}
  More generally, we let $(q_1,\ldots,q_k)$ denote the
greatest simultaneous common divisor of $q_1,\ldots,q_k$.  We note in particular that $(0,q) = q$ for any natural number $q$. 
Note that $a \mapsto (a,q)$ is periodic with period $q$, and so we may
also define $(a,q)$ for $a \in\Z/q\Z$ without ambiguity.  We also let
$$ (\Z/q\Z)^\times := \{ a\ (q): (a,q)=1 \}$$
denote the primitive congruence classes of $\Z/q\Z$.  More generally,
for any commutative ring $R$ (with unity) we use $R^\times$ to denote the
multiplicative group of units.  If $a \in (\Z/q\Z)^\times$, we use $\overline{a}$ to denote the inverse of $a$ in $\Z/q\Z$.

\par
For any real number $x$, we write $e(x) := e^{2\pi i x}$.  We denote
$e_q(a):=e(\frac{a}{q}) = e^{2\pi i a/q}$ (see also the conventions
concerning this additive character at the beginning of
Section~\ref{exp-sec}).

We use the following standard arithmetic functions:
\begin{itemize}
\item[(i)] $\phi(q) := |(\Z/q\Z)^\times|$ denotes the Euler totient function of $q$.
\item[(ii)] $\tau(q) := \sum_{d|q} 1$ denotes the divisor function of $q$.
\item[(iii)] $\Lambda(q)$ denotes the von Mangoldt function of $q$, thus $\Lambda(q)=\log p$ if $q$ is a power of a prime $p$, and $\Lambda(q)=0$ otherwise.
\item[(iv)] $\theta(q)$ is defined to equal $\log q$ when $q$ is a prime, and $\theta(q)=0$ otherwise.
\item[(v)] $\mu(q)$ denotes the M\"obius function of $q$, thus $\mu(q) = (-1)^k$ if $q$ is the product of $k$ distinct primes for some $k \geq 0$, and $\mu(q)=0$ otherwise.
\item[(vi)] $\Omega(q)$ denotes the number of prime factors of $q$ (counting multiplicity).
\end{itemize}

The \emph{Dirichlet convolution} $\alpha \star \beta \colon \N \to \C$
of two arithmetic functions $\alpha,\beta \colon \N \to \C$ is defined
in the usual fashion as
$$ \alpha\star\beta(n) := \sum_{d|n} \alpha(d)
\beta\left(\frac{n}{d}\right) =\sum_{ab=n}{\alpha(a)\beta(b)}.$$

Many of the key ideas in Zhang's work (as well as in the present article) concern the
uniform distribution of arithmetic functions in arithmetic
progressions. For any function $\alpha \colon \N \to \C$ with finite
support (that is, $\alpha$ is non-zero only on a finite set) and any
primitive congruence class $a\ (q)$, we define the (signed)
\emph{discrepancy} $\Delta(\alpha; a\ (q))$ to be the quantity
\begin{equation}\label{disc-def}
  \Delta(\alpha; a\ (q)) := \sum_{n = a\ (q)} \alpha(n) - 
  \frac{1}{\phi(q)} \sum_{(n,q)=1} \alpha(n).
\end{equation}

There are some additional concepts and terminology that will be used in multiple sections of this paper.  These are listed in Table \ref{tab:notation}.

\begin{table}[htbp]\caption{Notation and terminology}
\begin{center}
\begin{tabular}{r p{7cm} p{3cm}}
\toprule
$\varpi$ & Level of distribution & \S \ref{sec-technical} \\
$\delta$ & Smoothness/dense divisibility parameter & \S \ref{sec-technical} \\
$i$ & Multiplicity of dense divisibility & Definition \ref{mdd-def}\\
$\sigma$ & Type I/III boundary parameter & Definition \ref{type-def}\\
\midrule
$\MPZ^{(i)}[\varpi,\delta]$ & MPZ conjecture for densely divisible moduli & Claim \ref{mpzm-def} \\
$\TypeI^{(i)}[\varpi,\delta,\sigma]$ & Type I estimate & Definition \ref{type-def} \\
$\TypeII^{(i)}[\varpi,\delta]$ & Type II estimate & Definition \ref{type-def} \\
$\TypeIII^{(i)}[\varpi,\delta,\sigma]$ & Type III estimate & Definition \ref{type-def} \\
\midrule
$\Scal_I$ & Squarefree products of primes in $I$ & Definition \ref{sfi} \\
$P_I$ & Product of all primes in $I$ & Definition \ref{sfi} \\
${\mathcal D}^{(i)}(y)$ & $i$-tuply $y$-densely divisible integers & Definition \ref{mdd-def}\\
$\FT_q(f)$ & Normalized Fourier transform of $f$ & \eqref{ftq-def} \\
\midrule
& Coefficient sequence at scale $N$ & Definition \ref{Coef}\\
& Siegel-Walfisz theorem & Definition \ref{Coef} \\
& (Shifted) smooth sequence at scale $N$ & Definition \ref{Coef} \\
\bottomrule
\end{tabular}
\end{center}
\label{tab:notation}
\end{table}

We will often use the following simple estimates for the divisor
function $\tau$ and its powers.

\begin{lemma}[Crude bounds on $\tau$]\label{divisor-crude}\ 
\begin{enumerate}[(i)]
\item\label{div-bound} (Divisor bound) One has
\begin{equation}\label{divisor-bound}
\tau(d) \lessapprox 1
\end{equation}
whenever $d$ is of polynomial size.  In particular, $d$ has $o(\log
x)$ distinct prime factors.
\item\label{taud-bound} One has
\begin{equation}\label{taud}
\sum_{d \leq y} \tau^C(d) \ll y \log^{O(1)} x 
\end{equation}
for any fixed $C>0$ and any $y > 1$ of polynomial size.  
\item\label{talc-bound} More generally, one has
\begin{equation}\label{talc}
\sum_{\substack{d \leq y\\ d = a\ (q)}} \tau^C(d) \ll \frac{y}{q} \tau^{O(1)}(q) \log^{O(1)} x + x^{o(1)}
\end{equation}
for any fixed $C>0$, any residue class $a\ (q)$ (not necessarily primitive), and any $y>1$ of polynomial size.
\end{enumerate}
\end{lemma}

\begin{proof}  For the divisor bound \eqref{divisor-bound}, see
  e.g. \cite[Theorem 2.11]{mv-book}.  For the bound \eqref{taud}, see
  e.g. \cite[Corollary 2.15]{mv-book}.  Finally, to prove the bound
  \eqref{talc}, observe using \eqref{divisor-bound} that we may factor
  out any common factor of $a$ and $q$, so that $a\ (q)$ is primitive.
  Next, we may assume that $q \leq y$, since the case  $q>y$ is
  trivial by \eqref{divisor-bound}.  The claim now follows from the Brun-Titchmarsh inequality for multiplicative functions (see \cite{shiu} or \cite{barban}).
\end{proof}

Note that we have similar bounds for the higher divisor functions
$$ \tau_k(n) := \sum_{d_1,\ldots,d_k: d_1 \ldots d_k = n} 1$$
for any fixed $k \geq 2$, thanks to the crude upper bound $\tau_k(n) \leq \tau(n)^{k-1}$.

The following elementary consequence of the divisor bound will also be useful:

\begin{lemma}\label{ram-avg} Let $q\geq 1$ be an integer.  Then for
  any $K \geq 1$ we have
$$ 
\sum_{1 \leq k \leq K} (k,q) \leq K\tau(q).
$$
In particular, if $q$ is of polynomial size, then we have
$$ 
\sum_{a \in \Z/q\Z} (a,q) \lessapprox q,
$$
and we also have
$$ 
\sum_{|k| \leq K} (k,q) \ll Kq^{\eps} + q
$$
for any fixed $\eps>0$ and arbitrary $q$ (not necessarily of polynomial size).
\end{lemma}

\begin{proof}  We have
$$
(k,q)\leq \sum_{d|(q,k)} d
$$
and hence
$$
\sum_{1 \leq k \leq K}(k,q) \leq \sum_{d|q} \sum_{\substack{1 \leq k
    \leq K\\d|k}} d \leq K \tau(q).
$$
\end{proof}


\section{Preliminaries}\label{sec-technical}

\subsection{Statements of results}

In this section we will give the most general statements that we
prove, and in particular define the concept of ``dense divisibility'',
which weakens the smoothness requirement of Theorem~\ref{th-main}.

\begin{definition}[Multiple dense divisibility]\label{mdd-def}  Let $y \geq 1$.  For each natural number $i \geq 0$, we define a notion of \emph{$i$-tuply $y$-dense divisibility} recursively as follows:
\begin{itemize}
\item[(i)] Every natural number $n$ is $0$-tuply $y$-densely divisible.
\item[(ii)] If $i \geq 1$ and $n$ is a natural number, we say that $n$
  is $i$-tuply $y$-densely divisible if, whenever $j,k \geq 0$ are
  natural numbers with $j+k=i-1$, and $1 \leq R \leq yn$, one can find
  a factorisation 
\begin{equation}\label{eq-dd-factor}
  n = qr,\quad\quad\text{ with } y^{-1}R\leq r\leq R
\end{equation}
such that $q$ is $j$-tuply $y$-densely
divisible and $r$ is $k$-tuply $y$-densely divisible.
\end{itemize}
We let $\Dcal{i}{y}$ denote the set of $i$-tuply $y$-densely divisible
numbers.  We abbreviate ``$1$-tuply densely divisible'' as ``densely
divisible'', ``$2$-tuply densely divisible'' as ``doubly densely
divisible'', and so forth; we also abbreviate $\Dcal{1}{y}$ as
$\Dcone{y}$, and since we will often consider squarefree densely
divisible integers with prime factors in an interval $I$, we will
denote
\begin{equation}\label{eq-ddiv-interval}
  \DI{I}{j}{y}=\Scal_I\cap \Dcal{j}{y}.
\end{equation}
\end{definition}

A number of basic properties of this notion will be proved at the beginning of
Section~\ref{dhl-mpz-2-proof}, but the intent is that we want to have
integers which can always be factored, in such a way that we can
control the location of the divisors. For instance, the following fact
is quite easy to check: any $y$-smooth integer is also $i$-tuply
$y$-densely divisible, for any $i\geq 0$ (see Lemma~\ref{fq} (iii) for
details).

\begin{definition}\label{sfi} For any set $I \subset \R$ (possibly depending on $x$), let $\Scal_I$ denote the set
  of all squarefree natural numbers whose prime factors lie in $I$.  If $I$ is also a bounded set (with the bound allowed to depend on $x$),
 we let $P_I$ denote the
  product of all the primes in $I$, thus in this case $\Scal_I$ is the set of divisors of $P_I$.
\end{definition}

For every fixed $0 < \varpi < \frac{1}{4}$ and $0 < \delta <
\frac{1}{4}+\varpi$ and every natural number $i$, we let
$\MPZ^{(i)}[\varpi,\delta]$ denote the following claim:

\begin{claim}[Modified Motohashi-Pintz-Zhang estimate,
  {$\MPZ^{(i)}[\varpi,\delta]$}]\label{mpzm-def} Let $I \subset \R$ be
  a bounded set, which may vary with $x$, and let $Q \lessapprox
  x^{1/2+2\varpi}$.  If $a$ is an integer coprime to $P_I$, and $A \geq 1$ is
  fixed, then
\begin{equation}\label{qq-mpz-mod}
  \sum_{\substack{q \leq Q\\ q \in \DI{I}{i}{x^\delta}}}
  |\Delta(\Lambda\onef_{[x,2x]}; a\ (q))| \ll x \log^{-A} x.
\end{equation}
\end{claim}

We will prove the following cases of these estimates:

\begin{theorem}[Motohashi-Pintz-Zhang type estimates]\label{mpz-est}\ 
\begin{itemize}
\item[(i)]  We have $\MPZ^{(4)}[\varpi,\delta]$ for any fixed $\varpi,\delta>0$ such that $600 \varpi + 180 \delta < 7$.
\item[(ii)] We can prove $\MPZ^{(2)}[\varpi,\delta]$ for any fixed
  $\varpi,\delta>0$ such that $168 \varpi + 48 \delta < 1$, without
  invoking any of Deligne's results~\cite{WeilI, WeilII} on the
  Riemann Hypothesis over finite fields.
\end{itemize}
\end{theorem}

The statement $\MPZ^{(i)}[\varpi,\delta]$ is easier to establish as
$i$ increases.  If true for some $i\geq 1$, it implies
$$
  \sum_{\substack{q \leq x^{1/2+2\varpi-\eps}\\q\text{
        $x^{\delta}$-smooth, squarefree}}}
  |\Delta(\Lambda\onef_{[x,2x]}; a\ (q))| \ll x \log^{-A} x
$$
for any $A\geq 1$ and $\eps>0$. Using a dyadic decomposition and the
Chinese Remainder Theorem, this shows that Theorem~\ref{mpz-est} (i)
implies Theorem~\ref{th-main}.

\subsection{Bilinear and trilinear estimates}

As explained, we will reduce Theorem~\ref{mpz-est} to bilinear or
trilinear estimates.  In order to state these precisely, we
introduction some further notation.

\begin{definition}[Coefficient sequences]\label{Coef}  A \emph{coefficient sequence} is a finitely supported sequence $\alpha \colon \N \to \R$ (which may depend on $x$) that obeys the bounds
\begin{equation}\label{alpha-bound}
|\alpha(n)| \ll \tau^{O(1)}(n) \log^{O(1)}(x)
\end{equation}
for all $n$ (recall that $\tau$ is the divisor function).
\begin{itemize}
\item[(i)]  A coefficient sequence $\alpha$ is said to be \emph{located at scale $N$} for some $N \geq 1$ if it is supported on an interval of the form $[cN, CN]$ for some $1 \ll c < C \ll 1$.
\item[(ii)] A coefficient sequence $\alpha$ located at scale $N$ for
  some $N \geq 1$ is said to \emph{obey the Siegel-Walfisz theorem},
  or to \emph{have the Siegel-Walfisz property}, if one has
\begin{equation}\label{sig}
| \Delta(\alpha \onef_{(\cdot,r)=1}; a\ (q)) | \ll \tau(qr)^{O(1)} N \log^{-A} x 
\end{equation}
for any $q,r \geq 1$, any fixed $A$, and any primitive residue class $a\ (q)$.  
\item[(iii)]  A coefficient sequence $\alpha$ is said to be
  \emph{shifted smooth at scale $N$} for some $N \geq 1$ if it has the form $\alpha(n) = \psi(\frac{n-x_0}{N})$ for some smooth function $\psi \colon \R \to \C$ supported on an interval $[c,C]$ for some fixed $0 < c < C$, and some real number $x_0$, with $\psi$ obeying the derivative bounds
\begin{equation}\label{soso}
|\psi^{(j)}(x)| \ll \log^{O(1)} x
\end{equation}
for all fixed $j \geq 0$, where the implied constant may depend on
$j$, and where $\psi^{(j)}$ denotes the $j^{\operatorname{th}}$ derivative of $\psi$.
If we can take $x_0=0$, we call $\alpha$ \emph{smooth at scale $N$};
note that such sequences are also located at scale $N$.
\end{itemize}
\end{definition}

Note that for a coefficient sequence $\alpha$ at scale $N$, an
integer $q\geq 1$ and a primitive residue class $a\ (q)$, we have the
trivial estimate
\begin{equation}\label{eq-delta-trivial}
\Delta(\alpha;a\ (q))\ll \frac{N}{\varphi(q)}(\log x)^{O(1)}.
\end{equation}
In particular, we see that the Siegel-Walfisz property amounts to a
requirement that the sequence $\alpha$ be uniformly equidistributed in
arithmetic progressions to moduli $q\ll (\log x)^A$ for any $A$. In
the most important arithmetic cases, it is established using the
methods from the classical theory of $L$-functions.

\begin{definition}[Type I,II,III estimates]\label{type-def}  Let $0 <
  \varpi < 1/4$, $0 < \delta < 1/4+\varpi$, and $0 < \sigma < 1/2$ be
  fixed quantities, and let $i \geq 1$ be a fixed natural number.  We
  let $I$ be an arbitrary bounded subset of $\R$ and define $P_I =
  \prod_{p \in I} p$ as before. Let $a\ (P_I)$ be a primitive
  congruence class.
\begin{itemize}
\item[(i)]  We say that $\TypeI^{(i)}[\varpi,\delta,\sigma]$ holds if, for any $I$ and $a\ (P_I)$ as above, any quantities $M,N \gg 1$ with
\begin{equation}\label{lin} M N \sim x 
\end{equation}
and
\begin{equation}\label{xss}
 x^{1/2-\sigma} \lessapprox N \lessapprox x^{1/2-2\varpi-c}
\end{equation}
for some fixed $c>0$, any $Q \lessapprox x^{1/2+2\varpi}$, and any
coefficient sequences $\alpha,\beta$ located at scales $M,N$
respectively, with $\beta$ having the Siegel-Walfisz property, we have
\begin{equation}\label{alphabet}
 \sum_{\substack{q\leq Q\\q \in \DI{I}{i}{x^\delta}}} |\Delta(\alpha \star \beta; a\ (q))|
\ll x \log^{-A} x
\end{equation}
for any fixed $A>0$.  (Recall the definition~(\ref{eq-ddiv-interval})
of the set $\DI{I}{i}{x^{\delta}}$.)
\item[(ii)]  We say that $\TypeII^{(i)}[\varpi,\delta]$ holds if, for any $I$ and $a\ (P_I)$ as above, any quantities $M,N \gg 1$ obeying \eqref{lin} and
\begin{equation}\label{xss-2}
 x^{1/2-2\varpi-c} \lessapprox N \lessapprox x^{1/2}
\end{equation}
for some sufficiently small fixed $c>0$, any $Q \lessapprox
x^{1/2+2\varpi}$, and any coefficient sequences $\alpha,\beta$ located
at scales $M,N$ respectively, with $\beta$ having the Siegel-Walfisz
property, we have \eqref{alphabet} for any fixed $A>0$.
\item[(iii)] We say that $\TypeIII^{(i)}[\varpi,\delta,\sigma]$ holds if, for
  any $I$ and $a\ (P_I)$ as above, for any quantities $M, N_1,N_2,N_3 \gg
  1$ which satisfy the conditions
\begin{gather}
 M N_1 N_2 N_3 \sim x
\nonumber\\
\label{n1-o}
N_1 N_2, N_1 N_3, N_2 N_3 \gtrapprox x^{1/2+\sigma}
\\
\label{n2-o}
x^{2\sigma} \lessapprox N_1,N_2,N_3 \lessapprox x^{1/2-\sigma},
\end{gather}
for any coefficient sequences $\alpha,\psi_1,\psi_2,\psi_3$ located at scales $M$, $N_1$, $N_2$, $N_3$, respectively, with $\psi_1,\psi_2,\psi_3$ smooth, and
finally for any $Q \lessapprox x^{1/2+2\varpi}$, we have
\begin{equation}\label{psi}
\sum_{\substack{q\leq Q\\q \in \DI{I}{i}{x^\delta}}}
  |\Delta(\alpha \star \psi_1 \star \psi_2 \star \psi_3; a\ (q))|\\
\ll x \log^{-A} x
\end{equation}
for any fixed $A>0$.
\end{itemize}
\end{definition}

Roughly speaking, Type I estimates control the distribution of
Dirichlet convolutions $\alpha,\beta$ where $\alpha,\beta$ are rough
coefficient sequences at moderately different scales, Type II
estimates control the distribution of Dirichlet convolutions
$\alpha,\beta$ where $\alpha,\beta$ are rough coefficient sequences at
almost the same scale, and Type III estimates control the distribution
of Dirichlet convolutions $\alpha\star\psi_1\star\psi_2\star\psi_3$ where
$\psi_1,\psi_2,\psi_3$ are smooth and $\alpha$ is rough but supported
at a fairly small scale.  

In Section~\ref{heath-brown-sec}, we will use the Heath-Brown identity
to reduce $\MPZ^{(i)}[\varpi,\delta]$ to a combination of
$\TypeI^{(i)}[\varpi,\delta,\sigma]$, $\TypeII^{(i)}[\varpi,\delta]$,
and $\TypeIII^{(i)}[\varpi,\delta,\sigma]$:

\begin{lemma}[Combinatorial lemma]\label{comlem} Let $i \geq 1$ be a
  fixed integer, and let $0 < \varpi < \frac{1}{4}$, $0 < \delta <
  \frac{1}{4}+\varpi$, and $\frac{1}{10} < \sigma < \frac{1}{2}$ be
  fixed quantities with $\sigma > 2\varpi$, such that the estimates
  $\TypeI^{(i)}[\varpi,\delta,\sigma]$,
  $\TypeII^{(i)}[\varpi,\delta]$, and
  $\TypeIII^{(i)}[\varpi,\delta,\sigma]$ all hold.  Then
  $\MPZ^{(i)}[\varpi,\delta]$ holds.

Furthermore, if $\sigma > 1/6$, then the hypothesis $\TypeIII^{(i)}[\varpi,\delta,\sigma]$ may be omitted.
\end{lemma}

As stated earlier, this lemma is a simple consequence of the
Heath-Brown identity, a dyadic decomposition (or more precisely, a
finer-than-dyadic decomposition), some standard analytic number theory
estimates (in particular, the Siegel-Walfisz theorem) and some
elementary combinatorial arguments.

In Zhang's work~\cite{zhang}, the claims
$\TypeI[\varpi,\delta,\sigma]$, $\TypeII[\varpi,\delta]$,$
\TypeIII[\varpi,\delta,\sigma]$ are (implicitly) proven with
$\varpi=\delta=1/1168$ and $\sigma = 1/8 - 8 \varpi$.  In fact, if one
optimizes the numerology in his arguments, one can derive
$\TypeI[\varpi,\delta,\sigma]$ whenever $44 \varpi + 12\delta + 8
\sigma < 1$, $\TypeII[\varpi,\delta]$ whenever $116\varpi+20\delta <
1$, and $\TypeIII[\varpi,\delta,\sigma]$ whenever $\sigma >
\frac{3}{26} + \frac{32}{13}\varpi + \frac{2}{13}\delta$ (see
\cite{pintz} for details).  We will obtain the following improvements
to these estimates, where the dependency with respect to $\sigma$ is
particularly important:

\begin{theorem}[New Type I, II, III estimates]\label{newtype} Let
  $\varpi,\delta,\sigma > 0$ be fixed quantities.
\begin{enumerate}[(i)]
\item\label{typeI1} If $54\varpi + 15\delta + 5 \sigma < 1$, then $\TypeI^{(1)}[\varpi,\delta,\sigma]$ holds.
\item\label{typeI2}  If $56\varpi + 16\delta + 4 \sigma < 1$, then $\TypeI^{(2)}[\varpi,\delta,\sigma]$ holds.
\item\label{typeI4} If $\frac{160}{3} \varpi + 16\delta + \frac{34}{9} \sigma < 1$ and $64\varpi + 18 \delta + 2\sigma < 1$, then $\TypeI^{(4)}[\varpi,\delta,\sigma]$ holds.
\item\label{typeII1}  If $68\varpi + 14 \delta < 1$, then $\TypeII^{(1)}[\varpi,\delta]$ holds.
\item\label{typeIII1} If $\sigma > \frac{1}{18}  + \frac{28}{9} \varpi + \frac{2}{9} \delta$ and $\varpi<1/12$, then $\TypeIII^{(1)}[\varpi,\delta,\sigma]$ holds.
\end{enumerate}
The proofs of the claims in \eqref{typeI4} and \eqref{typeIII1}
require Deligne's work on the Riemann Hypothesis over finite fields,
but the claims in \eqref{typeI1}, \eqref{typeI2} and \eqref{typeII1}
do not.
\end{theorem}

In proving these estimates, we will rely on the following general
``bilinear'' form of the Bombieri-Vinogradov theorem (the principle of
which is due to Gallagher~\cite{gallagher} and
Motohashi~\cite{motohashi}).

\begin{theorem}[Bombieri-Vinogradov theorem]\label{bvt}  Let $N,M \gg
  1$ be such that $NM \sim x$ and $N \geq x^\eps$ for some fixed
  $\eps>0$.  Let $\alpha,\beta$ be coefficient sequences at scale
  $M,N$ respectively such that $\beta$ has the Siegel-Walfisz property.  Then for any fixed $A>0$ there exists a fixed $B>0$ such that
$$ \sum_{q \leq x^{1/2} \log^{-B} x} \sup_{a \in (\Z/q\Z)^\times} 
|\Delta(\alpha\star\beta; a\ (q))| \ll x \log^{-A} x.$$
\end{theorem}

See \cite[Theorem 0]{bfi} for the proof. Besides the assumption of the
Siegel-Walfisz property, the other main ingredient used to establish Theorem \ref{bvt} is the large sieve
inequality for Dirichlet characters, from which the critical
limitation to moduli less than $x^{1/2}$ arises.

The Type I and Type II estimates in Theorem \ref{newtype} will be
proven in Section \ref{typei-ii-sec}, with the exception of the more
difficult Type I estimate \eqref{typeI4} which is proven in Section
\ref{typei-advanced-sec}.  The Type III estimate is established in
Section \ref{typeiii-sec}.  In practice, the estimate in Theorem
\ref{newtype}\eqref{typeI1} gives inferior results to that in Theorem
\ref{newtype}\eqref{typeI2}, but we include it here because it has a
slightly simpler proof. 

The proofs of these estimates involve essentially all the methods that
have been developed or exploited for the study of the distribution of
arithmetic functions in arithmetic progressions to large moduli, for
instance the dispersion method, completion of sums, the Weyl
differencing technique, and the $q$-van der Corput $A$ process. All
rely ultimately on some estimates of (incomplete) exponential sums
over finite fields, either one-dimensional or
higher-dimensional. These final estimates are derived from forms of
the Riemann Hypothesis over finite fields, either in the (easier) form
due to Weil~\cite{weilrh}, or in the much more general form due to
Deligne~\cite{WeilII}.


\subsection{Properties of dense divisibility}\label{dhl-mpz-2-proof}

We present the most important properties of the notion of multiple
dense divisibility, as defined in Definition \ref{mdd-def}.  Roughly
speaking, dense divisibility is a weaker form of smoothness which
guaranteees a plentiful supply of divisors of the given number in any
reasonable range, and multiple dense divisibility is a hereditary
version of this property which also partially extends to some factors
of the original number.


\begin{lemma}[Properties of dense divisibility]\label{fq} Let $i \geq
  0$ and $y \geq 1$.
\begin{itemize}
\item[(0)] If $n$ is $i$-tuply $y$-densely divisible, and $y_1\geq y$,
  then $n$ is $i$-tuply $y_1$-densely divisible. Furthermore, if
  $0\leq j\leq i$, then $n$ is $j$-tuply $y$-densely divisible.
\item[(i)] If $n$ is $i$-tuply $y$-densely divisible, and $m$ is a
  divisor of $n$, then $m$ is $i$-tuply $y (n/m)$-densely
  divisible. Similarly, if $l$ is a multiple of $n$, then $l$ is
  $i$-tuply $y (l/n)$-densely divisible.
\item[(ii)] If $m,n$ are $y$-densely divisible, then $[m,n]$ is also $y$-densely divisible.
\item[(iii)] Any $y$-smooth number is $i$-tuply $y$-densely divisible.
\item[(iv)]  If $n$ is $z$-smooth and squarefree for some $z \geq y$,
  and 
\begin{equation}\label{eq-hyp-iv}
\prod_{\substack{p|n\\p \leq y}} p \geq \frac{z^i}{y},
\end{equation}
then $n$ is $i$-tuply $y$-densely divisible.  
\end{itemize}
\end{lemma}

\begin{proof}  
  We abbreviate ``$i$-tuply $y$-divisible'' in this proof by the
  shorthand ``$(i,y)$-d.d.''.
\par
(0) These monotony properties are immediate from the definition.
\par
Before we prove the other properties, we make the following
remark: in checking that an integer $n$ is $(i,y)$-d.d., it suffices
to consider parameters $R$ with $1\leq R\leq n$ when looking for
factorizations of the form~(\ref{eq-dd-factor}): indeed, if $n<R\leq yn$, the
factorization $n=qr$ with $r=n$ and $q=1$ satisfies the condition
$y^{-1}R\leq r\leq R$, and $r=n$ is $(j,y)$-d.d.\ (resp.\ $q=1$ is
$(k,y)$-d.d.) whenever $j+k=i-1$. We will use this reduction in (i),
(ii), (iii), (iv) below.
\par
(i) We prove the first part by induction on $i$. For $i=0$, the
statement is obvious since every integer is $(0,y)$-d.d.\ for every
$y\geq 1$. Now assume the property holds for $j$-tuply dense
divisibility for $j<i$, let $n$ be $(i,y)$-d.d., and let $m\mid n$
be a divisor of $n$. We proceed to prove that $m$ is $(i,ym_1)$-d.d.
\par
We write $n=mm_1$. Let $R$ be such that $1\leq R\leq m$, and let $j$,
$k\geq 0$ be integers with $j+k=i-1$. Since $R\leq n$, and $n$ is
$(i,y)$-d.d., there exists by definition a factorization $n=qr$ where
$q$ is $(j,y)$-d.d., $r$ is $(k,y)$-d.d., and $y/R\leq r\leq y$. Now we
write $m_1=n_1n'_1$ were $n_1=(r,m_1)$ is the gcd of $r$ and $m_1$. We
have then a factorization $m=q_1r_1$ where
$$
q_1=\frac{q}{n'_1},\quad\quad r_1=\frac{r}{n_1},
$$
and we check that this factorization satisfies the condition required
for checking that $m$ is $(i,ym_1)$-d.d. First, we have
$$
\frac{R}{ym_1}\leq \frac{r}{m_1}\leq \frac{r}{n_1}=r_1\leq R,
$$
so the divisor $r_1$ is well-located. Next, by induction applied to
the divisor $r_1=r/n_1$ of the $(k,y)$-d.d.\ integer $r$, this integer
is $(k,yn_1)$-d.d., and hence by (0), it is also
$(k,ym_1)$-d.d. Similarly, $q_1$ is $(j,yn'_1)$-d.d., and hence also
$(j,ym_1)$-d.d. This finishes the proof that $m$ is $(i,ym_1)$-d.d.
\par
The second part of (i) is similar and left to the reader.
\par
To prove (ii), recall that $y$-densely divisible means $(1,y)$-densely
divisible. We may assume that $m\leq n$. Denote $a=[m,n]n^{-1}$. Now
let $R$ be such that $1\leq R\leq [m,n]$. If $R\leq n$, then a
factorization $n=qr$ with $Ry^{-1}\leq r\leq R$, which exists since
$n$ is $y$-d.d., gives the factorization $[m,n]=aqr$, which has the
well-located divisor $r$. If $n<R\leq [m,n]$, we get
$$
1\leq \frac{n}{a}\leq \frac{R}{a}\leq n
$$
and therefore there exists a factorization $n=qr$ with $R(ay)^{-1}\leq
r\leq Ra^{-1}$. Then $[m,n]=q(ar)$ with $Ry^{-1}\leq ar\leq R$. Thus
we see that $[m,n]$ is $y$-d.d.
\par
We now prove (iii) by induction on $i$. The case $i=0$ is again
obvious, so we assume that (iii) holds for $j$-tuply dense
divisibility for $j<i$. Let $n$ be a $y$-smooth integer, let $j$,
$k\geq 0$ satisfy $j+k=i-1$, and let $1\leq R\leq n$ be given. Let $r$
be the largest divisor of $n$ which is $\leq R$, and let
$q=n/r$. Since all prime divisors of $n$ are $\leq y$, we have
$$
Ry^{-1}\leq r\leq R,
$$
and furthermore both $q$ and $r$ are $y$-smooth. By the induction hypothesis, $q$ is
$(j,y)$-d.d.\ and $r$ is $(k,y)$-d.d., hence it follows that $n$ is
$(i,y)$-d.d.
\par
We now turn to (iv).  The claim is again obvious for $i=0$.  Assume
then that $i=1$.  Let $R$ be such that $1 \leq R \leq n$.
Let 
$$
s_1=\prod_{\substack{p\mid n\\p\leq y}}p,\quad\quad
r_1=\prod_{\substack{p|n\\ p>y}} p.
$$ 
Assume first that $r_1\leq R$. Since $n/r_1=s_1$ is $y$-smooth, it is
$1$-d.d., and since $1\leq Rr_1^{-1}\leq s_1$, we can factor
$s_1=q_2r_2$ with $R(r_1y)^{-1}\leq r_2\leq Rr_1^{-1}$. Then
$n=q_2(r_1r_2)$ with
$$
Ry^{-1}\leq r_1r_2\leq R.
$$
So assume that $r_1>R$. Since $n$ and hence $r_1$ are $z$-smooth, we
can factor $r_1=r_2q_2$ with $Rz^{-1}\leq r_2\leq R$.
Let $r_3$ be the smallest divisor of $s_1$ such that $r_3 r_2\geq
Ry^{-1}$, which exists because $s_1r_2\geq zy^{-1}r_2\geq Ry^{-1}$ by
the assumption~(\ref{eq-hyp-iv}). Since $s_1$ is $y$-smooth, we have
$r_3 r_2\leq R$ (since otherwise we must have $r_3\neq 1$, hence $r_3$
is divisible by a prime $p\leq y$, and $r_3 p^{-1}$ is a smaller
divisor with the required property $r_3 p^{-1}r_2>Ry^{-1}$, contradicting the minimality of $r_3$). Therefore
$n=q(r_3r_2)$ with 
$$
\frac{R}{y}\leq r_3 r_2\leq R,
$$
as desired.
\par
Finally we consider the $i>1$ case.  We assume, by induction, that
(iv) holds for integers $j<i$. Let $j,k \geq 0$ be such that
$j+k=i-1$.  By assumption, using the notation $r_1$, $s_1$ as above,
we have
$$
s_1\geq z^iy^{-1}=z^j\cdot z^k \cdot \frac{z}{y}.
$$
We can therefore write $s_1=n_1n_2n_3$ where
\begin{equation}\label{eq-induct}
z^j y^{-1} \leq n_1 \leq z^j,\quad\quad
z^k y^{-1} \leq n_2 \leq z^k
\end{equation}
and thus
$$ n_3 \geq \frac{z}{y}.
$$
Now we divide into several cases in order to find a suitable
factorization of $n$.  Suppose first that $n_1 \leq R \leq
n/n_2$. Then 
$$
1 \leq \frac{R}{n_1} \leq \frac{n}{n_1 n_2}
$$ 
and the integer $n/(n_1n_2)=r_1n_3$ satisfies the assumptions of (iv)
for $i=1$. Thus, by the previous case, we can find a factorization
$r_1n_3= q'r'$ with $y^{-1} \frac{R}{n_1} \leq r' \leq
\frac{R}{n_1}$. We set $r = n_1 r'$ and $q = n_2 q'$, and observe that
by~(\ref{eq-induct}), $r$ (resp. $q$) satisfies the assumption of (iv)
for $i=j$ (resp. $i=k$). By induction, the factorization $n=qr$ has
the required property.

Next, we assume that $R < n_1$.  Since $n_1$ is $y$-smooth, we can
find a divisor $r$ of $n_1$ such that $y^{-1} R \leq r \leq R$. Then
$q=n/r$ is a multiple of $n_2$, and therefore it satisfies
$$
\prod_{\substack{p|q\\p \leq y}} p \geq n_2\geq z^k y^{-1}.
$$
By induction, it follows that $q$ is $(k,y)$-d.d. Since $r$ is
$y$-smooth, $q$ is also $(j,y)$-d.d.\ by (iii), and hence the factorization $n=qr$
is suitable in this case.

Finally, suppose that $R > n/n_2$, i.e., that $nR^{-1}<n_2$.  We then
find a factor $q$ of the $y$-smooth integer $n_2$ such that
$n(Ry)^{-1} \leq q \leq nR^{-1}$. Then the complementary factor $r =
n/q$ is a multiple of $n_1$, and therefore it satisfies
$$
\prod_{\substack{p|r\\ p \leq y}} p \geq z^j y^{-1},
$$
so that $r$ is $(j,y)$-d.d.\ by induction, and since $q$ is also $(j,y)$-d.d.\ by
(iii), we also have the required factorization in this case.
\end{proof}


\section{Applying the Heath-Brown identity}\label{heath-brown-sec}

The goal of this and the next sections is to prove the assumption
$\MPZ^{(i)}[\varpi,\delta]$ (Definition~\ref{mpzm-def}) for as wide a
range of $\varpi$ and $\delta$ as possible, following the outline in
Section~\ref{ssec-outline}. The first step, which we implement in this
section, is the proof of Lemma \ref{comlem}.  We follow standard
arguments, particularly those in \cite{zhang}. The main tool is the
Heath-Brown identity, which is combined with a purely combinatorial
result about finite sets of non-negative numbers.  We begin with the
latter statement:

\begin{lemma}\label{sum}  
  Let $1/10 < \sigma < 1/2$, and let $t_1,\ldots,t_n$ be non-negative
  real numbers such that $t_1+\cdots+t_n=1$.  Then at least one of the
  following three statements holds:
\begin{itemize}
\item[(Type 0)] There is a $t_i$ with $t_i \geq 1/2 + \sigma$.
\item[(Type I/II)] There is a partition $\{1,\ldots,n\} = S \cup T$ such that
$$ \frac{1}{2} - \sigma < \sum_{i \in S} t_i \leq \sum_{i \in T} t_i < \frac{1}{2} + \sigma.$$
\item[(Type III)] There exist distinct $i,j,k$ with $2\sigma \leq t_i
  \leq t_j \leq t_k \leq 1/2-\sigma$ and 
\begin{equation}\label{eq-type-iii-last}
t_i+t_j,\ t_i+t_k,\ t_j+t_k \geq \frac{1}{2} + \sigma.
\end{equation}
\end{itemize}
Furthermore, if $\sigma > 1/6$, then the Type III alternative cannot
occur.
\end{lemma}

\begin{proof}  We dispose of the final claim first: if $\sigma>1/6$, then $2\sigma > 1/2 - \sigma$, and so the inequalities $2\sigma \leq t_i \leq t_j \leq t_k \leq 1/2-\sigma$ of the Type III alternative are inconsistent.  

Now we prove the main claim.  Let $\sigma$ and $(t_1,\ldots,t_n)$ be
as in the statement. We assume that the Type 0 and Type I/II
statements are false, and will deduce that the Type III statement
holds.
\par
From the failure of the Type 0 conclusion, we know that 
\begin{equation}\label{ti-bound}
t_i < \frac{1}{2} + \sigma
\end{equation}
for all $i=1,\ldots,n$.  From the failure of the Type I/II conclusion,
we also know that, for any $S \subset \{1,\ldots,n\}$, we have
$$
\sum_{i \in S} t_i\notin
\Bigl(\frac{1}{2}-\sigma,\frac{1}{2}+\sigma\Bigr)
$$
since otherwise we would obtain the conclusion of Type I/II by taking
$T$ to be the complement of $S$, possibly after swapping the role of
$S$ and $T$.
\par
We say that a set $S \subset \{1,\ldots,n\}$ is \emph{large} if
$\sum_{i \in S} t_i \geq \frac{1}{2}+\sigma$, and that it is
\emph{small} if $\sum_{i \in S} t_i \leq \frac{1}{2} - \sigma$. Thus,
the previous observation shows that every set $S \subset
\{1,\ldots,n\}$ is either large or small, and also (from
\eqref{ti-bound}) that singletons are small, as is the empty set.
Also, it is immediate that the complement of a large set is small, and
conversely (since $t_1+\cdots+t_n=1$).
\par
Further, we say that an element $i \in \{1,\ldots,n\}$ is
\emph{powerful} if there exists a small set $S \subset \{1,\ldots,n\}
\backslash \{i\}$ such that $S \cup \{i\}$ is large, i.e., if $i$ can
be used to turn a small set into a large set. Then we say that an
element $i$ is \emph{powerless} if it is not powerful. Thus, adding or
removing a powerless element from a set $S$ cannot alter its smallness
or largeness, and in particular, the union of a small set and a set of
powerless elements is small.

We claim that there exist exactly three powerful elements. First,
there must be at least two, because if $P$ is the set of powerless
elements, then it is small, and hence its complement is large, and
thus contains at least two elements, which are powerful. But picking
one of these powerful $i$, the set $\{i\}\cup P$ is small, and
therefore its complement also has at least two elements, which
together with $i$ are three powerful elements.

Now, we observe that if $i$ is powerful, then $t_i\geq 2\sigma$, since
the gap between a large sum $\sum_{j \in S \cup \{i\}} t_j$ and a
small sum $\sum_{j \in S} t_j$ is at least $2\sigma$. In particular,
if $i\neq j$ are two powerful numbers, then
$$
t_i+t_j\geq 4\sigma>\frac{1}{2}-\sigma,
$$
where the second inequality holds because of the assumption
$\sigma>1/10$. Thus the set $\{i,j\}$ is not small, and is therefore
large. But then if  $\{i,j,k,l\}$ was a set of four powerful elements,
it would follow that
$$
1=t_1+\cdots+t_n\geq (t_i+t_j)+(t_k+t_l)\geq
2\Bigl(\frac{1}{2}+\sigma\Bigr)>1,
$$
a contradiction.

Let therefore $i$, $j$, $k$ be the three powerful elements. We may
order them so that $t_i\leq t_j\leq t_k$. We have
$$
2\sigma\leq t_i\leq t_j\leq t_k\leq \frac{1}{2}-\sigma
$$
by~(\ref{ti-bound}) and the previous argument, which also shows that
$\{i,j\}$, $\{i,k\}$ and $\{j,k\}$ are large, which
is~(\ref{eq-type-iii-last}).
\end{proof}

\begin{remark}\label{type-iv} For $1/10 < \sigma \leq 1/6$, the Type
  III case can indeed occur, as can be seen by considering the
  examples $(t_1,t_2,t_3) = (2\sigma, 1/2-\sigma, 1/2-\sigma)$.  The
  lemma may be extended to the range $1/14 < \sigma < 1/2$, but at the
  cost of adding two additional cases (corresponding to the case of
  four or five powerful elements respectively):
\begin{itemize}
\item[(Type IV)] There exist distinct $i,j,k,l$ with $2\sigma \leq t_i \leq t_j \leq t_k \leq t_l \leq 1/2-\sigma$ and $t_i+t_l \geq 1/2 + \sigma$.
\item[(Type V)] There exist distinct $i,j,k,l,m$ with $2\sigma \leq t_i \leq t_j \leq t_k \leq t_l \leq t_m \leq 1/2-\sigma$ and $t_i+t_j+t_k \geq 1/2 + \sigma$.
\end{itemize}
We leave the verification of this extension to the reader. Again, for
$1/14 < \sigma \leq 1/10$, the Type IV and Type V cases can indeed
occur, as can be seen by considering the examples $(t_1,t_2,t_3,t_4) =
(2\sigma, 2\sigma, 1/2-3\sigma, 1/2-\sigma)$ and
$(t_1,t_2,t_3,t_4,t_5) = (2\sigma,2\sigma,2\sigma,2\sigma,1-8\sigma)$.
With this extension, it is possible to extend Lemma \ref{comlem} to
the regime $1/14 < \sigma < 1/2$, but at the cost of requiring
additional ``Type IV'' and ``Type V'' estimates as hypotheses.
Unfortunately, while the methods in this paper do seem to be able to
establish some Type IV estimates, they do not seem to give enough Type
V estimates to make it profitable to try to take $\sigma$ below
$1/10$. 
\end{remark}

To apply Lemma \ref{sum} to distribution theorems concerning the von
Mangoldt function $\Lambda$, we recall the Heath-Brown identity
(see~\cite{hb-ident} or~\cite[Prop. 13.3]{ik}).

\begin{lemma}[Heath-Brown identity]\label{hbi}  For any $K \ge 1$, we have the identity
\begin{equation}\label{heath}
  \Lambda =\sum_{j=1}^K (-1)^{j-1} \binom{K}{j} \mu_{\leq}^{\star j} 
  \star \onef^{\star (j-1)} \star L
\end{equation}
on the interval $[x,2x]$, where $\onef$ is the constant function
$\onef(n) := 1$, $L$ is the logarithm function $L(n) := \log n$,
$\mu_{\leq}$ is the truncated M\"obius function
$$ \mu_{\leq}(n) := \mu(n) \onef_{n \leq (2x)^{1/K}},$$
and where we denote $f^{\star j} = f \star \ldots \star f$ the
$j$-fold Dirichlet convolution of an arithmetic function $f$, i.e.,
$$
f^{\star j}(n):=\multsum_{a_1\cdots a_j=n}{f(a_1)\cdots f(a_j)}.
$$
\end{lemma}

\begin{proof} Write $\mu = \mu_{\leq} + \mu_>$, where $\mu_>(n) :=
  \mu(n) \onef_{n > (2x)^{1/K}}$.  Clearly the convolution
$$ \mu_>^{\star K} \star \onef^{\star K-1} \star L$$
vanishes on $[1,2x]$.  Expanding out $\mu_> = \mu - \mu_\leq$ and
using the binomial formula, we conclude that
\begin{equation}\label{mang}
  0 = \sum_{j=0}^K (-1)^{j} \binom{K}{j} \mu^{\star (K-j)} \star 
  \mu_{\leq}^{\star j} \star \onef^{\star (K-1)} \star L
\end{equation}
on $[x,2x]$.  Since Dirichlet convolution is associative, the standard
identities $\Lambda = \mu \star L$ and $\delta = \mu \star \onef$
(where the Kronecker delta function $\delta(n) := \onef_{n=1}$ is the
unit for Dirichlet convolution) show that the $j=0$ term of
\eqref{mang} is
$$
\mu^{\star K}\star \onef^{\star (K-1)}\star L=\mu\star L=\Lambda.
$$
For all the other terms, we can use commutativity of Dirichlet
convolution and (again) $\mu\star \onef=\delta$ to write
$$
\mu^{\star K-j} \star \mu_{\leq}^{\star j} \star \onef^{\star
  K-1} \star L=
 \mu_{\leq}^{\star j} \star \onef^{\star (j-1)} \star L,
$$
so that we get \eqref{heath}.
\end{proof}

We will now prove Lemma \ref{comlem}, which the reader is invited to
review.  Let $i,\varpi,\delta,\sigma$ satisfy the hypotheses of that
lemma, and let $A_0 > 0$ be fixed. By the definition of
$\MPZ^{(i)}(\varpi,\delta)$, which is the conclusion of the lemma, it
suffices to show that for any $Q \lessapprox x^{1/2+2\varpi}$, any
bounded set $I\subset (0+\infty)$ and any residue class $a\ (P_I)$, we
have
\begin{equation}\label{show}
  \sum_{q \in {\mathcal Q}} |\Delta(\Lambda \onef_{[x,2x]}; a\ (q))| 
  \ll x \log^{-A_0+O(1)} x,
\end{equation}
where 
\begin{equation}\label{eq-Q}
{\mathcal Q} := \{ q \leq Q: q \in \DI{I}{i}{x^\delta} \}
\end{equation}
(recalling the definition~(\ref{eq-ddiv-interval})) and the $O(1)$
term in the exponent is independent of $A_0$.

Let $K$ be any fixed integer with 
\begin{equation}\label{1k}
\frac{1}{K} < 2\sigma
\end{equation} 
(e.g. one can take $K=10$).  We apply Lemma \ref{hbi} with this value of $K$.  By the triangle inequality, it suffices to show that
\begin{equation}\label{sss}
 \sum_{q \in {\mathcal Q}} |\Delta((\mu_{\leq}^{\star j} \star 1^{\star j-1} \star L) \onef_{[x,2x]}; a\ (q))| \ll x \log^{-A_0/2+O(1)} x
\end{equation}
for each $1 \leq j \leq K$, which we now fix.

The next step is a finer-than-dyadic decomposition (a standard idea
going back at least to Fouvry \cite{fouvry} and Fouvry-Iwaniec
\cite{fi-2}).  We denote $\Theta:=1+\log^{-A_0}x$.  Let $\psi \colon
\R \to \R$ be a smooth function supported on $[-\Theta,\Theta]$ that
is equal to $1$ on $[-1,1]$
and obeys the derivative estimates
$$ |\psi^{(m)}(x)| \ll \log^{mA_0} x$$
for $x\in\R$ and any fixed $m \geq 0$, where the implied constant
depends only on $m$.  We then have a smooth partition of unity
$$ 1 = \sum_{N \in {\mathcal D}} \psi_N(n)$$
indexed by the multiplicative semigroup 
$$ {\mathcal D} := \{ \Theta^m\colon  m \in \N \cup \{0\} \}$$
for any natural number $n$, where
$$ \psi_N(n) := \psi\Bigl( \frac{n}{N}\Bigr) - \psi\Bigl( \frac{\Theta
  n}{N}\Bigr)
$$ 
is supported in $[\Theta^{-1}N,\Theta N]$.  We thus have decompositions
$$
1 = \sum_{N \in {\mathcal D}} \psi_N,\quad\quad \mu_{\leq} = \sum_{N
  \in {\mathcal D}} \psi_N \mu_{\leq}, \quad\quad L = \sum_{N \in
  {\mathcal D}} \psi_N L.
$$
For $1 \leq j \leq K$, we have
\begin{multline*}
  (\mu_{\leq}^{\star j} \star 1^{\star (j-1)} \star L) \onef_{[x,2x]}
  = \multsum_{N_1,\ldots,N_{2j} \in {\mathcal D}} \{(\psi_{N_1} \mu_{\leq}
  ) \star \cdots \star (\psi_{N_j} \mu_{\leq} ) \\
\quad\quad\quad\quad\quad\quad\quad\quad
  \star \psi_{N_{j+1}} \star \cdots \star \psi_{N_{2j-1}} \star
  \psi_{N_{2j}} L\} \onef_{[x,2x]}
\\   =\multsum_{N_1,\ldots,N_{2j} \in {\mathcal D}} \log(N_{2j})
  \{(\psi_{N_1} \mu_{\leq}) \star \cdots \star (\psi_{N_j} \mu_{\leq})\\
  \star \psi_{N_{j+1}} \star \cdots \star \psi_{N_{2j-1}} \star
  \psi'_{N_{2j}}\} \onef_{[x,2x]}
\end{multline*}
where $\psi'_N := \psi_N \frac{L}{\log N}$ is a simple variant of
$\psi_N$.

For each $N_1, \ldots, N_{2j}$, the summand in this formula vanishes
unless
\begin{equation}\label{naj}
  N_1,\ldots,N_j \ll x^{1/K}
\end{equation}
and
$$
\frac{x}{\Theta^{2K}}
\leq N_1 \cdots N_{2j} \leq 2x\Theta^{2K}.
$$
In particular, it vanishes unless
\begin{equation}\label{nudge}
  x\Bigl(1-O\Bigl(\frac{1}{\log^{A_0}x}\Bigr)\Bigr)
  \leq N_1 \cdots N_{2j} \leq 
  2x\Bigl(1+O\Bigl(\frac{1}{\log^{A_0}x }\Bigr)\Bigr).
\end{equation}
We conclude that there are at most 
\begin{equation}\label{eq-nb-summands}
  \ll \log^{2j(A_0+1)}x
\end{equation}
tuples $(N_1,\ldots, N_{2j})\in\mathcal{D}^{2j}$ for which the summand is
non-zero. Let $\mathcal{E}$ be the set of these tuples. We then
consider the arithmetic function
\begin{multline}
  \alpha= \multsum_{(N_1,\ldots, N_{2j})\in\mathcal{E}} \log(N_{2j})
  \{(\psi_{N_1} \mu_{\leq} ) \star \cdots \star ( \psi_{N_j} \mu_{\leq})
  \star \psi_{N_{j+1}} \star \cdots \star \psi_{N_{2j-1}}
  \star \psi'_{N_{2j}}\}\\
  - (\mu_{\leq}^{\star j} \star 1^{\star j-1} \star L) \onef_{[x,2x]}.
\end{multline}
Note that the cutoff $\onef_{[x,2x]}$ is only placed on the second term in the definition of $\alpha$, and is not present in the first term.
\par
By the previous remarks, this arithmetic function is supported on
$$
[x(1-O(\log^{-A_0}x)),x]\cup [2x,2x(1+O(\log^{-A_0} x))]
$$
and using the divisor bound and trivial estimates, it satisfies
$$
\alpha(n)\ll \tau(n)^{O(1)}(\log n)^{O(1)},
$$
where the exponents are bounded independently of $A_0$. In particular,
we deduce from Lemma \ref{divisor-crude} that
$$
\Delta(\alpha; a\ (q))\ll x\log^{-A_0+O(1)} x
$$
for all $q\geq 1$. Using the estimate~(\ref{eq-nb-summands}) for the
number of summands in $\mathcal{E}$, we see that, in order to
prove~(\ref{sss}), it suffices to show that
\begin{equation}\label{tap}
\sum_{q\in\mathcal{Q}}|\Delta(\alpha_1\star\cdots\star \alpha_{2j}; a\
(q))|\ll x\log^{-A}x
\end{equation}
for $A>0$ arbitrary, where each $\alpha_i$ is an arithmetic function
of the form $\psi_{N_i} \mu_{\leq }$, $\psi_{N_i}$ or $\psi'_{N_i}$,
where $(N_1,\ldots, N_{2j})$ satisfy~(\ref{naj}) and~(\ref{nudge}).

We now establish some basic properties of the arithmetic functions
$\alpha_k$ that may occur. For a subset $S\subset \{1\ldots, 2j\}$, we
will denote  by
$$
\alpha_S:=\underset{k\in S}{\bigstar}{\alpha_k}
$$
the convolution of the $\alpha_k$ for $k\in S$.

\begin{lemma}\label{facts}  Let $1 \leq k \leq 2j$ and $S \subset
  \{1,\ldots,2j\}$. The following facts hold:
\begin{itemize}
\item[(i)] Each $\alpha_k$ is a coefficient sequence located at scale
  $N_k$, and more generally, the convolution $\alpha_S$ is a
  coefficient sequence located at scale $\prod_{k \in S} N_k$.
\item[(ii)] If $N_k \gg x^{2\sigma}$, then $\alpha_k$ is smooth at
  scale $N_k$.
\item[(iii)] If $N_k \gg x^\eps$ for some fixed $\eps>0$, then
  $\alpha_k$ satisfies the Siegel-Walfisz property.  More generally,
  $\alpha_S$ satisfies the Siegel-Walfisz property if $\prod_{k \in S}
  N_k \gg x^\eps$ for some fixed $\eps>0$.
\item[(iv)]  $N_1 \ldots N_{2j} \sim x$.
\end{itemize}
\end{lemma}

\begin{proof} 
  The first part of (i) is clear from construction.  For the second
  part of (i), we use the easily verified fact that if $\alpha,\beta$
  are coefficient sequences located at scales $N,M$ respectively, then
  $\alpha\star \beta$ is a coefficient sequence located at scale $NM$.

  For (ii), we observe that since $2\sigma >K^{-1}$, the condition
  $N_k\gg x^{2\sigma}$ can only occur for $k > j$ in view
  of~(\ref{naj}), so that $\alpha_k$ takes the form $\psi_{N_k}$ or
  $\psi'_{N_k}$, and the smoothness then follows directly from the
  definitions.

  For (iii), the Siegel-Walfisz property for $\alpha_k$ when $k \leq j$
  follows from the Siegel-Walfisz theorem for the M\"obius function
  and for Dirichlet characters (see e.g. \cite[Satz 4]{siebert}
  or~\cite[Th. 5.29]{ik}), using summation by parts to handle the
  smooth cutoff, and we omit the details.  For $k > j$, $\alpha_k$ is
  smooth, and the Siegel-Walfisz property for $\alpha_k$ follows from
  the Poisson summation formula (and the rapid decay of the Fourier
  transform of smooth, compactly supported functions; compare with the
  arguments at the end of this section for the Type 0 case).

  To handle the general case, it therefore suffices to check that if
  $\alpha,\beta$ are coefficient sequences located at scales $N,M$
  respectively with $x^\eps \ll M \ll x^C$ for some fixed $\eps,C>0$,
  and $\beta$ satisfies the Siegel-Walfisz property, then so does
  $\alpha\star \beta$. This is again relatively standard, but we give
  the proof for completeness.

  By Definition \ref{Coef}, our task is to show that
$$
| \Delta((\alpha \star \beta) \onef_{(\cdot,q)=1}; a\ (r)) | \ll
\tau(qr)^{O(1)} N \log^{-A} x
$$
for any $q,r \geq 1$, any fixed $A$, and any primitive residue class
$a\ (r)$.  We replace $\alpha,\beta$ by their restriction to integers
coprime to $qr$ (without indicating this in the notation), which allows
us to remove the constraint $\onef_{(n,q)=1}$. We may also assume that
$r = O( \log^{A+O(1)} x)$, since the desired estimate follows from the
trivial estimate~(\ref{eq-delta-trivial}) for the discrepancy
otherwise.

For any integer $n$, we have
$$ 
\sum_{n = a\ (r)} (\alpha \star \beta)(n) = \sum_{b\in
  (\Z/r\Z)^\times} \Bigl(\sum_{d = b\ (r)} \alpha(d)\Bigr)
\Bigl(\sum_{m = \bar{b}a\ (r)} \beta(m)\Bigr)
$$
and
\begin{align*}
  \sum_{n} (\alpha \star \beta)(n) &= \left(\sum_d \alpha(d)\right) \left(\sum_m \beta(m)\right)\\
  &= \sum_{b\in (\Z/r\Z)^\times} \Bigl(\sum_{d = b\ (r)}
  \alpha(d)\Bigr) \Bigl(\sum_m \beta(m)\Bigr)
\end{align*}
so that
$$
|\Delta(\alpha \star \beta, a\ (r))|\leq \sum_{b\in (\Z/r\Z)^\times}
\Bigl|\sum_{d = b\ (r)} \alpha(d)\Bigr| \ |\Delta(\beta; \bar{b}a\
(r))|.
$$ 
From \eqref{talc} (and Definition~(\ref{Coef})), we have
$$ 
\sum_{d = b\ (r)} \alpha(d) \ll \frac{N}{r} \tau(r)^{O(1)} \log^{O(1)}
x + N^{o(1)}
$$
for any $b\ (r)$, and since $\beta$ has the Siegel-Walfisz property, we
have
$$ 
| \Delta(\beta; \bar{b}a\ (r))| \ll \tau(r)^{O(1)} M \log^{-B} x
$$
for any $b\ (r)$ and any fixed $B>0$. Thus
\begin{align*}
  |\Delta(\alpha \star \beta, a\ (r))| &\ll \tau(r)^{O(1)} \varphi(r)
  \left( \frac{N}{r} + N^{o(1)} \right) M \log^{-B + O(1)} x
  \\
  &\ll \tau(r)^{O(1)} MN \log^{-B + O(1)} x
\end{align*}
by the assumption concerning the size of $r$.

Finally, the claim (iv) follows from \eqref{nudge}.  
\end{proof}

We can now conclude this section by showing how the assumptions
$\TypeI^{(i)}[\varpi,\delta,\sigma]$, $\TypeII^{(i)}[\varpi,\delta]$
and $\TypeIII^{(i)}[\varpi,\delta,\sigma]$ of Lemma~\ref{comlem} imply
the estimates \eqref{tap}.

Let therefore $(\alpha_1,\ldots, \alpha_{2j})$ be given with the
condition of~(\ref{tap}). By Lemma \ref{facts}(iv), we can write
$N_k \sim x^{t_k}$ for $k=1,\ldots,2j$, where the $t_k$ are
non-negative reals (not necessarily fixed) that sum to $1$.  By Lemma
\ref{sum}, the $t_i$ satisfy one of the three conclusions (Type 0),
(Type I/II), (Type III) of that lemma. We deal with each in turn.  The
first case can be dealt with directly, while the others require one of
the assumptions of Lemma~\ref{comlem}, and we begin with these.

Suppose that we are in the Type I/II case, with the partition
$\{1,\ldots, 2j\}=S\cup T$ given by the combinatorial lemma. We have
$$
\alpha_1\star\cdots\star \alpha_{2j}=\alpha_S\star \alpha_T.
$$
By Lemma \ref{facts}, $\alpha_S, \alpha_T$ are coefficient
sequences located at scales $N_S, N_T$ respectively, where
$$ 
N_S N_T \sim x,
$$ 
and (by (iii)) $\alpha_S$ and $\alpha_T$ satisfy the Siegel-Walfisz
property.  By Lemma~\ref{sum}, we also have
$$ 
x^{1/2-\sigma} \ll N_S \ll N_T \ll x^{1/2+\sigma}.
$$
Thus, directly from Definition \ref{type-def} and~(\ref{eq-Q}),
the required estimate \eqref{tap} follows either from the hypothesis
$\TypeI^{(i)}[\varpi,\delta,\sigma]$ (if one has $N_S \leq
x^{1/2-2\varpi-c}$ for some sufficiently small fixed $c>0$) or from
$\TypeII^{(i)}[\varpi,\delta]$ (if $N_S > x^{1/2-2\varpi-c}$, for the
same value of $c$).

Similarly, in the Type III case, comparing Lemmas \ref{facts}
and~\ref{sum} with Definition~\ref{type-def} and~(\ref{eq-Q}) shows
that \eqref{sss} is a direct translation of
$\TypeIII^{(i)}[\varpi,\delta,\sigma]$.

It remains to prove~(\ref{sss}) in the Type 0 case, and we can do this
directly.  In this case, there exists some $k \in \{ 1, \ldots, 2j \}$, such
that $t_k\geq 1/2+\sigma>2\sigma$. Intuitively, this means that
$\alpha_k$ is smooth (by Lemma~\ref{facts} (ii)) and has a long
support, so that it is very well-distributed in arithmetic
progressions to relatively large moduli, and we can just treat the
remaining $\alpha_j$ trivially.

Precisely, we write
$$
\alpha_1 \star \ldots \star \alpha_{2j}=\alpha_k \star
\alpha_S
$$
where $S=\{1,\ldots,2j\}\backslash \{k\}$. By Lemma~\ref{facts},
$\alpha_k$ is a coefficient sequence which is smooth at a scale $N_k
\gg x^{1/2+\sigma}$, and $\alpha_S$ is a coefficient sequence which is
located at a scale $N_S$ with $N_k N_S \sim x$.  We argue as in Lemma
\ref{facts}(iii):
we have
$$ 
\Delta(\alpha_k \star \alpha_S; a\ (q)) = \sum_{m\in
  (\Z/q\Z)^{\times}} \sum_{\ell = m\ (q)} \alpha_S(\ell) \Delta(\alpha_k; \bar{m} a\ (q))
$$
and since
$$ \sum_m |\alpha_S(m)| \lessapprox  N_S,$$
(by \eqref{taud} and Definition \ref{Coef}), we get
\begin{equation}\label{bosh}
  \sum_{q\in\mathcal{Q}}|\Delta(\alpha_1\star\cdots\star\alpha_{2j};a\ (q))|
  \lessapprox N_S \sum_{q \leq Q} 
  \sup_{b \in (\Z/q\Z)^\times} |\Delta(\alpha_k;b\ (q))|.
 \end{equation}
Since $\alpha_k$ is smooth at scale $N_k$, we can write
$$ \alpha_k(n) = \psi( n / N_k ) $$
for some smooth function $\psi \colon \R \to \R$ supported on an
interval of size $\ll 1$, which satisfies the estimates
$$
|\psi^{(j)}(t)| \lessapprox 1
$$
for all $t$ and all fixed $j \geq 0$.  By the Poisson summation
formula, we have
$$
\sum_{n = b\ (q)} \alpha_k(n) = \frac{N_k}{q} \sum_{m \in \Z} e_q(mb)
\hat \psi\Bigl( \frac{m N_k}{q}\Bigr)= \frac{N_k}{q}\hat{\psi}(0)+
\frac{N_k}{q} \sum_{m \neq 0}e_q(mb) \hat \psi\Bigl( \frac{m
  N_k}{q}\Bigr),
$$ 
for $q\geq 1$ and $b\ (q)$, where
$$
\hat \psi(s) := \int_\R \psi(t) e(-ts)\ dt
$$ 
is the Fourier transform of $\psi$.  From the smoothness and support
of $\psi$, we get the bound
$$
\Bigl|\hat\psi\Bigl(\frac{m N_k}{q}\Bigr)\Bigr| \lessapprox
\Bigl(\frac{mN_k}{q}\Bigr)^{-2}
$$ 
for $m \neq 0$ and $q\leq Q$, and thus we derive
$$ 
\sum_{n = b\ (q)} \alpha_k(n)= \frac{N_k}{q}\hat{\psi}(0)
+O\Bigl(\frac{N_k}{q} (N_k/q)^{-2}\Bigr).
$$ 
Since by definition
$$
\Delta(\alpha_k;b\ (q))|=
\sum_{n = b\ (q)} \alpha_k(n)
-\frac{1}{\varphi(q)}\sum_{c\in(\Z/q\Z)^{\times}} \sum_{n=c\
  (q)}\alpha_k(n),
$$
we get
$$ 
\Delta(\alpha_k;b\ (q))| \lessapprox \frac{N_k}{q} (N_k/q)^{-2}.
$$
Therefore, from~(\ref{bosh}), we have
\begin{align*}
  \sum_{q\in\mathcal{Q}} |\Delta(\alpha_1\star\cdots\star\alpha_{2j};a\ (q))| &
  \lessapprox N_SN_k\Bigl(\frac{Q}{N_k}\Bigr)^2\ll x^{1-2\sigma+4\varpi},
\end{align*}
and since $\sigma>2\varpi$ (by assumption in Lemma~\ref{comlem}), this
implies~(\ref{tap}), which concludes the proof  of Lemma~\ref{comlem}.


\begin{remark}\label{vaughan-rem} In the case $\sigma > 1/6$, one can
  replace the Heath-Brown identity of Lemma \ref{hbi} with other
  decompositions of the von Mangoldt function $\Lambda$, and in
  particular with the well-known \emph{Vaughan identity}
$$ \Lambda_\geq = \mu_< \star  L - \mu_< \star  \Lambda_< \star  1 + \mu_\geq \star  \Lambda_\geq \star  1$$
from \cite{vaughan}, where
\begin{gather}
\Lambda_\geq(n) := \Lambda(n) \onef_{n \geq V},\quad\quad 
\Lambda_<(n) := \Lambda(n) \onef_{n < V}\\
\mu_\geq(n) := \mu(n) \onef_{n \geq U},\quad\quad \mu_<(n) := \mu(n)
\onef_{n < U},
\end{gather}
where $U,V > 1$ are arbitrary parameters.  Setting $U=V=x^{1/3}$, we
then see that to show \eqref{show}, it suffices to establish the
bounds
\begin{align}
\sum_{q \in {\mathcal Q}} |\Delta((\mu_< \star  L) \onef_{[x,2x]}; a\ (q))| &\ll x \log^{-A_0/2+O(1)} x \label{v-1} \\
\sum_{q \in {\mathcal Q}} |\Delta((\mu_< \star  \Lambda_< \star  1) \onef_{[x,2x]}; a\ (q))| &\ll x \log^{-A_0/2+O(1)} x \label{v-2} \\
\sum_{q \in {\mathcal Q}} |\Delta((\mu_\geq \star  \Lambda_\geq \star  1) \onef_{[x,2x]}; a\ (q))| &\ll x \log^{-A_0/2+O(1)} x \label{v-3}.
\end{align}
To prove \eqref{v-1}, we may perform dyadic decomposition on $\mu_<$
and $L$, much as in the previous arguments.  The components of $L$
which give a non-trivial contribution to \eqref{v-1} will be located
at scales $\gg x^{2/3}$. One can then use the results of the Type
0 analysis above.  In order to prove \eqref{v-3}, we
similarly decompose the $\mu_\geq, \Lambda_\geq$, and $1$ factors and
observe that the resulting components of $\mu_\geq$ and $\Lambda_\geq
\star 1$ that give a non-trivial contribution to \eqref{v-3} will be
located at scales $M,N$ with $x^{1/3} \ll M,N \ll x^{2/3}$ and $MN
\sim x$, and one can then argue using Type I and Type II estimates as
before since $\sigma > 1/6$.  Finally, for \eqref{v-2}, we decompose
$\mu_< \star \Lambda_<$, and $1$ into components at scales $M,N$
respectively with $M \ll x^{2/3}$ and $MN \sim x$, so $N \gg x^{1/3}$.
If $N \gg x^{2/3}$, then the Type 0 analysis applies again, and
otherwise we may use the Type I and Type II estimates with $\sigma >
1/6$.
\end{remark}

\begin{remark} An inspection of the arguments shows that the interval
  $[x,2x]$ used in Lemma~\ref{comlem} may be replaced by a more general
  interval $[x_1,x_2]$ for any $x \leq x_1 \leq x_2 \leq 2x$, leading
  to a slight generalization of the conclusion
  $\MPZ^{(i)}[\varpi,\delta]$.  By telescoping series, one may then
  generalize the intervals $[x_1,x_2]$ further, to the range $1 \leq
  x_1 \leq x_2 \leq 2x$.
\end{remark}

In the next sections, we will turn our attention to the task of
proving distribution estimates of Type I, II and III.  All three turn
out to be intimately related to estimates for exponential sums over
$\Z/q\Z$, either ``complete'' sums over all of $\Z/q\Z$ or
``incomplete'' sums over suitable subsets, such as reductions modulo
$q$ of intervals or arithmetic progressions (this link goes back to
the earliest works in proving distribution estimates beyond the range
of the large sieve). In the next section, we consider the basic theory
of the simplest of those sums, where the essential results go back to
Weil's theory of exponential sums in one variable over finite
fields. These are enough to handle basic Type I and II estimates,
which we consider next. On the other hand, for Type III estimates and
the most refined Type I estimates, we require the much deeper results
and insights of Deligne's second proof of the Riemann Hypothesis for
algebraic varieties over finite fields.


\section{One-dimensional exponential sums}\label{exp-sec}

The results of this section are very general and are applicable to
many problems in analytic number theory. Since the account we provide
might well be useful as a general reference beyond the applications to
the main results of this paper, we will not use the asymptotic
convention of Definition~\ref{asym}, but provide explicit estimates
that can easily be quoted in other contexts. (In particular, we will
sometimes introduce variables named $x$ in our notation.)

\subsection{Preliminaries}\label{ssec-exp-prelim}

We begin by setting up some notation and conventions. We recall from
Section \ref{notation-sec} that we defined $e_q(a)=e^{2i\pi a/q}$ for
$a\in\Z$ and $q\geq 1$. This is a group homomorphism $\Z\rightarrow
\C^{\times}$, and since $q\Z\subset \ker e_q$, it induces naturally a
homomorphism, which we also denote $e_q$, from $\Z/q\Z$ to
$\C^{\times}$. In fact, for any multiple $qr$ of $q$, we can also view
$e_q$ as a homomorphism $\Z/qr\Z\ra \C^{\times}$.
\par
It is convenient for us (and compatible with the more algebraic theory
for multi-variable exponential sums discussed in
Section~\ref{deligne-sec}) to extend further $e_q$ to the projective
line $\P^1(\Z/q\Z)$ by extending it by zero to the point(s) at
infinity. Precisely, recall that $\P^1(\Z/q\Z)$ is the quotient of
$$
X_q=\{(a,b)\in (\Z/q\Z)^{2}\colon a\text{ and } b\text{ have no common
  factor}\},
$$
(where a common factor of $a$ and $b$ is a prime $p\mid q$ such that
$a$ and $b$ are zero modulo $p$) by the equivalence relation
$$
(a,b)=(ax,bx)
$$
for all $x\in(\Z/q\Z)^{\times}$. We identify $\Z/q\Z$ with a subset
of $\P^1(\Z/q\Z)$ by sending $x$ to the class of $(x,1)$. We note that
$$
|\P^1(\Z/q\Z)|=q\prod_{p\mid q}{\Bigl(1+\frac{1}{p}\Bigr)},
$$
and that a point $(a,b)\in \P^1(\Z/q\Z)$ belongs to $\Z/q\Z$ if and
only if $b\in (\Z/q\Z)^{\times}$, in which case $(a,b)=(ab^{-1},1)$.
\par
Thus, we can extend $e_q$ to $\P^1(\Z/q\Z)$ by defining
$$
e_q((a,b))=e_q(ab^{-1})
$$
if $b\in (\Z/q\Z)^{\times}$, and $e_q((a,b))=0$ otherwise. 
\par
We have well-defined reduction maps $\P^1(\Z/qr\Z)\ra \P^1(\Z/q\Z)$
for all integers $r\geq 1$, as well as $\P^1(\Q)\ra \P^1(\Z/q\Z)$, and
we can therefore also naturally define $e_q(x)$ for $x\in
\P^1(\Z/qr\Z)$ or for $x\in \P^1(\Q)$ (for the map $\P^1(\Q)\ra
\P^1(\Z/q\Z)$, we use the fact that any $x\in \P^1(\Q)$ is the class
of $(a,b)$ where $a$ and $b$ are coprime integers, so that $(a\ (q),b\
(q))\in X_q$).
\par
We will use these extensions especially in the following context: let
$P$, $Q\in\Z[X]$ be polynomials, with $Q$ non-zero, and consider the
rational function $f=\frac{P}{Q}\in \Q(X)$. This defines a map $\P^1(\Q)\ra
\P^1(\Q)$, and then by reduction modulo $q$, a map
$$
f\ (q)\colon \P^1(\Z/q\Z)\ra \P^1(\Z/q\Z).
$$
We can therefore consider the function $x\mapsto e_q(f(x))$ for
$x\in\Z/q\Z$. If $x\in \Z$ is such that $Q(x)$ is coprime to $q$, then
this is just $e_q(P(x)\overline{Q(x)})$. If $Q(x)$ is not coprime to
$q$, on the other hand, one must be a bit careful. If $q$ is prime,
then one should write $f\ (q)=P_1/Q_1$ with $P_1$, $Q_1\in(\Z/q\Z)[X]$
coprime, and then $e_q(f(x))=e_q(P_1(x)\overline{Q_1(x)})$ if
$Q_1(x)\neq 0$, while $e_q(f(x))=0$ otherwise. If $q$ is squarefree,
one combines the prime components according to the Chinese Remainder
Theorem, as we will recall later.

\begin{example}
Let $P=X$, $Q=X+3$ and $q=3$, and set $f := \frac{P}{Q}$. Then, although $P\ (q)$ and $Q\ (q)$
both take the value $0$ at $x=0\in\Z/q\Z$, we have $e_q(f(0))=1$.
\end{example}

In rare cases (in particular the proof of Proposition~\ref{lode} in
Section~\ref{sec-final-exp}) we will use one more convention:
quantities 
$$
e_p\Bigl(\frac{a}{b}\Bigr)
$$
may arise where $a$ and $b$ are integers that depend on other
parameters, and with $b$ allowed to be divisible by $p$.
However, this will only happen when the formula is to be interpreted as
$$
e_p\Bigl(\frac{a}{b}\Bigr) =\psi\Bigl(\frac{1}{b}\Bigr)=\psi(\infty)
$$
where $\psi(x)=e_p(ax)$ defines an additive character of $\Fp$. Thus
we use the convention
$$
e_p\Bigl(\frac{a}{b}\Bigr)=\begin{cases}
0&\text{ if } a\neq 0\ (p), b = 0\ (p)\\
1&\text{ if } a=0\ (p), b = 0\ (p),
\end{cases}
$$
since in the second case, we are evaluating the trivial character at
$\infty$. 
\par

\subsection{Complete exponential sums over a finite field}

As is well known since early works of Davenport and Hasse in
particular, the Riemann Hypothesis for curves over finite fields
(proven by Weil~\cite{weilrh}) implies bounds with ``square root
cancellation'' for one-dimensional exponential sums over finite
fields.  A special case is the following general bound:

\begin{lemma}[One-variable exponential sums with additive
  characters]\label{prime-exp} Let $P, Q \in
  \Z[X]$ be polynomials over $\Z$ in one indeterminate $X$. Let $p$ be
  a prime number such that $Q\ (p) \in \F_p[X]$ is non-zero, such that
  there is no identity of the form
\begin{equation}\label{Degenerate}
  \frac{P}{Q}\ (p) = g^p - g + c
\end{equation}
in $\F_p(X)$ for some rational function $g = g(X)\in\Fp(X)$ and some
$c \in \F_p$.  Then we have
\begin{equation}\label{pqp}
  \left|\sum_{x \in \F_p} e_p\left(\frac{P(x)}{Q(x)}\right)\right| \ll \sqrt{p}
\end{equation}
where the implicit constant depends only on $\max(\deg P, \deg Q)$,
and this dependency is linear.
\end{lemma}

Note that, by our definitions, we have
$$
\sum_{x \in \F_p} e_p\Bigl(\frac{P(x)}{Q(x)}\Bigr) = \sum_{\substack{x \in
    \F_p\\Q_1(x) \neq 0}} e_p(P_1(x) \overline{Q_1(x)}),
$$ 
where $P/Q\ (p)=P_1/Q_1$ with $P_1$, $Q_1\in \F_p[X]$ coprime
polynomials. 
\par
As key examples of Lemma \ref{prime-exp}, we record Weil's bound for
Kloosterman sums, namely
\begin{equation}\label{kloost}
\left|\sum_{x \in \F_p} e_p\left( ax + \frac{b}{x} \right)\right| \ll \sqrt{p}
\end{equation}
when $a,b \in \F_p$ are not both zero, as well as the variant
\begin{equation}\label{kloost-2}
  \left|\sum_{x \in \F_p} e_p\left( ax + \frac{b}{x} + 
      \frac{c}{x+l} +\frac{d}{x+m}+\frac{e}{x+l+m}\right)\right| \ll \sqrt{p}
\end{equation}
for $a,b,c,d,e,l,m \in \F_p$ with $b,c,d,e,l,m,l+m$ non-zero.  In
fact, these two estimates are almost the only two cases of Lemma
\ref{prime-exp} that are needed in our arguments. In both cases, one
can determine a suitable implied constant, e.g., the Kloosterman sum
in~(\ref{kloost}) has modulus at most $2\sqrt{p}$.

We note also that the case \eqref{Degenerate} must be excluded, since
$g^p(x)-g(x)+c=c$ for all $x\in\F_p$, and therefore the corresponding
character sum has size equal to $p$.

\begin{proof} 
  This estimate follows from the Riemann Hypothesis for the algebraic
  curve $C$ over $\F_p$ defined by the Artin-Schreier equation
$$
y^p-y = P(x)/Q(x).
$$
This was first explicitly stated by Perelmuter in \cite{perelmuter},
although this was undoubtedly known to Weil; an elementary proof based
on Stepanov's method may also be found in \cite{cochrane}. A full
proof for all curves using a minimal amount of theory of algebraic
curves is found in~\cite{bombieri-stepanov}.
\end{proof}

\begin{remark} For our purpose of establishing some non-trivial Type I
  and Type II estimates for a given choice of $\sigma$ (and in
  particular for $\sigma$ slightly above $1/6$) and for sufficiently small $\varpi,\delta$, it is not necessary
  to have the full square root cancellation in \eqref{pqp}, and any
  power savings of the form $p^{1-c}$ for some fixed $c>0$ would
  suffice (with the same dependency on $P$ and $Q$); indeed, one obtains a non-trivial estimate for a given value of $\sigma$ once one invokes the $q$-van der Corput method a sufficient number of times, and once the gains from the cancellation are greater than the losses coming from completion of sums, with the latter becoming negligible in the limit $\varpi,\delta \to 0$.  Such a power
  saving (with $c=1/4$) was obtained for the Kloosterman sum
  \eqref{kloost} by Kloosterman \cite{Kloosterman} using an elementary
  dilation argument (see also \cite{mordell} for a generalization),
  but this argument does not appear to be available for estimates such
  as \eqref{kloost-2}.
\end{remark}

In order to prove parts \eqref{typeI1}, \eqref{typeI2} and
\eqref{typeII1} of Theorem \ref{newtype}, we need to extend the bounds
of Lemma~\ref{prime-exp} in two ways: to sums over $\Z/q\Z$ for $q$
squarefree instead of prime, and to incomplete sums over suitable
subsets of $\Z/q\Z$ (the other two parts of the theorem also require
exponential sum estimates, but these require the much deeper work of
Deligne \cite{WeilII}, and will be considered in
Section~\ref{deligne-sec}).

\subsection{Complete exponential sums to squarefree moduli}

To extend Lemma \ref{prime-exp} to squarefree moduli, we first need
some preliminaries. We begin with a version of the Chinese Remainder
Theorem.

\begin{lemma}[Chinese Remainder Theorem]\label{crt}  If $q_1,q_2$ are
  coprime natural numbers, then for any integer $a$, or indeed for any
  $a\in \P^1(\Q)$, we have
\begin{equation}\label{eq-crt-easy}
e_{q_1 q_2}(a) = e_{q_1}\left(\frac{a}{q_2}\right) e_{q_2}\left(
  \frac{a}{q_1} \right).
\end{equation}
More generally, if $q_1,\ldots,q_k$ are pairwise coprime natural
numbers, then for any integer $a$ or any $a\in\P^1(\Q)$, we have
$$ 
e_{q_1 \cdots q_k}(a) = \prod_{i=1}^k e_{q_i}\left(\frac{a}{\prod_{j
      \neq i} q_j}\right).
$$
\end{lemma}

\begin{proof} It suffices to prove the former claim for
  $a\in\P^1(\Q)$, as the latter then follows by induction. 

If $a$ maps to a point at infinity in $\P^1(\Z/q_1q_2\Z)$, then it
must map to a point at infinity in $\P^1(\Z/q_1\Z)$ or
$\P^1(\Z/q_2\Z)$, so that both sides of~(\ref{eq-crt-easy}) are zero.

So we can assume that $a\in \Z/q_1q_2\Z$. Let $\overline{q_1},
\overline{q_2}$ be integers such that $q_1 \overline{q_1} = 1\ (q_2)$
and $q_2 \overline{q_2} = 1\ (q_1)$, respectively. Then we have $q_1
\overline{q_1} + q_2 \overline{q_2} = 1\ (q_1 q_2)$, and hence
$$ 
e_{q_1 q_2}(a) = e_{q_1q_2}(a(q_1\overline{q_1}+q_2\overline{q_2}))=
e_{q_1 q_2}( q_1 \overline{q_1} a ) e_{q_1 q_2}( q_2 \overline{q_2} a
).
$$ 
Since $e_{q_1 q_2}( q_1 \overline{q_1} a ) = e_{q_2}(\frac{a}{q_1})$
and $e_{q_1 q_2}( q_2 \overline{q_2} a ) = e_{q_1}(\frac{a}{q_2})$,
the claim follows.
\end{proof}



If $q\in\Z$ is an integer, we say that $q$ divides $f$, and write
$q|f$, if $q$ divides $f$ in $\Z[X]$.  We denote by $(q,f)$ the
largest factor of $q$ that divides $f$ (i.e., the positive generator
of the ideal of $\Z$ consisting of integers dividing $f$).  Thus for instance $(q,0)=q$. We also write $f\ (q) \in (\Z/q\Z)[X]$ for the reduction of $f$ modulo $q$.

We need the following algebraic lemma, which can be viewed as a
version of (a special case of) the fundamental theorem of calculus:

\begin{lemma}\label{fund} Let $f =
  \frac{P}{Q}\in \Q(X)$ with $P,Q \in \Z[X]$ coprime, and
  let $q$ be a natural number such that $Q\ (p)$ is a non-zero
  polynomial for all primes $p\mid q$ (automatic if $Q$ is monic).
\par
\begin{enumerate}[(i)]
\item  If $q|f'$ and all prime factors of $q$ are sufficiently
  large depending on the degrees of $P$ and $Q$, then there exists $c
  \in \Z/q\Z$ such that $q|f-c$. 
\item  If $q$ is squarefree, if $Q\ (p)$ has degree $\deg(Q)$ for
all $p\mid q$ and\footnote{We adopt the convention $\deg(0)=-\infty$.} $\deg(P)<\deg(Q)$, and if all prime factors of $q$
are sufficiently large depending on the degrees of $P$ and $Q$, then
$(q,f')$ divides $(q,f)$. In particular, if $(q,f)=1$ then $(q,f')=1$.
\end{enumerate}
\end{lemma}

\begin{proof}  
  We first prove (i). By the Chinese Remainder Theorem, we may assume that $q=p^j$ is
  the power of a prime.  Write $f'=P_1/Q_1$ where $P_1$ and
  $Q_1\in\Z[X]$ are coprime. By definition, the condition $q\mid f'$
  implies that $P_1(x)=0\ (q)$ for all $x\in \Z/q\Z$. On the other
  hand, since $Q_1\ (p)$ is non-zero in $\Z/p\Z[X]$, the rational
  function $f'\ (q)$ is well-defined at all $x\in \Z/q\Z$ except at
  most $\deg(Q)$ zeros of $Q_1$, and takes the value $0$ at all these
  $\geq q-\deg(Q)$ values. If $q$ is large enough in terms of
  $\deg(P)$ and $\deg(Q)$, this implies that $f'\ (q)=0\in \Z/q\Z[X]$,
  and therefore that $f\ (q)=c$ for some $c\in\Z/q\Z$, i.e., that
  $q\mid f-c$.
\par
Now we prove (ii). If a prime $p$ divides $(q,f')$, then by (i), there exists $c\in
\Z/p\Z$ such that $p\mid f-c$. If $p\nmid (q,f)$, we must have
$c \neq 0$. But then $p\mid P-cQ$, where $P-cQ\ (p)\in\Z/p\Z[X]$ is (by
assumption) a polynomial of degree $\deg(Q)\geq 1$. For $p>\deg(Q)$,
this is a contradiction, so that $p\mid (q,f)$.
\end{proof}

We use this to give an estimate for complete exponential sums, which
combines the bounds for  Ramanujan sums with those arising from the
Riemann Hypothesis for curves.

\begin{proposition}[Ramanujan-Weil bounds]\label{weil} Let $q$ be a
  squarefree natural number, and let $f = \frac{P}{Q}\in \Q(X)$, where
  $P,Q \in \Z[X]$ are coprime polynomials with $Q$
  non-zero modulo $p$ for every $p\mid q$, for instance $Q$
  monic. Then we have
$$ 
\Bigl|\sum_{n \in \Z/q\Z} e_q( f(n) )\Bigr| \leq C^{\Omega(q)}q^{1/2}
\frac{(f',q)}{(f'',q)^{1/2}}
$$
for some constant $C\geq 1$ depending only on $\deg(P)$ and $\deg(Q)$.
\end{proposition}

\begin{example}
  (1) Let $f(X):=b/X$ for some integer $b$. We get, after changing the
  summation variable, a slightly weaker version of the familiar
  Ramanujan sum bound
\begin{equation}\label{ram1}
  \Bigl|\sum_{n \in \Z/q\Z} e( b n) \onef_{(n,q)=1}\Bigr| \leq (b,q)
\end{equation}
since $(q,f')=(b,q)$ and $(q,f'') = c (b,q)$ in this case for some $c=1,2$.
\par
(2) More generally, let $f:=a/X+bX$ for some integers $a,b$. We get a
weaker form of Weil's bound for Kloosterman sums
$$ 
\Bigl|\sum_{n \in \Z/q\Z} e_q( a \overline{n} + bn )
\onef_{(n,q)=1}\Bigr| \leq 2^{\Omega(q)} q^{1/2}
\frac{(a,b,q)}{(a,q)^{1/2}},
$$ 
which generalizes \eqref{kloost}.
\end{example}

\begin{proof} By Lemma \ref{crt}, we can factor the sum as a product
  of exponential sums over the prime divisors of $q$:
$$ 
\sum_{n \in \Z/q\Z} e_q( f(n) ) = \prod_{p|q} \sum_{n \in \Z/p\Z}
e_p\left( \frac{f(n)}{(q/p)} \right).
$$
Since, for each $p\mid q$, the constant $q/p$ is an invertible element
in $\Z/p\Z$, we see that it suffices to prove the estimates
\begin{align}
\sum_{n \in \Z/p\Z} e_p(f(n))&\ll p,\quad\text{when  $p|f'$ (which implies $p|f''$)}\\
\label{ram}
\sum_{n \in \Z/p\Z} e_p(f(n))& \ll 1,\quad\text{when $p|f''$ but $p \notdivides f'$}\\
\label{epf}
\sum_{n \in \Z/p\Z} e_p(f(n)) &\ll \sqrt{p},\quad\text{otherwise,}
\end{align}
where the implied constants, in all three cases, depend only on
$\deg(P)$ and $\deg(Q)$. Thus we may always assume that $p\mid q$ is
large enough in terms of $\deg(P)$ and $\deg(Q)$, since otherwise the
result is trivial.

The first bound is clear, with implied constant equal to $1$.
For~(\ref{ram}), since $p|f''$, we conclude from Lemma \ref{fund}
(since $p$ is large enough) that there exists $c \in \Z/p\Z$ such that
$p|f'-c$. Since $p \notdivides f'$, we see that $c$ must be non-zero.
Then, since $f'-c = (f-ct)'$, another application of Lemma \ref{fund}
shows that there exists $d \in \Z/p\Z$ such that $p|f-ct-d$.  This
implies that $f(n) = cn+d\ (p)$ whenever $n$ is not a pole of $f\
(p)$. The denominator $Q$ of $f$ (which is non-zero modulo $p$ by
assumption) has at most $\deg(Q)$ zeroes, and therefore we see that
$e_p(f(n)) = e_p(cn+d)$ for all but $\leq \deg(Q)$ values of $n \in
\Z/p\Z$.  Thus (by orthogonality of characters) we get
$$ 
\Bigl|\sum_{n \in \Z/p\Z} e_p(f(n))\Bigr|=\Bigl|\sum_{n \in \Z/p\Z}
e_p(f(n))-\sum_{n \in\Z/p\Z} e_p( cn+d)\Bigr|\leq \deg(Q).
$$

Now we prove \eqref{epf}. This estimate follows immediately from Lemma
\ref{prime-exp}, except if the reduction $\tilde{f}\in\F_p(X)$ of $f$
modulo $p$ satisfies an identity
\begin{equation}\label{tptq}
  \tilde{f} = g^p - g + c
\end{equation}
for some $g \in \F_p(X)$ and $c \in \F_p$. We claim that, if $p$ is
large enough, this can only happen if $p\mid f'$, which contradicts
the assumption of~(\ref{epf}) and therefore concludes the proof.

To prove the claim, we just observe that if~(\ref{tptq}) holds, then
any pole of $g$ would be a pole of $\tilde{f}$ of order $p$, and thus
$g$ must be a polynomial if $p$ is large enough. But then~(\ref{tptq})
implies that $\tilde{f}-c$ either vanishes or has degree at least $p$.  If $p$ is large enough, the latter conclusion is not possible, and thus $p\mid f'$.
\end{proof}

We also need a variant of the above proposition, which is a slight
refinement of an estimate appearing in the proof of \cite[Proposition
11]{zhang}:

\begin{lemma}\label{dork} Let $d_1,d_2$ be squarefree
  integers, so that $[d_1,d_2]$ is squarefree, and let
  $c_1,c_2,l_1,l_2$ be integers.  Then there exists $C\geq 1$ such
  that
$$ 
\Bigl|\sum_{n \in \Z/[d_1,d_2]\Z} e_{d_1}\left( \frac{c_1}{n+l_1}
\right) e_{d_2}\left( \frac{c_2}{n+l_2} \right) \Bigr|\leq
C^{\Omega([d_1,d_2])} (c_1,\delta_1) (c_2,\delta_2) (d_1,d_2)
$$ 
where $\delta_i := d_i / (d_1,d_2)$ for $i=1,2$.
\end{lemma}

\begin{proof} As in the proof of Proposition \ref{weil},
  we may apply Lemma \ref{crt} to reduce to the case where
  $[d_1,d_2]=p$ is a prime number.  The bound is then trivial if
  $(c_1,\delta_1)$, $(c_2, \delta_2)$, or $(d_1,d_2)$ is equal to $p$, so we may
  assume without loss of generality that $d_1 = p$, $d_2 = 1$, and
  that $c_1$ is coprime to $p$.  We then need to prove that
$$ 
\sum_{n \in \Z/p\Z} e_p\left( \frac{c_1}{n+l} \right)
 \ll 1,
$$
but this is clear since, by the change of variable $m = c_1 / (n+l)$,
this sum is just a Ramanujan sum.
\end{proof}

\subsection{Incomplete exponential sums}\label{incomplete-exp-sec}

The bounds in the previous section control ``complete'' additive
exponential sums in one variable in $\Z/q\Z$, by which we mean sums
where the variable $n$ ranges over all of $\Z/q\Z$.  For our
applications, as well as for many others, one needs also to have good
estimates for ``incomplete'' versions of the sums, in which the
variable $n$ ranges over an interval, or more generally over the
integers weighted by a coefficient sequence which is (shifted) smooth
at some scale $N$.

The most basic technique to obtain such estimates is the method of
completion of sums, also called the P\'olya-Vinogradov method. In
essence, this is an elementary application of discrete Fourier
analysis, but the importance of the results cannot be overestimated.
\par
We begin with some facts about the discrete Fourier transform. Given a
function 
$$
f\colon\Z/q\Z\ra\C
$$ 
we define its \emph{normalized Fourier transform} $\FT_q (f)$ to be
the function on $\Z/q\Z$ given by
\begin{equation}\label{ftq-def}
\FT_q (f)(h):=\frac{1}{q^{1/2}}\sum_{x\in\Zz/q\Zz}f(x)e_q(hx).
\end{equation}
The normalization factor ${1}/{q^{1/2}}$ is convenient
because the resulting Fourier transform operator is then unitary with
respect to the inner product 
$$
\langle f,g\rangle:=\sum_{x\in \Z/q\Z}f(x)\overline{g(x)}
$$
on the space of functions $\Z/q\Z\ra \C$. In other words, the
Plancherel formula
$$
\sum_{x\in\Zz/q\Zz}f(x)\ov{g(x)}=\sum_{h\in\Zz/q\Zz}\FT_q(f)(h)\ov{\FT_q(g)(h)}
$$
holds for any functions $f$, $g\colon \Z/q\Z\ra \C$. Furthermore, by
the orthogonality of additive characters, we have the discrete Fourier
inversion formula
$$
\FT_q(\FT_q (f))(x)=f(-x)
$$
for all $x\in\Z/q\Z$.

\begin{lemma}[Completion of sums]\label{com} 
  Let $M\geq 1$ be a real number and let $\psi_M$ be a function on
  $\R$ defined by
$$
\psi_M(x) = \psi\left(\frac{x-x_0}{M}\right)
$$
where $x_0\in \R$ and $\psi$ is a smooth function supported on $[c,C]$
satisfying 
$$
|\psi^{(j)}(x)| \ll \log^{O(1)} M
$$
for all fixed $j \geq 0$, where the implied constant may depend on
$j$.  Let $q\geq 1$ be an integer, and let
$$
M' := \sum_{m\geq 1}\psi_M(m)\ll M(\log M)^{O(1)}.
$$
We have:
\begin{enumerate}[(i)]
\item If $f\colon \Z/q\Z \to \C$ is a function, then
\begin{equation}\label{complete-1}
  \Bigl|\sum_m \psi_M(m) f(m) - \frac{M'}{q} 
  \sum_{m \in \Z/q\Z} f(m)\Bigr| 
  \ll  q^{1/2}(\log M)^{O(1)}\sup_{h \in \Z/q\Z \backslash \{0\}}|\FT_q(f)(h)|.
\end{equation}
In particular, if $M \ll q(\log M)^{O(1)}$, then
\begin{equation}\label{complete-2}
  \Bigl|\sum_m \psi_M(m) f(m)\Bigr| \ll q^{1/2}(\log M)^{O(1)} \|\FT_q(f)\|_{\ell^\infty(\Z/q\Z)}.
\end{equation}
We also have the variant
\begin{multline}\label{complete-1b}
  \Bigl|\sum_m \psi_M(m) f(m) - \frac{M'}{q} \sum_{m \in \Z/q\Z}
  f(m)\Bigr| \ll (\log M)^{O(1)} \frac{M}{q^{1/2}} \sum_{0 < |h|\leq
    qM^{-1+\eps}} |\FT_q(f)(h)|
  \\
  + M^{-A} \sum_{m \in \Z/q\Z} |f(m)|
\end{multline}
for any fixed $A>0$ and $\eps>0$, where the implied constant depends
on $\eps$ and $A$.
\item If $I$ is a finite index set, and for each $i \in I$, $c_i$ is a
  complex number and $a_i\ (q)$ is a residue class, then for each
  fixed $A>0$ and $\eps>0$, one has
\begin{multline}\label{complete-3}
  \Bigl|\sum_{i \in I} c_i \sum_m \psi_M(m) \onef_{m=a_i\ (q)} -
  \frac{M'}{q} \sum_{i \in I} c_i\Bigr| \ll (\log M)^{O(1)}\frac{M}{q}
  \sum_{0 < |h|\leq  qM^{-1+\eps}} \Bigl|\sum_{i \in I} c_i e_q(a_i h)\Bigr| \\
  + M^{-A} \sum_{i \in I} |c_i|,
\end{multline}
where the implied constant depends on $\eps$ and $A$.
\end{enumerate}
\end{lemma}

\begin{remark} One could relax the derivative bounds on $\psi$ to $|\psi^{(j)}(x)| \ll M^{\eps_j}$ for various small fixed $\eps_j>0$, at the cost of similarly worsening the various powers of $\log M$ in the conclusion of the lemma to small powers of $M$, and assuming the $\eps_j$ small enough depending on $\eps$ and $A$; however this variant of the lemma is a little tricky to state, and we will not have use of it here.
\end{remark}

\begin{proof}  Define the function
$$
\psi_{M,q}(x)=\sum_{n\in\Zz}\psi_M(x+qn).
$$
This is a smooth $q$-periodic function on $\R$. By periodization and
by the Plancherel formula, we have
\begin{equation}\label{plancherel-result}
  \sum_m \psi_M(m) f(m)=\sum_{x\in\Zz/q\Zz}f(x)\psi_{M,q}(x)=\sum_{h\in\Zz/q\Zz}\FT_q(f)(h)\FT_q(\psi_{M,q})(-h).
\end{equation}
The contribution of the frequency $h=0$ is given by
$$
\FT_q(f)(0)\FT_q(\psi_{M,q})(0)=\frac{1}{q} \sum_{m \in \Z/q\Z}f(m)
\sum_{m\in\Z/q\Z} \psi_{M,q}(m)=\frac{M'}{q} \sum_{m \in \Z/q\Z}f(m).
$$
We now consider the contribution of the non-zero frequencies. For
$h\in \Z/q\Z$, the definition of $\psi_{M,q}$ leads to
$$
q^{1/2}\FT_q(\psi_{M,q})(-h)=\Psi\Bigl(\frac{h}{q}\Bigr),
$$
where the function $\Psi$ is defined on $\R/\Z$ by
$$
  \Psi(y):=\sum_m \psi_M(m) e(-my).
$$
This is a smooth function $\Psi\colon \Rr/\Zz\ra \C$. We then have
$$
\Bigl|
\sum_{h\in\Zz/q\Zz\backslash\{0\}}\FT_q(f)(h)\FT_q(\psi_{M,q})(-h)
\Bigr|\leq \sup_{h\in \Z/q\Z\backslash\{0\}}
|\FT_q(f)(h)|q^{-1/2}\sum_{\substack{-q/2<h\leq q/2\\h\neq 0}}
\Bigl|\Psi\Bigl(\frac{h}{q}\Bigr)\Bigr|.
$$
Applying the Poisson summation formula and the definition
$\psi_M(x)=\psi((x-x_0)/M)$, we have
$$ 
\Psi(y) = M \sum_{n\in \Z} \hat \psi(M(n+y)) e(-(n+y)x_0)
$$
where
$$ 
\hat \psi(s) =\int_\R \psi(t) e(-st)\ dt.
$$
By repeated integrations by parts, the assumption on the size of the
derivatives of $\psi$ gives the bounds
$$
 |\hat \psi(s)| \ll (\log M)^{O(1)} (1+|s|)^{-A}
$$
for any fixed $A \geq 0$, and therefore
\begin{equation}\label{eq-decay-psi}
|\Psi(y)|\ll M (\log M)^{O(1)} (1 + |y|M)^{-A}
\end{equation}
for any fixed $A \geq 0$ and any $-1/2<y\leq 1/2$. Taking, e.g.,
$A=2$, we get
$$
\sum_{\substack{-q/2<h\leq q/2\\h\neq 0}}
\Bigl|\Psi\Bigl(\frac{h}{q}\Bigr)\Bigr| \ll (\log M)^{O(1)}\sum_{1\leq
  h\leq q/2}\frac{M}{(1+|h|M/q)^2} \ll q(\log M)^{O(1)},
$$
and therefore we obtain \eqref{complete-1}. From this,
\eqref{complete-2} follows immediately.

We now turn to \eqref{complete-1b}. Fix $A>0$ and $\eps>0$. Arguing as
above, we have
\begin{multline*}
  \left|\sum_m \psi_M(m) f(m) - \frac{M'}{q} \sum_{m \in \Z/q\Z}
    f(m)\right|\leq \frac{1}{q^{1/2}}\sum_\stacksum{-q/2<h\leq
    q/2}{h\neq 0}\Bigl|\Psi\Bigl(\frac{h}{q}\Bigr)\Bigr|
  |\FT_q(f)(h)| \\
  \ll (\log M)^{O(1)}\frac{M}{q^{1/2}} \sum_{0 < |h| \leq
    qM^{-1+\eps}}
  |\FT_q(f)(h)| \\
  + (\log M)^{O(1)}\sum_{n \in\Z/q\Z} |f(n)| \sum_{|h| > q M^{-1+\eps}}
  \frac{M}{q(1+|h|M/q)^{A}}
\end{multline*}
Changing $A$ to a large value, we conclude that
\begin{multline*}
  \left|\sum_m \psi_M(m) f(m) - \frac{M'}{q} \sum_{m \in \Z/q\Z}
    f(m)\right| \ll Mq^{-1/2}(\log M)^{O(1)} \sum_{0 < |h| \leq
    qM^{-1+\eps}} |\FT_q(f)(h)| \\ + M^{-A}\sum_{n \in\Z/q\Z}
  |f(n)|,
\end{multline*}
as claimed.

Finally, the claim (ii) follows immediately from \eqref{complete-1b}
by setting
$$
f(m) := \sum_{\substack{i \in I\\ a_i = m\ (q)}} c_i,\quad\text{ so that }\quad
\FT_q(f)(h)=\frac{1}{\sqrt{q}}\sum_{i\in I}c_ie_q(a_ih).
$$
\end{proof}

\begin{remark}
  In Section~\ref{typeiii-sec}, we will use a slightly refined version
  where the coefficients $\Psi(h/q)$ above are not estimated
  trivially.
\end{remark}

By combining this lemma with Proposition \ref{weil}, we can obtain
non-trivial bounds for incomplete exponential sums of the form
$$ 
\sum_n \psi_N(n) e_q(f(n))
$$
for various moduli $q$, which are roughly of the shape
$$ 
\sum_n \psi_N(n) e_q(f(n)) \ll q^{1/2+\eps}
$$
when $N \ll q$. A number of bounds of this type were used by Zhang~\cite{zhang} 
to obtain his Type I and Type II estimates.  However, it
turns out that we can improve this bound for certain regimes of $q,N$
when the modulus $q$ is smooth, or at least densely divisible, by
using the ``$q$-van der Corput $A$-process'' of Heath-Brown
\cite{heath-hybrid} and Graham-Ringrose \cite{graham}. This method was
introduced 
to handle incomplete \emph{multiplicative} character sums, but it is
also applicable to incomplete additive character sums.  It turns out
that these improved estimates lead to significant improvements in the
Type I and Type II numerology over that obtained in \cite{zhang}.

Here is the basic estimate on incomplete one-dimensional exponential
sums that we will need for the Type I and Type II estimates.  Essentially the same bounds were obtained in \cite[Theorem 2]{hb-large}.

\begin{proposition}[Incomplete additive character sums]\label{inc} Let
  $q$ be a squarefree integer, and let $f = \frac{P}{Q}\in \Q(X)$ with
  $P$, $Q\in\Z[X]$, such that the degree of $Q\
  (p)$ is equal to $\deg(Q)$ for all $p\mid q$. Assume that
  $\deg(P)<\deg(Q)$.  Set $q_1 := q/(f,q)$.  Let further $N \geq 1$ be
  given with $N\ll q^{O(1)}$ and let $\psi_N$ be a function on $\R$
  defined by
$$
\psi_N(x) = \psi\left(\frac{x-x_0}{N}\right)
$$
where $x_0\in \R$ and $\psi$ is a smooth function with compact support
satisfying
$$
|\psi^{(j)}(x)| \ll \log^{O(1)} N
$$
for all fixed $j \geq 0$, where the implied constant may depend on
$j$. 
\begin{enumerate}[(i)]
\item (Poly\'a-Vinogradov + Ramanujan-Weil) We have the bound
\begin{equation}\label{vdc-0}
  \sum_n \psi_N(n) e_q(f(n))\ll q^{\eps}\Bigl(q_1^{1/2} +
  \frac{N}{q_1}\onef_{N \geq q_1} 
  \Bigl|\sum_{n \in \Z/q_1\Z} e_{q_1}(f(n) / (f,q) )\Bigr|\Bigr)
\end{equation}
for any $\eps>0$. In particular, lifting the $\Z/q_1\Z$ sum to a $\Z/q\Z$ sum, we have
\begin{equation}\label{vdc-0-alt}
  \sum_n \psi_N(n) e_q(f(n))\ll q^{\eps}\Bigl(q^{1/2} +
  \frac{N}{q} 
  \Bigl|\sum_{n \in \Z/q\Z} e_{q}(f(n) )\Bigr|\Bigr)
\end{equation}
\item (one van der Corput + Ramanujan-Weil) If $q = rs$, then we have the
  additional bound
\begin{equation}\label{vdc-1}
  \sum_n \psi_N(n) e_q(f(n)) \ll q^{\eps}\left( \Bigl(N^{1/2}r_1^{1/2} 
  + N^{1/2} s_1^{1/4}\Bigr) +
  \frac{N}{q_1}\onef_{N \geq q_1} 
  \Bigl|\sum_{n \in \Z/q_1\Z} e_{q_1}(f(n) / (f,q) )\Bigr|\right)
\end{equation}
for any $\eps>0$, where $r_1 := (r,q_1)$ and $s_1:=(s,q_1)$.  In particular, we have
\begin{equation}\label{vdc-1-alt}
  \sum_n \psi_N(n) e_q(f(n)) \ll q^{\eps}\left( \Bigl(N^{1/2}r^{1/2} 
  + N^{1/2} s^{1/4}\Bigr) +
  \frac{N}{q} 
  \Bigl|\sum_{n \in \Z/q\Z} e_{q}(f(n) )\Bigr|\right).
\end{equation}
\end{enumerate}
In all cases the implied constants depend on $\eps$, $\deg(P)$,
$\deg(Q)$ and the implied constants in the estimates for the
derivatives of $\psi$.
\end{proposition}

\begin{remark} The estimates obtained by completion of sums are
  usually inefficient in the regime $M = o(q)$, and they become
  trivial for $M\ll q^{1/2}$. For instance, when $f$ is
  bounded in magnitude by $1$, the trivial bound for the right-hand
  side of \eqref{complete-2} is $q$, whereas the trivial bound for the
  left-hand side is of size about $M$, which means that one needs a
  cancellation at least by a factor $q/M$ in the right-hand side to
  even recover the trivial bound. This becomes a prohibitive
  restriction if this factor is larger than $\sqrt{M}$.  In this
  paper, this inefficiency is a major source of loss in our final
  exponents (the other main source being our frequent reliance on the
  Cauchy-Schwarz inequality, as each invocation of this inequality
  tends to halve all gains in exponents arising from application of
  the Riemann Hypothesis over finite fields). It would thus be of
  considerable interest to find stronger estimates for incomplete
  exponential sums. But the only different (general) method we are
  aware of is the recent ``sliding sum method'' of Fouvry, Kowalski
  and Michel~\cite{sliding}, which however only improves on the
  completion technique when $M$ is very close to $q^{1/2}$, and does
  not give stronger bounds than Lemma~\ref{com} and Proposition~\ref{inc} in most
  ranges of interest. (Note however that uniformity of estimates is
  often even more crucial to obtaining good results, and for this
  purpose, the completion techniques are indeed quite efficient.)
\end{remark}

\begin{proof} 
We begin with some technical reductions.  First of all, we may assume
that $q$ has no prime factor smaller than any fixed $B$ depending on
$\deg(P)$ and $\deg(Q)$, as the general case then follows by factoring out a bounded factor from $q$ and splitting the summation over $n$ into a bounded number of pieces.

Second, we also observe that, in all cases, we may replace $f$ by
$f/(f,q)$ and $q$ by $q_1$ and (in the case when $q=rs$) $r$ by $r_1$
and $s$ by $s_1$, since if we write $q=q_1q_2$ we have
$$
e_q(f(n))=e_{q_1}\Bigl(\frac{P(n)}{q_2Q(n)}\Bigr).
$$
Thus we can reduce to a situation where $(f,q)=1$, so $q=q_1$, $r=r_1$
and $s=s_1$. In this case, the condition $\deg(P)<\deg(Q)$ implies
also $(f',q)=(f'',q)=1$ by Lemma \ref{fund}(ii), provided $q$ has no
prime factor less than some constant depending on $\deg(P)$ and
$\deg(Q)$, which we may assume to be the case, as we have seen.

We now establish \eqref{vdc-0}.  We apply~(\ref{complete-1b}), and put
the ``main term'' with $h=0$ in the right-hand side, to get
$$
\sum_n \psi_N(n) e_q(f(n)) \ll \frac{N^{1+\eps}}{q} \sum_{|h|\leq
  qN^{-1+\eps}}\Bigl|\sum_{n \in \Z/q\Z} e_q(f(n) + hn)\Bigr| + 1
$$
for $\eps>0$ arbitrarily small (by selecting $A$ large enough
in~(\ref{complete-1b}) using the assumption $N\ll q^{O(1)}$).

If $N<q$, Proposition~\ref{weil} applied for all $h$ gives
$$
\sum_n \psi_N(n) e_q(f(n)) \ll \frac{N^{1+\eps}}{q^{1/2}}
\sum_{0\leq |h|\leq qN^{-1+\eps}}(f'+h,q).
$$
Since $(f'',q)=1$, we also have $(f'+h,q)=1$, and therefore
$$
\sum_n \psi_N(n) e_q(f(n)) \ll q^{1/2}N^{2\eps}
$$
which implies~(\ref{vdc-0}).  If $N\geq q$, on the other hand, we only
apply~(\ref{weil}) for $h \neq 0$, and we get in the same way
$$
\sum_n \psi_N(n) e_q(f(n)) \ll \frac{N^{1+\eps}}{q}
\Bigl|\sum_{n\in\Z/q\Z}e_q(f(n))\Bigr|+ q^{1/2}N^{2\eps},
$$
which is again~(\ref{vdc-0}).



Consider now \eqref{vdc-1}.  We may assume that $N \leq s$, since
otherwise the claim follows simply from \eqref{vdc-0}, and we may
similarly assume that $ r\leq N$, since otherwise we can use
the trivial bound
$$
\sum_n \psi_N(n) e_q(f(n)) \ll N(\log N)^{O(1)}\ll r^{1/2}N^{1/2}(\log
N)^{O(1)}.
$$

Let $K := \lfloor N/r\rfloor$.  Using translation invariance, we can
write
$$
\sum_n \psi_N(n)e_q(f(n))=\frac{1}{K} \sum_n \sum_{k=1}^K
\psi_N(n+kr) e_q(f(n+kr)).
$$
Since $q=rs$, we have 
$$
e_q(f(n+kr)) = e_{r}(\overline{s} f(n))e_{s}( \overline{r} f(n+kr) )
$$
by Lemma \ref{crt} (and periodicity), and hence we obtain
\begin{align*}
  \Bigl|\sum_n \psi_N(n) e_q(f(n))\Bigr| &\leq
  \frac{1}{K} \sum_n \left|\sum_{k=1}^K \psi_N(n+kr) e_{s}( \overline{r} f(n+kr) )\right| \\
  & \ll \frac{N^{1/2}}{K} \left(\sum_n \left|\sum_{k=1}^K \psi_N(n+kr)
      e_{s}( \overline{r} f(n+kr) )\right|^2\right)^{1/2},
\end{align*}
where the factor $N^{1/2}$ arises because the summand is (as a
function of $n$) supported on an interval of length $O(N)$.  Expanding
the square, we obtain
\begin{equation}\label{eq-square-sum}
\Bigl|\sum_n \psi_N(n) e_q(f(n))\Bigr|^2\ll  \frac{N}{K^2}
\sum_{1\leq k,l\leq K}A(k,l),
\end{equation}
where
$$
A(k,l)=\sum_n \psi_N(n+kr) \overline{\psi_N(n+lr)}\
e_{s}\left(\overline{r} (f(n+kr) - f(n+lr) ) \right).
$$
We have
$$
A(k,k)=\sum_n |\psi_N(n+kr)|^2\ll N(\log N)^{O(1)}.
$$
and therefore 
\begin{equation}\label{eq-diag}
\sum_{1\leq k\leq K} |A(k,k)|\ll KN(\log N)^{O(1)}.
\end{equation}

There remains to handle the off-diagonal terms. For each $k\neq l$, we
have
$$
\frac{f(n+kr) - f(n+lr)}{r}=g(n)
$$
where $g=P_1/Q_1\in \Q(X)$ with integral polynomials
\begin{gather*}
P_1(X)=P(X+kr)Q(X+lr)-Q(X+kr)P(X+lr),\\
Q_1(X)=rQ(X+kr)Q(X+lr).
\end{gather*}
Note that $P_1$ and $Q_1$ satisfy the assumptions of~(\ref{vdc-0})
with respect to the modulus $s$ (although they might not be coprime).

We now claim that (provided all prime factors of $q$ are large enough)
we have 
$$
(s,g')\mid (s,k-l)\quad\text{ and }\quad (s,g)\mid (s,k-l).
$$
Indeed, since $\deg(P)<\deg(Q)$ and the degree of the reduction of $Q$
modulo primes dividing $q$ is constant, it is enough to show that
$(s,g)\mid (s,k-l)$ by Lemma~\ref{fund}(ii). So suppose that a prime
$p$ divides $(s,g)$. Then, by change of variable we have
$$
p\mid (s, f(X+(k-l)r)-f(X)).
$$
By induction, we thus have
$$ p\mid (s, f(X+i(k-l)r)-f(X))$$
for any integer $i$.  If $p\nmid k-l$, then $(k-l)r$ generates $\Z/p\Z$ as an additive group, and we conclude that $p\mid (s,f(X+a)-f(X))$ for all $a\in
\Z/p\Z$. This implies that $f\ (p)$ is constant
where it is defined. But since $\deg(P)<\deg(Q)$ holds modulo $p$, for
$p$ large enough in terms of $\deg(Q)$, this would imply that $p\mid
f$ (as in Lemma~\ref{fund}(ii)), contradicting the assumption
$(s,f)=1$. Thus we have $p\mid k-l$, and we conclude $(s,g)\mid
(s,k-l)$, and then $(s,g')\mid (s,k-l)$, as claimed.

By~(\ref{vdc-0}) and Proposition~\ref{weil}, we have
\begin{align*}
  A(k,l)&\ll q^{\eps}\Bigl(s^{1/2}+\frac{N}{s}\onef_{N\geq s/(s,k-l)}
  \Bigl|\sum_{n\in\Z/s\Z} e_s(g(n))\Bigr|\Bigr)\\
  &\ll q^{\eps}\Bigl(s^{1/2}+\frac{N}{s^{1/2}}
  (s,k-l)^{1/2}\onef_{N\geq s/(s,k-l)}\Bigr).
\end{align*}
Summing over $k$ and $l$, we have
\begin{equation}\label{eq-non-diag}
\sumsum_{1\leq k\neq l\leq K} |A(k,l)| \ll
q^{\eps}K^2s^{1/2}+q^{\eps}Ns^{-1/2} \sum_{1\leq k\neq l\leq K}
(s,k-l)^{1/2}\onef_{N\geq s/(s,k-l)}.
\end{equation}
We use the simple bound
$$
\onef_{N\geq s/(s,k-l)} \leq \sqrt{(s,k-l)} \sqrt{\frac{N}{s}}
$$
to estimate the last sum as follows:
\begin{align*}
  Ns^{-1/2}\sum_{1\leq k\neq l\leq K} (s,k-l)^{1/2}\onef_{N\geq
    s/(s,k-l)}&\leq \frac{N^{3/2}}{s}\sum_{1\leq k\neq l\leq K}
  (s,k-l)\\
&\ll N^{3/2}s^{-1}\times K^2q^{\eps}\ll K^2s^{1/2}q^{\eps}
\end{align*}
using Lemma \ref{ram-avg} and the bound $N<s$. We combine this
with~(\ref{eq-diag}) and~(\ref{eq-non-diag}) in the
bound~(\ref{eq-square-sum}) to obtain
$$
\Bigl|\sum_n \psi_N(n) e_q(f(n))\Bigr|^2\ll  q^{\eps}\frac{N}{K^2}
\Bigl(KN(\log N)^{O(1)}+K^2s^{1/2}\Bigr)\ll
q^{\eps}(Nr+Ns^{1/2}),
$$
from which~(\ref{vdc-1}) follows.
\end{proof}

\begin{remark} 
(1) Assuming that $(f,q)=1$, the first bound \eqref{vdc-0}
  is non-trivial (i.e., better than $O(N)$) as long as $N$ is a bit
  larger than $q^{1/2}$. As for \eqref{vdc-1}, we see that in the
  regime where the factorization $q=rs$ satisfies $r\approx
  q^{1/3}\approx s^{1/2}$, the bound is non-trivial in the
  significantly wider range where $N$ is a bit larger than $q^{1/3}$.
\par
(2) The procedure can also be generalized with similar results to more
general $q$-periodic functions than $n\mapsto e_q(f(n))$, and this
will be important for the most advanced Type I estimates (see
Section~\ref{ssec-vdc}). 
\end{remark}

\begin{remark}\label{vdCiterate} One can iterate the above argument
  and show that
\begin{multline*}
  \Bigl|\sum_n \psi_N(n) e_q(f(n)) \Bigr| \ll q^{\eps}\Bigl(
  \sum_{i=1}^{l-1} N^{1-1/2^i} \tilde{r}_i^{1/2^i} +
  N^{1-1/2^{l-1}} \tilde{r}_l^{1/2^l} \\
  + \frac{N}{q_1} \onef_{N \geq q_1} \Bigl|\sum_{n \in \Z/q_1\Z}
    e_{q_1}(f(n) / (f,q) )\Bigr|\Bigr)
\end{multline*}
for any fixed $l \geq 1$ and any factorization $q = r_1 \ldots r_l$,
with $\tilde{r}_i = (r_i,q_1)$; see \cite{graham}, \cite{hb-large}.  However,
we have found in practice that taking $l$ to be $3$ or higher
(corresponding to two or more applications of the $q$-van der Corput
$A$-process) ends up being counterproductive, mainly because the power
of $q$ that one can save over the trivial bound decays exponentially
in $l$.  However, it is possible that some other variation of the
arguments (for instance, taking advantage of the Parseval identity,
which would be a $q$-analogue of the van der Corput $B$-process) may
give further improvements.
\end{remark}

In our particular application, we only need a special case of the
above proposition. This is a strengthening of \cite[Lemma 11]{zhang},
and it shows how an assumption of dense divisibility of a modulus may
be exploited in estimates for exponential sums.

\begin{corollary}\label{dons} 
  Let $N\geq 1$ and let $\psi_N$ be a function on $\R$ defined by
$$
\psi_N(x) = \psi\left(\frac{x-x_0}{N}\right)
$$
where $x_0\in \R$ and $\psi$ is a smooth function with compact support
satisfying
$$
|\psi^{(j)}(x)| \ll \log^{O(1)} N
$$
for all fixed $j \geq 0$, where the implied constant may depend on
$j$.
\par
Let $d_1,d_2$ be squarefree integers, not necessarily coprime. Let
$c_1,c_2,l_1,l_2$ be integers. Let $y\geq 1$ be a real number, and
suppose that $[d_1,d_2]$ is $y$-densely divisible.  Let $d$ be a divisor of $[d_1,d_2]$ and let $a\ (d)$ be any residue
class.
\par
If $N\leq [d_1,d_2]^{O(1)}$, then we have
\begin{multline*}
  \Bigl|\sum_{n = a\ (d)} \psi_N(n) e_{d_1}\left( \frac{c_1}{n+l_1}
  \right) e_{d_2}\left( \frac{c_2}{n+l_2}
  \right)\Bigr| \\
  \ll [d_1,d_2]^{\eps}\Bigl(d^{-1/2} N^{1/2} [d_1,d_2]^{1/6} y^{1/6} +
  d^{-1} \frac{(c_1,\delta'_1)}{\delta'_1} \frac{(c_2,\delta'_2)}{\delta'_2} N\Bigr),
\end{multline*}
for any $\eps>0$, where $\delta_i := d_i/(d_1,d_2)$ and $\delta'_i := \delta_i / (d,\delta_i)$ for $i=1,2$.  We also
have the variant bound
\begin{multline*}
  \Bigl|\sum_{n = a\ (d)} \psi_N(n) e_{d_1}\left( \frac{c_1}{n + l_1}
  \right) e_{d_2}\left( \frac{c_2}{n+l_2} \right)\Bigr| \ll\\
  [d_1,d_2]^{\eps}\Bigl(d^{-1/2} [d_1,d_2]^{1/2} + d^{-1}
  \frac{(c_1,\delta'_1)}{\delta'_1} \frac{(c_2,\delta'_2)}{\delta'_2} N\Bigr).
\end{multline*}
In both cases the implied constant depends on $\eps$.
\end{corollary}

\begin{proof} Denote $q=[d_1,d_2]$. We first consider the case $d=1$,
  so that the congruence condition $n = a\ (d)$ is vacuous.  Since
  $R=y^{1/3}q^{1/3}\leq yq$, the dense divisibility hypothesis implies
  that there exists a factorization $q = rs$ for some integers $r$,
  $s$ such that
$$ 
y^{-2/3}q^{1/3} \leq r \leq y^{1/3} q^{1/3}
$$
and
$$ 
y^{-1/3}q^{2/3} \leq s \leq y^{2/3} q^{2/3}.
$$

Note now that, by the Chinese Remainder Theorem (as in
Lemma~\ref{crt}), we can write
$$
e_{d_1}\left( \frac{c_1}{n+l_1} \right) e_{d_2}\left(
  \frac{c_2}{n+l_2} \right)=
e_q(f(n))
$$
for a rational function $f=P/Q\in \Q(X)$ satisfying the assumptions of
Proposition~\ref{inc} (in particular $\deg(P)<\deg(Q)$).  The first bound follows
immediately from Proposition \ref{inc}(ii), combined with the complete
sum estimate
$$ 
\left|\sum_{n \in \Z/[d_1,d_2]\Z} e_{d_1}\left( \frac{c_1}{n+l_1}
  \right) e_{d_2}\left( \frac{c_2}{n+l_2} \right)\right| \ll
q^{\eps}(c_1,\delta_1) (c_2,\delta_2) (d_1,d_2)
$$
of Lemma \ref{dork}.  The second bound similarly follows from
Proposition \ref{inc}(i).

Now we consider the case when $d> 1$.  Making the substitution $n = n' d + a$ and applying the previous argument (with $N$ replaced by $N/d$, and with suitable modifications to $x_0$ and $f$), we reduce to showing that
$$ 
\left|\sum_{n \in \Z/[d_1,d_2]\Z: n = a\ (d)} e_{d_1}\left( \frac{c_1}{n+l_1}
  \right) e_{d_2}\left( \frac{c_2}{n+l_2} \right)\right| \ll
q^{\eps}(c_1,\delta'_1) (c_2,\delta'_2) (d'_1,d'_2)
$$
where $d'_i := d_i / (d,d_i)$ for $i=1,2$ (note that $\frac{d (d'_1,d'_2)}{[d_1,d_2]} = \frac{1}{\delta'_1 \delta'_2}$).  However, this again follows from Lemma \ref{dork} after making the change of variables $n = n'd+a$.
\end{proof}


\section{Type I and Type II estimates}\label{typei-ii-sec}

Using the estimates of the previous section, we can now prove the Type
I and Type II results of Theorem \ref{newtype}, with the exception of
part \eqref{typeI4} of that theorem in which we only make a
preliminary reduction for now. The rest of the proof of that part,
which depends on the concepts and results of
Section~\ref{deligne-sec}, will be found in
Section~\ref{typei-advanced-sec}.

We recall the statements (see Definition \ref{type-def}).

\begin{theorem}[New Type I and Type II estimates]\label{newtype-i-ii}
  Let $\varpi,\delta,\sigma > 0$ be fixed quantities, let $I$ be a
  bounded subset of $\R$, let $i\geq 1$ be fixed, let $a\ (P_I)$ be a
  primitive congruence class, and let $M,N \gg 1$ be quantities with
\begin{equation}\label{lin-2}
MN \sim x
\end{equation}
and
\begin{equation}\label{slop}
 x^{1/2-\sigma} \lessapprox N \lessapprox x^{1/2}.
\end{equation}
Let $\alpha,\beta$ be coefficient sequences located at scales $M,N$
respectively, with $\beta$ satisfying the Siegel-Walfisz property.
Then we have the estimate
\begin{equation}\label{alphabet-a}
  \sum_{\substack{d \in \DI{I}{i}{x^\delta}\\ d \lessapprox x^{1/2+2\varpi}}} |\Delta(\alpha \star \beta; a\ (d))|
  \ll x \log^{-A} x
\end{equation}
for any fixed $A>0$, provided that one of the following
hypotheses holds:
\begin{enumerate}[(i)]
\item\label{typeI1-again} $i=1$, $54\varpi + 15\delta + 5 \sigma < 1$, and $N \lessapprox x^{1/2-2\varpi-c}$ for some fixed $c>0$.
\item\label{typeI2-again} $i=2$, $56\varpi + 16\delta + 4 \sigma < 1$, and $N \lessapprox x^{1/2-2\varpi-c}$ for some fixed $c>0$.
\item\label{typeI4-again} $i=4$, $\frac{160}{3} \varpi + 16\delta + \frac{34}{9} \sigma < 1$, $64\varpi + 18 \delta + 2\sigma < 1$, and $N \lessapprox x^{1/2-2\varpi-c}$ for some fixed $c>0$.
\item\label{typeII1-again} $i=1$, $68\varpi + 14 \delta < 1$, and $N \gtrapprox x^{1/2-2\varpi-c}$ for some sufficiently small fixed $c>0$.
\end{enumerate}
The proof of the case \eqref{typeI4-again} uses the general form
of the Riemann Hypothesis over finite fields~\cite{WeilII}, but the
proofs of \eqref{typeI1-again}, \eqref{typeI2-again},
\eqref{typeII1-again} only need the Riemann Hypothesis for curves over
finite fields.
\end{theorem}

Before we begin the rigorous proof of Theorem \ref{newtype-i-ii}, we
give an informal sketch of our strategy of proof for these estimates,
which is closely modeled on the arguments of Zhang \cite{zhang}.  The
basic idea is to reduce the estimate \eqref{alphabet-a} to a certain
exponential sum estimate, of the type found in Corollary \ref{dons} (and,
for the estimate \eqref{typeI4-again}, in Corollary \ref{inctrace-q} of
the next section).  The main tools for these reductions are completion
of sums (Lemma \ref{com}), the triangle inequality, and many
techniques related to the Cauchy-Schwarz inequality (viewed in a broad
sense), for instance Vinogradov's bilinear form method, the $q$-van
der Corput $A$-process, the method of Weyl differencing, or the
dispersion method of Linnik.

\subsection{Bilinear form estimates}

We begin with a short discussion of typical instances of applications
of the Cauchy-Schwarz inequality (some examples already appeared in
previous sections).  We want to estimate a sum
$$
\sum_{s  \in S} c_s
$$
of (typically) complex numbers $c_s$ indexed by some finite set $S$ of
large size. Suppose we can parameterize $S$ (possibly with repetition)
by a \emph{non-trivial} product set $A \times B$, i.e., by a product
where neither factor is too small, or otherwise prove an inequality
$$ 
\Bigl|\sum_{s \in S} c_s\Bigr| \leq \Bigl|\sum_{a \in A} \sum_{b \in
  B} \alpha_a \beta_b k_{a,b}\Bigr|
$$ 
for certain coefficients $\alpha_a$, $\beta_b$ and $k_{a,b}$.  The
crucial insight is that one can often derive non-trivial estimates for
an expression of this type with little knowledge of the coefficients
$\alpha_a$, $\beta_b$, by exploiting the bilinear structure and
studying the coefficients $k_{a,b}$. 
\par
Precisely, one can apply the Cauchy-Schwarz inequality to bound the right-hand side by
$$ 
\left(\sum_{a \in A} |\alpha_a|^2\right)^{1/2} \left(\sum_{a \in A}
  \left|\sum_{b\in B} \beta_bk_{a,b}\right|^2\right)^{1/2}.
$$ 
The first factor in the above expression is usually easy to
estimate, and the second factor can be expanded as
$$
\Bigl|\sum_{b,b'\in B}\beta_b\overline{\beta_{b'}}
C(b,b')\Bigr|^{1/2}, \quad\quad C(b,b')=\sum_{a\in
  A}k_{a,b}\overline{k_{a,b'}}.
$$ 
One can then distinguish between the \emph{diagonal contribution}
defined by $b=b'$ and the \emph{off-diagonal contribution} where $b
\neq b'$.  The contribution of the former is
$$
\sum_{b\in B}\sum_{a\in A}|\beta_b|^2|k_{a,b}|^2
$$
which is (usually) not small, since there cannot be cancellation
between these non-negative terms. It may however be estimated
satisfactorily, provided $B$ is large enough for the diagonal $\{
(b,b): b\in B \}$ to be a ``small'' subset of the square $B \times
B$. (In practice, there might be a larger subset of $B\times B$ than
the diagonal where the coefficient $C(b,b')$ is not small, and that is
then incorporated in the diagonal; in this paper, where $b$ and $b'$
are integers, it is the size of a gcd $(b-b',q)$ that will dictate
which terms can be considered diagonal).
\par
On the other hand, the individual off-diagonal terms $C(b,b')$ can be
expected to exhibit cancellation that makes them individually small. In
order for the sum over $b\neq b'$ to remain of manageable size, one
needs $B$ to remain not too large. In order to balance the two
contributions, it turns out to be extremely useful to have a flexible
\emph{family} of parameterizations $(a,b) \mapsto s$ of $S$ by product
sets $A \times B$, so that one can find a parameterization for which
the set $B$ is close to the optimum size arising from various
estimates of the diagonal and non-diagonal parts. This idea of
flexibility is a key idea at least since Iwaniec's
discovery~\cite{iwaniec} of the bilinear form of the error term in
the linear sieve.
\par
One of the key ideas in Zhang's paper \cite{zhang} is that if one is
summing over smooth moduli, then such a flexible range of
factorizations exists; to put it another way, the restriction to
smooth moduli is essentially a ``well-factorable'' weight in the sense
of Iwaniec.  In this paper, we isolated the key property of smooth
moduli needed for such arguments, namely the property of \emph{dense
  divisibility}. 
The general strategy is thus to keep exploiting
the smoothness or dense divisibility of the moduli to split the sums
over such moduli into a ``well-factorable'' form to which the
Cauchy-Schwarz inequality may be profitably applied. (Such a strategy
was already used to optimize the use of the $q$-van der Corput
$A$-process in Corollary \ref{dons}.)

\subsection{Sketch of proofs}\label{typei-ii-sketch}

We now give a more detailed, but still very informal, sketch of the
proof of Theorem~\ref{newtype-i-ii}, omitting some steps and some
terms for sake of exposition (e.g., smooth cutoffs are not
mentioned). For simplicity we will pretend that the quantities
$\varpi,\delta$ are negligible, although the quantity $\sigma$ will
still be of a significant size (note from Lemma \ref{comlem} that we
will eventually need to take $\sigma$ to be at least $1/10$).  The
first step is to exploit the dense divisibility of the modulus $d$ to
factor it as $d=qr$, with $q,r$ located at certain scales $Q, R$ which
we will specify later; with $\varpi$ negligible, we expect $QR$ to be
approximately equal to $x^{1/2}$ but a bit larger.  Our task is then
to obtain a non-trivial bound on the quantity
$$ 
\sum_{q\sim Q} \sum_{r \sim R} |\Delta(\alpha \star \beta; a\ (qr))|,
$$
or equivalently to obtain a non-trivial bound on
$$ 
\sum_{q\sim Q} \sum_{r \sim R} c_{q,r} \Delta(\alpha \star \beta; a\
(qr))
$$
for an arbitrary bounded sequence $c_{q,r}$. We suppress here, and
later, some additional information on the moduli $q,r$, e.g. that they
are squarefree and coprime, to simplify this informal exposition. 
For similar reasons we are being vague on what a ``non-trivial bound''
means, but roughly speaking, it should improve upon the ``trivial
bound'' by a factor of $\log^{-A}x$ where $A$ is very large (or
arbitrarily large).
\par
If we insert the definition \eqref{disc-def}, and denote generically
by $\emt$ the contribution of the second term in that definition
(which is the ``expected main term''), we see that we need a
non-trivial bound on the quantity
$$
\sum_{q\sim Q} \sum_{r \sim R} c_{q,r} \sum_{n= a\ (qr)} \alpha \star
\beta(n)-\emt.
$$ 
For simplicity, we will handle the $r$ averaging trivially, and thus seek to control
the sum
$$
 \sum_{q\sim Q} c_{q,r} \sum_{n= a\ (qr)} \alpha \star \beta(n)-\emt
$$
for a single $r\sim R$.  We rearrange this as
$$
\sum_{m \sim M} \alpha(m) \sum_{q\sim Q} c_{q,r} \sum_{\substack{n
    \sim N\\ nm = a\ (qr)}} \beta(n)-\emt.
$$
Note that for fixed $m$ coprime with $q$, the number of pairs $(q,n)$ with $q \sim Q$,
$n \sim N$, and $nm = a\ (qr)$ is expected to be about $\frac{QN}{QR}
= \frac{N}{R}$.  Thus, if we choose $R$ to be a little bit less than
$N$, e.g. $R = x^{-\eps} N$, then the number of pairs $(q,n)$
associated to a given value of $m$ is expected to be non-trivial. This
opens up the possibility of using the dispersion method of Linnik
\cite{linnik}, as the diagonal contribution in that method is expected
to be negligible.  Accordingly, we apply Cauchy-Schwarz in the 
variable $m$, eliminating the rough coefficient sequence $\alpha$, and end
up with the task of controlling an expression of the shape
$$
\sum_{m \sim M} \Bigl|\sum_{q\sim Q} c_{q,r} \sum_{\substack{n \sim
    N\\ nm = a\ (qr)}} \beta(n)-\emt\Bigr|^2.
$$ 
Opening the square as sketched above, this is equal to
$$
\sum_{q_1,q_2 \sim Q} c_{q_1,r} \overline{c_{q_2,r}}
\sumsum_{n_1,n_2\sim N} \beta(n_1) \overline{\beta(n_2)} \Bigl(
\sum_{\substack{m \sim M\\ n_1m =a\ (q_1r)\\ n_2m= a\ (q_2 r)}}
1-\emt\Bigr).
$$ 
Note that, since $a\ (qr)$ is a primitive residue class, the
constraints $n_1 m = a\ (q_1r), n_2 m = a\ (q_2 r)$ imply $n_1 = n_2\
(r)$. Thus we can write $n_2 = n_1 + \ell r$ for some $\ell = O(N/R)$, which
will be rather small (compare with the method of Weyl differencing).
\par
For simplicity, we consider only\footnote{Actually, for technical reasons, in the rigorous argument we will dispose of the $\ell=0$ contribution by a different method, so the discussion here should be viewed as an oversimplification.} the case $\ell=0$ here.  We are thus led
to the task of controlling sums such as
\begin{equation}\label{cqq}
  \sum_{q_1,q_2 \sim Q} c_{q_1,r} \overline{c_{q_2,r}} \sum_{n \sim N} \beta(n) \overline{\beta(n)} \Bigl(\sum_{\substack{m \sim M\\ nm =a\ (q_1r)\\ nm= a\ (q_2 r)}} 1-\emt\Bigr).
 \end{equation}
 It turns out (using a technical trick of Zhang which we will describe
 below) that we can ensure that the moduli $q_1,q_2$ appearing here
 are usually coprime, in the sense that the contribution of the
 non-coprime pairs $q_1,q_2$ are negligible.  Assuming this, we can
 use the Chinese Remainder Theorem to combine the two constraints $nm
 = a\ (q_1r)$, $nm = a\ (q_2r)$ into a single constraint $nm =a\ (q_1
 q_2 r)$ on $m$.  Now, we note that if $R$ is slightly less than $N$, then
 (since $MN$ is close to $x$, and $QR$ is close to $x^{1/2}$) the
 modulus $q_1 q_2 r$ is comparable to $M$.  This means that the inner
 sum 
$$
\sum_{\substack{m \sim M\\ nm =a\ (q_1q_2r)}} 1-\emt
$$ 
is essentially a complete sum, and can therefore be very efficiently
handled by Lemma~\ref{com}. This transforms \eqref{cqq} into
expressions such as
$$ 
\sum_{0 < |h| \leq H} c_h \sum_{q_1,q_2 \sim Q} c_{q_1,r}
\overline{c_{q_2,r}} \sum_{n \sim N} \beta(n) \overline{\beta(n)}
e_{q_1 q_2 r}\left(\frac{ah}{n}\right),
$$ 
where $H \approx \frac{Q^2 R}{M}$ is a fairly small quantity, and the
coefficients $c_h$ are bounded. At this point, the contribution of the
zero frequency $h=0$ has cancelled out with the expected main term
$\emt$ (up to negligible error).

This expression involves the essentially unknown (but bounded)
coefficients $c_{q_1,r}$, $c_{q_2,r}$, $\beta(n)$, and as before, we
can not do much more than eliminate them using the Cauchy-Schwarz
inequality.  This can be done in several ways here, depending on which
variables are taken ``outside'' of the Cauchy-Schwarz inequality. For
instance, if we take $n$ to eliminate the $\beta(n)
\overline{\beta(n)}$ term, one is led, after expanding the square and
exchanging the sum in the second factor of the Cauchy-Schwarz
inequality, to expressions such as
$$ 
\sum_{0 < |h_1|, |h_2| \leq H} \sumsum_{q_1,q_2,s_1,s_2\sim Q} \left|
  \sum_{n \sim N} e_{q_1q_2 r}\left(\frac{ah_1}{n}\right) e_{s_1s_2r}
  \left(-\frac{ah_2}{n}\right)\right|.
$$ 
The sum over $n$ has length $N$ close to the modulus
$[q_1q_2r,s_1s_2r]\approx Q^4R$, and therefore can be estimated
non-trivially using Corollary \ref{dons}. As we will see, this
arrangement of the Cauchy-Schwarz inequality is sufficient to
establish the Type II estimate \eqref{typeII1-again}.
\par
The Type I estimates are obtained by a slightly different application
of Cauchy-Schwarz. Indeed, note for instance that as the parameter
$\sigma$ (which occurs in the Type I condition, but not in Type II)
gets larger, the length $N$ in the sum may become smaller in
comparison to the modulus $q_1 q_2s_1s_2r$ in the
exponential sum
$$
\sum_{n \sim N} e_{q_1q_2 r}\left(\frac{ah_1}{n}\right) e_{s_1 s_2
  r}\left(-\frac{a h_2}{n}\right),
$$ 
and this necessitates more advanced exponential sum estimates to
recover non-trivial cancellation. Here, the $q$-van der Corput
$A$-method enlarges the range of parameters for which we can prove
that such a cancellation occurs. This is one of the main reasons why
our Type I estimates improve on those in \cite{zhang}.  (The other
main reason is that we will adjust the Cauchy-Schwarz inequality to
lower the modulus in the exponential sum to be significantly smaller
than $q_1 q_2s_1s_2 r \sim Q^4 R$, while still keeping both the
diagonal and off-diagonal components of the Cauchy-Schwarz estimate
under control.)

\subsection{Reduction to exponential sums}\label{prelim-red}

We now turn to the details of the above strategy.  We begin with the
preliminary manipulations (mostly following \cite{zhang}) to reduce
the estimate \eqref{alphabet-a} to a certain exponential sum
estimate. This reduction can be done simultaneously in the four cases
\eqref{typeI1-again}, \eqref{typeI2-again}, \eqref{typeI4-again},
\eqref{typeII1-again}, but the verification of the exponential sum
estimate requires a different argument in each of the four cases.

In the remainder of this section
$\varpi,\delta,\sigma,I,i,a,M,N,\alpha,\beta$ are as in Theorem
\ref{newtype-i-ii}.  First of all, since $\beta$ satisfies the
Siegel-Walfisz property, the Bombieri-Vinogradov Theorem \ref{bvt}
implies
\begin{equation}\label{load}
  \sum_{d \leq x^{1/2} \log^{-B} x} |\Delta(\alpha \star \beta; a\ (d))| \ll x \log^{-A} x
\end{equation}
for any fixed $A >0$ and some $B$ depending on $A$.  From this and
dyadic decomposition, we conclude that to prove \eqref{alphabet-a}, it
suffices to establish the estimate
$$
 \sum_{d \in \DI{I}{i}{x^\delta} \cap [D,2D]} |\Delta(\alpha \star \beta; a\ (d))| \ll x \log^{-A} x
$$
for any fixed $A>0$, and for all $D$ such that
\begin{equation}\label{xpd}
  x^{1/2} \lessapprox D \lessapprox x^{1/2+2\varpi}
\end{equation}
(recall that this means $x^{1/2}\ll x^{o(1)}D$ and $D\ll
x^{1/2+2\varpi+o(1)}$ for any $\eps>0$).

We now fix one such $D$.  In the spirit of \cite{zhang}, we first restrict $d$ to moduli which do not have too many small prime
factors.  Precisely, let 
\begin{equation}\label{eq-d0}
D_0:=\exp(\log^{1/3} x),
\end{equation}
and let $\mathcal{E}(D)$ be the set of $d\in [D,2D]$ such that
\begin{equation}\label{eq-many-small}
\prod_{\substack{p|d\\ p\leq D_0}} p > \exp(\log^{2/3} x).
\end{equation}
\par
We have (cf. \cite[lemme 4]{ft}):

\begin{lemma}\label{roar} For any fixed $A>0$, and $D$ obeying \eqref{xpd}, we have
$$
|\mathcal{E}(D)|\ll D\log^{-A} x.
$$
\end{lemma}

\begin{proof}
  If $d\geq 1$ satisfies \eqref{eq-many-small}, then
$$
\prod_{\substack{p|d\\ p \leq D_0}} p > \exp( \log^{2/3} x ) =
D_0^{\log^{1/3} x}.
$$
In particular, $d$ has at least $\log^{1/3} x$ prime factors, and
therefore
$$ 
\tau(d) \geq 2^{\log^{1/3} x}.
$$
On the other hand, we have
$$
\sum_{\substack{D\leq d\leq 2D\\\tau(d)\geq \kappa}}1
\leq \frac{1}{\kappa}\sum_{D\leq d\leq 2D}{\tau(d)}
\ll \frac{D}{\kappa}\log x
$$
for any $\kappa>0$ by the standard bound
$$ 
\sum_{D \leq d \leq 2D} \tau(d) \ll D \log x
$$
(see \eqref{taud}), and the result follows.
\end{proof}

This allows us to dispose of these exceptional moduli:

\begin{corollary} We have
$$
\sum_{\substack{d \in \DI{I}{i}{x^\delta} \\ d\in \mathcal{E}(D)}}
|\Delta(\alpha \star \beta; a\ (d))| \ll x \log^{-A} x
$$
for any fixed $A>0$.
\end{corollary}

\begin{proof} 
  From \eqref{talc} we derive the trivial bound
$$ 
|\Delta(\alpha \star \beta; a\ (d))| \ll x D^{-1} \tau(d)^{O(1)}
\log^{O(1)} x,
$$ 
for every $d \sim D$, and hence the Cauchy-Schwarz inequality gives
$$
\sum_{\substack{d \in \DI{I}{i}{x^\delta}\\ d\in \mathcal{E}(D)}}
|\Delta(\alpha \star \beta; a\ (d))| \ll |\mathcal{E}(D)|^{1/2} xD^{-1}\log^{O(1)}x\Bigl(
 \sum_{d\in \mathcal{E}(D)} \tau(d)^{O(1)}
\Bigr)^{1/2} \ll x\log^{-A} x
$$
by Lemma~\ref{roar} and \eqref{taud}.
\end{proof}

It therefore suffices to show that
\begin{equation}\label{summa}
  \sum_{\substack{d \in \DI{I}{i}{x^\delta}\\d \in [D,2D] \backslash \mathcal{E}(D)}} 
  |\Delta(\alpha \star \beta; a\ (d))| \ll x \log^{-A} x
\end{equation}
for any fixed $A>0$.

Let $\eps>0$ be a small fixed quantity to be chosen later.  From
\eqref{slop} and \eqref{xpd} we have
$$
1\leq x^{-3\eps}N\leq D
$$
for $x$ large enough. Let $j\geq 0$ and $k\geq 0$ be fixed integers so that
\begin{equation}\label{waffle}
  i-1=j+k
\end{equation}
Then any integer $d\in \DI{I}{i}{x^{\delta}}$ can by definition (see
Definition \ref{mdd-def}) be factored as $d = qr$, where $q \in
\DI{I}{j}{x^\delta}$, $r \in \DI{I}{k}{x^\delta}$, and
$$
x^{-3\eps-\delta} N \leq r \leq x^{-3\eps} N.
$$

\begin{remark}
  The reason that $r$ is taken to be slightly less than $N$ is to
  ensure that a diagonal term is manageable when the time comes to
  apply the Cauchy-Schwarz inequality. The factor of $3$ in the
  exponent is merely technical, and should be ignored on a first
  reading ($\eps$ will eventually be set to be very small, so the
  constants in front of $\eps$ will ultimately be irrelevant).
\end{remark}

Let $d\ \in [D,2D] \backslash \mathcal{E}(D)$, so that
$$
s=\prod_{\substack{p|d\\ p\leq D_0}} p \lessapprox 1.
$$
Then replacing $q$ by $q/(q,s)$ and $r$ by $r(q,s)$, we obtain a
factorization $d=qr$ where $q$ has no prime factor $\leq D_0$ and
\begin{equation}\label{los}
  x^{-3\eps-\delta} N \lessapprox r \lessapprox x^{-3\eps} N.
\end{equation}
By Lemma \ref{fq}(0), (i), we have
$$
q\in \Dcal{j}{sx^{\delta}}=\Dcal{j}{x^{\delta+o(1)}},\quad\quad
r\in \Dcal{k}{sx^{\delta}}=\Dcal{k}{x^{\delta+o(1)}}.
$$
In particular, $q \in \DI{J}{j}{x^{\delta+o(1)}}$ where $J := I \cap
(D_0,+\infty)$.  As $i \geq 1$, we also have $qr = d \in
\DIone{I}{x^\delta}=\DI{I}{1}{x^{\delta}}$.

\begin{remark}
  The reason for removing all the small prime factors from $q$ will
  become clearer later, when the Cauchy-Schwarz inequality is invoked
  to replace the single parameter $q$ with two parameters $q_1,q_2$ in
  the same range. By excluding the small primes from $q_1,q_2$, this
  will ensure that $q_1$ and $q_2$ will almost always be coprime,
  which will make things much simpler.
\end{remark}

The next step is to perform dyadic decompositions of the range of the
$q$ and $r$ variables, which (in view of \eqref{lin-2}) reduces the
proof of \eqref{summa} to the proof of the estimates
$$
\sumsum_{\substack{q \in \DI{J}{j}{x^{\delta+o(1)}}\cap
    [Q,2Q]\\
    r \in \DI{I}{k}{x^{\delta+o(1)}} \cap [R,2R]\\
    qr \in \DIone{I}{x^\delta}} } |\Delta(\alpha \star \beta; a\
(qr))| \ll MN \log^{-A} x
$$
for any fixed $A>0$ and any $Q,R$ obeying the conditions
\begin{gather}
  x^{-3\eps-\delta} N \lessapprox R \lessapprox x^{-3\eps} N,
\label{crop}
\\
x^{1/2} \lessapprox QR \lessapprox x^{1/2+2\varpi}.  \label{crop2}
\end{gather}
We note that these inequalities also imply that
\begin{equation}\label{eq-nr}
  NQ\lessapprox x^{1/2+2\varpi+\delta+3\eps}.
\end{equation}
For future reference we also claim the bound
\begin{equation}\label{rq}
RQ^2 \ll x.
\end{equation}
In the cases (i)-(iii) of Theorem \ref{newtype-i-ii}, we have $\sigma+4\varpi+\delta < \frac{1}{2}$ (with plenty of room to spare), and \eqref{rq} then easily follows from \eqref{crop}, \eqref{crop2}, \eqref{slop}.  For case (i), we have $6\varpi+\delta < \frac{1}{2}$, and we may argue as before, but with \eqref{slop} replaced by the bound $N \gg  x^{1/2-2\varpi-c}$.

Let $Q,R$ be as above.  We will abbreviate
\begin{equation}\label{q-sum}
  \sum_q A_q= \sum_{q \in \DI{J}{j}{x^{\delta+o(1)}} \cap [Q,2Q]} A_q
\end{equation}
and
\begin{equation}\label{r-sum}
  \sum_r A_r= \sum_{r \in \DI{I}{k}{x^{\delta+o(1)}} \cap [R,2R]} A_r
\end{equation}
for any summands $A_q,A_r$.

We now split the discrepancy by writing
$$ 
\Delta(\alpha \star \beta; a\ (qr)) = 
\Delta_1(\alpha \star \beta; a\ (qr))+
\Delta_2(\alpha \star \beta; a\ (qr))
$$
where
\begin{align*}
  \Delta_1(\alpha \star \beta; a\ (qr))&:= \sum_{n = a\ (qr)} (\alpha
  \star \beta)(n) - \frac{1}{\phi(q)}
  \sum_{\substack{(n,q)=1\\ n = a\ (r)}} (\alpha \star \beta)(n),\\
  \Delta_2(\alpha \star \beta; a\ (qr))&:= \frac{1}{\phi(q)}
  \sum_{\substack{(n,q)=1\\n = a\ (r)}} (\alpha \star \beta)(n) -
  \frac{1}{\phi(qr)} \sum_{(n,qr)=1} (\alpha \star \beta)(n).
\end{align*}

The second term can be dealt with immediately: 

\begin{lemma} We have
$$
\sumsum_{q,r:qr \in \DIone{I}{x^\delta}}| \Delta_2(\alpha \star \beta;
a\ (qr))| \ll NM \log^{-A} x
$$
for any fixed $A>0$.
\end{lemma}

\begin{proof}
  Since $r\leq 2R\ll x^{1/2+o(1)-3\eps}$, the Bombieri-Vinogradov
  Theorem~\ref{bvt}, applied for each $q$ to $\alpha_q\star \beta_q$,
  where $\alpha_q=\alpha\onef_{(n,q)=1}$, $\beta_q=\onef_{(n,q)=1}$,
  gives
$$
\sum_{\substack{R\leq r\leq 2R\\qr \in \DIone{I}{x^\delta}}}
\Bigl|\sum_{\substack{(n,q)=1\\ n = a\ (r)}} (\alpha \star \beta)(n) -
\frac{1}{\phi(r)} \sum_{(n,qr)=1} (\alpha \star \beta)(n)\Bigr| \ll NM
\log^{-A} x,
$$
since $\beta_q$ inherits the Siegel-Walfisz property from
$\beta$. Dividing by $\phi(q)$ and summing over $q\leq 2Q$, we get the
result using the standard estimate 
$$
\sum_q \frac{1}{\phi(q)} \ll \log x.
$$
\end{proof}

To deal with $\Delta_1$, it is convenient to define
$$
\Delta_0(\alpha\star\beta;a,b_1,b_2)= \sum_\stacksum{n = a\ (r)}{ n=
  b_1\ (q)} (\alpha \star \beta)(n) - \sum_\stacksum{n = a\ (r)}{n =
  b_2\ (q)} (\alpha \star \beta)(n)
$$
for all integers $a$, $b_1$, $b_2$ coprime to $P_I$. Indeed, we have
$$
\sumsum_{\substack{q,r\\qr \in \DIone{I}{x^\delta}}} |\Delta_1(\alpha
\star \beta; a\ (qr))| \leq \frac{1}{\phi(P_I)} \sum_{\substack{b\
    (P_I)\\(b,P_I)=1}} \sumsum_{\substack{q,r\\qr \in
    \DIone{I}{x^\delta}}} |\Delta_0(\alpha\star\beta;a,a,b)|
$$
by the triangle inequality and the Chinese Remainder Theorem. Hence it
is enough to prove that
\begin{equation}\label{straw-2}
  \sumsum_{\substack{q,r\\qr \in \DIone{I}{x^\delta}}}|\Delta_0(\alpha\star\beta;a,b_1,b_2)|
  \ll  NM \log^{-A} x
\end{equation}
for all $a,b_1,b_2$ coprime to $P_I$, and this will be our goal.  The
advantage of this step is that the two terms in $\Delta_0$ behave
symmetrically, in contrast to those in $\Delta_1$ (or $\Delta$), and
this will simplify the presentation of the dispersion method: in the
notation of \cite{bfi,linnik,zhang}, one only
needs to control ${\mathcal S}_1$, and one avoids dealing explicitly
with ${\mathcal S}_2$ or ${\mathcal S}_3$.  This is mostly an
expository simplification, however, since the estimation of ${\mathcal
  S}_1$ is always the most difficult part in applications of the
dispersion method.

The fact that $r\leq R$ is slightly less than $N$ ensures that the
constraint $n = a\ (r)$ leaves room for non-trivial averaging of the
variable $n$, and allows us to profitably use the dispersion method of
Linnik.  We begin by writing
$$
\sumsum_{\substack{q,r\\qr \in
    \DIone{I}{x^\delta}}}|\Delta_0(\alpha\star\beta;a,b_1,b_2)| =
\sumsum_{\substack{q,r\\qr \in \DIone{I}{x^\delta}}} c_{q,r}
\Bigl(\sum_{\substack{n = a\ (r)\\ n= b_1\ (q)}} (\alpha \star
\beta)(n) - \sum_{\substack{n = a\ (r)\\ n = b_2\ (q)}} (\alpha
\star \beta)(n)\Bigr)
$$
where $c_{q,r}$ are complex numbers of modulus $1$.  Expanding the
Dirichlet convolution and exchanging the sums, we obtain
$$
\sumsum_{\substack{q,r\\qr \in
    \DIone{I}{x^\delta}}}|\Delta_0(\alpha\star\beta;a,b_1,b_2)| =
\sum_r \sum_m \alpha(m) \Bigl(\sumsum_{\substack{mn = a\ (r)\\qr \in
    \DIone{I}{x^\delta}}} c_{q,r} \beta(n) (\onef_{mn = b_1\ (q)} -
\onef_{mn = b_2\ (q)})\Bigr).
$$ 
By the Cauchy-Schwarz inequality applied to the $r$ and $m$
sums,~(\ref{alpha-bound}),~(\ref{soso}) and Lemma \ref{divisor-crude}
we have
\begin{multline*}
  \sumsum_{\substack{q,r\\qr \in
      \DIone{I}{x^\delta}}}|\Delta_0(\alpha\star\beta;a,b_1,b_2)| \leq
  R^{1/2}M^{1/2}(\log x)^{O(1)} \Bigl(\sum_{r}\sum_{m} \psi_M(m)\\
  \times \Bigl| \sumsum_{\substack{mn = a\ (r)\\qr \in
      \DIone{I}{x^\delta}}} c_{q,r} \beta(n) (\onef_{mn = b_1\ (q)} -
  \onef_{mn = b_2\ (q)}) \Bigr|^2\Bigr)^{1/2}
\end{multline*}
for any smooth coefficient sequence $\psi_M$ at scale $M$ such that
$\psi_M(m)\geq 1$ for $m$ in the support of $\beta$. This means in
particular that it is enough to prove the estimate
\begin{equation}\label{sq}
  \sum_r \sum_{m} \psi_M(m) \Bigl|\sumsum_{\substack{ mn = a\ (r)\\qr
      \in \DIone{I}{x^\delta}}} c_{q,r} \beta(n) (\onef_{mn = b_1\
    (q)} - \onef_{mn = b_2\ (q)})\Bigr|^2
  \ll N^2 M R^{-1} \log^{-A} x
\end{equation}
for any fixed $A>0$, where $\psi_M$ is a smooth coefficient sequence
at scale $M$.  
\par
Let $\Sigma$ denote the left-hand side of~(\ref{sq}). Expanding the
square, we find
\begin{equation}\label{eq-sigma1}
\Sigma=\Sigma(b_1,b_1)-\Sigma(b_1,b_2)-\Sigma(b_2,b_1)+\Sigma(b_2,b_2),
\end{equation}
where
$$
\Sigma(b_1,b_2):=\sum_r \sum_{m} \psi_M(m)
\multsum_{\substack{q_1,q_2,n_1,n_2\\ mn_1=mn_2 = a\ (r)\\ q_1 r,q_2 r
    \in \DIone{I}{x^\delta}}} c_{q_1,r} \overline{c_{q_2,r}}
\beta(n_1) \overline{\beta(n_2)} \onef_{mn_1 = b_1\ (q_1)} \onef_{mn_2
  = b_2\ (q_2)}
$$
for any integers $b_1$ and $b_2$ coprime to $P_I$ (where the variables
$q_1$ and $q_2$ are subject to the constraint~(\ref{q-sum})). We will
prove that
\begin{equation}\label{sq2}
  \Sigma(b_1,b_2)= X + O( N^2 M R^{-1} \log^{-A} x )
\end{equation}
for all $b_1$ and $b_2$, where the main term $X$ is independent of
$b_1$ and $b_2$. From~(\ref{eq-sigma1}), the desired
conclusion~(\ref{sq}) then follows.

Since $a$ is coprime to $qr$, so are the variables $n_1$ and $n_2$ in
the sum. In particular, they satisfy the congruence $n_1=n_2\ (r)$. We
write $n_2=n_1+\ell r$ in the sum, rename $n_1$ as $n$, and therefore
obtain
\begin{multline*}
  \Sigma(b_1,b_2)=\sum_r\sum_{\ell}
  \sumsum_{\substack{q_1,q_2\\q_1r,q_2r\in \DIone{I}{x^{\delta}}}}
  c_{q_1,r}\overline{c_{q_2,r}} \sum_n\beta(n)\overline{\beta(n+\ell
    r)} \\\sum_{m} \psi_M(m)\onef_{mn = b_1\ (q_1)} \onef_{m(n+\ell r)
    = b_2\ (q_2)}\onef_{mn=a\ (r)}
\end{multline*}
after some rearranging (remembering that $(n,q_1r)=(n+\ell r,q_2 r)=1$).
Note that the sum over $\ell$ is restricted to a range $0\leq
|\ell|\ll L:=NR^{-1}$.  
\par
We will now complete the sum in $m$ (which is long since $M$ is just a
bit smaller than the modulus $[q_1,q_2]r\leq Q^2R$) using
Lemma~\ref{com} (ii), but first we handle separately the diagonal case
$n_1=n_2$, i.e., $\ell=0$. This contribution, say $T(b_1,b_2)$, satisfies
\begin{align*}
  |T(b_1,b_2)|&\leq \sum_r \sumsum_{\substack{q_1,q_2\\q_1r,q_2r\in
      \DIone{I}{x^{\delta}}}} \sum_n|\beta(n)|^2 \sum_{m}
  \psi_M(m)\onef_{mn = b_1\ (q_1)}
  \onef_{mn = b_2\ (q_2)}\onef_{mn=a\ (r)}\\
  &\lessapprox \sum_{r\sim R}\sumsum_{q_1,q_2\sim Q} \sum_{s\sim x}
  \tau(s)\onef_{s = b_1\ (q_1)} \onef_{s = b_2\ (q_2)}\onef_{s=a\
    (r)}\\
  &\lessapprox \sum_{r\sim R}\sumsum_{q_1,q_2\sim Q} \frac{x}{r[q_1,q_2]}\lessapprox x\ll N^2 M R^{-1} \log^{-A} x
\end{align*}
(since $RQ^2\ll x$ (from \eqref{rq}) and $R\lessapprox x^{-3\epsilon }N$).
\par
Now we consider the contributions where $\ell \neq 0$. First, since $n$
and $n+\ell r$ are coprime to $q_1r$ and $q_2r$ respectively, we have
\begin{equation}\label{eq-gamma}
\onef_{mn = b_1\ (q_1)} \onef_{m(n+\ell r) = b_2\ (q_2)}
\onef_{mn=a\ (r)}=\onef_{m=\gamma\ ([q_1,q_2]r)}
\end{equation}
for some residue class $\gamma\ ([q_1,q_2]r)$ (which depends on $b_1$,
$b_2$, $\ell$, $n$ and $a$). We will denote $q_0=(q_1,q_2)$, and observe
that since $q_1$, $q_2$ have no prime factor less than $D_0$, we have
either $q_0=1$ or $q_0\geq D_0$. (The first case gives the
principal contribution, and the reader may wish to assume that $q_0=1$
in a first reading.) The sum over $n$ is further restricted by the
congruence
\begin{equation}\label{eq-n-congruence}
\frac{b_1}{n}=\frac{b_2}{n+\ell r}\ (q_0),
\end{equation}
and we will use
\begin{equation}\label{cn-def}
C(n):=\onef_{\frac{b_1}{n}=\frac{b_2}{n+\ell r}\ (q_0)}
\end{equation}
to denote
the characteristic function of this condition (taking care of the fact that it
depends on other parameters). Observe that, since $q_0$ is coprime to
$rb_1$, this is the characteristic function of a union of at most
$(b_1-b_2,q_0,\ell rb_1)\leq (q_0,\ell)$ congruence classes modulo
$q_0$.
\par 
By Lemma~\ref{com} (ii) applied to each choice of $q_1,q_2,r,\ell$ (where $I$ is the range of the remaining parameter $n$) and summing, we derive
$$
\Sigma(b_1,b_2)=\Sigma_0(b_1,b_2)+\Sigma_1(b_1,b_2)+
O(MN^2R^{-1}\log^{-A}x),
$$
where
$$
  \Sigma_0(b_1,b_2):= \Bigl(\sum_m\psi_M(m)\Bigr)
  \sum_rr^{-1}\sum_{\ell\neq 0}
  \sumsum_{\substack{q_1,q_2\\q_1r,q_2r\in
      \DIone{I}{x^{\delta}}}} \frac{c_{q_1,r}\overline{c_{q_2,r}}}{[q_1,q_2]}\\
  \sum_n\beta(n)\overline{\beta(n+\ell r)}C(n)
$$
and
$$
\Sigma_1(b_1,b_2)\ll 1+ x^{\eps} \hat{\Sigma}_1(b_1,b_2) 
$$
with
$$
\hat{\Sigma}_1(b_1,b_2):= \sum_r\sum_{\ell\neq 0}
\sumsum_{\substack{q_1,q_2\\q_1r,q_2r\in \DIone{I}{x^{\delta}}}}
c_{q_1,r}\overline{c_{q_2,r}}\ \frac{1}{H} \sum_{1\leq |h|\leq
  H}\left|\sum_n\beta(n)\overline{\beta(n+\ell r)}C(n)e_{[q_1,q_2]r}(\gamma
h)\right|,
$$
where $H:=x^{\eps}[q_1,q_2]rM^{-1}\ll x^{\eps}Q^2RM^{-1}$.  We caution that $H$ depends on $q_1$ and $q_2$, so one has to take some care if one is to interchange the $h$ and $q_1,q_2$ summations.

\begin{remark} 
  Before going further, note that $H$ is rather small since $M$ and
  $R$ are close to $x^{1/2}$ and $\eps>0$ will be very small:
  precisely, we have
$$ 
H \ll H_0 :=x^\eps \times (QR)^2 \times \frac{N}{R} \times \frac{1}{NM}
$$
and using \eqref{crop}, \eqref{crop2}, \eqref{lin-2}, we see that
\begin{equation}\label{ao}
  x^{4\eps}\lessapprox 
  H_0 \lessapprox x^{4\varpi+\eps} (N/R) \lessapprox x^{4\varpi + \delta + 4\eps}.
\end{equation}
As we will be using small values of $\varpi,\delta,\eps$, one should
thus think of $H$ as being quite small compared to $x$.
\end{remark}

We can deal immediately with $\Sigma_0(b_1,b_2)$. We distinguish
between the contributions of $q_1$ and $q_2$ which are coprime, and
the remainder. The first is independent of $b_1$ and $b_2$ (since
these parameters are only involved in the factor
$C(n)=\onef_{b_1/n=b_2/(n+\ell r)\ (q_0)}$, which is then always $1$)
and it will be the main term $X$, thus
$$
  X := \Bigl(\sum_m\psi_M(m)\Bigr)
  \sum_rr^{-1}\sum_{\ell\neq 0}
  \sumsum_{\substack{q_1,q_2\\q_1r,q_2r\in
      \DIone{I}{x^{\delta}} \\ (q_1,q_2)=1}} \frac{c_{q_1,r}\overline{c_{q_2,r}}}{[q_1,q_2]}\\
  \sum_n\beta(n)\overline{\beta(n+\ell r)}.
$$
\par
The remaining contribution to $\Sigma_0(b_1,b_2)$, say $\Sigma'_0(b_1,b_2)$, satisfies
$$
  \Sigma'_0(b_1,b_2)\ll \frac{M(\log x)^{O(1)}}{R} \sum_{r\sim R}
  \sum_{|\ell|\ll L} \sum_{\substack{1\neq q_0\ll Q\\ q_0 \in
      {\mathcal S}_{J}}}\frac{1}{q_0}\sumsum_{q_1,q_2 \sim Q/q_0}
  \frac{1}{q_1q_2}
  \sum_{n}(\tau(n)\tau(n+\ell r))^{O(1)}C(n).
$$
We rearrange to sum over $\ell$ first (remember that $C(n)$ depends on
$\ell$ also). Since $rb_1$ is coprime with $q_0$, the condition
$b_1/n=b_2/(n+\ell r)\ (q_0)$ is a congruence condition modulo $q_0$
for $\ell$, and therefore
$$
\sum_{|\ell|\ll L} \tau(n+\ell r)^{O(1)}\onef_{b_1/n=b_2/(n+\ell r)\
  (q_0)}\ll \Bigl(1+\frac{L}{q_0}\Bigr)\log^{O(1)} x
=\Bigl(1+\frac{N}{q_0R}\Bigr)\log^{O(1)} x
$$
by Lemma~\ref{divisor-crude}. Since all $q_0\neq 1$ in the sum satisfy
$D_0\leq q_0\ll Q$, we get
\begin{align*}
  \Sigma'_0(b_1,b_2)&\ll \frac{MN(\log x)^{O(1)}}{R} \sum_{r\sim R}
  \sum_{D_0\leq q_0\ll Q}\frac{1}{q_0}\Bigl(1+\frac{N}{q_0R}\Bigr)
  \sumsum_{q_1,q_2 \sim Q/q_0}  \frac{1}{q_1q_2}\\
  &\ll MN\log^{O(1)} x\sum_{D_0\leq q_0\ll
    Q}\frac{1}{q_0}\Bigl(1+\frac{N}{q_0R}\Bigr)\\
    &\ll MN\log^{O(1)}x +\frac{1}{D_0}\frac{MN^2}{R}\log^{O(1)} x\\
    &\ll MN^2R^{-1}\log^{-A} x,
\end{align*}
since $R\ll x^{-3\eps}N$ and $D_0\gg \log^A x$ for all $A>0$.

Hence we have shown that
\begin{equation}\label{eq-last-1}
\Sigma(b_1,b_2)=X+
O(x^{\eps}|\hat{\Sigma}_1(b_1,b_2)|)+O(MN^2R^{-1}\log^{-A} x).
\end{equation}
From the definition, and in particular the localization of $r$ and the
value of $H$, we have
\begin{align}
  |\hat{\Sigma}_1(b_1,b_2)|&\leq\sum_r\sum_{\ell\neq 0}
  \sumsum_{\substack{q_1,q_2\\q_1r,q_2r\in \DIone{I}{x^{\delta}}}}
  \frac{1}{H}\sum_{0<|h|\leq H}
  \Bigl|\sum_nC(n)\beta(n)\overline{\beta(n+\ell
    r)}e_{[q_1,q_2]r}(\gamma
  h)\Bigr|\nonumber\\
  &\ll x^{-\eps}\frac{M}{RQ^2}\sum_{1\leq |\ell|\ll
    L}\sum_{q_0\ll Q}q_0\sum_r\Upsilon_{\ell,r}(b_1,b_2;q_0)\label{eq-last}
\end{align}
where $q_0$ is again $(q_1,q_2)$ and 
\begin{multline}\label{eq-upsilon}
  \Upsilon_{\ell,r}(b_1,b_2;q_0):= \sumsum_{\substack{q_1,q_2\sim
      Q/q_0\\(q_1,q_2)=1}} \onef_{\substack{q_0q_1,q_0q_2
      \in\DI{I}{j}{x^{\delta+o(1)}} \\ q_0q_1 r,q_0q_2 r \in
      \DIone{I}{x^\delta} }}\\
  \sum_{1\leq |h|\ll \tfrac{x^{\eps}RQ^2}{q_0M}} \Bigl|\sum_{n} C(n)
  \beta(n)\overline{\beta(n+\ell r)}
  \Phi_{\ell}(h,n,r,q_0,q_1,q_2)\Bigr|.
\end{multline}
The latter expression involves the phase function $\Phi_{\ell}$, which we define for parameters
$\uple{p}=(h,n,r,q_0,q_1,q_2)$ by
\begin{equation}\label{eq-phiell}
  \Phi_{\ell}(\uple{p}):= e_r\left( \frac{ah}{nq_0q_1 q_2}
  \right) e_{q_0q_1}\left( \frac{b_1h}{n r q_2} \right) e_{q_2}\left(
    \frac{b_2 h}{(n+\ell r)  rq_0q_1} \right).
\end{equation}
Here we have spelled out and split, using~(\ref{eq-gamma}) and the
Chinese Remainder Theorem, the congruence class of $\gamma$ modulo
$[q_1,q_2]r$, and changed variables so that $q_1$ is $q_0q_1$, $q_2$
is $q_0q_2$ (hence $[q_1,q_2]r$ becomes $q_0q_1q_2r$).  Moreover, the
$r$ summation must be interpreted using \eqref{r-sum}. It will be
important for later purposes to remark that we also have
$$
\hat{\Sigma}_1(b_1,b_2)=0
$$
unless
\begin{equation}\label{eq-h-condition}
\frac{ x^\eps Q^2 R}{q_0M}\gg 1,
\end{equation}
since otherwise the sum over $h$ is empty.
\par

Gathering these estimates, we obtain the following general reduction
statement, where we pick a suitable value of $(j,k)$ in each of the
four cases of Theorem~\ref{newtype-i-ii}: 




\begin{theorem}[Exponential sum estimates]\label{eset}  Let
  $\varpi,\delta,\sigma > 0$ be fixed quantities, let $I$ be a bounded subset
  of $\R$, let $j$, $k\geq 0$ be fixed, let $a\ (P_I)$, $b_1\ (P_I)$,
  $b_2\ (P_I)$ be primitive congruence classes, and let $M,N \gg 1$ be
  quantities satisfying the conditions \eqref{lin-2} and
  \eqref{slop}. Let $\eps>0$ be a sufficiently small fixed quantity,
  and let $Q,R$ be quantities obeying \eqref{crop}, \eqref{crop2}. Let
  $\ell$ be an integer with $1\leq |\ell|\ll N/R$, and let $\beta$ be a
  coefficient sequence located at scale $N$.
\par
Let $\Phi_{\ell}(\uple{p})$ be the phase function defined by
\eqref{eq-phiell} for parameters $\uple{p}=(h,n,r,q_0,q_1,q_2)$, let $C(n)$ be the cutoff \eqref{cn-def} and let $\Upsilon_{\ell,r}(b_1,b_2;q_0)$ be defined in terms of $\beta, \Phi, C$ by \eqref{eq-upsilon}.  Then
we have
\begin{equation}\label{keyo}
  \sum_{r}\Upsilon_{\ell,r}(b_1,b_2;q_0)
  \lessapprox x^{-\eps} Q^2R N  (q_0,\ell)q_0^{-2}
\end{equation}
for all $q_0 \in \Scal_I$, where the sum over $r$ is over
$r\in \DI{I}{k}{x^{\delta+o(1)}} \cap [R,2R]$, provided that one of
the following hypotheses is satisfied:
\begin{enumerate}[(i)]
\item\label{typeI1-yetagain} $(j,k) = (0,0)$, $54\varpi + 15\delta + 5 \sigma < 1$, and $N \lessapprox x^{1/2-2\varpi-c}$ for some fixed $c>0$.
\item\label{typeI2-yetagain} $(j,k) = (1,0)$, $56\varpi + 16\delta + 4 \sigma < 1$, and $N \lessapprox x^{1/2-2\varpi-c}$ for some fixed $c>0$.
\item\label{typeI4-yetagain} $(j,k) = (1,2)$, $\frac{160}{3} \varpi + 16\delta + \frac{34}{9} \sigma < 1$, $64\varpi + 18 \delta + 2\sigma < 1$, and $N \lessapprox x^{1/2-2\varpi-c}$ for some fixed $c>0$.
\item\label{typeII1-yetagain} $(j,k)=(0,0)$, $68\varpi + 14 \delta < 1$, and $N \gtrapprox x^{1/2-2\varpi-c}$ for some sufficiently small fixed $c>0$.
\end{enumerate}
The proof of the estimate \eqref{typeI4-yetagain} requires
Deligne's form of the Riemann Hypothesis for algebraic varieties over
finite fields, but the proofs of \eqref{typeI1-yetagain},
\eqref{typeI2-yetagain}, \eqref{typeII1-yetagain} do not.
\end{theorem}

Indeed, inserting this bound in~(\ref{eq-last}) we obtain
$$
x^{\eps}|\hat{\Sigma}(b_1,b_2)|\lessapprox x^{-\eps}MN\sum_{q_0\ll Q}
\frac{1}{q_0} \sum_{1\leq |\ell|\ll NR^{-1}}(q_0,\ell) \lessapprox
x^{-\eps}MN^2R^{-1}
$$
(by Lemma~\ref{ram-avg}, crucially using the fact that we have previously removed the $\ell=0$ contribution), and hence using~(\ref{eq-last-1}), we derive
the goal~(\ref{sq2}).

\begin{remark} As before, one should consider the $q_0=1$ case as the main case, so that the technical factors of $q_0$, $(\ell,q_0)$, and $C(n)$ should be ignored at a first reading; in practice, we will usually (though not always) end up discarding several powers of $q_0$ in the denominator in the final bounds for the $q_0>1$ case.
The trivial bound for $\Upsilon_{\ell,r}(b_1,b_2;q_0)$
  is about $ (Q/q_0)^2 N H$ with $H=x^{\eps}RQ^2M^{-1}q_0^{-1}$.  Thus
  one needs to gain about $H$ over the trivial bound.  As observed
  previously, $H$ is quite small, and even a modestly non-trivial
  exponential sum estimate can suffice for this purpose (after using
  Cauchy-Schwarz to eliminate factors such as $\beta(n)
  \overline{\beta(n+\ell r)}$).
\end{remark}

It remains to establish Theorem \ref{eset} in the four cases
indicated.  We will do this for \eqref{typeI1-yetagain},
\eqref{typeI2-yetagain}, \eqref{typeII1-yetagain} below, and defer the
proof of \eqref{typeI4-yetagain} to Section \ref{typei-advanced-sec}.
In all four cases, one uses the Cauchy-Schwarz inequality to
eliminate non-smooth factors such as $\beta(n)$ and $\beta(n+\ell r)$, and
reduces matters to incomplete exponential sum estimates.  In the cases
\eqref{typeI1-yetagain}, \eqref{typeI2-yetagain}, \eqref{typeII1-yetagain} treated below, the one-dimensional exponential sum estimates from Section
\ref{incomplete-exp-sec} suffice; for the final case \eqref{typeI4-yetagain}, a
multidimensional exponential sum estimate is involved, and we will
prove it using Deligne's formalism of the Riemann Hypothesis over
finite fields, which we survey in Section~\ref{deligne-sec}.

\subsection{Proof of Type II estimate}\label{typeii-sec}

We begin with the proof of Theorem \ref{eset}\eqref{typeII1-yetagain},
which is the simplest of the four estimates to prove. We fix notation
and hypotheses as in this statement. 

To prove \eqref{keyo}, we will not exploit any averaging in the 
variable $r$, and more precisely, we will show that
\begin{equation}\label{eq-iv-target}
\Upsilon_{\ell,r}(b_1,b_2;q_0)
\lessapprox x^{-\eps} Q^2 N (q_0,\ell)q_0^{-2}
\end{equation}
for each $q_0 \geq 1$, $r \sim R$ and $\ell\ll N/R$.  We abbreviate
$\Upsilon=\Upsilon_{\ell,r}(b_1,b_2;q_0)$ in the remainder of this
section and denote
$$
H=x^{\eps}RQ^2M^{-1}q_0^{-1}.
$$
By \eqref{eq-upsilon}, we can then write
\begin{equation}\label{eq-start}
\Upsilon= \sumsum_{\substack{q_1,q_2 \sim Q/q_0\\ (q_1,q_2) = 1}} \sum_{1
  \leq |h| \leq H}c_{h,q_1,q_2} \sum_{n} C(n)\beta(n)
\overline{\beta(n+\ell r)} \Phi_{\ell}(h,n, r, q_0, q_1,q_2)
\end{equation}
for some coefficients $c_{h,q_1,q_2}$ with modulus at most $1$. We then
exchange the order of summation to move the sum over $n$ (and the terms $C(n) \beta(n) \overline{\beta(n+\ell r)}$) outside. Since $C(n)$
is the characteristic function of at most $(q_0,\ell)$ congruence
classes modulo $q_0$ (as observed after~(\ref{eq-n-congruence})), we
have
\begin{equation}\label{baz}
  \sum_n C(n)|\beta(n)|^2 |\beta(n+\ell r)|^2 \lessapprox 
  N\frac{(q_0,\ell)}{q_0}
\end{equation}
by~(\ref{divisor-crude}) (and the Cauchy-Schwarz inequality), using
the fact that $Q\leq N$.
\par
By another application of the Cauchy-Schwarz inequality, and after inserting (by positivity) a suitable coefficient sequence
$\psi_N(n)$, smooth at scale $N$ and $\geq 1$ for $n$ in the support
of $\beta(n)\overline{\beta(n+\ell r)}$, 
we conclude the bound
\begin{align*}
  |\Upsilon|^2&\lessapprox N\frac{(q_0,\ell)}{q_0}
  \sum_{n}\psi_N(n)C(n) \Bigl| \sumsum_{\substack{q_1,q_2 \sim Q/q_0\\
      (q_1,q_2) = 1}} \sum_{1 \leq |h|
    \leq H} c_{h,q_1,q_2} \Phi_{\ell}(h,n, r, q_0, q_1,q_2)\Bigr|^2\\
  &\lessapprox N\frac{(q_0,\ell)}{q_0}
  \multsum_{\substack{q_1,q_2,s_1,s_2\sim Q/q_0\\
      (q_1,q_2)=(s_1,s_2)=1 }} \sumsum_{1\leq h_1,h_2\leq |H|}
  |S_{\ell,r}(h_1,h_2,q_1,q_2,s_1,s_2)|,
\end{align*}
where the exponential sum  $S_{\ell,r}=S_{\ell,r}(h_1,h_2,q_1,q_2,s_1,s_2)$ is given by
\begin{equation}\label{S-def}
  S_{\ell,r}:=  \sum_n C(n)\psi_N(n) \Phi_{\ell}(h_1,n, r, q_0,q_1,q_2)
  \overline{\Phi_{\ell}( h_2,n, r,q_0,s_1,s_2)}.
\end{equation}

We will prove below the following estimate for this exponential sum
(compare with \cite[(12.5)]{zhang}):

\begin{proposition}\label{exse} For any
$$
\uple{p}=(h_1,h_2,q_1,q_2,s_1,s_2)
$$
with $(q_0q_1q_2s_1s_2,r)=1$, any $\ell \neq 0$ and $r$ as above with
$$
q_0q_i,\ q_0s_i\ll Q,\quad\quad r\ll R,
$$
we have
$$ 
|S_{\ell,r}(\uple{p})| \lessapprox (q_0,\ell)\Bigl( q_0^{-2} Q^2
R^{1/2} + \frac{N}{q_0R} (h_1s_1s_2-h_2q_1q_2,r)\Bigr).
$$
\end{proposition}

Assuming this, we obtain
$$
|\Upsilon|^2\lessapprox N\Bigl(\frac{(q_0,\ell)}{q_0}\Bigr)^2
\multsum_{\substack{q_1,q_2,s_1,s_2\sim Q/q_0\\
    (q_1,q_2)=(s_1,s_2)=1 }} \sumsum_{1\leq h_1,h_2\leq |H|}
\Bigl(\frac{1}{q_0}Q^2 R^{1/2} +
\frac{N}{R}(h_1s_1s_2-h_2q_1q_2,r)\Bigr)
$$
(since $S_{\ell,r}=0$ unless $(q_0q_1q_2s_1s_2,r)=1$, by the
definition~(\ref{eq-phiell}) and the definition of $e_q$ in
Section~\ref{exp-sec}).
\par
Making the change of variables $\Delta = h_1 s_1 s_2 - h_2 q_1 q_2$, and noting that each $\Delta$ has at most $\tau_3(\Delta) = |\{ (a,b,c): abc = \Delta \}|$ representations in terms of $h_2,q_1,q_2$ for each fixed $h_1,s_1,s_2$, we have
\begin{align*}
  \multsum_{\substack{q_1,q_2,s_1,s_2\sim Q/q_0\\
      (q_1,q_2)=(s_1,s_2)=1 }} \sumsum_{1\leq h_1,h_2\leq |H|}
  (h_1s_1s_2-h_2q_1q_2,r) &\leq \sum_{|\Delta|\ll
    H(Q/q_0)^2}(\Delta,r) \multsum_{h_1,s_1,s_2}
  \tau_3(h_1s_1s_2-\Delta)\\
  &\lessapprox H\Bigl(\frac{Q}{q_0}\Bigr)^2\sum_{0\leq |\Delta|\ll
    H(Q/q_0)^2}(\Delta,r)\\
&\lessapprox H\Bigl(\frac{Q}{q_0}\Bigr)^2\Bigl(\frac{HQ^2}{q_0^2}+R\Bigr)
\end{align*}
by Lemma~\ref{divisor-crude} (bounding $\tau_3 \leq \tau^2$) and
Lemma~\ref{ram-avg}. Therefore we obtain
\begin{align}
  |\Upsilon|^2&\lessapprox N\frac{(q_0,\ell)^2}{q_0^2}\Bigl\{
  \frac{H^2Q^2R^{1/2}}{q_0}\Bigl(\frac{Q}{q_0}\Bigr)^4
  +\frac{H^2N}{R}\Bigl(\frac{Q}{q_0}\Bigr)^4+
  NH\Bigl(\frac{Q}{q_0}\Bigr)^2
  \Bigr\}\nonumber\\
  &\lessapprox \frac{N^2Q^4(q_0,\ell)^2}{q_0^4} \Bigl\{
  \frac{H^2Q^2R^{1/2}}{N}+\frac{H^2}{R}+\frac{H}{Q^2} \Bigr\}
  \nonumber\\
  &\lessapprox \frac{N^2Q^4(q_0,\ell)^2}{q_0^4} \Bigl\{
  x^{2\eps}\frac{Q^6R^{5/2}}{M^2N}+x^{2\eps}\frac{RQ^4}{M^2}+\frac{x^{\eps}R}{M}
  \Bigr\}
\label{eq-iv-final}
\end{align}
where we have discarded some powers of $q_0 \geq 1$ in the denominator to reach the second and third lines.
We now observe that
\begin{gather*}
  \frac{Q^6R^{5/2}}{M^2N} \sim \frac{(NQ)(QR)^5}{x^2R^{5/2}}
  \lessapprox \frac{x^{1+12\varpi+\delta+3\eps}}{R^{5/2}} \lessapprox
  \frac{x^{1+12\varpi+7\delta/2+21\eps/2}}{N^{5/2}}
  \\
  \frac{Q^4R}{M^2}\sim \frac{N^2RQ^4}{x^2}=\frac{(QR)(NQ)^3}{x^2N}
  \lessapprox \frac{x^{8\varpi +3\delta+9\eps}}{N}
  \\
  \frac{R}{M}\sim \frac{NR}{x}\lessapprox x^{-1-3\eps}N^2\lessapprox
  x^{-3\eps}
\end{gather*}
by~(\ref{crop2}) and~(\ref{eq-nr}) and $N\lessapprox M$. Under the
Type II assumption that $N\ggcurly x^{1/2-2\varpi-c}$ for a small
enough $c>0$ and that $\eps>0$ is small enough, we see
that~(\ref{eq-iv-final}) implies~(\ref{eq-iv-target}) provided
$\varpi$ and $\delta$ satisfy
$$
\begin{cases}
1+12\varpi+\frac{7\delta}{2}<\frac{5}{2}(\frac{1}{2}-2\varpi)\\
8\varpi+3\delta<\frac{1}{2}-2\varpi
\end{cases}
\Leftrightarrow\quad
\begin{cases}
68\varpi+14\delta<1\\
20\varpi+6\delta<1,
\end{cases}
$$
both of which are, indeed, consequences of the hypotheses of
Theorem~\ref{eset} (iv) (the first implies the second because
$\varpi>0$ so $\delta<1/14$).

To finish this treatment of the Type II sums, it remains to prove the
proposition.

\begin{proof}[Proof of Proposition~\ref{exse}]
 For fixed $(r,\ell,q_0,a, b_1,b_2)$ we can use \eqref{eq-phiell} to express the phase
  $\Phi_{\ell}$ in the form
$$
\Phi_{\ell}(h,n,r,q_0,q_1,q_2)=
\uple{e}^{(1)}_r\Bigl(\frac{h}{q_1q_2n}\Bigr)
\uple{e}^{(2)}_{q_0q_1}\Bigl(\frac{h}{nq_2}\Bigr)
\uple{e}^{(3)}_{q_2}\Bigl(\frac{h}{(n+\tau)q_0q_1}\Bigr)
$$
where $\uple{e}^{(i)}_{d}$ denotes various non-trivial additive characters
modulo $d$ which may depend on $(r,\ell,q_0,a,b_1,b_2)$ and $\tau=\ell
r$. 
\par
We denote $\Phi_1(n)=\Phi_{\ell}(h_1,n,r,q_0,q_1,q_2)$ and
$\Phi_2(n)=\Phi_{\ell}(h_2,n,r,q_0,s_1,s_2)$, and thus we have
\begin{multline}\label{eq-phi1-phi2}
  \Phi_1(n)\overline{\Phi_2(n)}
  =\uple{e}_r^{(1)}\Bigl(\frac{h_1}{q_1q_2n}-\frac{h_2}{s_1s_2n}\Bigr)
  \uple{e}_{q_0q_1}^{(2)}\Bigl(\frac{h_1}{nq_2}\Bigr)
  \uple{e}_{q_0s_1}^{(2)}\Bigl(-\frac{h_2}{ns_2}\Bigr)\\
  \uple{e}_{q_2}^{(3)}\Bigl(\frac{h_1}{(n+\tau)q_0q_1}\Bigr)
  \uple{e}_{s_2}^{(3)}\Bigl(-\frac{h_2}{(n+\tau)q_0s_1}\Bigr)
\end{multline}
and this can be written
$$
\Phi_1(n)\overline{\Phi_2(n)} =\uple{e}^{(4)}_{d_1} \Bigl( \frac{c_1}{n}
\Bigr) \uple{e}_{d_2}^{(5)}\Bigl( \frac{c_2}{n+\tau} \Bigr)
$$
where
\begin{gather*}
  d_1 := r q_0 [q_1,s_1],\quad\quad
  d_2 := [q_2,s_2] 
\end{gather*}
for some $c_1$ and $c_2$. 
\par
Now, since $C(n)$ is the characteristic function
of $\leq (q_0,\ell)$ residue classes modulo $q_0$, we deduce
$$
|S_{\ell,r}|=
\Bigl|\sum_{n}C(n)\psi_N(n)\Phi_1(n)\overline{\Phi_2(n)}\Bigr|
\leq (q_0,\ell)
\max_{t\in \Z/q_0\Z}
\Bigl|\sum_{n=t\ (q_0)} 
\psi_N(n)\Phi_1(n)\overline{\Phi_2(n)}\Bigr|,
$$
and by the second part of Corollary \ref{dons}, we derive
$$
|S_{\ell,r}|\lessapprox (q_0,\ell)
\Bigl(\frac{[d_1,d_2]^{1/2}}{q_0^{1/2}}+
\frac{N}{q_0}\frac{(c_1,\delta'_1)}{\delta'_1}\frac{(c_2,\delta'_2)}{\delta'_2}
\Bigr) \lessapprox (q_0,\ell)
\Bigl(R^{1/2}\Bigl(\frac{Q}{q_0}\Bigr)^2+
\frac{N}{q_0}\frac{(c_1,\delta'_1)}{\delta'_1} \Bigr)
$$
where $\delta_i=d_i/(d_1,d_2)$ and $\delta'_i = \delta_i/(q_0,\delta_i)$, since
$$
[d_1,d_2]\leq rq_0q_1q_2s_1s_2\ll
q_0R\Bigl(\frac{Q}{q_0}\Bigr)^4,\quad\quad
\frac{(c_2,\delta'_2)}{\delta'_2}\leq 1.
$$
\par
Finally, we have
$$ 
\frac{(c_1,\delta'_1)}{\delta'_1}=\prod_{\substack{p\mid
    \delta_1\\p\nmid c_1, q_0}}p \leq \frac{(c_1,r)}{r}
$$
(since $r\mid \delta_1$ and $(r,q_0)=1$). But a prime $p\mid r$ divides $c_1$
precisely when the $r$-component of~(\ref{eq-phi1-phi2}) is constant,
which happens exactly when $p\mid h_1s_1s_2-h_2q_1q_2$, so that
$$
S_{\ell,r}\lessapprox (q_0,\ell)R^{1/2}\Bigl(\frac{Q}{q_0}\Bigr)^2+
\frac{(q_0,\ell)N}{q_0R}(r,h_1s_1s_2-h_2q_1q_2).
$$
\end{proof}


\begin{remark} By replacing the lower bound $N \gtrapprox
  x^{1/2-2\varpi-c}$ with the lower bound $N \gtrapprox
  x^{1/2-\sigma}$, the above argument also yields the estimate
  $\TypeI^{(1)}[\varpi,\delta,\sigma]$ whenever
  $48\varpi+14\delta+10\sigma < 1$.  However, as this constraint does
  not allow $\sigma$ to exceed $1/10$, one cannot use this estimate as
  a substitute for Theorem \ref{newtype}\eqref{typeI2} or Theorem
  \ref{newtype}\eqref{typeI4}.  If one uses the first estimate of
  Corollary \ref{dons} in place of the second, one can instead obtain
  $\TypeI^{(1)}[\varpi,\delta,\sigma]$ for the range $56 \varpi+16\delta+6\sigma< 1$, which
  now does permit $\sigma$ to exceed $1/10$, and thus gives some
  version of Zhang's theorem after combining with a Type III estimate.
  However, $\sigma$ still does not exceed $1/6$, and so one cannot
  dispense with the Type III component of the argument entirely with
  this Type I estimate.  By using a second application of $q$-van der
  Corput, though (i.e. using the $l=3$ case of Proposition \ref{inc}
  rather than the $l=2$ case), it is possible to raise $\sigma$ above
  $1/6$, assuming sufficient amounts of dense divisibility; we leave
  the details to the interested reader.  However, the Cauchy-Schwarz
  arguments used here are not as efficient in the Type I setting as
  the Cauchy-Schwarz arguments in the sections below, and so these
  estimates do not supersede their Type I counterparts.
\end{remark}

\subsection{Proof of first Type I estimate}\label{typei-1-sec}

We will establish Theorem \ref{eset}\eqref{typeI1-yetagain}, which is
the easiest of the Type I estimates to prove. The strategy follows
closely that of the previous section. The changes, roughly speaking,
are that the Cauchy-Schwarz argument is slightly modified (so that
only the $q_2$ variable is duplicated, rather than both $q_1$ and
$q_2$) and that we use an exponential sum estimate based on the first
part of Corollary \ref{dons} instead of the second.

As before, we will establish the bound~(\ref{eq-iv-target}) for each
individual $r$. We abbreviate again
$\Upsilon=\Upsilon_{\ell,r}(b_1,b_2;q_0)$ and denote
$$
H=x^{\eps}RQ^2M^{-1}q_0^{-1}.
$$
We begin with the formula~(\ref{eq-start}) for $\Upsilon$, move the
$q_1$ and $n$ sums outside, apply the Cauchy-Schwarz inequality (and
insert a suitable smooth coefficient sequence $\psi_N(n)$ at scale $N$
to the $n$ sum), so that we get
$$
|\Upsilon|^2\leq \Upsilon_1 \Upsilon_2
$$
with
$$
\Upsilon_1:= \sum_{q_1\sim Q/q_0} \sum_n C(n)|\beta(n)|^2 |\beta(n+\ell
r)|^2 \lessapprox \frac{NQ(q_0,\ell)}{q_0^2},
$$
(as in~(\ref{baz})) and
\begin{align*}
  \Upsilon_2&:= \sum_{n}\psi_N(n)C(n)\sum_{q_1\sim Q/q_0} \Bigl|
  \sum_{\substack{q_2 \sim Q/q_0\\(q_1,q_2) = 1}} \sum_{1 \leq |h|
    \leq H} c_{h,q_1,q_2}
  \Phi_{\ell}(h,n, r, q_0, q_1,q_2)\Bigr|^2\\
  &= \sum_{q_1\sim Q/q_0}\sumsum_{\substack{q_2,s_2\sim Q/q_0\\
      (q_1,q_2)=(q_1,s_2)=1 }} \sumsum_{1\leq h_1,h_2\leq |H|}
  c_{h_1,q_1,q_2}\overline{c_{h_2,q_1,s_2}}S_{\ell,r}(h_1,h_2,q_1,q_2,q_1,s_2),
\end{align*}
where $S_{\ell,r}$ is the same sum~(\ref{S-def}) as before and the
variables $(q_1,q_2,s_2)$ are restricted by the condition $q_0q_1r$,
$q_0q_2r$, $q_0s_2r\in \DIone{I}{x^{\delta}}$ (recall the
definition~(\ref{eq-upsilon})).

We will prove the following bound:



\begin{proposition}\label{exse-2} 
For any
$$
\uple{p}=(h_1,h_2,q_1,q_2,q_1,s_2)
$$
with $(q_0q_1q_2s_2,r)=1$ and for any $\ell \neq 0$ and $r$ as above
with
\begin{gather*}
  q_0q_ir, q_0s_2r\in \DIone{I}{x^{\delta}}\\
  q_0q_i\ll Q,\quad q_0s_2\ll Q,\quad r\ll R,
\end{gather*}
we have
$$ 
|S_{\ell,r}(\uple{p})| \lessapprox q_0^{1/6} N^{1/2} x^{\delta/6} (Q^3
R)^{1/6} + R^{-1} N (h_1 s_2 - h_2q_2, r).
$$
\end{proposition}

We first conclude assuming this estimate: arguing as in the previous
section to sum the gcd $(h_1s_2-h_2q_2,r)$, we obtain
$$
\Upsilon_2\lessapprox 
\Bigl(\frac{Q}{q_0}\Bigr)^3H^2
\Bigl\{
q_0^{1/6}N^{1/2}(Q^3R)^{1/6}x^{\delta/6}+\frac{N}{R}
\Bigr\}
+HN\Bigl(\frac{Q}{q_0}\Bigr)^2,
$$
and therefore
\begin{align*}
|\Upsilon|^2&
\lessapprox \frac{NQ(q_0,\ell)}{q_0^2}\Bigl\{
q_0^{1/6}\Bigl(\frac{Q}{q_0}\Bigr)^3H^2N^{1/2}(Q^3R)^{1/6}x^{\delta/6}+
\Bigl(\frac{Q}{q_0}\Bigr)^3\frac{H^2N}{R}+
HN\Bigl(\frac{Q}{q_0}\Bigr)^2
\Bigr\}\\
&\lessapprox \frac{N^2Q^4(q_0,\ell)^2}{q_0^4}
\Bigl\{
\frac{H^2Q^{1/2}R^{1/6}x^{\delta/6}}{N^{1/2}}+
\frac{H^2}{R}+\frac{H}{Q}
\Bigr\}
\end{align*}
where we once again discard some powers of $q_0 \geq 1$ from the denominator.
Using again~(\ref{crop2}) and~(\ref{eq-nr}) and $N\lessapprox M$, we
find that
\begin{gather*}
\frac{H^2Q^{1/2}R^{1/6}x^{\delta/6}}{N^{1/2}}\lessapprox
x^{\delta/6+2\eps}\frac{R^{13/6}Q^{9/2}}{M^2N^{1/2}}\lessapprox
x^{-2+\delta/6+2\eps}\frac{N^{3/2}(QR)^{9/2}}{R^{7/3}}
\lessapprox \frac{x^{1/4+9\varpi+5\delta/2+9\eps}}{N^{5/6}}
\\
\frac{H^2}{R}\lessapprox \frac{x^{8\varpi +3\delta+11\eps}}{N}\\
\frac{H}{Q}\leq x^{\eps}\frac{RQ}{M}\lessapprox
\frac{x^{1/2+2\varpi+\eps}}{M}
\lessapprox x^{-c+\eps},
\end{gather*}
and using the assumption $N\ggcurly x^{1/2-\sigma}$ from \eqref{slop}, we will
derive~(\ref{eq-iv-target}) if $c=3\eps$, $\eps>0$ is small enough,
and
$$
\begin{cases}
\frac{1}{4}+9\varpi+5\frac{\delta}{2}<\frac{5}{6}
\Bigl(\frac{1}{2}-\sigma\Bigr)\\
8\varpi+3\delta<\frac{1}{2}-\sigma
\end{cases}
\quad
\Leftrightarrow\quad\quad
\begin{cases}
54\varpi+15\delta+5\sigma<1\\
16\varpi+6\delta+2\sigma<1.
\end{cases}
$$
For $\varpi$, $\delta$, $\sigma>0$, the first condition implies the
second (as its coefficients are larger). Since the first
condition is the assumption of Theorem
\ref{eset}\eqref{typeI1-yetagain}, we are then done.

We now prove the exponential sum estimate.

\begin{proof}[Proof of Proposition~\ref{exse-2}]
  We denote 
$$
\Phi_1(n)=\Phi_{\ell}(h_1,n,r,q_0,q_1,q_2),\quad\quad
\Phi_2(n)=\Phi_{\ell}(h_2,n,r,q_0,q_1,s_2),
$$
as in the proof of Proposition~\ref{exse}, and we write
$$
\Phi_1(n)\overline{\Phi_2(n)}= \uple{e}^{(4)}_{d_1} \Bigl( \frac{c_1}{n}
\Bigr) \uple{e}^{(5)}_{d_2}\Bigl( \frac{c_2}{n+\tau} \Bigr)
$$
for some $c_1$ and $c_2$, where
\begin{gather*}
  d_1 := r q_0 q_1,\quad\quad d_2 := [q_2,s_2].
\end{gather*}
Since $rq_0q_1$, $rq_0q_2$ and $rq_0s_2$ are $x^{\delta}$-densely
divisible, Lemma~\ref{fq}(ii) implies that the lcm
$[d_1,d_2]=[rq_0q_1,rq_0q_2,rq_0s_2]$ is also $x^{\delta}$-densely
divisible.
\par
Splitting again the factor $C(n)$ into residue classes modulo $q_0$,
and applying the first part of Corollary~\ref{dons} to each residue class,
we obtain
$$
|S_{\ell,r}|\lessapprox (q_0,\ell) \Bigl(
\frac{N^{1/2}}{q_0^{1/2}} [d_1,d_2]^{1/6} x^{\delta/6} +
\frac{N}{q_0}\frac{(c_1,\delta'_1)}{\delta_1}
\frac{(c_2,\delta'_2)}{\delta'_2} \Bigr)
$$
where $\delta_i=d_i/(d_1,d_2)$ and $\delta'_i = \delta_i/(q_0,\delta_i)$. Again, as in the proof of
Proposition~\ref{exse}, we conclude by observing that $[d_1,d_2]\leq
Q^3R/q_0$, $(c_2,\delta'_2)/\delta'_2\leq 1$ while
$$ 
\frac{(c_1,\delta'_1)}{\delta'_1}\leq\frac{(c_1,r)}{r},
$$
and inspection of the $r$-component of $\Phi_1(n)\overline{\Phi_2(n)}$
using~(\ref{eq-phiell}) shows that a prime $p\mid r$ divides $c_1$ if
and only if $p\mid h_1s_2-h_2q_2$.
\end{proof}

\subsection{Proof of second Type I estimate}\label{second-typei}

We finish this section with the proof of Theorem
\ref{eset}\eqref{typeI2-yetagain}.  The idea is very similar to the
previous Type I estimate, the main difference being that since $q_1$
(and $q_2$) is densely divisible in this case, we can split the sum
over $q_1$ to obtain a better balance of the factors in the
Cauchy-Schwarz inequality. 

As before, we will prove the bound~(\ref{eq-iv-target}) for individual
$r$, and we abbreviate $\Upsilon=\Upsilon_{\ell,r}(b_1,b_2;q_0)$ and
denote
$$
H=x^{\eps}RQ^2M^{-1}q_0^{-1}.  
$$
We may assume that $H\geq 1$, since otherwise the bound is trivial.
We note that $q_0q_1$ is, by assumption,
$x^{\delta+o(1)}$-densely divisible, and therefore by
Lemma~\ref{fq}(i), $q_1$ is $y$-densely divisible with
$y=q_0x^{\delta+o(1)}$. Furthermore we have
$$
x^{-2\eps} Q/H\ggcurly x^{c-3\eps}
$$
by~(\ref{crop2}) and $M\ggcurly x^{1/2+2\varpi+c}$, and
$$
x^{-2\eps} Q/H\lessapprox q_1y=q_1q_0x^{\delta+o(1)}
$$
since $q_1q_0\sim Q$ and $H\geq 1$. Thus (assuming $c>3\eps$) we can factor
$$
q_1=u_1v_1
$$
where $u_1,v_1$ are squarefree with
\begin{gather*}
  q_0^{-1} x^{-\delta-2\eps} Q/H \lessapprox u_1 \lessapprox
  x^{-2\eps} Q/H\\
  q_0^{-1} x^{2\eps} H \lessapprox v_1 \lessapprox x^{\delta+2\eps} H
\end{gather*}
(either from dense divisibility if $x^{-2\eps}Q/H\lessapprox q_1$, or
taking $u_1=q_1$, $v_1=1$ otherwise).
\par
Let
$$
\Upsilon_{U,V}:=\sum_{1 \leq |h| \leq H} \sum_{u_1 \sim U} \sum_{v_1
  \sim V} \sum_{\substack{q_2 \sim Q/q_0\\ (u_1v_1,q_0q_2)= 1}}\Bigl|
\sum_n C(n)\beta(n) \overline{\beta(n+\ell r)}
\Phi_{\ell}(h,n,r,q_0,u_1v_1,q_2)\Bigr|,
$$
where $u_1,v_1$ are understood to be squarefree.
\par
By dyadic decomposition of the sum over $q_1=u_1v_1$ in $\Upsilon$, it
is enough to prove that
\begin{equation}\label{eq-upsilon-dyadic}
\Upsilon_{U,V}\lessapprox x^{-\eps}(q_0,\ell)Q^2Nq_0^{-2}
\end{equation}
whenever
\begin{align}
  q_0^{-1} x^{-\delta-2\eps} Q/H \lessapprox U &\lessapprox x^{-2\eps} Q/H \label{s-bound}\\
  q_0^{-1} x^{2\eps} H \lessapprox V &\lessapprox x^{\delta+2\eps} H \label{t-bound}\\
  UV &\sim Q/q_0. \label{st}
\end{align}

We replace the modulus by complex numbers $c_{h,u_1,v_1,q_2}$ of
modulus at most $1$, move the sum over $n$, $u_1$ and $q_2$ outside and apply
the Cauchy-Schwarz inequality as in the previous sections to obtain
$$
|\Upsilon_{U,V}|^2\leq \Upsilon_1\Upsilon_2
$$
with
$$
\Upsilon_1:=\sumsum_{\substack{u_1\sim U\\ q_2\sim Q/q_0}} \sum_n C(n)
|\beta(n)|^2|\beta(n+\ell r)|^2\lessapprox (q_0,\ell)\frac{NQU}{q_0^2}
$$
as in~(\ref{baz}) and
\begin{align*}
  \Upsilon_2&:= \sumsum_{\substack{u_1\sim U\\ q_2\sim Q/q_0}}
  \sum_{n}\psi_N(n)C(n) \Bigl| \sum_{v_1\sim V; (u_1v_1,q_0q_2)= 1}
  \sum_{1 \leq |h| \leq H} c_{h,u_1,v_1,q_2}
  \Phi_{\ell}(h,n, r, q_0, u_1v_1,q_2)\Bigr|^2\\
  &=\sumsum_{\substack{u_1\sim U\\q_2\sim Q/q_0}} 
  \sumsum_{v_1,v_2\sim V; (u_1v_1v_2,q_0q_2)= 1} \sumsum_{1\leq |h_1|,|h_2|\leq H}
  c_{h_1,u_1,v_1,q_2}\overline{c_{h_2,u_1,v_2,q_2}}
  T_{\ell,r}(h_1,h_2,u_1,v_1,v_2,q_2,q_0)
\end{align*}
where the exponential sum $T_{\ell,r}$ is a variant of $S_{\ell,r}$ given by
\begin{equation}\label{eq-tell}
T_{\ell,r} :=\sum_n C(n)\psi_N(n) \Phi_{\ell}(h_1,n, r,q_0,u_1v_1, q_2)
\overline{ \Phi_{\ell}(h_2,n, r,q_0,u_1v_2,q_2)}.
\end{equation}
The analogue of Propositions~\ref{exse} and~\ref{exse-2} is:

\begin{proposition}
\label{exse-3}  
For any
$$
\uple{p}=(h_1,h_2,u_1,v_1,v_2,q_2,q_0)
$$
with $(u_1v_1v_2,q_0q_2)=(q_0,q_2)=1$, any $\ell \neq 0$ and $r$
as above, we have
$$
|T_{\ell,r}( \uple{p})| \lessapprox (q_0,\ell)\Bigl(q_0^{-1/2} N^{1/2}
x^{\delta/3+\eps/3} (R H Q^2)^{1/6} + \frac{N}{q_0R}(h_1 v_2-h_2v_1,
r)\Bigr).
$$
\end{proposition}

Assuming this, we derive as before 
$$
\Upsilon_2\lessapprox  (q_0,\ell) 
H^2UV^2\Bigl(\frac{Q}{q_0}\Bigr)
\Bigl\{
N^{1/2}(RHQ^2)^{1/6}x^{\delta/3+\eps/3}+\frac{N}{R}
\Bigr\}
+HNUV\Bigl(\frac{Q}{q_0^2}\Bigr),
$$
and then
\begin{align*}
  |\Upsilon_{U,V}|^2&\lessapprox (q_0,\ell)^2\frac{NQU}{q_0} \Bigl\{
  \frac{H^2Q^3N^{1/2}(HQ^2R)^{1/6}x^{\delta/3+\eps/3}}{Uq_0^3}
  +\frac{H^2NQ^3}{URq_0^3} +HN\Bigl(\frac{Q^2}{q_0^3}\Bigr)
  \Bigr\}\\
  &\lessapprox (q_0,\ell)^2\frac{N^2Q^4}{q_0^4} \Bigl\{
  \frac{H^{13/6}Q^{1/3}R^{1/6}x^{\delta/3+\eps/3}}{N^{1/2}} +\frac{H^2}{R}+
  \frac{H}{Vq_0} \Bigr\}
\end{align*}
since $UV\sim Q/q_0$, where we have again discarded a factor of $q_0$ in the first line.  Using
again~(\ref{crop2}),~(\ref{eq-nr}),~(\ref{t-bound}), we find that
\begin{align*}
  \frac{H^{13/6}Q^{1/3}R^{1/6}x^{\delta/3+\eps/3}}{N^{1/2}} &\lessapprox
  x^{\delta+5\eps/2}\frac{R^{7/3}Q^{14/3}}{N^{1/2}M^{13/6}}
  \lessapprox
  x^{1/6+28\varpi/3+\delta/3+5\eps/2}\frac{N^{5/3}}{R^{7/3}}\\
  & \lessapprox
  \frac{x^{28\varpi/3+8\delta/3+1/6+19\eps/2}}{N^{2/3}}
  \\
  \frac{H^2}{R}&\lessapprox \frac{x^{8\varpi +3\delta+11\eps}}{N}\\
  \frac{H}{Vq_0}&\lessapprox x^{-2\eps},
\end{align*}
and therefore~(\ref{eq-upsilon-dyadic}) holds for sufficiently small $\eps$ provided
$$
\begin{cases}
  \frac{28\varpi}{3}+\frac{8\delta}{3}+\frac{1}{6}
  <\frac{2}{3}\left(\frac{1}{2}-\sigma\right)
  \\
  8\varpi+3\delta<\frac{1}{2}-\sigma
\end{cases}\quad
\Leftrightarrow
\quad\quad
\begin{cases}
  56\varpi+16\delta+4\sigma <1
  \\
  16\varpi+6\delta+2\sigma<1.
\end{cases}
$$
Again the first condition implies the second, and the proof is
completed.

\begin{proof}[Proof of Proposition~\ref{exse-3}]
We proceed as in the previous cases. Denoting
$$
\Phi_1(n):= \Phi_{\ell}(h_1,n, r,q_0,u_1v_1, q_2),\quad\quad \Phi_2(n):=
\Phi_{\ell}(h_2,n, r,q_0,u_1v_2,q_2)
$$
for brevity, we may write
$$
\Phi_1(n)\overline{\Phi_2(n)}= \uple{e}^{(4)}_{d_1} \Bigl( \frac{c_1}{n}
\Bigr) \uple{e}_{d_2}^{(5)}\Bigl( \frac{c_2}{n+\tau} \Bigr)
$$
by~(\ref{eq-phiell}) for some $c_1$ and $c_2$ and $\tau$, where
\begin{gather*}
  d_1 := r q_0u_1[v_1,v_2],\quad\quad d_2 := q_2.
\end{gather*}
Since $rq_0u_1v_1$, $rq_0u_1v_2$ and $rq_0q_2$ are
$x^{\delta}$-densely divisible, Lemma~\ref{fq}(ii) implies that their
gcd $[d_1,d_2]$ is also $x^{\delta}$-densely divisible.
\par
Splitting again the factor $C(n)$ into residue classes modulo $q_0$,
and applying the first part of Corollary~\ref{dons} to each residue class,
we obtain
$$
|T_{\ell,r}|\lessapprox (q_0,\ell) \Bigl( \frac{N^{1/2}}{q_0^{1/2}}
[d_1,d_2]^{1/6} x^{\delta/6} +
\frac{N}{q_0}\frac{(c_1,\delta'_1)}{\delta'_1}
\frac{(c_2,\delta'_2)}{\delta'_2} \Bigr)
$$
where $\delta_i=d_i/(d_1,d_2)$ and $\delta'_i = \delta_i / (q_0,\delta_i)$. We conclude as before by observing
that 
$$
[d_1,d_2]\ll QRUV^2\lessapprox x^{\delta+2\eps}\frac{HQ^2R}{q_0},
$$
by~(\ref{t-bound}) and~(\ref{st}), that $(c_2,\delta_2)/\delta_2\leq
1$ and that $(c_1,\delta)/\delta_1\leq(c_1,r)/r$, where inspection of
the $r$-component of $\Phi_1(n)\overline{\Phi_2(n)}$
using~(\ref{eq-phiell}) shows that a prime $p\mid r$ divides $c_1$ if
and only if $p\mid h_1v_2-h_2v_1$.
\end{proof}

\section{Trace functions and multidimensional exponential sum
  estimates}\label{deligne-sec}

In this section (as in Section \ref{exp-sec}), we do not use the standard asymptotic convention
(Definition~\ref{asym}), since we discuss general ideas that are of
interest independently of the goal of bounding gaps between primes.
\par
We will discuss some of the machinery and formalism of
\emph{$\ell$-adic sheaves} $\mcF$ on curves\footnote{In our
  applications, the only curves $U$ we deal with are obtained by
  removing a finite number of points from the projective line $\Pl$.}
and their associated \emph{Frobenius trace functions} $t_{\mcF}$. This
will allow us to state and then apply the deep theorems of Deligne's
general form of the Riemann Hypothesis over finite fields for such
sheaves. We will use these theorems to establish certain estimates for
multi-variable exponential sums which go beyond the one-dimensional
estimates obtainable from Lemma \ref{prime-exp} (specifically, the
estimates we need are stated in Corollary~\ref{inctrace-q} and 
Corollary~\ref{corr-2}).  
\par
The point is that these Frobenius trace functions significantly
generalize the rational phase functions $x \mapsto
e_p\left(\frac{P(x)}{Q(x)}\right)$ which appear in
Lemma~\ref{prime-exp}. They include more general functions, such as the
hyper-Kloosterman sums 
$$
x \mapsto \frac{(-1)^{m-1}}{p^{\frac{m-1}{2}}}
\multsum_{\substack{y_1,\ldots,y_m \in \Fp\\ y_1 \ldots y_m = x}} e_p(
y_1 + \cdots + y_m ),
$$
and satisfy a very flexible formalism. In particular, the class of Frobenius 
trace functions is (essentially)
closed under basic operations such as pointwise addition and
multiplication, complex conjugation, change of variable
(pullback),
and the normalized Fourier transform.  Using these closure properties
allows us to build a rich class of useful trace functions from just a
small set of basic trace functions.  In fact, the sheaves we actually
use in this paper are ultimately obtained from only two sheaves: the
Artin-Schreier sheaf and the third hyper-Kloosterman sheaf.\footnote{One can even reduce the number of generating
  sheaves to one, because the sheaf-theoretic Fourier transform,
  combined with pullback via the inversion map $x \mapsto\frac{1}{x}$,
  may be used to iteratively build the hyper-Kloosterman sheaves from
  the Artin-Schreier sheaf.} However, we have chosen to discuss more
general sheaves in this section in order to present the
sheaf-theoretic framework in a more natural fashion.

Because exponential sums depending on a parameter are often themselves
trace functions, one can recast many multidimensional exponential sums
(e.g.
$$
\sum_{x_1,\ldots,x_n \in \Fp} e_p(f(x_1,\ldots,x_n))
$$ 
for some rational function $f\in\Fp(X_1,\ldots,X_n)$) in terms of
one-dimensional sums of Frobenius trace functions. As a very rough
first approximation, Deligne's results~\cite{WeilII} imply that the
square root cancellation exhibited in Lemma~\ref{prime-exp} is also
present for these more general sums of Frobenius trace functions, as
long as certain degenerate cases are avoided.  Therefore, at least
in principle, this implies square root cancellation for many
multidimensional exponential sums.  
\par
In practice, this is often not entirely straightforward, as we will
explain. One particular issue is that the bounds provided by Deligne's
theorems depend on a certain measure of complexity of the $\ell$-adic
sheaf defining the trace function, which is known as the
\emph{conductor} of a sheaf. In estimates for sums of trace functions,
this conductor plays the same role that the degrees of the polynomials
$f,g$ play in Lemma \ref{prime-exp}.  We will therefore have to expend
some effort to control the conductors of various sheaves before we can
extract usable estimates from Deligne's results.

This section is not self-contained, and assumes a certain amount of
prior formal knowledge of the terminology of $\ell$-adic cohomology on
curves. For readers who are not familiar with this material, we would
recommend as references such surveys as \cite[\S 11.11]{ik},
\cite{exp-rism}, \cite{pizza}, and some of the books and papers of
Katz, in particular \cite{katz-sommes, katzWII, GKM}, as well as
Deligne's own account~\cite[Sommes trig.]{deligne41/2}. We would like
to stress that, if the main results of the theory are assumed and the
construction of some main objects (e.g. the Artin-Schreier and
hyper-Kloosterman sheaves) is accepted, working with $\ell$-adic
sheaves essentially amounts to studying
certain 
finite-dimensional continuous representations of the Galois group of the field
$\Fp(X)$ of rational functions over $\Fp$.  
\par
Alternatively, for the purposes of establishing only the bounds on
(incomplete) multi-variable exponential sums used in the proofs of the
main theorems of this paper (namely the bounds in Corollary~\ref{inctrace-q} and 
Corollary~\ref{corr-2}), it is possible to ignore all
references to sheaves, if one accepts the estimates on complete
multi-dimensional exponential sums in Proposition \ref{kls} and
Theorem \ref{kf-bound} as ``black boxes''; the estimates on  incomplete
exponential sums will be deduced from these results via completion of sums and the
$q$-van der Corput $A$-process.

\subsection{$\ell$-adic sheaves on the projective line}
For $p$ a prime, we fix an algebraic closure $\ov\Fp$ of $\Fp$ and
denote by $k\subset \ov\Fp=\ov k$ a finite extension of $\Fp$. Its
cardinality is usually denoted
$|k|=p^{[k:\Fp]}=p^{\operatorname{deg}(k)}=q$.  For us, the Frobenius
element relative to $k$ means systematically the \emph{geometric
  Frobenius} $\Fr_k$, which is the inverse in $\Gal(\ov k/k)$ of the
\emph{arithmetic Frobenius}, $x \mapsto x^q$ on $\ov k$. 

We denote by $K=\Fp(t)$ the function field of the projective line
$\Pl_{\Fp}$ and by $\overline K\supset \overline \Fp$ some
separable closure; let $\ov\eta=\mathrm{Spec}(\overline K)$ be the
corresponding geometric generic point.

We fix another prime $\ell\neq p$, and we denote by $\iota \colon
\ov{\Q_\ell}\hookrightarrow \C$ an algebraic closure of the field
$\Q_\ell$ of $\ell$-adic numbers, together with
an embedding into the complex numbers. By an \emph{$\ell$-adic
  sheaf} $\mcF$ on a noetherian scheme $X$ (in practice, a curve), we
always mean a constructible sheaf of finite-dimensional $\overline{\Q_\ell}$-vector spaces with respect to the \'etale topology on $X$,
and we recall that the category of $\ell$-adic sheaves is abelian.

We will be especially interested in the case $X=\mathbb{P}_k^1$ (the
projective line) and we will use the following notation for the
translation, dilation, and fractional linear maps from $\Pl$
to itself:
\begin{align*}
  [+l] \colon x&\mapsto x+l,\\
  [\times a] \colon x&\mapsto ax,\\
  \gamma \colon x&\mapsto \gamma\cdot x=\frac{ax+b}{cx+d}\text{ for }
  \gamma=\begin{pmatrix}a&b\\c&d
  \end{pmatrix}\in \mathrm{GL}_2(\Fp).
\end{align*}
We will often transform a sheaf $\mcF$ on $\Pl_k$ by applying
pullback by one of the above maps, and we denote these pullback
sheaves by $[+l]^*\mcF$, $[\times a]^*\mcF$ and $\gamma^*\mcF$.

\subsubsection{Galois representations}

The category of $\ell$-adic sheaves on $\mathbb{P}_k^1$ admits a
relatively concrete description in terms of representations of the
Galois group $\Gal(\ov K/k.K)$. We recall some important features of
it here and we refer to \cite[4.4]{katz-sommes} for a complete
presentation.

For $j:U\hookrightarrow \mathbb{P}_k^1$ some non-empty open subset
defined over $k$, we denote by $\pi_1(U)$ (resp. $\pi^g_1(U)$) the
\emph{arithmetic (resp. geometric) fundamental group} of $U$, which may
be defined as the quotient of $\Gal(\overline K/k.K)$ (resp.\ of
$\Gal(\overline K/\ov k.K)$) by the smallest closed normal subgroup
containing all the inertia subgroups above the closed points of
$U$. We have then a commutative diagram of short exact sequences of
groups
\begin{equation}\label{exactseq}
\begin{tikzcd} 1\arrow{r}&  \Gal(\ov K/\ov k.K)\arrow{r}\arrow[two heads]{d}&\Gal(\ov K/k.K)\arrow{r}\arrow[two heads]{d}&\Gal(\ov k/k)\arrow{r}\arrow{d}{=}&1\\1\arrow{r}&\pige(U)\arrow{r}&\piar(U)\arrow{r}&\Gal(\ov k/k)\arrow{r}&1
\end{tikzcd}
\end{equation}

Given an $\ell$-adic sheaf $\mcF$ on $\mathbb{P}_k^1$, there exists
some non-empty (hence dense, in the Zariski topology) open set
$j:U\hookrightarrow \mathbb{P}_k^1$ such that the pullback $j^*\mcF$
(the restriction of $\mcF$ to $U$) is \emph{lisse}, or in other words,
for which $j^*\mcF$ ``is'' a finite-dimensional continuous representation
$\rho_{\mcF}$ of $\Gal(\overline{K}/k.K)$ factoring through $\pi_1(U)$
$$
\rho_\mcF \colon \Gal(\overline K/k.K )\twoheadrightarrow
\pi_1(U)\rightarrow \mathrm{GL}(\mcF_{\ov\eta}),
$$
where the geometric generic stalk $\mcF_{\overline\eta}$ of $\mcF$ is
a finite-dimensional $\overline{\Q_\ell}$-vector space. Its dimension
is the \emph{(generic) rank} of $\mcF$ and is denoted
$\mathrm{rk}(\mcF)$. There is a maximal (with respect to inclusion) open subset on
which $\mcF$ is lisse, which will be denoted $U_\mcF$.

We will freely apply the terminology of representations to $\ell$-adic
sheaves. The properties of $\rho_\mcF$ as a representation of the
arithmetic Galois group $\Gal(\ov K/k.K)$ (or of the arithmetic
fundamental group $\pi_1(U)$) will be qualified as ``arithmetic'',
while the properties of its restriction $\rho_\mcF^g$ to the
geometric Galois group $\Gal(\overline K/\ov k. K )$ (or the geometric
fundamental group $\pi_1^g(U)$) will be qualified as
``geometric''. For instance, we will say that $\mcF$ is {\em
  arithmetically irreducible} (resp.\ \emph{geometrically irreducible})
or \emph{arithmetically isotypic} (resp.\ \emph{geometrically
  isotypic}) if the corresponding arithmetic representation
$\rho_\mcF$ (resp.\ the geometric representation $\rho_\mcF^g$) is. 

We will be mostly interested in the geometric properties of a sheaf,
therefore we will usually omit the adjective ``geometric'' in our
statements, so that ``isotypic'' will mean ``geometrically
isotypic''. We will always spell out explicitly when an arithmetic
property is intended, so that no confusion can arise.

\subsubsection{Middle-extension sheaves} 
An $\ell$-adic sheaf is a called a \emph{middle-extension} sheaf if,
for some (and in fact, for any) non-empty open subset
$j:U\hookrightarrow \mathbb{P}_k^1$ such that $j^*\mcF$ is lisse, we
have an arithmetic isomorphism
$$
\mcF\simeq j_*j^*\mcF,
$$
or equivalently if, for every $\ov x\in\Pl(\ov k)$, the
specialization maps (cf. \cite[4.4]{katz-sommes}) 
$$
s_{\ov x}:\mcF_{\ov x}\rightarrow \mcF_{\ov \eta}^{I_{\ov x}}
$$ 
are isomorphisms, where $I_{\ov{x}}$ is the inertia subgroup at
$\ov{x}$. Given an $\ell$-adic sheaf, its associated middle-extension
is the sheaf
$$
\mcF^{\mathrm{me}}=j_*j^*\mcF
$$
for some non-empty open subset $j:U\hookrightarrow \mathbb{P}_k^1$ on
which $\mcF$ is lisse. This sheaf is a middle-extension sheaf, and is
(up to arithmetic isomorphism) the unique middle-extension sheaf whose
restriction to $U$ is arithmetically isomorphic to that of $\mcF$. In
particular, $\mcF^{\mathrm{me}}$ does not depend on the choice of $U$.

\subsection{The trace function of a sheaf}
Let $\mcF$ be an $\ell$-adic sheaf on the projective line over
$\Fp$. For each finite extension $k/\Fp$, $\mcF$ defines a complex
valued function
$$
x \mapsto t_{\mcF}(x;k)
$$ 
on $k\cup\{\infty\}=\Pl(k)$, which is called the \emph{Frobenius trace function}, or
just \emph{trace function}, associated with $\mcF$ and $k$. It is
defined by
$$
\Pl(k) \ni x\mapsto t_{\mcF}(x;k):=\iota(\Tr(\Fr_{x,k}|\mcF_{\overline
  x})).
$$
Here $\ov x:\mathrm{Spec}(\ov k)\rightarrow \Pl_k$ denotes a geometric
point above $x$, and $\mcF_{\overline x}$ is the stalk of $\mcF$ at
that point, which is a finite-dimensional $\ov{\Q_\ell}$-vector space
on which $\Gal(\ov k/k)$ acts linearly, and $\Fr_{x,k}$ denotes the
geometric Frobenius of that Galois group. The trace of the action of
this operator is independent of the choice of $\ov{x}$.

If $k=\Fp$, which is the case of importance for the applications in
this paper, we will write $t_{\mcF}(x;p)$ or simply $t_{\mcF}(x)$
instead of $t_{\mcF}(x;\Fp)$.

If $x\in U_\mcF(k)$, the quantity $t_\mcF(x;k)$ is simply the trace of
the geometric Frobenius conjugacy class of a place of $\ov K$ above
$x$ acting through the associated representation $\mcF_{\ov\eta}$,
i.e., the value (under $\iota$) of the character of the
representation at this conjugacy class:
$$
t_{\mcF}(x;k)=\iota(\Tr(\Fr_{x,k}|\mcF_{\ov\eta})).
$$
If $\mcF$ is a middle-extension sheaf one has more generally
$$
t_{\mcF}(x;k)=\iota(\Tr(\Fr_{x,k}|\mcF_{\ov\eta}^{I_{\ov x}})).
$$

For any sheaf $\mcF$, the trace function of $\mcF$ restricted to
$U_\mcF(k)$ coincides with the restriction of the trace function of
$\mcF^{\mathrm{me}}$.

\subsubsection{Purity and admissibility}

The following notion was introduced by Deligne~\cite{WeilII}.

\begin{definition}[Purity] For $i\in \Z$, an $\ell$-adic sheaf on
  $\Pl_{\Fp}$ is \emph{generically pure} (or \emph{pure}, for short) of
  weight $i$ if, for any $k/\Fp$ and any $x\in U_\mcF(k)$, the
  eigenvalues of $\Fr_{x,k}$ acting on $\mcF_{\ov\eta}$ are $\Q$-algebraic
  numbers whose Galois conjugates have complex absolute value equal to
  $q^{i/2}=|k|^{i/2}$.
\end{definition}

\begin{remark} Deligne proved (see \cite[(1.8.9)]{WeilII}) that if
  $\mcF$ is generically pure of weight $i$, then for any $k/\Fp$ and
  any $x\in \Pl(k)$, the eigenvalues of $\Fr_{x,k}$ acting on
  $\mcF_{\ov\eta}^{I_{\ov x}}$ are $\Q$-algebraic numbers whose Galois
  conjugates have complex absolute value $\leq q^{i/2}$.

In particular, if $\mcF$ is a middle-extension sheaf which is pointwise pure
of weight $i$, then we get
\begin{equation}\label{traceupperbound}
  |t_\mcF(x;k)|=|\iota(\Tr(\Fr_x|\mcF_{\ov\eta}^{I_{\ov x}}))|\leq \rk(\mcF)q^{i/2}
\end{equation}
for any $x\in\Pl(k)$.
\end{remark}


We can now describe the class of sheaves and trace functions that we
will work with.

\begin{definition}[Admissible sheaves] Let $k$ be a finite extension
  of $\Fp$. An \emph{admissible sheaf} over $k$ is a middle-extension
  sheaf on $\Pl_{k}$ which is pointwise pure of weight $0$.  An
  \emph{admissible trace function} over $k$ is a function
  $k\rightarrow \C$ which is equal to the trace function of some
  admissible sheaf restricted to $k\subset \Pl(k)$.
\end{definition}

\begin{remark} The weight $0$ condition may be viewed as a
  normalization to ensure that admissible trace functions typically
  have magnitude comparable to $1$.  Sheaves which are pure of some
  other weight can be studied by reducing to the $0$ case by the
  simple device of \emph{Tate twists}. However, we will not need to do
  this, as we will be working exclusively with sheaves which are pure
  of weight $0$.
\end{remark}

\subsubsection{Conductor}

Let $\mcF$ be a middle-extension sheaf on $\Pl_{k}$. The
\emph{conductor} of $\mcF$ is defined to be
$$
\cond(\mcF):=\rk(\mcF)+|(\Pl-U_{\mcF})(\ov{k})|+\sum_{x\in (\Pl-
  U_\mcF)(\overline k)}\swan_x(\mcF)
$$
where $\swan_x(\mcF)$ denotes the Swan conductor of the representation
$\rho_\mcF$ at $x$, a non-negative integer measuring the ``wild
ramification'' of $\rho_{\mcF}$ at $x$ (see e.g. \cite[Definition
1.6]{GKM} for the precise definition of the Swan conductor).  If
$\swan_x(\mcF)=0$, one says that $\mcF$ is \emph{tamely ramified} at
$x$, and otherwise that it is \emph{wildly ramified}.
\par
The invariant $\cond(\mcF)$ is a non-negative integer (positive if
$\mcF\neq 0$) and it measures the complexity of the sheaf
$\mcF$ and of its trace function $t_\mcF$.  For instance, if $\mcF$ is
admissible, so that it is also pure of weight $0$, then we deduce from
\eqref{traceupperbound} that
\begin{equation}\label{traceupperbound-2}
|t_\mcF(x;k)| \leq \rk(\mcF) \leq \cond(\mcF)
\end{equation}
for any $x\in k$.

\subsubsection{Dual and Tensor Product} Given admissible sheaves
$\mcF$ and $\mcG$ on $\Pl_k$, their tensor product, denoted
$\mcF\otimes\mcG$, is by definition the middle-extension sheaf
associated to the tensor product representation
$\rho_\mcF\otimes\rho_\mcG$ (computed over the intersection of
$U_{\mcF}$ and $U_{\mcG}$, which is still a dense open set of
$\Pl_k$). Note that this sheaf may be different from the tensor product
of $\mcF$ and $\mcG$ as constructible sheaves (similarly to the fact
that the product of two primitive Dirichlet characters is not
necessarily primitive).

Similarly, the dual of $\mcF$, denoted $\wcheck{\mcF}$, is defined as
the middle extension sheaf associated to the contragredient
representation $\wcheck{\rho_\mcF}$.

We have
$$
U_\mcF\cap U_\mcG\subset U_{\mcF\otimes\mcG},\quad\quad U_{\wcheck
  \mcF}=U_{\mcF}.
$$

It is not obvious,
but true, that the tensor product and the dual of
admissible sheaves are admissible. We then have
\begin{equation}\label{hom}
  t_{\mcF\otimes\mcG}(x;k)=t_{\mcF}(x;k)t_{\mcG}(x;k),\quad\quad
  t_{\wcheck\mcF}(x;k)=\ov{t_{\mcF}(x;k)}
\end{equation}
for $x \in U_\mcF(k) \cap U_\mcG(k)$ and $x \in \Pl(k)$, respectively.
In particular, the product of two admissible trace functions
$t_{\mcF}$ and $t_{\mcG}$ coincides with an admissible trace function
outside a set of at most $\cond(\mcF)+\cond(\mcG)$ elements, and the complex
conjugate of an admissible trace function is again an admissible trace function.

We also have
\begin{equation}\label{cond-1}
\cond(\wcheck \mcF) = \cond(\mcF)
\end{equation}
(which is easy to check from the definition of Swan conductors) and
\begin{equation}\label{cond-2}
  \cond(\mcF\otimes\mcG)\ll \rk(\mcF)\rk(\mcG)\cond(\mcF)\cond(\mcG)
  \leq \cond(\mcF)^2\cond(\mcG)^2
\end{equation}
where the implied constant is absolute (which is also relatively
elementary, see \cite[Prop. 8.2(2)]{FKM1} or~\cite[Lemma
4.8]{transforms}).

\subsection{Irreducible components and isotypic decomposition}

Let $k$ be a finite field, let $\mcF$ be an admissible sheaf over
$\Pl_k$ and consider $U=U_{\mcF}$ and the corresponding open immersion
$j\colon U\hookrightarrow \Pl_{k}$.  A fundamental result of
Deligne~\cite[(3.4.1)]{WeilII} proves that $\rho_{\mcF}$ is then
geometrically semisimple. Thus there exist lisse sheaves $\mcG$ on
$U\times\ov{k}$, irreducible and pairwise non-isomorphic, and integers
$n(\mcG)\geq 1$, such that we have
$$
j^*\mcF\simeq \bigoplus_{\mcG}\mcG^{n(\mcG)}
$$
as an isomorphism of lisse sheaves on $U\times\ov{k}$ (the $\mcG$
might not be defined over $k$). Extending with $j_*$ to $\Pl_{\ov{k}}$
we obtain a decomposition
$$
\mcF\simeq \bigoplus_{\mcG}j_*\mcG^{n(\mcG)}
$$
where each $j_*\mcG$ is a middle-extension sheaf over $\ov{k}$. We
call the sheaves $j_*\mcG$ the \emph{geometrically irreducible
  components} of $\mcF$.
\par
Over the open set $U_{\mcF}$, we can define the arithmetic
semisimplification $\rho_{\mcF}^{\text{ss}}$ as the direct sum of the
Jordan-H\"older arithmetically irreducible components of the
representation $\rho_{\mcF}$. Each arithmetically irreducible
component is either geometrically isotypic or induced from a proper
finite index subgroup of $\pi_1(U_{\mcF})$.  If an arithmetically
irreducible component $\pi$ is induced, it follows that the trace
function of the middle-extension sheaf corresponding to $\pi$ vanishes
identically. Thus, if we denote by $\mathrm{Iso}(\mcF)$ the set of
middle-extensions associated to the geometrically isotypic components
of $\rho_{\mcF}^{\text{ss}}$, we obtain an identity
\begin{equation}\label{eq-isotypic}
 t_\mcF=\sum_{\mcG\in \mathrm{Iso}(\mcF)}t_{\mcG}
\end{equation}
(indeed, these two functions coincide on $U_{\mcF}$ and are both trace
functions of middle-extension sheaves), where each summand is
admissible. For these facts, we refer to~\cite[\S 4.4, \S
4.5]{katz-sommes} and~\cite[Prop. 8.3]{FKM1}.



\subsection{Deligne's main theorem and quasi-orthogonality}

The generalizations of complete exponential sums over finite fields
that we consider are sums
$$
S(\mcF;k)=\sum_{x\in k}t_\mcF(x;k)
$$
for any admissible sheaf $\mcF$ over $\Pl_k$. By
\eqref{traceupperbound-2}, we have the trivial bound
$$
|S(\mcF;k)|\leq \cond(\mcF)|k|=\cond(\mcF)q.
$$
Deligne's main theorem \cite[Thm.~1]{WeilII} provides strong
non-trivial estimates for such sums, at least when $p$ is large
compared to $\cond(\mcF)$.

\begin{theorem}[Sums of trace functions]\label{propGL1}
  Let $\mcF$ be an admissible sheaf on $\Pl_k$ where $|k|=q$ and
  $U=U_{\mcF}$. We have
$$
S(\mcF;k)=q\Tr\left(\Fr_{k}|(\mcF_{\ov\eta})_{\pi_1^g(U)}\right)
+O(\cond(\mcF)^2q^{1/2})
$$ 
where $(\mcF_{\ov\eta})_{\pi_1^g(U)}$ denotes the Tate-twisted
$\pi_1^g(U_{\mcF})$-coinvariant space\footnote{Recall that the
  coinvariant space of a representation of a group $G$ is the largest
  quotient on which the group $G$ acts trivially.} of $\rho_{\mcF}$,
on which $\Gal(\ov k/k)$ acts canonically, and where the implied
constant is effective and absolute.
\end{theorem}

\proof Using \eqref{traceupperbound-2}, we have
$$
S(\mcF;k)=\sum_{x\in U(k)}t_\mcF(x;k)+O(\cond(\mcF)^2)
$$
where the implied constant is at most $1$. The Grothendieck-Lefshetz
trace formula (see, e.g., \cite[Chap.\ 3]{GKM}) gives
$$
S_\mcF(U,k)=\sum_{i=0}^2(-1)^i\Tr\left(\Fr_{k}|H_c^i(U\otimes_k\ov
  k,\mcF)\right)
$$
where $H^i_c(U\otimes_k\ov{k},\mcF)$ is the $i^{\operatorname{th}}$ compactly supported
\'etale cohomology group of the base change of $U$ to $\ov{k}$ with
coefficients in $\mcF$, on which the global Frobenius automorphism
$\Fr_k$ acts.
\par
Since $U$ is affine and $\mcF$ is lisse on $U$, it is known that
$H_c^0(U\otimes_k\ov k,\mcF)=0$. For $i=1$, Deligne's main theorem
shows that, because $\mcF$ is of weight $0$, all eigenvalues of
$\Fr_k$ acting on $H^1_c(U\times_k\ov{k},\mcF)$ are algebraic numbers
with complex absolute value $\leq |k|^{1/2}$, so that
$$
\left|\Tr\left(\Fr_{k}|H_c^1(U\otimes_k\ov k,\mcF)\right)\right| \leq
\dim\left(H_c^1(U\otimes_k\ov k,\mcF)\right)q^{1/2}.
$$
Using the Euler-Poincar\'e formula and the definition of the
conductor, one easily obtains
$$
\dim\left(H_c^1(U\otimes_k\ov k,\mcF)\right)\ll \cond(\mcF)^2
$$
with an absolute implied constant (see, e.g.,~\cite[Chap.\ 2]{GKM}
or~\cite[Th. 2.4]{FKMGowersnorms}).
\par
Finally for $i=2$, it follows from Poincar\'e duality that
$H_c^2(U\otimes_k\ov k,\mcF)$ is isomorphic to the Tate-twisted space of
$\pi_1^g(U)$-coinvariants of $\mcF_{\ov \eta}$ (see,
e.g.,~\cite[Chap. 2.]{GKM}), and hence the contribution of
this term is the main term in the formula.
\qed

\subsubsection{Correlation and quasi-orthogonality of trace functions}
An important application of the above formula arises when estimating
the \emph{correlation} between the trace functions $t_\mcF$ and
$t_\mcG$ associated to two admissible sheaves $\mcF, \mcG$, i.e., when
computing the sum associated to the tensor product sheaf
$\mcF\otimes\wcheck\mcG$.  
We define the correlation sum
$$
C(\mcF,\mcG;k):=\sum_{x\in k}t_\mcF(x;k)\ov{t_\mcG(x;k)}.
$$
From \eqref{traceupperbound-2} we have the trivial bound
$$
|C_{\mcF,\mcG}(k)|\leq \cond(\mcF)\cond(\mcG)q.
$$
The Riemann Hypothesis allows us improve upon this bound when $\mcF$,
$\mcG$ are ``disjoint'':

\begin{corollary}[Square root cancellation]\label{correlationcor} Let
  $\mcF,\mcG$ be two admissible sheaves on $\Pl_k$ for a finite field
  $k$. If $\mcF$ and $\mcG$ have no irreducible constituent in common,
  then we have
$$
|C(\mcF,\mcG;k)| \ll (\cond(\mcF)\cond(\mcG))^4q^{1/2}
$$
where the implied constant is absolute.  In particular, if in addition
$\cond(\mcF)$ and $\cond(\mcG)$ are bounded by a fixed constant, then
$$ 
|C(\mcF, \mcG;k)| \ll q^{1/2}.
$$
\end{corollary}

\proof We have
$$
t_{\mcF\otimes\wcheck\mcG}(x;k)=t_{\mcF}(x;k)\ov{t_{\mcG}(x;k)}
$$
for $x\in U_\mcF(k)\cap U_\mcG(k)$ and
$$
|t_{\mcF\otimes\wcheck\mcG}(x;k)|,\quad
|t_{\mcF}(x;k)\ov{t_{\mcG}(x;k)}|\leq \cond(\mcF)\cond(\mcG).
$$
Thus the previous proposition applied to the sheaf
$\mcF\otimes\widecheck\mcG$ gives
\begin{align*}
  C(\mcF,\mcG;k)&=S(\mcF\otimes\widecheck\mcG;k)+
  O((\cond(\mcF)+\cond(\mcG))\cond(\mcF)\cond(\mcG))\\
  &=q\Tr\left(\Fr_{k}|(\mcF\otimes\wcheck{\mcG})_{\ov\eta})_{\pi_1^g(U)}\right)
  +O((\cond(\mcF)\cond(\mcG))^4q^{1/2})
\end{align*}
using \eqref{cond-1} and \eqref{cond-2}.  We conclude by observing
that, by Schur's Lemma and the geometric semisimplicity of admissible
sheaves (proved by Deligne~\cite[(3.4.1)]{WeilII}), our disjointness
assumption on $\mcF$ and $\mcG$ implies that the coinvariant space
vanishes. \qed

\subsection{The Artin-Schreier sheaf}

We will now start discussing specific important admissible sheaves.
Let $p$ be a prime and let $\psi\colon (\Fp,+)\rightarrow \C^\times$
be a non-trivial additive character.  For any finite extension $k$ of
$\Fp$, we then have an additive character
$$
\psi_k \colon \begin{cases}
  k  \to \C^\times\\
  x\mapsto \psi(\Tr_{k/\Fp}(x)),
\end{cases}
$$
where $\Tr_{k/\Fp}$ is the trace map from $k$ to $\Fp$.
\par
One shows (see \cite[Chap.\ 4]{GKM}, \cite[\S 1.4]{deligne41/2},
\cite[p.\ 302--303]{ik}) that there exists an admissible sheaf
$\mcL_{\psi}$, called the \emph{Artin-Schreier sheaf} associated to
$\psi$, with the following properties:
\begin{itemize}
\item the sheaf $\mcL_{\psi}$ has rank $1$, hence is automatically
  geometrically irreducible, and it is geometrically non-trivial;
\item the sheaf $\mcL_{\psi}$ is lisse on $\mathbb{A}^1_{\Fp}$, and
  wildly ramified at $\infty$ with $\swan_\infty(\mcL_\psi)=1$, so
  that in particular $\cond(\mcL_{\psi})=3$, independently of $p$ and
  of the non-trivial additive character $\psi$;
\item the trace function is given by the formula
$$
t_{\mcL_\psi}(x;k)=\psi_{k}(x)
$$
for every finite extension $k/\Fp$ and every $x\in \mathbb{A}^1(k)=k$,
and
$$
t_{\mcL_{\psi}}(\infty;k)=0.
$$
\end{itemize}

Let $f\in \Fp(X)$ be a rational function not of the shape $g^p-g+c$
for $g\in\Fp(X),\ c\in\Fp$ (for instance whose zeros or poles have
order prime to $p$). Then $f$ defines a morphism $f\colon \Pl_{\F_p}
\rightarrow \Pl_{\F_p}$, and we denote by $\mcL_{\psi(f)}$ the
pull-back sheaf $f^*\mcL_{\psi}$, which we call the
\emph{Artin-Schreier sheaf associated to $f$ and $\psi$}. Then
$\mcL_{\psi(f)}$ has the following properties:
\begin{itemize}
\item it has rank $1$, hence is geometrically irreducible, and it is
  geometrically non-trivial (because $f$ is not of the form $g^p-g+c$
  for some other function $g$, by assumption);
\item it is lisse outside the poles of $f$, and wildly ramified at
  each pole with Swan conductor equal to the order of the pole, so
  that if the denominator of $f$ has degree $d$ (coprime to $p$) we
  have $\cond(\mcL_{\psi(f)})=1+e+d$, where $e$ is the number of
  distinct poles of $f$;
\item it has trace function given by the formula
$$
t_{\mcL_{\psi(f)}}(x;k)=\psi(\tr_{k/\Fp}(f(x)))
$$
for any finite extension $k/\Fp$ and any $x\in\Pl(k)$ which is not a
pole of $f$, and $t_{\mcL_{\psi(f)}}(x;k)=0$ if $x$ is a pole of $f$.
\end{itemize}

In particular, from Theorem \ref{propGL1}, we thus obtain the
estimate
$$ 
\Bigl|\sum_{x\in \Fp}\psi(f(x))\Bigr| \ll \deg(f)^2p^{1/2}
$$
for such $f$, which is a slightly weaker form of the Weil bound from
Lemma~\ref{prime-exp}. Note that this weakening, which is immaterial
in our applications, is only due to the general formulation of
Theorem \ref{propGL1} which did not attempt to obtain the best
possible estimate for specific situations.

%
 
\subsection{The $\ell$-adic Fourier transform}

Let $p$ be a prime, $k/\Fp$ a finite extension and $\psi$ a
non-trivial additive character of $k$. For a finite extension $k/\Fp$
and a function $x\mapsto t(x)$ defined on $k$, we define the
\emph{normalized Fourier transform} $\FT_\psi t(x)$ by the formula
$$
\FT_\psi t(x):=-\frac{1}{q^{1/2}}\sum_{y\in k}t(y)\psi(xy)
$$
(which is similar to \eqref{ftq-def} except for the sign).  It is a
very deep fact that, when applied to trace functions, this
construction has a sheaf-theoretic incarnation. This was defined by
Deligne and studied extensively by Laumon~\cite{laumon} and
Katz~\cite{GKM}.  However, a restriction on the admissible sheaves is
necessary, in view of the following obstruction: if $t(x)=\psi(bx)$
for some $b\in k$, then its Fourier transform is a Dirac-type function
$$
\FT_\psi(t)(x)=-q^{1/2}\delta_{-b}(x)=
\begin{cases}-q^{1/2}&\text{ if $x=-b$}\\
  0&\text{ otherwise.}
\end{cases}
$$
But this cannot in general be an admissible trace function with
bounded conductor as this would violate \eqref{traceupperbound} at
$x=-b$ if $q$ is large enough.  We make the following definition, as
in~\cite{GKM}:

\begin{definition}[Admissible Fourier sheaves] An admissible sheaf
  over $\Pl_k$ is a \emph{Fourier sheaf} if its geometrically
  irreducible components are neither trivial nor Artin-Schreier
  sheaves $\mcL_{\psi}$ for some non-trivial additive character
  $\psi$.
\end{definition}

\begin{theorem}[Sheaf-theoretic Fourier
  transform]\label{fourier-sheaf}
  Let $p$ be a prime and $k/\Fp$ a finite extension, and let
  $\psi$ be a non-trivial additive character of $k$.  Let $\mcF$ be an
  admissible $\ell$-adic Fou\-rier sheaf on $\Pl_k$. There exists an
  $\ell$-adic sheaf
$$
\mcG=\FT_\psi(\mcF),
$$
called the \emph{Fourier transform of $\mcF$}, which is also an
admissible $\ell$-adic Fourier sheaf, with the property that for any
finite extension $k'/k$, we have
$$
t_\mcG(\cdot;k')=\FT_{\psi_{k'}} t_\mcF(\cdot;k),
$$
in particular
$$
t_{\mcG}(x;k)=-\frac{1}{\sqrt{|k|}}
\sum_{y\in k}t_{\mcF}(y;k)\psi(xy).
$$
\par
Moreover, the following additional assertions hold:
\begin{itemize}
\item The sheaf $\mcG$ is geometrically irreducible, or geometrically
  isotypic, if and only if $\mcF$ is;
\item The Fourier transform is (almost) involutive, in the sense that
  we have a canonical arithmetic isomorphism
\begin{equation}\label{eq-fourier-inv}
\FT_\psi\mcG\simeq [\times (-1)]^*\mcF
\end{equation}
where $[\times (-1)]^*$ denotes the pull-back by the map $x\mapsto
-x$;
\item We have
\begin{equation}
 \cond(\mcG)\leq 10\cond(\mcF)^2.\label{fourierconductorbound}
\end{equation}
\end{itemize}
\end{theorem}

\begin{proof}  
  These claims are established for instance in \cite[Chap. 8]{GKM},
  with the exception of \eqref{fourierconductorbound} which is proved
  in \cite[Prop. 8.2(1)]{FKM1}.
\end{proof}

\subsection{Kloosterman sheaves}

Given a prime $p\geq 3$, a non-trivial additive character $\psi$ of $\Fp$ and
an integer $m\geq 1$, the $m^{\operatorname{th}}$ hyper-Kloosterman sums are defined by
the formula
\begin{equation}\label{klm-def}
  \hypk_m(x;k):=\frac{1}{q^{\frac{m-1}2}}\sum_\stacksum{y_1,\cdots,y_m\in
    k}{y_1\cdots y_m=x}\psi_{k}(y_1+\cdots+y_m)
\end{equation}
for any finite extension $k/\Fp$ and any $x \in k$. Thus, we have for
instance $\hypk_1(x;k) = \psi_k(x)$, while $\hypk_2$ is essentially a classical Kloosterman sum.
\par
The following deep result shows that, as functions of $x$, these sums
are trace functions of admissible sheaves.

\begin{proposition}[Deligne; Katz]\label{kl-sheaf} There exists an
  admissible Fourier sheaf $\Kl_m$ such that, for any $k/\Fp$ and any
  $x\in k^\times$, we have
$$
t_{\Kl_m}(x;k)=(-1)^{m-1}\hypk_m(x;k).
$$
Furthermore:
\begin{itemize}
\item $\Kl_m$ is lisse on $\Gm=\Pl-\{0,\infty\}$; if $m\geq 2$, it is
  tamely ramified at $0$, and for $m=1$ it is lisse at $0$; for all
  $m\geq 1$, it is wildly ramified at $\infty$ with Swan conductor
  $1$; 
\item $\Kl_m$ is of rank $m$, and is geometrically irreducible;
\item if $p$ is odd, then the Zariski closure of the image
  $\rho_{\Kl_m}(\pi^g_1(\Gm))$, which is called the \emph{geometric
    monodromy group} of $\Kl_m$, is isomorphic to $\mathrm{SL}_m$ if
  $m$ is odd, and to $\mathrm{Sp}_m$ if $m$ is even.
\end{itemize}
It follows that $\cond(\Kl_m)=m+3$ for all $m\geq 2$ and all $p$, and
that $\cond(\Kl_1)=3$.
\end{proposition}

\begin{proof}
  All these results can be found in the book of Katz~\cite{GKM}; more
  precisely, the first two points are part of Theorem 4.1.1
  in~\cite{GKM} and the last is part of Theorem 11.1 in the same
  reference. 
\end{proof}

\begin{remark}\label{romeo} In particular, for $x\neq 0$, we get the
  estimate
$$
|\hypk_m(x;k)|\leq m,
$$
first proved by Deligne. Note that this exhibits square-root
cancellation in the $(m-1)$-variable character sum defining
$\hypk(x;k)$. For $x=0$, it is elementary that
$$
\hypk_m(0;k)=(-1)^{m-1}q^{-(m-1)/2}.
$$
\end{remark}

We have the following bounds for hyper-Kloosterman sums, where the
case $m=3$ is the important one for this paper:

\begin{proposition}[Estimates for hyper-Kloosterman sums]\label{kls}
  Let $m\geq 2$ be an integer, and $\psi'$ an additive character of
  $\Fp$, which may be trivial. We have
\begin{equation}\label{klsumsfirstmoment}
\Bigl|\sum_{x\in\Fpt}\hypk_m(x;p)\psi'(x)\Bigr| \ll p^{1/2}.
\end{equation}
Further, let $a\in\Fpt$. If either $a\neq 1$ or $\psi'$ is
non-trivial, we have
\begin{equation}\label{klsumscorrelation}
  \Bigl|\sum_{x\in\Fpt}\hypk_m(x;p)\ov{\hypk_m(ax;p)}\psi'(x)\Bigr| \ll p^{1/2}
\end{equation}
In these bounds, the implied constants depend only, and at most
polynomially, on $m$.
\end{proposition}

\proof The first bound~\eqref{klsumsfirstmoment} follows directly from
Corollary~\ref{correlationcor} and~\eqref{cond-2} because $\Kl_m$ is,
for $m\geq 2$, geometrically irreducible of rank $>1$, and therefore
not geometrically isomorphic to the rank $1$ Artin-Schreier sheaf
$\mcL_{\psi'}$.

For the proof of \eqref{klsumscorrelation}, we use the identity\footnote{One could use this identity to recursively build the hyper-Kloosterman sheaf from the Artin-Schreier sheaf, Theorem \ref{fourier-sheaf}, and pullback via the map $x \mapsto \frac{1}{x}$, if desired.} 
$$
\hypk_m(x)=\frac{1}{p^{1/2}}\sum_{y\in\Fpt}\hypk_{m-1}(y^{-1})\psi(xy)=
-\FT_\psi([y^{-1}]^*\hypk_{m-1})(x)
$$
which is valid for all $x\in\Fp$ (including $x=0$). If we let
$b\in\Fp$ be such that $\psi'(x)=\psi(bx)$ for all $x$, then by the
Plancherel formula, we deduce
\begin{align*}
  \sum_{x\in\Fp}\hypk_m(x;p)\ov{\hypk_m(ax;p)}\psi'(x)&=
  \sum_{y\in\Fp \backslash \{0,-b\}}\hypk_{m-1}(y^{-1})\ov{\hypk_{m-1}(a(y+b)^{-1})}\\
  &=\sum_\stacksum{y\in\Fp,}{ y\neq
    0,-1/b}\hypk_{m-1}(y;p)\ov{\hypk_{m-1}(\gamma\cdot y;p)}
\end{align*}
where
$$
\gamma:=\begin{pmatrix}
a&0\\b&1
\end{pmatrix}.
$$  
We are in the situation of Corollary \ref{correlationcor}, with both
sheaves $\Kl_{m-1}$ and $\gamma^*\Kl_{m-1}$ admissible and
geometrically irreducible.  If $m\geq 3$, $\Kl_{m-1}$ is tamely
ramified at $0$ and wildly ramified at $\infty$, and
$\gamma^*\Kl_{m-1}$ is therefore tame at $\gamma^{-1}(0)$ and wild at
$\gamma^{-1}(\infty)$, so that a geometric isomorphism
$\Kl_{m-1}\simeq \gamma^*\Kl_{m-1}$ can only occur if $\gamma(0)=0$
and $\gamma(\infty)=\infty$, or in other words if $b=0$. 
If $b=0$, we have $\gamma^*\Kl_{m-1}=[\times a]^*\Kl_{m-1}$ which is
known to be geometrically isomorphic to $\Kl_{m-1}$ if and only if
$a=1$, by \cite[Prop. 4.1.5]{GKM}. Thus \eqref{klsumscorrelation}
follows from Corollary \ref{correlationcor} for $m\geq 3$,
using~(\ref{cond-2}) and the formulas
$\cond(\Kl_{m-1})=\cond(\gamma^*\Kl_{m-1})=m+3$. 
\par
The case $m=2$ is easy since the sum above is then simply
$$
\sum_\stacksum{y\in\Fp,}{ y\neq 0,-1/b}\psi(y-ay/(by+1))
$$
where the rational function $f(y)=y-ay/(by+1)$ is constant if and only
if $a=1,b=0$, so that we can use Lemma \ref{prime-exp} in this case.
\qed

\begin{remark} 
  A similar result was proved by Michel using a different method
  in \cite[Cor.\ 2.9]{DMJPhM}. That method requires more information
  (the knowledge of the geometric monodromy group of $\Kl_m$) but
  gives more general estimates.  The case $m=3$ is (somewhat
  implicitly) the result used by Friedlander and Iwaniec
  in~\cite{fi-3}, which is proved by Birch and Bombieri in the
  Appendix to \cite{fi-3} (with in fact two proofs, which are rather
  different and somewhat more ad-hoc than the argument presented
  here). This same estimate is used by Zhang \cite{zhang} to control
  Type III sums.
\end{remark}

\subsection{The van der Corput method for trace functions}

Let $t=t_\mcF$ be the trace function associated to an admissible sheaf
$\mcF$.  In the spirit of Proposition \ref{inc}, the $q$-van der
Corput method, when applied to incomplete sums of $t$, followed by
completion of sums, produces expressions of the form
$$
\sum_{x\in\Fp}t(x)\ov{t(x+l)}\psi(hx)
$$
for $(h,l) \in\Fp\times\Fpt$, and for some additive character $\psi$. We
seek sufficient conditions that ensure square-root cancellation in the
above sum, for any $l\neq 0$ and any $h$.

Observe that if 
$$
t(x)=\psi(ax^2+bx),
$$
then the sum is sometimes of size $p$. Precisely, this happens if and
only if $h=2al$. As we shall see, this phenomenon is essentially the
only obstruction to square-root cancellation.

\begin{definition}[No polynomial phase]
  For a finite field $k$ and $d\geq 0$, we say that an admissible
  sheaf $\mcF$ over $\Pl_k$ has \emph{no polynomial phase} of degree
  $\leq d$ if no geometrically irreducible component of $\mcF$ is
  geometrically isomorphic to a sheaf of the form $\mcL_{\psi(P(x))}$
  where $P(X)\in\Fp[X]$ is a polynomial of degree $\leq d$.
\end{definition}

Thus, for instance, an admissible sheaf is Fourier if and only if it
has no polynomial phase of degree $\leq 1$.

\begin{remark}\label{irrednopolynomialphase} 
  An obvious sufficient condition for $\mcF$ not to contain any
  polynomial phase (of any degree) is that each geometrically
  irreducible component of $\mcF$ be irreducible of rank $\geq 2$, for
  instance that $\mcF$ itself be geometrically irreducible of rank
  $\geq 2$.
\end{remark}

The following inverse theorem is a variant of an argument of Fouvry,
Kowalski and Michel \cite[Lemma 5.4]{FKMGowersnorms}.

\begin{theorem}\label{th-pol-phase}
  Let $d\geq 1$ be an integer, and let $p$ be a prime such that
  $p>d$. Let $\mcF$ be an isotypic admissible sheaf over $\Pl_{\Fp}$
  with no polynomial phase of degree $\leq d$.
  Then either $\cond(\mcF)\geq p+1$, or for any $l\in\Fp^\times$ the
  sheaf $\mcF\otimes[+l]^*\wcheck\mcF$ contains no polynomial phase of
  degree $\leq d-1$. 
\par
In all cases, for any $l\in\Fpt$ and any $P(X)\in\Fp[X]$ of degree
$d-1$, we have
\begin{equation}\label{squarerootforWeylshifts}
  \Bigl|\sum_{x\in\Fp}t_\mcF(x+l)\ov{t_\mcF(x)}\psi(P(x))\Bigr|\ll p^{1/2}
\end{equation}
where the implied constant depends, at most polynomially, on
$\cond(\mcF)$ and on $d$.  Furthermore, this estimate holds also if
$l=0$ and $P(x)=hx$ with $h\neq 0$.
\end{theorem}


\proof First suppose that $l \neq 0$.  Observe that if
$\cond(\mcF)\geq p+1$, the bound \eqref{squarerootforWeylshifts}
follows from the trivial bound 
$$
|t_\mcF(x+l)\ov{t_\mcF(x)}\psi(P(x))|\leq \rk(\mcF)^2\leq \cond(\mcF)^2,
$$
and that if the sheaf $[+l]^*\mcF\otimes\wcheck\mcF$ contains no
polynomial phase of degree $\leq d-1$, then the bound is a consequence
of Corollary \ref{correlationcor}.

We now prove that one of these two properties holds. We assume that
$[+l]^*\mcF\otimes\wcheck\mcF$ contains a polynomial phase of degree
$\leq d-1$, and will deduce that $\cond(\mcF)\geq p+1$. 
\par
Since $\mcF$ is isotypic, the assumption implies that there is a
geometric isomorphism
$$
[+l]^*\mcF\simeq \mcF\otimes\mcL_{\psi(P(x))}
$$
for some polynomial $P(X)\in\Fp[X]$ of degree $\leq d-1$. Then,
considering the geometric irreducible component $\mcG$ of $\mcF$
(which is a sheaf on $\Pl_{\ov{\Fp}}$) we also have
\begin{equation}\label{plus-l}
[+l]^*\mcG\simeq \mcG\otimes\mcL_{\psi(P(x))}.
\end{equation}
If $\mcG$ is ramified at some point $x\in \Aa^1(\ov k)$, then since
$\mcL_{\psi(P(x))}$ is lisse on $\Aa^1(\ov k)$, we conclude by
iterating \eqref{plus-l} that $\mcG$ is ramified at $x,x+l,\
x+2l,\cdots, x+(p-1)l$, which implies that $\cond(\mcF)\geq
\cond(\mcG)\geq p+\rk(\mcG)$. Thus there remains to handle the case
when $\mcG$ is lisse outside $\infty$. It then follows from
\cite[Lemma 5.4 (2)]{FKMGowersnorms} that either $\cond(\mcG)\geq
\rk(\mcG)+p$, in which case $\cond(\mcF)\geq p+1$ again, or that
$\mcG$ is isomorphic (over $\ov{\Fp}$) to a sheaf of the form
$\mcL_{\psi(Q(x))}$ for some polynomial of degree $\leq d$. Since
$\mcG$ is a geometrically irreducible component of $\mcF$, this
contradicts the assumption on $\mcF$.
\par
Finally, consider the case where $l=0$ and $P(x)=hx$ with $h\neq
0$. Using Corollary~\ref{correlationcor} and~(\ref{cond-2}), the
result holds for a given $h\in\Fpt$ unless the geometrically
irreducible component $\mcG$ of $\mcF$ satisfies
$$
\mcG\simeq \mcG\otimes\mcL_{\psi(hx)}.
$$
Since $d\geq 1$, $\mcF$ is a Fourier sheaf, and hence so are $\mcG$
and $\mcG\otimes\mcL_{\psi(hx)}$. Taking the Fourier transform of both
sides of this isomorphism, we obtain easily
$$
[+h]^*\FT_\psi\mcG\simeq \FT_\psi\mcG
$$
and it follows from~\cite[Lemma 5.4 (2)]{FKMGowersnorms} again that
$\cond(\FT_\psi\mcG)\geq p+1$. Using the Fourier inversion
formula~(\ref{eq-fourier-inv}) and~(\ref{fourierconductorbound}), we
derive
$$
\cond(\mcF)\geq \cond(\mcG)\gg p^{1/2},
$$   
so that the bound \eqref{squarerootforWeylshifts} also holds trivially
in this case.
\qed

\begin{remark}\label{invarianceremark}
  For later use, we observe that the property of having \emph{no
    polynomial phase of degree $\leq 2$} of an admissible sheaf $\mcF$
  is invariant under the following transformations:
\begin{itemize}
\item Twists by an Artin-Schreier sheaf associated to a polynomial
  phase of degree $\leq 2$, i.e., $\mcF\mapsto
  \mcF\otimes\mcL_{\psi(ax^2+bx)}$;
\item Dilations and translations: $\mcF\mapsto [\times a]^*\mcF$ and $\mcF \mapsto [+b]^* \mcF$ where
  $a \in \Fp^\times$ and $b \in \Fp$;
\item Fourier transforms, if $\mcF$ is Fourier: $\mcF\mapsto
  \FT_\psi\mcF$. Indeed, the Fourier transform of a sheaf
  $\mcL_{\psi(P(x))}$ with $\deg(P)=2$ is geometrically isomorphic to
  $\mcL_{\psi(Q(x))}$ for some polynomial $Q$ of degree $2$.
\end{itemize}
\end{remark}

\subsection{Study of some specific exponential sums}\label{study-sec}

We now apply the theory above to some specific multi-dimensional
exponential sums which appear in the refined treatment of the Type I
sums in Section \ref{typei-advanced-sec}. For parameters
$(a,b,c,d,e)\in \Fp$, with $a\neq c$, we consider the rational
function
$$
f(X,Y):=\frac{1}{(Y+aX+b)(Y+cX+d)}+eY\in \Fp(X,Y).
$$
For a fixed non-trivial additive character $\psi$ of $\Fp$ and for any
$x\in \Fp$, we define the character sum
\begin{equation}\label{Kfdef}
  K_f(x;p):=-\frac{1}{p^{1/2}}
  \sum_{\substack{y\in \Fp\\(y+ax+b)(y+cx+d)\neq 0}}\psi(f(x,y)).
\end{equation}
For any $x\in \Fp$, the specialized rational function
$f(x,Y)\in\Fp(Y)$ is non-constant (it has poles in $\Af_{\F_p}$), and
therefore by Lemma~\ref{prime-exp} (or Theorem~\ref{propGL1}) we have
\begin{equation}\label{klo}
|K_f(x;p)|\leq 4
\end{equation}
We will prove  the following additional properties of the sums
$K_f(x;p)$:

\begin{theorem}\label{kf-bound} For a prime $p$ and parameters
  $(a,b,c,d,e)\in\Fp^5$ with $a\neq c$, the function $x\mapsto
  K_f(x;p)$ on $\Fp$ is the trace function of an admissible
  geometrically irreducible sheaf $\mcF$ whose conductor is bounded by a
  constant independent of $p$. Furthermore, $\mcF$ contains no
  polynomial phase of degree $\leq 2$.
\par
In particular, we have
\begin{equation}\label{squareroot}
  \Bigl|\sum_{x\in\Fp}K_f(x;p)\psi(hx)\Bigr|\ll p^{1/2}
\end{equation}
for all $h \in \Fp$ and
\begin{equation}\label{squarerootforWeylshift}
  \Bigl|\sum_{x\in\Fp}K_f(x;p)\ov{K_f(x+l;p)}\psi(hx)\Bigr|\ll p^{1/2}
\end{equation}
for any $(h,l)\in\Fp^2-\{(0,0)\}$, where the implied constants are
absolute.
\end{theorem}

\proof Note that the estimates~(\ref{squareroot})
and~(\ref{squarerootforWeylshift}) follow from the first assertion
(see Theorem~\ref{th-pol-phase}).
\par
We first normalize most of the parameters: we have
$$
K_f(x;p)=-\frac{\psi(-eax-eb)}{p^{1/2}}\sum_{z\in\Fp}
\psi\Bigl(ez+\frac{1}{z(z+(c-a)x+d-b)}\Bigr),
$$
and by Remark~\ref{invarianceremark}, this means that we may assume
that $c=d=0$, $a\neq 0$. Furthermore, we have then
$$
K_f(x;p)=K_{\tilde{f}}(ax+b;p)
$$
where $\tilde{f}$ is the rational function $f$ with parameters
$(1,0,0,0,e)$. Again by Remark~\ref{invarianceremark}, we are reduced
to the special case $f=\tilde f$, i.e., to the sum
$$
K_f(x;p) =-\frac{1}{p^{1/2}}\sum_{\substack{y\in \Fp\\(y+x)y\neq 0}}
\psi\left(\frac{1}{(y+x)y} + ey\right).
$$

We will prove that the Fourier transform of $K_f$ is the trace
function of a geometrically irreducible Fourier sheaf with bounded
conductor and no polynomial phase of degree $\leq 2$. By the Fourier
inversion formula~(\ref{eq-fourier-inv})
and~(\ref{fourierconductorbound}), and the invariance of the property
of not containing a polynomial phase of degree $\leq 2$ under Fourier
transform (Remark~\ref{invarianceremark} again), this will imply the
result for $K_f$.

For $z\in \Fp$, we have
$$ 
\FT_\psi(K_f)(z) =\frac{1}{p} \sumsum_{y+x,y\neq 0}
\psi\left(\frac{1}{(y+x)y}+ ey + zx\right)
$$
If $z\neq 0$, the change of variables
$$
y_1:=\frac{1}{(y+x)y},\quad\quad y_2 := z(y+x)
$$ 
is a bijection
$$
\{(x,y)\in\Fp\times\Fp\colon y(x+y)\neq 0\}\to \{(y_1,y_2)\in
\Fpt\times\Fpt\}
$$
(with inverse $y=z/(y_1y_2)$ and $x=y_2/z-z/(y_1y_2)$), which
satisfies
$$
\frac{1}{(y+x)y}+ ey +
zx=y_1+\frac{ez}{y_1y_2}+y_2-\frac{z^2}{y_1y_2}=
y_1+y_2+\frac{z(e-z)}{y_1y_2}
$$
for $y(x+y)\neq 0$. Thus
$$
\FT_\psi(K_f)(z) =
\frac{1}{p}\sumsum_{y_1,y_2\in\Fpt}\psi\Bigl(
y_1+y_2+\frac{z(e-z)}{y_1y_2} \Bigr)=\hypk_3(z(e-z);p)
$$
for $z(e-z)\neq 0$.
\par
Similar calculations reveal that this identity also holds when $z=0$
and $z=e$ (treating the doubly degenerate case $z=e=0$ separately),
i.e., both sides are equal to $\frac{1}{p}$ in these cases. This means
that $\FT_\psi(K_f)$ is the trace function of the pullback sheaf
$$
\mcG_f := \phi^* \Kl_3,
$$
where $\phi$ is the quadratic map $\phi: z\mapsto z(e-z)$.

The sheaf $\mcG_f$ has bounded conductor (it has rank $3$ and is lisse on
$U=\Pl_{\Fp}-\{0,e,\infty\}$, with wild ramification at $\infty$ only,
where the Swan conductor can be estimated using~\cite[1.13.1]{GKM},
for $p\geq 3$).  We also claim that $\mcG_f$ is geometrically irreducible. Indeed, it suffices
to check that $\pi_1^g(U)$ acts irreducibly on the underlying vector
space of $\rho_{\Kl_3}$. But since $z\mapsto z(e-z)$ is a non-constant
morphism $\Pl_{\F_p} \to \Pl_{\F_p}$, $\pi_1^g(U)$ acts by a finite-index subgroup of
the action of $\pi_1^g(\Gm)$ on $\Kl_3$. Since the image of
$\pi_1^g(\Gm)$ is Zariski-dense in $\mathrm{SL}_3$ (as recalled in
Proposition~\ref{kl-sheaf}), which is a connected algebraic group, it
follows that the image of $\pi_1^g(U)$ is also Zariski-dense in
$\mathrm{SL}_3$, proving the irreducibility.
\par
Since $\mcG_f$ is geometrically irreducible of rank $3>1$, it does not
contain any polynomial phase (see Remark
\ref{irrednopolynomialphase}), concluding the proof.
\qed

\begin{remark} Another natural strategy for proving this theorem would
  be to start with the observation that the function $x\mapsto
  K_f(x;k)$ is the trace function of the constructible $\ell$-adic
  sheaf
$$
\mcK_f=R^1\pi_{1,!}\mcL_{\psi(f)}(1/2),\quad\quad
\mcL_{\psi(f)}=f^*\mcL_\psi
$$
where $\pi_1:\Aa_{\Fp}^2\to \Af_{\Fp}$ is the projection on the
first coordinate and $R^1\pi_{1,!}$ denotes the operation of
higher-direct image with compact support associated to that map (and
$(1/2)$ is a Tate twist). This is known to be mixed of weights $\leq
0$ by Deligne's work~\cite{WeilII}, and it follows from the general
results\footnote{Which were partly motivated by the current paper.} of
Fouvry, Kowalski and Michel in \cite{transforms} that the conductor of
this sheaf is absolutely bounded as $p$ varies. To fully implement
this approach, it would still remain to prove that the weight $0$ part
of $\mcK_f$ is geometrically irreducible with no polynomial phase of
degree $\leq 2$.  Although such arguments might be necessary in more
advanced cases, the direct approach we have taken is simpler here.
\end{remark}


\begin{remark} In the remainder of this paper, we will only use the
  bounds \eqref{squareroot} and \eqref{squarerootforWeylshift} from
  Theorem~\ref{kf-bound}.  These bounds can also be expressed in terms
  of the Fourier transform $\FT_\psi(K_f)$ of $K_f$, since they are
  equivalent to
$$
|\FT_\psi( K_f )(h)| \ll p^{1/2}
$$
and
$$ 
\Bigl|\sum_{x \in \Fp}\FT_\psi(K_f)(x+h)\overline{\FT_\psi(K_f)(x)}
  \psi( -lx )\Bigr| \ll p^{1/2},
$$ 
respectively.  As such, we see that it is in fact enough to show that
$\FT_\psi(K_f)$, rather than $K_f$, is the trace function of a
geometrically irreducible admissible sheaf with bounded conductor and
no quadratic phase component.  Thus, in principle, we could avoid any
use of Theorem \ref{fourier-sheaf} in our arguments (provided that we took the existence of the Kloosterman sheaves for granted).  However, from a
conceptual point of view, the fact that $K_f$ has a good trace
function interpretation is more important than the corresponding fact
for $\FT_\psi$ (for instance, the iterated van der Corput bounds in
Remark \ref{vdCiterate-trace} rely on the former fact rather than the
latter).
\end{remark}

\subsection{Incomplete sums of trace
  functions}\label{compositetracefunctions}

In this section, we extend the discussion of Section \ref{exp-sec} to
general admissible trace functions. More precisely, given a
squarefree integer $q$, we say that a $q$-periodic arithmetic function 
$$
t\colon \Zz\rightarrow \Zz/q\Zz\rightarrow \Cc
$$ 
is an \emph{admissible trace function} if we have
\begin{equation}\label{tfactor}
t(x)=\prod_{p|q}t(x;p)
\end{equation}
for all $x$ where, for each prime $p\mid q$, $x\mapsto t(x;p)$ is the
composition of reduction modulo $p$ and the trace function
associated to an admissible sheaf $\mcF_{p}$ on $\Pl_{\Fp}$.

An example is the case discussed in Section \ref{exp-sec}: for a
rational function $f(X)=P(X)/Q(X)\in\Q(X)$ with $P,Q\in\Zz[X]$ and a
squarefree integer $q$ such that $Q\ (q)\neq 0$, we can write 
$$
e_q(f(x))=e_q\left(\frac{P(x)}{Q(x)}\right)=
\prod_{p|q}e_p\left(\ov{q_p}f(x)\right),\quad\text{ where } q_p=q/p.
$$ 
(by Lemma~\ref{crt}). In that case, we take
$$
\mcF_p=\mcL_{\psi(f)},\quad\text{ where}\quad \psi(x)=e_p(\ov{q_p}x).
$$
Another example is given by the Kloosterman sums defined for $q$
squarefree and $x\in\Z$ by
\begin{equation}\label{kl-def}
  \hypk_m(x;q)=\frac{1}{q^{{m-1}/2}}
  \sum_{\substack{x_1,\ldots,x_m\in \Z/q\Z\\x_1\cdots x_m=x}}e_q(x_1+\cdots+x_m),
\end{equation}
for which we have
$$
\hypk_m(x;q)=\prod_{p|q}\hypk_m(\ov{q_p}^mx;p) =\prod_{p|q}([\times
\ov{q_p}^m]^*\hypk_m(\cdot;p))(x).
$$ 
and hence
$$
\hypk_m(x;q)=(-1)^{(m-1)\Omega(q)}t(x)
$$
where
$$
t(x)=\prod_{p|q}(-1)^{m-1}t_{\mcF_p}(x;p)\text{ with }\mcF_p=[\times
\ov{q_p}^m]^*\Kl_m
$$
is an admissible trace function modulo $q$.

Given a tuple of admissible sheaves $\uple{\mcF}=(\mcF_p)_{p\mid q}$,
we define the conductor $\cond(\uple{\mcF})$ as
$$
\cond(\uple{\mcF})=\prod_{p\mid q}\cond(\mcF_p).
$$

Thus, for the examples above, the conductor is bounded by
$C^{\Omega(q)}$ for some constant $C$ depending only on $f$ (resp. on
$m$). This will be a general feature in applications.


\subsubsection{A generalization of Proposition \ref{inc}}\label{ssec-vdc}

Thanks to the square root cancellation for complete sums of trace
functions provided by Corollary \ref{correlationcor}, we may extend
Proposition \ref{inc} to general admissible trace functions to
squarefree moduli.

\begin{proposition}[Incomplete sum of trace function]\label{inctrace}
  Let $q$ be a squarefree natural number of polynomial size and let
  $t(\cdot;q) \colon \Z\to \C$ be an admissible trace function modulo
  $q$ associated to admissible sheaves $\uple{\mcF}=(\mcF_p)_{p\mid
    q}$.
\par
Let further $N \geq 1$ be given with $N\ll q^{O(1)}$ and let $\psi_N$
be a function on $\R$ defined by
$$
\psi_N(x) = \psi\left(\frac{x-x_0}{N}\right)
$$
where $x_0\in \R$ and $\psi$ is a smooth function with compact support
satisfying
$$
|\psi^{(j)}(x)| \ll \log^{O(1)} N
$$
for all fixed $j \geq 0$, where the implied constant may depend on
$j$.
\begin{enumerate}[(i)]
\item (P\'olya-Vinogradov + Deligne) Assume that, for every $p|q$, the sheaf
  $\mcF_p$ has no polynomial phase of degree $\leq 1$. Then we have
\begin{equation}\label{vdctrace-0}
  \Bigl|\sum_n \psi_N(n) t(n;q)\Bigr| 
  \ll q^{1/2+\eps} \Bigl(1 + \frac{N}{q}\Bigr).
\end{equation}
for any $\eps>0$.
\item (one van der Corput + Deligne) Assume that, for every $p|q$, the sheaf
  $\mcF_p$ has no polynomial phase of degree $\leq 2$. Then, for any
  factorization $q = rs$ and $N \leq q$, we have
\begin{equation}\label{vdctrace-1}
  \Bigl|\sum_n \psi_N(n) t(n;q)\Bigr| \ll q^{\eps}\Bigl(
  N^{1/2}r^{1/2} + N^{1/2} s^{1/4}\Bigr).
\end{equation}
\end{enumerate}
In all cases the implied constants depend on $\eps$, $\cond(\uple{\mcF})$
and the implied constants in the estimates for the derivatives of
$\psi$.
\end{proposition}

\begin{remark} In the context of Proposition \ref{inc}, where
  $t(n;q)=e_q\left(\frac{P(n)}{Q(n)}\right)$, the assumptions $\deg
  P<\deg Q$ and $\deg(Q\ (p))=\deg(Q)$ (for all $p\mid q$), ensure that
  the sheaves $\mcL_{e_p\left(\ov{q_p}\frac{P(x)}{Q(x)}\right)}$ do
  not contain any polynomial phase of any degree.
\end{remark}

\begin{remark}\label{anyfunction} 
  For future reference, we observe that, in the proof of \eqref{vdctrace-1} below, we will not use
  any of the properties of the functions $x\mapsto t(x;p)$ for $p\mid
  r$ for a given factorization $q=rs$, except for their boundedness.
\end{remark}

\proof For each $p \mid q$, the trace function $t_{\mcF_p}$ decomposes
by~(\ref{eq-isotypic}) into a sum of at most $\rk(\mcF_p)\leq
\cond(\mcF_p)\leq \cond(\uple{\mcF})$ trace functions of isotypic admissible
sheaves, and therefore $n\mapsto t(n;q)$ decomposes into a sum of at
most $C^{\omega(q)}$ functions, each of which is an admissible trace
function modulo $q$ associated to isotypic admissible
sheaves. Moreover, if no $\mcF_p$ contains a polynomial phase of
degree $\leq d$, then all isotypic components share this property (in
particular, since $d\geq 1$ for both statements, each component is
also a Fourier sheaf). Thus we may assume without loss of generality
that each $\mcF_p$ is isotypic.

We start with the proof of \eqref{vdctrace-0}. By \eqref{complete-1} we have 
\begin{align*}
  \left|\sum_n \psi_N(n) t(n;q)\right| &\ll
  q^{1/2+\eps}\Bigl(1+\frac{|N'|}q\Bigr)
  \sup_{h\in\Zz/q\Zz}|\FT_q(t(h;q))|\\
  &\ll
  q^{1/2+\eps}\left(1+\frac{N}q\right)\sup_{h\in\Zz/q\Zz}|\FT_q(t(h;q))|
\end{align*}
for any $\eps>0$, where $N'=\sum_n\psi_N(n)$.  By Lemma \ref{crt},
\eqref{tfactor} and the definition of the Fourier transform, we have
$$
\FT_q(t(\cdot;q))(h)=\prod_{p|q}\FT_{p}(t(\cdot;p))(\ov{q_p}h).
$$
Since $t(\cdot;p)=t_{\mcF_p}$ is the trace function of a Fourier
sheaf, we have
$$
|\FT_{p}(t(\cdot;p))(\ov{q_p}h)|\leq 10\cond(\mcF_p)^2\leq
10\cond(\uple{\mcF})^2
$$
for all $h$ by~(\ref{fourierconductorbound}) (or Corollary
\ref{correlationcor} applied to the sheaves $\mcF_p$ and
$\mcL_{e_p(-\ov{q_p}x)}$).  Combining these bounds, we obtain
\eqref{vdctrace-0}.

The proof of \eqref{vdctrace-1} follows closely that of
\eqref{vdc-1}. It is sufficient to prove this bound in the case $r\leq s$. We may also assume that $r\leq N\leq s$, since,
  otherwise, the result follows either from the trivial bound or \eqref{vdctrace-0}. Then, denoting $K := \lfloor N/r\rfloor$,
we write
$$
\sum_n \psi_N(n)t(n;q)=\frac{1}{K} \sum_n \sum_{k=1}^K \psi_N(n+kr)
t(n+kr;q).
$$
Since $q=rs$, we have 
$$
t(n+kr;q) = t(n;r)t(n+kr;s),
$$
where
$$
t(n;r)=\prod_{p|r}t(n;p),\quad\quad t(n;s)=\prod_{p|s}t(n;p)
$$
are admissible trace functions modulo $r$ and $s$, respectively. 
Hence
\begin{align*}
  \Bigl|\sum_n \psi_N(n) t(n;q)\Bigr| &\ll
  \frac{1}{K} \sum_n \Bigl|\sum_{k=1}^K \psi_N(n+kr) t(n+kr;s) \Bigr| \\
  & \ll \frac{N^{1/2}}{K} \Bigl(\sum_n \Bigl|\sum_{k=1}^K \psi_N(n+kr)
  t(n+kr;s)\Bigr|^2\Bigr)^{1/2}\\
  &\ll \frac{N^{1/2}}{K} \Bigl(\sum_{1\leq k,l\leq K}A(k,l)\Bigr)^{1/2},
\end{align*}
where
$$
A(k,l)=\sum_n \psi_N(n+kr)
\overline{\psi_N(n+lr)}t(n+kr;s)\overline{t(n+lr;s)}.
$$
The diagonal contribution satisfies
$$
\sum_{1\leq k\leq K}A(k,k) \ll q^{\eps}KN
$$
for any $\eps>0$, where the implied constant depends on
$\cond(\uple{\mcF})$.


Instead of applying~(\ref{vdctrace-0}) for the off-diagonal terms, it
is slightly easier to just apply~(\ref{complete-1}). For given
$k\neq l$, since $kr$, $lr\ll N$, the sequence $\Psi_N(n)=\psi_N(n+kr)
\overline{\psi_N(n+lr)}$ satisfies the assumptions
of~(\ref{complete-1}). Denoting
$$
w(n;s)=t(n+kr;s)\overline{t(n+lr;s)},
$$
we obtain 
$$
|A(k,l)|=\Bigl|\sum_n \Psi_{N}(n)w(n;s)\Bigr| \ll
q^{\eps}s^{1/2}\sup_{h\in\Z/s\Z}|\FT_{s}(w(\cdot;s))(h)|
$$
by~(\ref{complete-1}) (since $N\leq s$). We have
$$
\FT_s(w(\cdot;s))(h)= \prod_{p\mid s}\FT_p(w(\cdot;p))(\ov{s_p}h)
$$
with $s_p=s/p$. For $p\mid k-l$, we use the trivial bound
$$
|\FT_p(w(\cdot;p))(\ov{s_p}h)|\ll p^{1/2}
$$
and for $p\nmid k-l$, we have
$$
\FT_p(w(\cdot;p))(\ov{s_p}h) =\frac{1}{p^{1/2}}
\sum_{x\in\Fp}t(x+kr;p)\ov{t(x+lr;p)}e_p(\ov{s_p}hx)\ll 1
$$
by the change of variable $x\mapsto x+kq_1$ and
\eqref{squarerootforWeylshifts}, which holds for $\mcF_p$ by our
assumptions. In all cases, the implied constant depends only on
$\cond(\mcF_p)$. Therefore we have
$$
A(k,l)\ll (k-l,s)^{1/2}q^{\eps}s^{1/2},
$$
and summing over $k\neq l$, we derive
\begin{align*}
  \Bigl|\sum_n \psi_N(n) e_q(f(n))\Bigr|&\ll \frac{q^{\eps}N^{1/2}}{K}
  \Bigl(KN+s^{1/2}\sum_{1\leq k\neq l\leq K}(k-l,s)^{1/2}\Bigr)^{1/2}
  \\
  &\ll \frac{q^{\eps}N^{1/2}}{K} (K^{1/2}N^{1/2}+s^{1/4}K)
\end{align*}
which gives the desired conclusion~(\ref{vdctrace-1}).
\qed

\begin{remark}\label{vdCiterate-trace} Similarly to Remark
  \ref{vdCiterate}, one can iterate the above argument and conclude
  that, for any $l \geq 1$, and any factorization $q = q_1 \cdots q_l$
$$
\left| \sum_n \psi_N(n) t(n;q) \right| \ll q^{\eps}\left(
\left(\sum_{i=1}^{l-1} N^{1-1/2^i} q_i^{1/2^i}\right) + N^{1-1/2^{l-1}}
q_l^{1/2^l}\right),
$$ 
assuming that $N<q$ and the $\mcF_p$ do not contain any polynomial
phase of degree $\leq l$.
\end{remark}

Specializing Proposition \ref{inctrace} to the functions in Theorem
\ref{kf-bound}, we conclude:

\begin{corollary}\label{inctrace-q} Let $q\geq 1$ be a squarefree
  integer and let $K(\cdot;q)$ be given by
$$
K(x;q):=\frac{1}{q^{1/2}} \sum_{y\in \Z/q\Z} e_q(f(x,y))
$$
where
$$
f(x,y)=\frac{1}{(y+ax+b)(y+cx+d)}+ey
$$
and $a,b,c,d,e$ are integers with $(a-c,q)=1$. 
Let further $N \geq 1$ be given with $N\ll q^{O(1)}$ and let $\psi_N$
be a function on $\R$ defined by
$$
\psi_N(x) = \psi\left(\frac{x-x_0}{N}\right)
$$
where $x_0\in \R$ and $\psi$ is a smooth function with compact support
satisfying
$$
|\psi^{(j)}(x)| \ll \log^{O(1)} N
$$
for all fixed $j \geq 0$, where the implied constant may depend on
$j$.
\par
Then we have
\begin{equation}\label{vdctrace-0-q}
  \Bigl|\sum_n \psi_N(n) K(n;q)\Bigr| 
  \ll q^{1/2+\eps}\left(1 + \frac{N}{q}\right)
\end{equation}
for any $\eps>0$.
\par
Furthermore, for any factorization $q = rs$ and $N \leq q$, we have the
additional bound
\begin{equation}\label{vdctrace-1-q}
  \Bigl|\sum_n \psi_N(n) K(n;q)\Bigr| \ll q^{\eps}\Bigl( N^{1/2}r^{1/2} +
  N^{1/2} s^{1/4}\Bigr).
\end{equation}
\end{corollary}

Indeed, it follows from Theorem~\ref{kf-bound} and the assumption
$(a-c,q)=1$ that $K_f(\cdot;q)$ is an admissible trace function modulo
$q$ associated to sheaves which do not contain any polynomial phase of
degree $\leq 2$.

\subsubsection{Correlations of hyper-Kloosterman sums of composite moduli}

Finally, we extend Proposition \ref{kls} to composite moduli:

\begin{lemma}[Correlation of hyper-Kloosterman sums]\label{lm-c-kl} Let $s,r_1,r_2$
  be squarefree integers with $(s,r_1)=(s,r_2)=1$. Let $a_1 \in
  (\Z/r_1 s)^\times$, $a_2 \in (\Z/r_2s)^\times$, and $n \in
  \Z/([r_1,r_2] s)\Z$.  Then we have
\begin{multline*}
  \sum_{h \in (\Z/s[r_1, r_2] \Z)^\times} \hypk_3(a_1 h; r_1 s)
  \overline{\hypk_3(a_2 h; r_2 s)} e_{[r_1,r_2] s}( nh )\ll
  \\
  (s[r_1,r_2])^{\eps} s^{1/2} [r_1,r_2]^{1/2}
  (a_2-a_1,n,r_1,r_2)^{1/2}(a_2 r_1^3 - a_1 r_2^3, n, s)^{1/2}
\end{multline*}
for any $\eps>0$, where the implied constant depends only on $\eps$.
\end{lemma}

\begin{proof} Let $S$ be the sum to estimate. From Lemma \ref{crt}, we
  get
$$ 
\hypk_3( a_i h; r_i s ) = \hypk_3( a_i \bar{s}^3 h; r_i) \hypk_3( a_i
\overline{r_i}^3 h; s )
$$ 
for $i=1,2$, as well as
$$
 e_{[r_1,r_2]s}(nh) = e_{[r_1,r_2]}(\bar{s}nh)
 e_s(\overline{[r_1,r_2]}nh).
$$
and therefore $S=S_1S_2$ with
\begin{align*}
  S_1&=\sum_{h \in (\Z/[r_1, r_2] \Z)^\times} \hypk_3(a_1 \bar{s}^3 h;
  r_1)
  \overline{\hypk_3(a_2 \bar{s}^3 h; r_2)} e_{[r_1,r_2]}( \bar{s} nh ),\\
  S_2&=\sum_{h \in (\Z/s \Z)^\times} \hypk_3(a_1 \overline{r_1}^3 h;
  s) \overline{\hypk_3(a_2 \overline{r_2}^3 h; s)} e_{s}(
  \overline{[r_1,r_2]} nh ).
\end{align*}
Splitting further the summands as products over the primes dividing
$[r_1,r_2]$ and $s$, respectively, we see that it is enough to prove
the estimate
\begin{equation}\label{eq-ckl-goal}
\Bigl|\sum_{h \in (\Z/p \Z)^\times} \hypk_3(b_1 h; d_1)
\overline{\hypk_3(b_2 h; d_2)} e_{p}( mh )\Bigr| \ll p^{1/2}
(b_1-b_2,m,d_1,d_2)^{1/2}
\end{equation}
for $p$ prime and integers $d_1$, $d_2\geq 1$ such that $[d_1,d_2] =
p$ is prime, and for all $m \in \Z/p\Z$, and $b_1,b_2 \in
(\Z/p\Z)^\times$.

We now split into cases.  First suppose that $d_2=1$, so that $d_1=p$.
Then we have $\hypk_3(b_2 h; d_2) = 1$, and the left-hand side
of~(\ref{eq-ckl-goal}) simplifies to
$$
\sum_{h \in (\Z/p \Z)^\times} \hypk_3(b_1 h; p) e_{p}( mh ) \ll p^{1/2}
$$
by the first part of Proposition~\ref{kls}. Similarly, we
obtain~(\ref{eq-ckl-goal}) if $d_1=1$.

If $d_1=d_2=p$ and $b_1-b_2=m=0\ (p)$, then the claim follows from the
bound $|\hypk_3(h; p)| \ll 1$ (see Remark \ref{romeo}).

Finally, if $d_1=d_2=p$ and $b_1-b_2 \neq 0\ (p)$ or $m \neq 0\ (p)$,
then~(\ref{eq-ckl-goal}) is a consequence of the second part of
Proposition \ref{kls}.
\end{proof}

Finally, from this result, we obtain the following corollary:

\begin{corollary}[Correlation of hyper-Kloosterman sums,
  II]\label{corr-2} Let $s$, $r_1$, $r_2$ be squarefree integers with
  $(s,r_1)=(s,r_2)=1$. Let $a_1 \in (\Z/r_1 s)^\times$, $a_2 \in
  (\Z/r_2s)^\times$.  Let further $H \geq 1$ be given with $H\ll
  (s[r_1,r_2])^{O(1)}$ and let $\psi_H$ be a function on $\R$ defined
  by
$$
\psi_H(x) = \psi\left(\frac{x-x_0}{H}\right)
$$
where $x_0\in \R$ and $\psi$ is a smooth function with compact support
satisfying
$$
|\psi^{(j)}(x)| \ll \log^{O(1)} H
$$
for all fixed $j \geq 0$, where the implied constant may depend on
$j$.  Then we have
\begin{multline*}
  \Bigl|
  \sum_{(h,s[r_1,r_2])=1} \Psi_H(h) \hypk_3(a_1 h; r_1 s) \overline{\hypk_3(a_2 h; r_2 s)}\Bigr| \\
  \ll (s[r_1,r_2])^{\eps}\left(\frac{H}{[r_1,r_2]s} + 1\right) s^{1/2}
  [r_1,r_2]^{1/2} (a_2-a_1,r_1,r_2)^{1/2}(a_2 r_1^3 - a_1 r_2^3,s)^{1/2}
\end{multline*}
for any $\eps>0$ and any integer $n$.
\end{corollary}

This exponential sum estimate will be the main estimate used for
controlling Type III sums in Section~\ref{typeiii-sec}.

\begin{proof}
This follows almost directly from Lemma~\ref{lm-c-kl} and the
completion of sums in Lemma~\ref{com}, except that we must incorporate
the restriction $(h,s[r_1,r_2])=1$. We do this using M\"obius
inversion: the sum $S$ to estimate is equal to
$$
\sum_{\delta\mid s[r_1,r_2]}\mu(\delta)t_1(\delta)S_1(\delta)
$$
where $t_1(\delta)$ satisfies $|t_1(\delta)|\leq \delta^{-2}$ because
$\hypk_3(0;p)=p^{-1}$ for any prime $p$, and
\begin{align*}
  S_1(\delta)&=\sum_{\delta\mid h} \Psi_H(h) \hypk_3(\alpha_1 h; r_1
  s/(\delta,r_1s)) \overline{\hypk_3(\alpha_2 h; r_2 s/(\delta,r_2s))}
  \\
  &=\sum_{h} \Psi_{H/\delta}(h) \hypk_3(\delta\alpha_1 h; r_1
  s/(\delta,r_1s)) \overline{\hypk_3(\delta\alpha_2 h; r_2
    s/(\delta,r_2s))}
\end{align*}
for some $\alpha_i\in (\Z/r_is/(\delta,r_is)\Z)^{\times}$. By
Lemma~\ref{lm-c-kl} and Lemma~\ref{com}, we have
$$
S_1(\delta)\ll (s[r_1,r_2])^{\eps}
\Bigl(\frac{H}{\delta s[r_1,r_2]}+1\Bigr)
\Bigl(\frac{s[r_1,r_2]}{\delta}\Bigr)^{1/2}
(a_2-a_1,r_1,r_2)^{1/2}
(a_2r_1^3-a_1r_2^3,s)^{1/2}
$$
(the gcd factors for $S_1(\delta)$ are divisors of those for
$\delta=1$). Summing over $\delta\mid s[r_1,r_2]$ then gives the
result.
\end{proof}

\subsubsection{The Katz Sato-Tate law over short intervals}

In this section, which is independent of the rest of this paper, we
give a sample application of the van der Corput method to Katz's
equidistribution law for the angles of the Kloosterman sums
$\hypk_2(n;q)$.  
\par
Given a squarefree integer $q\geq 1$ with $\omega(q)\geq 1$ prime
factors, we define the \emph{Kloosterman angle}
$\theta(n;q)\in[0,\pi]$ by the formula
$$
2^{\omega(q)}\cos(\theta(n;q))=\hypk_2(n;q).
$$
In \cite{GKM}, as a consequence of the determination of the geometric
monodromy group of the Kloosterman sheaf $\Kl_2$, Katz proved (among
other things) a result which can be phrased as follows:

\begin{theorem}[Katz's Sato-Tate equidistribution law] 
  As $p\ra\infty$, the set of angles
$$
\{\theta(n;p),\ 1\leq n\leq p\}\in [0,\pi]\}
$$ 
becomes equidistributed in $[0,\pi]$ with respect to the Sato-Tate
measure $\mu_{ST}$ with density
$$
\frac{2}{\pi}\sin^2(\theta)d\theta,
$$
i.e., for any continuous function $f\,:\, [0,\pi]\ra \Cc$, we have
$$
\int f(x)d\mu_{ST}(x)=\lim_{p\ra +\infty}
\frac{1}{p-1}\sum_{1\leq n\leq p}f(\theta(n;p)).
$$
\end{theorem}

By the P\'olya-Vinogradov method one can reduce the length of the
interval $[1,p]$:
\begin{proposition} 
  For any $\eps>0$, the set of angles
$$
\{\theta(n;p),\ 1\leq n\leq p^{1/2+\eps}\}\in [0,\pi]
$$ 
becomes equidistributed on $[0,\pi]$ with respect to the Sato-Tate
measure $\mu_{ST}$ as $p\ra +\infty$.
\end{proposition}

(In fact, using the ``sliding sum method'' \cite{sliding}, one can
reduce the range to $1\leq n\leq p^{1/2}\Psi(p)$ for any increasing function
$\Psi$ with $\Psi(p)\ra +\infty$).

As we show here, as a very special example of application of the van
der Corput method, we can prove a version of Katz's Sato-Tate law for
Kloosterman sums of composite moduli over shorter ranges:

\begin{theorem} 
  Let $q$ denote integers of the form $q=rs$ where $r$, $s$ are two
  distinct primes satisfying
$$ 
s^{1/2}\leq r\leq 2s^{1/2}.
$$ 
For any $\eps>0$, the set of pairs of angles
$$
\{(\theta(n\ov s^2;r),\theta(n\ov r^2;s)),\ 1\leq n\leq
q^{1/3+\eps}\}\in [0,\pi]^2
$$
becomes equidistributed on $[0,\pi]^2$ with respect to the product
measure $\mu_{ST}\times\mu_{ST}$ as $q\ra +\infty$ among such
integers.  
\par 
Consequently the set
$$
\{\theta(n;q),\ 1\leq n\leq q^{1/3+\eps}\}\in [0,\pi]
$$ 
becomes equidistributed on $[0,\pi]$ with respect to the measure
$\mu_{ST,2}$ obtained as the push forward of the measure
$\mu_{ST}\times\mu_{ST}$ by the map $(\theta,\theta')\mapsto
\mathrm{acos}(\cos\theta\cos\theta')$.
\end{theorem}

\proof The continuous
functions 
$$
\mathrm{sym}_{k,k'}(\theta,\theta'):=\mathrm{sym}_k(\theta)
\mathrm{sym}_{k'}(\theta')=\frac{\sin((k+1)\theta)}{\sin\theta}\frac{\sin((k+1)\theta')}{\sin\theta'}
$$
for $(k,k')\in\Nn_{\geq 0}-\{(0,0)\}$ generate a dense subspace of the
space of continuous functions on $[0,\pi]^2$ with mean $0$ with
respect to $\mu_{ST}\times\mu_{ST}$. Thus, by the classical Weyl
criterion, it is enough to prove that
$$
\sum_{1\leq n\leq q^{1/3+\eps}}\mathrm{sym}_k(\theta(\ov
s^2n;r))\mathrm{sym}_{k'}(\theta(\ov r^2n;s))=o(q^{1/3+\eps}).
$$
By partition of unity, it is sufficient to prove that
\begin{equation}\label{STbound}
  \sum_{n}\Psi\Bigl(\frac{n}{N}\Bigr)\mathrm{sym}_k(\theta(\ov s^2n;r))
  \mathrm{sym}_{k'}(\theta(\ov r^2n;s))\ll_{k,k'} q^{1/3+9\eps/10}.
\end{equation}
for any $N\leq q^{1/3+\eps}\log q$ and any smooth function $\Psi$ as
above, where the subscripting in $\ll_{k,k'}$ indicates that the implied constant is allowed to depend on $k,k'$.  For any fixed $(k,k')$, the function
$$
x\mapsto \mathrm{sym}_{k'}(\theta(\ov r^2x;s))
$$
is a trace function modulo $s$, namely the trace function associated
to the lisse sheaf obtained by composing the representation
corresponding to the rank $2$ pullback of the Kloosterman sheaf
$[\times \ov r^2]^*\Kl_2$ with the $k^{\operatorname{th}}$ symmetric power
representation $\mathrm{sym}_{k'}:\GL_2\ra \GL_{k'+1}$. By \cite{GKM},
this sheaf $\mathrm{sym}_{k'}\Kl_2$ is non-trivial if $k'\geq 1$, and
geometrically irreducible of rank $k'+1>1$. Therefore, if $k'\geq 1$,
the van der Corput method \eqref{vdctrace-1} (see also Remark
\ref{anyfunction}) gives
$$
\sum_{n}\Psi_N(n)\mathrm{sym}_k(\theta(\ov
s^2n;r))\mathrm{sym}_{k'}(\theta(\ov r^2n;s))\ll
N^{1/2}q^{1/6}\ll_{k,k'} q^{1/3+9\eps/10}.
$$
Indeed, $\mathrm{sym}_{k'}\Kl_2$, being geometrically irreducible of
rank $>1$, does not contain any quadratic phase. 
\par
If $k'=0$ (so that the function modulo $s$ is the constant function
$1$), then we have $k\geq 1$ and $\mathrm{sym}_{k}\Kl_2$ is
geometrically irreducible of rank $>1$. Therefore it does not contain
any linear phase, and by the P\'olya-Vinogradov method
\eqref{vdctrace-0}, we deduce
$$
\sum_{n}\Psi_N(n)\mathrm{sym}_k(\theta(\ov
s^2n;r))\mathrm{sym}_{k'}(\theta(\ov r^2n;s))\ll
r^{1/2+\eta}(1+N/r)\ll_\eta q^{1/6+\eta+\eps}
$$
for any $\eta>0$.
\qed


\section{The Type III estimate}\label{typeiii-sec}

In this section we establish Theorem \ref{newtype}\eqref{typeIII1}. Let us 
recall the statement:

\begin{theorem}[New Type III estimates]\label{th-newtypeiii}
  Let $\varpi,\delta,\sigma > 0$ be fixed quantities, 
let $I$ be a bounded subset of $\R$, let $i\geq 1$ be fixed, let $a\
(P_I)$ be a primitive congruence class, and let $M$, $N_1$, $N_2$,
$N_3 \gg 1$ be quantities with
\begin{gather}
\label{eq-x}
MN_1N_2N_3 \sim x\\
\label{eq-long}
N_1N_2,\ N_1N_3,\ N_2N_3\ggcurly x^{1/2+\sigma}
\\
\label{eq-short}
x^{2\sigma}\lessapprox N_1,\ N_2,\ N_3\lessapprox x^{1/2-\sigma}.
\end{gather}
Let $\alpha$, $\psi_1$, $\psi_2$, $\psi_3$ be smooth coefficient
sequences located at scales $M$, $N_1$, $N_2$, $N_3$, respectively.
Then we have the estimate
$$
  \sum_{\substack{d \in \DIone{I}{x^\delta}\\ d \lessapprox
      x^{1/2+2\varpi}}} |\Delta(\alpha \star \psi_1\star\psi_2\star \psi_3; a\ (d))|
  \ll x \log^{-A} x
$$
for any fixed $A>0$, provided that
\begin{equation}\label{eq-iii-cond}
  \varpi<\frac{1}{12},\quad\quad \sigma>\frac{1}{18}+\frac{28}{9}
  \varpi + \frac{2}{9}\delta.
\end{equation}
\end{theorem}

Our proof of this theorem is inspired in part by the recent work of Fouvry, Kowalski
and Michel~\cite{FKM3}, in which the value of the exponent of
distribution of the ternary divisor function $\tau_3(n)$ in arithmetic
progressions to large (prime) moduli is improved upon the earlier
results of Friedlander-Iwaniec~\cite{fik-3} and
Heath-Brown~\cite{hb-d3}. Our presentation is also more streamlined. 
The present argument moreover exploits the existence of an
averaging over divisible moduli to derive further improvements to the
exponent.

\subsection{Sketch of proofs}

Before we give the rigorous argument, let us first sketch the solution
of the model problem (in the spirit of Section \ref{typei-ii-sketch}),
of obtaining a non-trivial estimate for
\begin{equation}\label{model}
 \sum_{q \sim Q} |\Delta( \psi_1 \star \psi_2 \star \psi_3,a\ (q))|
\end{equation}
for $Q$ slightly larger than $x^{1/2}$ in logarithmic scale (i.e. out
of reach of the Bombieri-Vinogradov theorem). Here
$\psi_1,\psi_2,\psi_3$ are smooth coefficient sequences at scales
$N_1$, $N_2$, $N_3$ respectively with $N_1N_2N_3 \sim x$ and
$N_1,N_2,N_3 \lessapprox \sqrt{x}$, and $q$ is implicitly restricted
to suitably smooth or densely divisible moduli (we do not make this
precise to simplify the exposition).  The trivial bound for this sum
is $\ll \log^{O(1)} x$, and we wish to improve it at least by a factor
$\log^{-A} x$ for arbitrary fixed $A>0$.
\par
This problem is equivalent to that of estimating
$$ 
\sum_{q \sim Q} c_q \Delta( \psi_1 \star \psi_2 \star \psi_3,a\ (q))
$$
when $c_q$ is an arbitrary bounded sequence.  As in Section
\ref{typei-ii-sketch}, we write $\emt$ for unspecified main terms, and
we wish to control the expression
$$ 
\sum_{q \sim Q} c_q \sum_{n = a\ (q)} \psi_1 \star \psi_2 \star
\psi_3(n)-\emt
$$
to accuracy better than $x$.  After expanding the convolution and
completing the sums, this 
sum can be transformed to a sum roughly of the form
$$ 
\frac{1}{H} \trpsum_{1\leq |h_i|\ll H_i} \sum_{q \sim Q} c_q
\trpsum_{\substack{n_1,n_2,n_3 \in\Z/q\Z\\ n_1 n_2 n_3 = a\ (q)}} e_q(
h_1 n_1 + h_2 n_2 + h_3 n_3 )
$$
where $H_i:= Q/N_i$ and $H := H_1H_2 H_3 \sim Q^3/x$, the main term
having cancelled out with the zero frequencies. As we are taking $Q$
close to $x^{1/2}$, $H$ is thus close to $x^{1/2}$ as well.  Ignoring the
degenerate cases when $h_1,h_2,h_3$ share a common factor with $q$, we
see from \eqref{kl-def} that
$$ 
\trpsum_{\substack{n_1,n_2,n_3 \in\Z/q\Z\\ n_1 n_2 n_3 = a\ (q)}} e_q(
h_1 n_1 + h_2 n_2 + h_3 n_3 ) = q \hypk_3( a h_1 h_2 h_3; q ),
$$ 
so we are now dealing essentially with the sum of hyper-Kloosterman
sums
$$
\frac{Q}{H} \trpsum_{1\leq |h_i|\ll H_i} \sum_{q \sim Q} c_q \hypk_3(
ah_1h_2h_3; q)= \frac{Q}{H} \sum_{1\leq |h|\ll H} \tilde \tau_3(h)
\sum_{q \sim Q} c_q \hypk_3(ah;q)
$$ 
where 
$$
\tilde \tau_3(h) := \trpsum_{\substack{1\leq |h_i|\ll H_i\\
    h_1h_2h_3=h}}1
$$ 
is a variant of the divisor function $\tau_3$.  

A direct application of the deep Deligne bound
\begin{equation}\label{del}
  |\hypk_3(ah;q)|\lessapprox 1
\end{equation}
for hyper-Kloosterman sums (see Remark \ref{romeo}) gives the trivial
bound $\lessapprox Q^2$, which just fails to give the desired result,
so the issue is to find some extra cancellation in the phases of the
hyper-Kloosterman sums. 
\par
One can apply immediately the Cauchy-Schwarz inequality to eliminate
the weight $\tilde\tau_3(h)$, but it turns out to be more efficient to
first use the assumption that $q$ is restricted to densely divisible
moduli and to factor $q=rs$ where $r\sim R$, $s\sim S$, in which $R$ and
$S$ are well-chosen in order to balance the diagonal and off-diagonal
components resulting from the Cauchy-Schwarz inequality (it turns out
that the optimal choices here will be $R,S \approx x^{1/4}$).  
\par
Applying this factorization, and arguing for each $s$ separately, we
are led to expressions of the form
$$ 
\frac{Q}{H}
\sum_{1\leq |h|\ll H} \tilde \tau_3(h) \sum_{r \sim R} c_{rs}
\hypk_3(ah;rs),
$$
where we must improve on the bound $\lessapprox QR$ coming from \eqref{del}
for any given $s \sim S$.  If we then apply the Cauchy-Schwarz
inequality to the sum over $h$, we get
\begin{align*}
  \frac{Q}{H} \sum_{1\leq |h|\ll H} \tilde \tau_3(h) \sum_{r \sim R}
  c_{rs} \hypk_3(ah;rs) &\lessapprox \frac{Q}{H^{1/2}} \Bigl(\sum_{1\leq |h|\ll
    H} \Bigl|\sum_{r \sim R} c_{rs}
  \hypk_3(ah;rs)\Bigr|^2\Bigr)^{1/2}\\
  &\lessapprox \frac{Q}{H^{1/2}}\Bigl(\sumsum_{r_1,r_2 \sim R} \sum_{1\leq
    |h|\ll H} \hypk_3(ah;r_1s)
  \overline{\hypk_3(ah;r_2s)}\Bigr)^{1/2}.
\end{align*}
The inner sum over $h$ is now essentially of the type considered by
Corollary \ref{corr-2}, and this result gives an adequate bound.
Indeed, the contribution of the diagonal terms $r_1=r_2$ is
$\lessapprox R H$ (using \eqref{del}) and the contribution of each
non-diagonal sum (assuming we are in the model case where $r_1$, $r_2$ are coprime, and the other greatest common divisors appearing in Corollary \ref{corr-2} are negligible) is
$$
\sum_{1\leq |h|\ll H} \hypk_3(ah;r_1s) \overline{\hypk_3(ah;r_1s)}
\lessapprox (r_1r_2s)^{1/2}\lessapprox RS^{1/2}
$$ 
by Corollary \ref{corr-2}, leading to a total estimate of size
$$
\lessapprox \frac{Q}{H^{1/2}}\Bigl( R^{1/2}H^{1/2}+R^{3/2} S^{1/4}\Bigr).
$$
If $R=S\approx x^{1/4}$, this is very comfortably better than what we
want, and this strongly suggests that we can take $Q$ quite a bit
larger than $x^{1/2}$.


\begin{remark}\label{higher-type} It is instructive to run the same
  analysis for the fourth order sum
$$ 
\sum_{q \sim Q} |\Delta( \psi_1 \star \psi_2 \star \psi_3 \star
\psi_4,a\ (q))|
$$ 
where $\psi_1,\psi_2,\psi_3,\psi_4$ are smooth at scales
$N_1,N_2,N_3,N_4$ with $N_1 \ldots N_4 \sim x$ and $N_1$ ,\ldots, $N_4
\lessapprox x^{1/2} \approx Q$. This is a model for the ``Type IV''
sums mentioned in Remark \ref{type-iv}, and is clearly related to the
exponent of distribution for the divisor function $\tau_4$. 
\par
The quantity $H$ is now of the form $H \approx Q^4/x \approx x$, and
one now has to estimate the sum
$$ 
\sum_{1\leq |h|\ll H} \tilde \tau_4(h) \sum_{q \sim Q} c_q
\hypk_4(ah;q)
$$ 
to accuracy better than $H x / Q^{3/2} \approx x^{5/4}$.  If we apply
the Cauchy-Schwarz inequality in the same manner after exploiting a
factorization $q=rs$ with $r\sim R$, $s\sim S$ and $RS \sim Q \approx
x^{1/2}$, we end up having to control
$$ 
\sumsum_{r_1,r_2 \sim R} \left|\sum_{1\leq |h|\ll H} \hypk_4(ah;r_1s)
  \overline{\hypk_4(ah;r_2s)}\right|
$$ 
with accuracy better than $(x^{5/4}/S)^2 / H \approx x^{3/2}/S^2$.
The diagonal contribution $r_1=r_2$ is $\lessapprox RH \approx
x^{3/2}/S$, and the off-diagonal contribution is $\approx R^2
(R^2S)^{1/2} \approx x^{3/2} / S^{5/2}$. However even with the optimal
splitting $S \approx 1$, $R \approx Q$, one cannot make both of these
terms much smaller than the target accuracy of $x^{3/2}/S^2$.  Thus
the above argument does not improve upon the Bombieri-Vinogradov
inequality for Type IV sums.  (It is known, due to Linnik, that the
exponent of distribution for $\tau_4$ is at least $1/2$, in the
stronger sense that the asymptotic formula holds for all moduli $\leq
x^{1/2-\eps}$ for $\eps>0$.)  The situation is even worse, as the
reader will check, for the Type V sums, in that one now cannot even
recover Bombieri-Vinogradov with this method.
\end{remark}

We will give the rigorous proof of Theorem
\ref{newtype}\eqref{typeIII1} in the next two sections, by first
performing the reduction to exponential sums, and then concluding the
proof.

\subsection{Reduction to exponential sums}


By Theorem \ref{bvt} (the general version of the Bombieri-Vinogra\-dov
theorem) we have
$$
\sum_{q\leq x^{1/2}\log^{-B(A)} x}|\Delta(\alpha\star \psi_1\star
\psi_2\star \psi_3)|\ll x\log^{-A} x
$$
for some $B(A)\geq 0$. We may therefore restrict our attention to moduli $q$ in the range
$x^{1/2}/\log^Bx\leq q \lessapprox x^{1/2+2\varpi}$.
\par
We also write $N=N_1N_2N_3$. From~(\ref{eq-long})
and~(\ref{eq-short}), we deduce
\begin{equation}\label{eq-n}
  x^{3/4+3\sigma/2}\lessapprox
  (N_1N_2)^{1/2}(N_1N_3)^{1/2}(N_2N_3)^{1/2}=N\lessapprox x^{3/2-3\sigma}.
\end{equation}

It is convenient to restrict $q$ to a finer-than-dyadic interval $\ftd(Q)$ in order to separate variables later using Taylor
expansions.  More precisely, 
for a small fixed $\eps>0$ and some fixed
$c\geq 1$, we denote by $\ftd=\ftd(Q)$ a finer-than-dyadic
interval of the type
$$
\ftd(Q) := \{ q: Q(1-cx^{-\eps})\leq q\leq Q(1+cx^{- \eps})\},
$$
(assuming, as always, that $x$ is large, so that $c x^{-\eps}$ is less than (say) $1/2$)
and abbreviate
$$ 
\sum_q A_q= \sum_{\substack{q \in \DIone{I}{x^\delta}\\ q
    \in\ftd(Q)}}A_q
$$
for any $A_q$.
\par
Theorem~\ref{th-newtypeiii} will clearly follow if we prove that, for
$\eps>0$ sufficiently small, we have
\begin{equation}\label{eq-target-iii}
  \sum_{q}
  |\Delta(\alpha \star \psi_1 \star \psi_2 \star \psi_3; a\ (q))|
  \lessapprox x^{-2\eps} M N
\end{equation}
for all $Q$ such that
\begin{equation}\label{qp}
 x^{1/2} \lessapprox Q \lessapprox x^{1/2+2\varpi}.
\end{equation}

We fix $Q$ as above and denote by $\Sigma(Q;a)$ the left-hand side
of~(\ref{eq-target-iii}).  We have
$$
\Sigma(Q;a)=\sum_q c_q \Delta(\alpha \star \psi_1 \star \psi_2 \star
\psi_3; a\ (q))
$$ 
for some sequence $c_q$ with $|c_q|=1$. We will prove that, for any
$a\ (q)$, we have
\begin{equation}\label{ton}
  \sum_q c_q \sum_{n = a\ (q)} (\alpha \star \psi_1 \star \psi_2 \star
  \psi_3)(n) 
  = X + O( x^{-2\eps+o(1)} M N  )
\end{equation}
for some $X$ that is independent of $a$ (but that can depend on all other
quantities, such as $c_q$, $\alpha$, or
$\psi_1,\psi_2,\psi_3$). Then~(\ref{eq-target-iii}) follows by
averaging over all $a$ coprime to $P_I$ (as in the reduction
to~(\ref{straw-2}) in Section~\ref{typei-ii-sec}).

The left-hand side of \eqref{ton}, say $\Sigma_1(Q;a)$, is equal to
\begin{equation}\label{tin}
  \Sigma_1(Q;a)=
  \sum_q c_q \sum_{(m,q)=1} \alpha(m) \trpsum_{n_1,n_2,n_3} 
  \psi_1(n_1) \psi_2(n_2) \psi_3(n_3) \onef_{mn_1n_2n_3 = a\ (q)}.
\end{equation}
The next step is a variant of the completion of sums technique from
Lemma \ref{com}.  In that lemma, the Fourier coefficients of the
cutoff functions were estimated individually using the fast decay of
the Fourier transforms. In our current context, we want to keep track
to some extent of their dependency on the variable $q$. Since we have
restricted $q$ to a rather short interval, we can separate the
variables fairly easily using a Taylor expansion.
\par
Note first that for $i=1,2,3$, one has
$$ 
N_i \lessapprox x^{1/2-\sigma} \lessapprox x^{-\sigma} Q,
$$
so in particular $\psi_i$ is supported in $(-q/2,q/2]$ if $x$ is large
enough.  By discrete Fourier inversion, we have
\begin{equation}\label{dfi}
\psi_i(x)=\frac{1}{q}\sum_{-q/2<h\leq
  q/2}\Psi_i\Bigl(\frac{h}{q}\Bigr) e\Big(\frac{hx}{q}\Bigr)
\end{equation}
where
$$
\Psi_i(y)=\sum_n{\psi_i(n)e(-ny)}
$$
is the analogue of the function $\Psi$ in the proof of
Lemma~\ref{com}. As in that lemma, using the smoothness of $\psi_i$,
Poisson summation, and integration by parts, we derive the bound
$$ 
|\Psi_i(y)| \lessapprox N_i (1 + N_i |y|)^{-C}
$$
for any fixed $C \geq 0$ and any $-1/2 \leq y \leq 1/2$
(see~(\ref{eq-decay-psi})). More generally we obtain
$$ 
|\Psi^{(j)}_i(y)| \lessapprox N_i^{1+j} (1 + N_i |y|)^{-C}
$$
for any fixed $C\geq 0$, any $j \geq 0$ and any $-1/2 \leq y \leq
1/2$.  
\par
Denoting $H_i:=Q/N_i\ggcurly x^{\sigma}$, we thus have
$$
\Psi_i^{(j)}\Bigl(\frac{h}{q}\Bigr) \ll x^{-100}
$$
(say) for $x^{\eps/2} H_i < |h| \leq q/2$ and all fixed $j$.  On the other hand, for
$|h| \leq x^{\eps/2} H_i$ and $q\in \ftd$, a Taylor expansion using the definition of $\ftd$ and $H_i$ gives
$$
\frac{1}{q} \Psi_i\Bigl(\frac{h}{q}\Bigr)=
\frac{1}{q}\sum_{j=0}^J\frac{1}{j!}\Psi_i^{(j)}(h/Q)\eta^j
+O(N_i^{2+J}|\eta|^{J+1})
$$
for any fixed $J$ where $\alpha$ is the $q$-dependent quantity
$$
\eta:=\frac{h}{q}-\frac{h}{Q} =\frac{h(Q-q)}{qQ}\ll
x^{-\eps}\frac{h}{Q} \ll x^{-\eps/2}\frac{1}{N_i}.
$$
Thus we obtain
$$
\frac{1}{q} \Psi_i\Bigl(\frac{h}{q}\Bigr)=
\frac{1}{q}\sum_{j=0}^J\frac{1}{j!}\Psi_i^{(j)}\Bigl(\frac{h}{Q}\Bigr)
\Bigl(\frac{h}{Q}\Bigr)^j \Bigl(\frac{q-Q}{q}\Bigr)^j
+O(x^{-(J+1)\eps/2}N_i).
$$
Taking $J$ large enough, depending on $\eps>0$ but still fixed, this gives an
expansion
\begin{equation}\label{expand}
\frac{1}{q} \Psi_i\Bigl(\frac{h}{q}\Bigr) = 1_{|h| < x^{\eps/2} H_i} \frac{1}{H_i} \sum_{j=0}^J
c_i(j,h) \frac{Q}{q}\left(\frac{q-Q}{q}\right)^j + O( x^{-100} )
\end{equation}
with coefficients that satisfy
$$
c_i(j,h)=\frac{1}{j!}\Psi_i^{(j)}\Bigl(\frac{h}{Q}\Bigr)
\Bigl(\frac{h}{Q}\Bigr)^j\frac{H_i}{Q}\ll 1,
$$
as well as
$$
\Bigl(\frac{Q}{q}\Bigr) \Bigl(\frac{q-Q}{q}\Bigr)^j\ll 1.
$$
Let
\begin{equation}\label{eq-def-h}
H:=H_1H_2H_3=Q^3/N.
\end{equation}
Inserting \eqref{expand} for $i=1$, $2$, $3$ into \eqref{dfi} and the definition \eqref{tin} of
$\Sigma_1(Q;a)$, we see that $\Sigma_1(Q;a)$ can be expressed (up to errors of $O(x^{-100})$) as a sum
of a bounded number (depending on $\eps$) of expressions, each of the form
$$
\Sigma_2(Q;a)=\frac{1}{H} \sum_q \eta_q \sum_{(m,q)=1} \alpha(m)
\sum_{\uple{h}} c(\uple{h}) \sum_{\uple{n}\in (\Z/q\Z)^3} e_{q}( h_1
n_1 + h_2 n_2 + h_3 n_3 ) \onef_{mn_1n_2n_3 = a\ (q)},
$$
where $\eta_q$ is a bounded sequence supported on $\ftd\cap
\DIone{I}{x^{\delta}}$, $\uple{h}:=(h_1,h_2,h_3)$ and $c(\uple{h})$ are
bounded coefficients supported on $|h_i|\leq x^{\eps/2}H_i$, and
$\uple{n}$ abbreviates $(n_1,n_2,n_3)$.
 Our task is now to show that
$$
\Sigma_2(Q;a)  = X_2 + O( x^{-2\eps+o(1)} M N  )
$$
for some quantity $X_2$ that can depend on quantities such as $\eta_q$, $\alpha$, $c$, $H$, but which is independent of $a$.
\par
We use $F(\uple{h}, a; q)$ to denote the hyper-Kloosterman type sum
\begin{equation}\label{fuple-def}
F(\uple{h}, a; q) := \frac{1}{q} \sum_{\uple{n}\in
  ((\Z/q\Z)^\times)^3} e_{q}( h_1 n_1 + h_2 n_2 + h_3 n_3 )
\onef_{n_1n_2n_3 = a\ (q)}
\end{equation}
for $\uple{h}=(h_1,h_2,h_3) \in (\Z/q\Z)^3$ and $a \in
(\Z/q\Z)^\times$ (note that the constraint $n_1n_2n_3=a\ (q)$ forces $n_1,n_2,n_3$ coprime to $q$), so that 
$$
\Sigma_2(Q;a)=\frac{Q}{H} \sum_q \eta'_q \sum_{(m,q)=1} \alpha(m)
\sum_{\uple{h}} c(\uple{h}) F(\uple{h},a\overline{m};q)
$$
where $\eta'_q := \frac{q}{Q} \eta_q$ is a slight variant of $\eta_q$.


We next observe that $F(\uple{h}, a \overline{m}; q)$ is independent
of $a$ if $h_1h_2h_3=0$ (as can be seen by a change of variable). Thus
the contribution $X_2$ to the sum from tuples $\uple{h}$ with
$h_1h_2h_3=0$ is independent of $a$. The combination of these terms
$X_2$ in the decomposition of $\Sigma_1(Q;a)$ in terms of instances of
$\Sigma_2(Q;a)$ is the quantity $X$ in~(\ref{ton}). We denote by
$\Sigma'_2(Q;a)$ the remaining contribution.  Our task is now to show that
\begin{equation}\label{sqa}
\Sigma'_2(Q,a) \lessapprox x^{-2\eps} MN.
\end{equation}

We must handle possible common factors of $q$ and $h_1h_2h_3$ for
$h_1h_2h_3\neq 0$ (the reader may skip the necessary technical details
and read on while assuming that $q$ is always coprime to each of the
$h_i$, so that all the $b$-factors appearing below become equal to $1$).  

For $i=1,2,3$, we write
$$
h_i=b_il_i
$$ 
where $(l_i,q)=1$ and $b_i\mid q^{\infty}$ (i.e., $b_i$ is the product
of all the primes in $h_i$, with multiplicity, that also divide
$q$). We also write
\begin{equation}\label{eq-b}
  b :=\prod_{p\mid b_1b_2b_3}p= (h_1h_2h_3,q),
\end{equation}
so that we have a factorization $q=bd$, where $d\in
\DIone{I}{bx^{\delta}}$ by Lemma~\ref{fq}(i), since $q$ is
$x^\delta$-densely divisible. 

By Lemma \ref{crt}, we have 
$$ 
F( \uple{h}, a \overline{m}; q) = F( \bar{d}\uple{h}, a \overline{m};
b) F( \bar{b}\uple{h}, a \overline{m}; d)
$$
where $\bar{b}\uple{h} := (\bar{b}h_1, \bar{b}h_2, \bar{b}h_3)$.
By an easy change of variable, the second factor satisfies
$$
F( \bar{b}\uple{h}, a \overline{m}; d) =\hypk_3 ( a h_1 h_2 h_3
\overline{mb^3}; d )= \hypk_3\Bigl( \frac{ab_1b_2b_3}{b^3}
\frac{l_1l_2l_3}{m};d\Bigr).
$$
We observe that the residue class $ab_1b_2b_3\overline{mb^3}\ (d)$ is
invertible.

Denoting $\uple{b}:=(b_1,b_2,b_3)$, $\uple{l}:=(l_1,l_2,l_3)$, we can
thus write
$$
\Sigma'_2(Q;a)=\frac{Q}{H}\sum_{\uple{b}}\sum_{\uple{l}}
c(\uple{b},\uple{l}) \sum_{\substack{d\in
    \DIone{I}{bx^{\delta}}\\(d,bl_1l_2l_3)=1}}\eta'_{bd}
\sum_{(m,bd)=1} \alpha(m) F(\bar{d}\uple{h},a\overline{m};b)
\hypk_3\Bigl(\frac{ab_1b_2b_3}{b^3} \frac{l_1l_2l_3}{m};d\Bigr)
$$
where $b$ is defined as in~(\ref{eq-b}), 
$c(\uple{b},\uple{l}):=c(b_1l_1,b_2l_2,b_3l_3)$, and the sum over $l_i$
is now over the range
\begin{equation}\label{eq-l-range}
  0 < |l_i| \leq \frac{x^{\eps/2} H_i}{b_i}.
\end{equation}

To control the remaining factor of $F$, we have the following
estimate, where we denote by $n^{\flat}$ the largest squarefree
divisor of an integer $n\geq 1$ (the \emph{squarefree radical of
  $n$}). Note that $b=(b_1b_2b_3)^{\flat}$.

\begin{lemma}  Let the notation and hypotheses be as above.

\emph{(1)} We have
$$ 
|F( \overline{d}\uple{h}, a \overline{m}; b)| \leq \frac{b_1^{\flat}
  b_2^{\flat}b_3^{\flat}}{b^2}.
$$
\par
\emph{(2)} The sum $F( \overline{d}\uple{h},a\overline{m};b)$ is
independent of $d$ and $m$.
\end{lemma}

\begin{proof}
  By further applications of Lemma \ref{crt} it suffices for (1) to
  show that
$$ 
|F( \uple{c}, a; p)| \leq \frac{(c_1,p) (c_2,p) (c_3,p)}{p^2}
$$
whenever $p$ is prime, $\uple{c}=(c_1,c_2,c_3) \in (\Z/p\Z)^3$, with
$c_1c_2c_3 = 0\ (p)$, and $a \in (\Z/p\Z)^\times$. Without loss of
generality we may assume that $c_3 = 0\ (p)$, and then 
$$ 
F(\uple{c},a;p)=\frac{1}{p} \sumsum_{n_1,n_2 \in (\Z/p\Z)^\times} e_p(c_1
n_1 + c_2 n_2 ),
$$ 
from which the result follows by direct computation of Ramanujan
sums (see e.g. \cite[(3.5)]{ik}). Similarly, we see that the value of $F(\uple{c},a;p)$ only depends
on which $c_i$ are divisible by $p$ and which are not, and this gives
(2). 
\end{proof}

This lemma leads to the estimate
\begin{align}
  |\Sigma'_2(Q;a)|&\ll \frac{Q}{H} \sum_{\uple{b}}
  \frac{b_1^{\flat}b_2^{\flat}b_3^{\flat}}{b^2} \sum_{\uple{l}}
  \Bigl|\sum_{\substack{d\in \DIone{I}{bx^\delta}\\(b l_1l_2l_3,d)=1}}
  \eta'_{bd} \sum_{(m,bd)=1} \alpha(m) \hypk_3\Bigl(
  \frac{ab_1b_2b_3l_1l_2l_3}{b^3m};d\Bigr)\Bigr|\nonumber\\
  &\ll \frac{Q}{H} \sum_{\uple{b}}
  \frac{b_1^{\flat}b_2^{\flat}b_3^{\flat}}{b^2}T(\uple{b})
\label{eq-tb}
\end{align}
with
$$
T(\uple{b}):= \sum_{0 < |\ell|\leq x^{3\eps/2}H/b_1b_2b_3} \tau_3(\ell)
\Bigl|\sum_{\substack{d\in b^{-1} \DIone{I}{bx^\delta}\cap \ftd\\(b
    \ell,d)=1}} \eta'_{bd} \sum_{(m,bd)=1} \alpha(m) \hypk_3\Bigl(
\frac{a\ell b_1b_2b_3}{b^3m};d\Bigr)\Bigr|;$$
following \cite{hb-d3} (particularly the arguments on p. 42), we have collected common values of $\ell=l_1l_2l_3$, and also
replaced the bounded coefficients $\eta'_{bd}$, supported on
$\ftd$, with their absolute values.  This is the desired reduction of
Type III estimates to exponential sums.


\subsection{End of the proof}

We now focus on estimating $T(\uple{b})$. First of all, we may assume
that
\begin{equation}\label{bqb}
\frac{Q}{b}\gg 1,\quad\quad x^{3\eps/2}\frac{H}{b_1b_2b_3}\gg 1
\end{equation}
since otherwise $T(\uple{b})=0$. 

Let $y=bx^{\delta}$ and let $S$ be a parameter such that
\begin{equation}\label{q}
  1 \leq S \leq y\frac{Q}{2b}=\frac{x^{\delta}Q}{2}.
\end{equation}
The moduli $d$ in the definition of $T(\uple{b})$ are $y$-densely
divisible and we have $1\leq S\leq dy$ (for $x$ sufficiently large), so that there exists a
factorization $d=rs$ with
$$
y^{-1}S\leq s\leq S,\quad\quad \frac{Q}{bS}\ll r\ll \frac{yQ}{bS},
$$
and $(r,s)=1$ (if $d<S\leq dy$, we take $s=d$ and $r=1$).

Thus we may write
$$
T(\uple{b}) \ll \sum_{\substack{y^{-1} S \leq s \leq S \\ (b\ell,s)=1}} \sum_{0 < |\ell| \leq
  H_{\uple{b}}} \tau_3(\ell)  \Bigl| \sum_{\substack{r \in \Scal_I\\ \frac{Q}{bS}\ll r\ll \frac{yQ}{bS}
	\\ (b
    \ell s,r)=1}} \eta'_{b,r,s} \sum_{ (m,brs)=1} \alpha(m)
\hypk_3\Bigl(\frac{a\ell b_1b_2b_3}{b^3m};rs\Bigr)\Bigr|
$$ 
where $\eta'_{b,r,s}$ is some bounded sequence and
$$ 
H_{\uple{b}} := \frac{x^{3\eps/2} H}{b_1 b_2 b_3}.
$$

We apply the Cauchy-Schwarz inequality to the sum over $s$ and $l$. As
usual, we may insert a smooth coefficient sequence $\psi_{H_{\uple{b}}}$ at scale
$H_{\uple{b}}$, equal to one on $[-H_{\uple{b}},H_{\uple{b}}]$, and
derive
$$
|T(\uple{b})|^2\leq T_1T_2
$$
where
$$
T_1:=\sum_{y^{-1} S \leq s \leq S}\frac{1}{s} \sum_{0 < |\ell| \leq
  H_{\uple{b}}} \tau_3(\ell)^2\lessapprox H_{\uple{b}}
$$
(by Lemma \ref{divisor-bound}) and
$$
T_2:= \sum_{y^{-1} S \leq s \leq S}\sum_{\ell}
s\psi_{H_{\uple{b}}}(\ell)\Bigl|  \sum_{\substack{r \in \Scal_I\\ \frac{Q}{bS}\ll r\ll \frac{yQ}{bS} \\ (b \ell,rs)=(r,s)=1}}
\eta'_{b,r,s} \sum_{ (m,brs)=1} \alpha(m)
\hypk_3\Bigl(\frac{a\ell b_1b_2b_3}{b^3m};rs\Bigr)\Bigr|^2.
$$


We expand the square and find 
$$
|T_2| \leq  \sum_{y^{-1}S\leq s\leq S}s\sumsum_{r_1,r_2} \sumsum_{m_1,m_2}
|\alpha(m_1)| |\alpha(m_2)| |U(r_1,r_2,s,m_1,m_2)|,
$$
where we have omitted the summation conditions
$$ r_i \in \Scal_I; \frac{Q}{bS}\ll r_i \ll \frac{yQ}{bS}; (b \ell, r_i s ) = (r_i,s) = (m_i,br_is) = 1 \hbox{ for } i=1,2$$
on $r_1$, $r_2$ and $m_1$,
$m_2$ for brevity, and where
$$
U(r_1,r_2,s,m_1,m_2):=
\sum_{\ell: (\ell,r_1r_2s)=1}\psi_{H_{\uple{b}}}(\ell)\hypk_3\Bigl(\frac{a\ell b_1b_2b_3}{b^3m_1};r_1s\Bigr)
\overline{\hypk_3\Bigl(\frac{a\ell b_1b_2b_3}{b^3m_2};r_2s\Bigr)}
$$
is exactly the type of sum considered in Corollary~\ref{corr-2}
(recall that  $ab_1b_2b_3$ is coprime to $r_1r_2 s$). 

We first consider the ``diagonal terms'', which here mean the
cases where 
$$
\frac{ab_1b_2b_3}{b^3m_1}r_2^3-\frac{ab_1b_2b_3}{b^3m_2}r_1^3=
\frac{ab_1b_2b_3}{b^3m_1m_2}(m_2r_2^3-m_1r_1^3)=0.
$$ 
Using the Deligne bound $|\hypk_3(x;d)|\lessapprox 1$ when $(d,x)=1$
(Remark~\ref{romeo}), this contribution $T'_2$ is bounded by
\begin{align*}
  T'_2&\lessapprox H_{\uple{b}}\sumsum_{r_1,r_2}\sum_{y^{-1}S\leq
    s\leq
    S}s \sumsum_{\substack{m_1,m_2\\ m_1r_1^3=m_2r_2^3}}|\alpha(m_1)\alpha(m_2)|\\
  &\lessapprox H_{\uple{b}}M\sum_{Q/(bS)\ll r_1\ll
    yQ/(bS)}\Bigl(\frac{Q}{br_1}\Bigr)^2
\end{align*}
since each pair $(r_1,m_1)$ determines $\lessapprox 1$ pairs
$(r_2,m_2)$, and since $s$ is, for each $r_1$, constrained to be $\sim
Q/(br_1)$ by the condition $r_1s\sim Q/b$. Summing, we obtain
\begin{equation}\label{eq-tp2}
T'_2\lessapprox \frac{H_{\uple{b}}MQS}{b}.
\end{equation}


We now turn to the off-diagonal case $m_1 r_1^3 - m_2 r_2^3 \neq 0$.
By Corollary \ref{corr-2}, we have
\begin{multline*}
U(r_1,r_2,s,m_1,m_2)\lessapprox
\Bigl(\frac{H_{\uple{b}}}{[r_1,r_2]s}+1\Bigr)
(s[r_1,r_2])^{1/2}\\
\times (r_1,r_2,m_2-m_1)^{1/2}
(m_1r_1^3-m_2r_2^3,s)^{1/2}
\end{multline*}
in this case. We now sum
these bounds to estimate the non-diagonal contribution $T''_2$ to
$T_2$. This is a straightforward, if a bit lengthy, computation, and
we state the result first:

\begin{lemma}\label{lm-sumup}
We have
$$
T''_2\lessapprox \frac{M^2Q^2}{b^2}\Bigl(
\frac{H_{\uple{b}}b^{1/2}}{Q^{1/2}}\Bigl(\frac{bS}{Q}\Bigr)^{1/2}
+\frac{Q^{1/2}}{b^{1/2}}\Bigl(\frac{x^{\delta}Q}{S}\Bigr)^{1/2} \Bigr).
$$
\end{lemma}

We first finish the proof of the Type III estimate using this. We
first derive
$$
T_2=T'_2+T''_2\lessapprox \frac{MQH_{\uple{b}}S}{b}+
\frac{M^2QS^{1/2}H_{\uple{b}}}{b}+ \frac{y^{1/2}M^2Q^3}{b^3S^{1/2}}.
$$
We select the parameter $S$ now, by optimizing it to minimize the sum
of the first and last terms, subject to the constraint $S \leq (yQ)/(2b)$. Precisely, let
$$
S=\min\Bigl(\Bigl(\frac{Q}{b}\Bigr)^{4/3}\frac{y^{1/3}M^{2/3}}{H_{\uple{b}}^{2/3}},
\frac{yQ}{2b}\Bigr).
$$
This satisfies~(\ref{q}) if $x$ is large enough: we have $S\leq
(yQ)/(2b)$ by construction, while $S\geq 1$ (for $x$ large enough)
follows either from $(yQ)/(2b)\gg y/2$ (see~(\ref{bqb})), or from
$$
\Bigl(\frac{Q}{b}\Bigr)^{4}\frac{yM^{2}}{H_{\uple{b}}^{2}}=
\frac{(b_1b_2b_3)^2}{b^2} \frac{(MN)^2x^{\delta-3\eps}}{bQ^2} \gg
x^{2+\delta-3\eps}Q^{-3}\gg x^{1/2+\delta-6\varpi-3\eps}\gg x^\eps
$$
if $\eps>0$ is small enough (using $b\ll Q$ and $\varpi<1/12$). 

This value of $S$ leads to
$$
|T(\uple{b})|^2\lessapprox H_{\uple{b}}\Bigl(
\frac{y^{1/3}H_{\uple{b}}^{1/3}M^{5/3}Q^{7/3}}{b^{7/3}}
+\frac{y^{1/6}H_{\uple{b}}^{2/3}M^{7/3}Q^{5/3}}{b^{5/3}}
+M^2\Bigl(\frac{Q}{b}\Bigr)^{5/2}
\Bigr)
$$
(where the third term only arises if $S=(yQ)/(2b)$), which gives
$$
T(\uple{b})\lessapprox \frac{x^{5\eps/4}}{(b_1b_2b_3)^{1/2}b}
\Bigl(
x^{\delta/6}H^{2/3}M^{5/6}Q^{7/6}+
x^{\delta/12}H^{5/6}M^{7/6}Q^{5/6}+
H^{1/2}MQ^{5/4}
\Bigr)
$$
using the definition of $H_{\uple{b}}$ and the bound $b_i \ge 1$
(to uniformize the three denominators involving $b$ and $\uple{b}$).
\par
We will shortly establish the following elementary fact.

\begin{lemma}\label{lm-euler}
The unsigned series
$$
\trpsum_{b_1,b_2,b_3\geq
  1}\frac{b_1^{\flat}b_2^{\flat}b_3^{\flat}}{(b_1b_2b_3)^{1/2}b^3}
$$
converges to a finite value.
\end{lemma}

Now from~(\ref{eq-tb}) and this lemma, we get
$$
\Sigma'_2(Q;a)\lessapprox \frac{x^{5\eps/4}Q}{H} \Bigl(
x^{\delta/6}H^{2/3}M^{5/6}Q^{7/6}+ x^{\delta/12}H^{5/6}M^{7/6}Q^{5/6}+
H^{1/2}MQ^{5/4} \Bigr),
$$
We now show that this implies \eqref{sqa} under suitable conditions on
$\delta$, $\varpi$ and $\sigma$. Indeed, we have
$$
\frac{x^{5\eps/4}Q}{H} \Bigl( x^{\delta/6}H^{2/3}M^{5/6}Q^{7/6}+
x^{\delta/12}H^{5/6}M^{7/6}Q^{5/6}+ H^{1/2}MQ^{5/4} \Bigr) \lessapprox
MN( E_1+E_2+E_3)
$$
where
\begin{align*}
  E_1&:=\frac{x^{5\eps/4+\delta/6}Q^{13/6}}{H^{1/3}M^{1/6}N}
  =\frac{x^{5\eps/4+\delta/6-1/6}Q^{7/6}}{N^{1/2}} \lessapprox
  Q^{7/6}x^{5\eps/4+\delta/6-3\sigma/4-13/24}
  \\
  E_2&:=\frac{x^{5\eps/4+\delta/12}Q^{11/6}M^{7/6}}{H^{1/6} MN}
  =\frac{x^{5\eps/4+\delta/12+1/6}Q^{4/3}}{N} \lessapprox
  Q^{4/3}x^{5\eps/4+\delta/12-3\sigma/2-7/12}
  \\
  E_3&:=\frac{x^{5\eps/4}Q^{9/4}}{H^{1/2}N}
  =\frac{x^{5\eps/4}Q^{3/4}}{N^{1/2}}\lessapprox
  Q^{3/4}x^{5\eps/4-3/8-3\sigma/4}
\end{align*}
using the definition~(\ref{eq-def-h}) of $H$ and the lower
bound~(\ref{eq-n}) for $N$. Using $Q \lessapprox x^{1/2+2\varpi}$, we see that we will
have $E_1+E_2+E_3\lessapprox x^{-2\eps}$ for some small positive
$\eps>0$ provided
$$
\begin{cases}
  \tfrac{7}{6}(\tfrac{1}{2}+2\varpi)+\tfrac{\delta}{6}-\tfrac{3\sigma}{4}
  - \tfrac{13}{24}<0\\
  \tfrac{4}{3}(\tfrac{1}{2}+2\varpi)+\tfrac{\delta}{12}-\tfrac{3\sigma}{2}
  - \tfrac{7}{12}<0\\
  \tfrac{3}{4}(\tfrac{1}{2}+2\varpi)-\tfrac{3\sigma}{4}
  -\tfrac{3}{8}<0
\end{cases}\quad
\Leftrightarrow \quad\quad
\begin{cases}
\sigma>\tfrac{28}{9}\varpi+\tfrac{2}{9}\delta+\tfrac{1}{18}\\
\sigma>\tfrac{16}{9}\varpi+\tfrac{1}{18}\delta+\tfrac{1}{18}\\
\sigma>2\varpi.
\end{cases}
$$
However, the first condition implies the second and third. Thus we
deduce Theorem~\ref{th-newtypeiii}, provided that we prove the two lemmas
above, which we will now do.

\begin{proof}[Proof of Lemma~\ref{lm-sumup}]
  We will relax somewhat the conditions on $r_1$, $r_2$ and $s$. We
  recall first that
$$
\frac{Q}{bS}\ll r_1,r_2\ll \frac{yQ}{bS}=\frac{x^{\delta}Q}{S}.
$$
Furthermore, the summation conditions imply $r_1s\sim Q/b\sim r_2s$,
and in particular $r_1$ and $r_2$ also satisfy $r_1\sim r_2$. In
addition, as above, we have $s\sim Q/(br_1)$ for a given $r_1$.
\par
Using this last property to fix the size of $s$, we have
\begin{multline*}
  T''_2\lessapprox \frac{Q}{b}\sumsum_{\frac{Q}{bS}\ll r_1\sim r_2\ll
    \frac{yQ}{bS}}\frac{1}{r_1}
  \Bigl(\frac{H_{\uple{b}}(br_1)^{1/2}}{(Q[r_1,r_2])^{1/2}}+
  \frac{(Q[r_1,r_2])^{1/2}}{(br_1)^{1/2}}\Bigr)\\
  \sumsum_{\substack{m_1,m_2\sim M\\ r_1^3m_1\neq r_2^3m_2}}
  (r_1,r_2,m_1-m_2)^{1/2} \sum_{s\sim
    Q/(br_1)}(r_1^3m_1-r_2^3m_2,s)^{1/2}.
\end{multline*}
By Lemma~\ref{ram-avg}, the inner sum is $\lessapprox Q/(br_1)$ for
all $(r_1,r_2,m_1,m_2)$, and similarly, we get
$$
\sumsum_{m_1,m_2\sim M} (r_1,r_2,m_1-m_2)^{1/2}
\lessapprox M^2+M(r_1,r_2)^{1/2},
$$
so that
$$
T''_2\lessapprox \Bigl(\frac{Q}{b}\Bigr)^2\sumsum_{\frac{Q}{bS}\ll
  r_1\sim r_2\ll \frac{yQ}{bS}}\frac{1}{r_1^2}(M^2+M(r_1,r_2)^{1/2})
\Bigl(\frac{H_{\uple{b}}(br_1)^{1/2}}{(Q[r_1,r_2])^{1/2}}+
\frac{(Q[r_1,r_2])^{1/2}}{(br_1)^{1/2}}\Bigr).
$$
We denote $r=(r_1,r_2)$ and write $r_i=rt_i$, and thus obtain
\begin{align*}
  T''_2&\lessapprox \Bigl(\frac{Q}{b}\Bigr)^2 \sum_{r\ll
    \frac{yQ}{bS}}\frac{M^2+r^{1/2}M}{r^2}\sumsum_{\frac{Q}{rbS}\ll
    t_1\sim t_2\ll \frac{yQ}{rbS}}\frac{1}{t_1^2}
  \Bigl(\frac{H_{\uple{b}}b^{1/2}}{(Qt_2)^{1/2}}+
  \frac{(Qt_2)^{1/2}}{b^{1/2}}\Bigr)\\
  &\lessapprox \Bigl(\frac{Q}{b}\Bigr)^2 \sum_{r\ll
    \frac{yQ}{bS}}\frac{M^2+r^{1/2}M}{r^2} \sum_{\frac{Q}{rbS}\ll
    t_2\ll \frac{yQ}{rbS}} \Bigl(
  \frac{H_{\uple{b}}b^{1/2}}{Q^{1/2}t_2^{3/2}} +
  \frac{Q^{1/2}}{b^{1/2}t_2^{1/2}}
  \Bigr)\\
  &\lessapprox \Bigl(\frac{MQ}{b}\Bigr)^2 \Bigl(
  \frac{H_{\uple{b}}b^{1/2}}{Q^{1/2}} \Bigl(\frac{Q}{bS}\Bigr)^{-1/2}+
  \frac{Q^{1/2}}{b^{1/2}}\Bigl(\frac{yQ}{bS}\Bigr)^{1/2} \Bigr),
\end{align*}
as claimed. (Note that it was important to keep track of the condition
$r_1\sim r_2$.)
\end{proof}

\begin{proof}[Proof of Lemma~\ref{lm-euler}]  If we write $t_i := b_i^\flat$, $b_i = t_i u_i$, then we have $t_i | b$ and $u_i | t_i^\infty$ and
$$ 
\frac{b_1^\flat b_2^\flat b_3^\flat}{(b_1b_2b_3)^{1/2} b^3} = \frac{1}{b^3} \prod_{i=1}^3 \frac{t_i^{1/2}}{u_i^{1/2}} $$
and thus we can bound the required series by
$$ \sum_{b \geq 1} \frac{1}{b^3} \Bigl(\sum_{t|b} t^{1/2} \sum_{u|t^\infty} \frac{1}{u^{1/2}}\Bigr)^3.$$
Using Euler products, we have
$$ \sum_{u|t^\infty} \frac{1}{u^{1/2}} \leq \tau(t)^{O(1)}$$
and thus
$$ \sum_{t|b} t^{1/2} \sum_{u|t^\infty} \frac{1}{u^{1/2}} \leq \tau(b)^{O(1)} b^{1/2}$$
and the claim now follows from another Euler product computation.
\end{proof}

\section{An improved Type I estimate}\label{typei-advanced-sec}

In this final section, we prove the remaining Type I estimate from
Section~\ref{typei-ii-sec}, namely Theorem
\ref{newtype-i-ii}\eqref{typeI4-again}. In Section \ref{prelim-red},
we reduced this estimate to the exponential sum estimate of Theorem
\ref{eset}\eqref{typeI4-yetagain}.  

\subsection{First reduction}

The reader is invited to review the definition and notation of
Theorem~\ref{eset}. We consider the sum
$$
\Upsilon:= \sum_{r}\Upsilon_{\ell,r}(b_1,b_2;q_0)
$$
of~(\ref{keyo}) for each $1\leq |\ell|\ll N/R$, where $\Upsilon_{\ell,r}$ was defined in \eqref{eq-phiell} and the sum over $r$
is restricted to $r\in \DI{I}{2}{x^{\delta+o(1)}} \cap [R,2R]$ (the
property that $r$ is doubly densely divisible being part of the
assumptions of \ref{eset}\eqref{typeI4-yetagain}).  Our task is to show the bound
$$ \Upsilon  \lessapprox x^{-\eps} Q^2R N  (q_0,\ell)q_0^{-2}$$
under the hypotheses of Theorem \ref{eset}\eqref{typeI4-yetagain}.
\par
In contrast to the Type I and II estimates of
Section~\ref{typei-ii-sec} (but similarly to the Type III estimate),
we will exploit here the average over $r$, and hence the treatment
will combine some features of all the methods used before.
\par
As before, we denote
\begin{equation}\label{eq-h-bis}
H:=x^{\eps}RQ^2M^{-1}q_0^{-1}.
\end{equation}
We recall that, from~(\ref{eq-h-condition}), we have $H\gg 1$.  We
begin as in Section \ref{second-typei} by exploiting the $x^\delta$-dense
divisibility of $q_0q_1$, which implies the $x^\delta q_0$-dense divisibility of $q_1$ by Lemma \ref{fq}(i). Thus we reduce by dyadic decomposition to
the proof of
\begin{equation}\label{eq-dyadic-2}
\sum_{r}\Upsilon_{U,V}\lessapprox  x^{-\eps}(q_0,\ell)RQ^2Nq_0^{-2}
\end{equation}
(which corresponds to~(\ref{eq-upsilon-dyadic}) with the average over
$r$ preserved) where
$$
\Upsilon_{U,V}:=\sum_{1 \leq |h| \leq H} \sum_{u_1 \sim U} \sum_{v_1
  \sim V} \sum_{\substack{q_2 \sim Q/q_0\\ (u_1v_1,q_0 q_2)= 1}}\Bigl|
\sum_n C(n)\beta(n) \overline{\beta(n+\ell r)}
\Phi_{\ell}(h,n,r,q_0,u_1v_1,q_2)\Bigr|
$$
as in Section~\ref{second-typei}, whenever
\begin{gather}
  q_0^{-1} x^{-\delta-2\eps} Q/H \lessapprox U \lessapprox x^{-2\eps} Q/H \label{s-bound-2}\\
  q_0^{-1} x^{2\eps} H \lessapprox V \lessapprox x^{\delta+2\eps} H \label{t-bound-2}\\
  UV \sim Q/q_0 \label{st-2}
\end{gather}
(which are identical constraints to~(\ref{s-bound}),~(\ref{t-bound})
and~(\ref{st})), and whenever the parameters $(\varpi,\delta,\sigma)$
satisfy the conditions of Theorem~\ref{eset}\eqref{typeI4-yetagain}.  As before, $u_1,v_1$ are understood to be squarefree.
\par

We replace again the modulus by complex numbers $c_{r,h,u_1,v_1,q_2}$
of modulus $\leq 1$, which we may assume to be supported on parameters
$(r,h,u_1,v_1,q_2)$ with
$$
(u_1v_1,q_2)=1
$$
and with
$$
q_0u_1v_1r,\quad q_0q_2r\text{ squarefree}.
$$
(These numbers $c_{r,h,u_1,v_1,q_2}$ are unrelated to the exponent $c$ in Theorem \ref{newtype-i-ii}.)
We then move the sums over $r$, $n$, $u_1$ and $q_2$ outside
and apply the Cauchy-Schwarz inequality as in the previous sections to
obtain
$$
\Bigl|\sum_r\Upsilon_{U,V}\Bigr|^2\leq \Upsilon_1\Upsilon_2
$$
with
$$
\Upsilon_1:=\sum_r\sumsum_{\substack{u_1\sim U\\ q_2\sim Q/q_0}} \sum_n
C(n) |\beta(n)|^2|\beta(n+\ell r)|^2\lessapprox (q_0,\ell)\frac{NQRU}{q_0^2}
$$
(again as in~(\ref{baz})) and
\begin{align*}
  \Upsilon_2&:= \sum_r \sumsum_{\substack{u_1\sim U\\ q_2\sim Q/q_0}}
  \sum_{n}\psi_N(n)C(n) \Bigl| \sum_{v_1\sim V} \sum_{1 \leq |h| \leq
    H} c_{h,r,u_1,v_1,q_2}
  \Phi_{\ell}(h,n, r, q_0, u_1v_1,q_2)\Bigr|^2\\
  &=\sum_r\sumsum_{\substack{u_1\sim U\\q_2\sim Q/q_0}} \sumsum_{v_1,v_2\sim
    V} \sumsum_{1\leq |h_1|,|h_2|\leq H}
  c_{h_1,r,u_1,v_1,q_2}\overline{c_{h_2,r,u_1,v_2,q_2}}
  T_{\ell,r}(h_1,h_2,u_1,v_1,v_2,q_2),
\end{align*}
where $T_{\ell,r}$ is defined by~(\ref{eq-tell}) and $\psi_N$ is a smooth coefficient sequence at scale $N$.
\par

The analysis of $\Upsilon_2$ will now diverge from
Section~\ref{second-typei}. In our setting, the modulus $r$ is doubly
$x^{\delta+o(1)}$-densely divisible. As in the previous section, we
will exploit this divisility to split the average and apply the
Cauchy-Schwarz inequality a second time.
\par
Let $D$ be a parameter such that
\begin{equation}\label{dr}
  1 \lessapprox D \lessapprox x^{\delta}R,
\end{equation}
which will be chosen and optimized later. By definition (see
Definition~\ref{mdd-def}) of doubly densely divisible integers, for
each $r$, there exists a factorization $r=dr_1$ where
$$
x^{-\delta} D \lessapprox d \lessapprox D,
$$
and where $r_1$ is $x^{\delta+o(1)}$-densely divisible (and
$(d,r_1)=1$, since $r$ is squarefree).  As before, in the case $D \geq R$ one can simply take $d=r$ and $r_1=1$.
\par
We consider the sums
$$
\Upsilon_3:= \sum_{\substack{d\sim \Delta\\(d,r_1)=1}}
 \sumsum_{1\leq |h_1|,|h_2|\leq H}
\sumsum_{\substack{v_1,v_2\sim V\\(v_1v_2,dr_1q_0u_1q_2)=1}}
|T_{\ell,dr_1}(h_1,h_2,u_1,v_1,v_2,q_2)|,
$$
with $d$ understood to be squarefree,
for all  $\Delta$ such that
\begin{equation}\label{dbounds}
\max(1,x^{-\delta} D) \lessapprox \Delta \lessapprox D
\end{equation}
and all $(r_1,u_1,q_2)$ such that
\begin{equation}\label{rp}
r_1 \sim R/\Delta,\quad  u_1 \sim U,\quad q_2 \sim Q/q_0,
\end{equation}
and such that $r_1 q_0 u_1q_2$ is squarefree and the integers $r_1$,
$q_0u_1v_1$, $q_0u_1v_2$ and $q_0q_2$ are $x^{\delta+o(1)}$-densely
divisible.
\par
For a suitable choice of $D$, we will establish the bound
\begin{equation}\label{eq-upsilon-3}
  \Upsilon_3\lessapprox (q_0,\ell)x^{-2\eps}\Delta NV^2q_0
\end{equation}
for all such sums. It then follows by dyadic subdivision of the 
variable $d$ and by trivial summation over $r_1$, $u_1$ and $q_2$ that
$$
\Upsilon_2\lessapprox (q_0,\ell)x^{-2\eps}NV^2q_0
\frac{RUQ}{q_0}=(q_0,\ell)x^{-2\eps}NRUV^2Q,
$$
and hence that
$$
\Bigl|\sum_r\Upsilon_{U,V}\Bigr|^2\lessapprox
(q_0,\ell)^2x^{-2\eps} N^2R^2\Bigl(\frac{Q}{q_0}\Bigr)^4,
$$
which gives the desired result.
\par

We first write $\Upsilon_3=\Upsilon'_3+\Upsilon''_3$, where
$\Upsilon'_3$ is the diagonal contribution determined by $h_1
v_2=h_2v_1$.  The number of quadruples $(h_1,v_1,h_2,v_2)$ satisfying
this condition is  $\lessapprox HV$ by the divisor bound, and
therefore a trivial bound $\lessapprox N$ for
$T_{\ell,r}(h_1,h_2,u_1,v_1,v_2,q_2)$ gives
$$
\Upsilon'_3\lessapprox \Delta HNV\lessapprox
(q_0,\ell)x^{-2\eps}\Delta NV^2q_0
$$
by \eqref{t-bound-2}. 
We now write
$$
\Upsilon''_3=\sumsum_{\substack{(h_1,v_1,h_2,v_2)\\
    h_1v_2\neq h_2v_1}} \Upsilon_4(h_1,v_1,h_2,v_2)
$$
where $h_1,v_1,h_2,v_2$ obey the same constraints as in the definition of $\Upsilon_3$, and
$$
\Upsilon_4(h_1,v_1,h_2,v_2):= \sum_{\substack{d\sim \Delta\\(d,r_1)=1}}
|T_{\ell,dr_1}(h_1,h_2,u_1,v_1,v_2,q_2)|.
$$

We will shortly establish the following key estimate.

\begin{proposition}\label{pr-upsilon-4}
  If $\eps>0$ is small enough,
  then we have
$$
\Upsilon_4(h_1,v_1,h_2,v_2)\lessapprox (q_0,\ell)x^{-2\eps}\Delta
NH^{-2}q_0 \ \left(h_1v_2-h_2v_1,q_0q_2r_1u_1[v_1,v_2]\right),
$$
if we take
\begin{equation}\label{eq-def-d}
  D:=x^{-5\eps}\frac{N}{H^4}
\end{equation}
and if
\begin{equation}\label{eq-condition}
\begin{cases}
\tfrac{160}{3}\varpi+16\delta+\tfrac{34}{9}\sigma<1\\
64\varpi+18\delta+2\sigma<1.
\end{cases}
\end{equation}
\end{proposition}

Assuming this proposition, we obtain 
$$
\Upsilon''_3\lessapprox (q_0,\ell)x^{-2\eps}\Delta NV^2q_0,
$$
and hence~(\ref{eq-upsilon-3}), by the following lemma, which will be
proved later.


\begin{lemma}\label{lm-gcd} We have
$$
\sumsum_{\substack{(h_1,v_1,h_2,v_2)\\
    h_1v_2\neq h_2v_1}} 
\left(h_1v_2-h_2v_1,q_0q_2r_1u_1[v_1,v_2]\right)
\lessapprox H^2V^2.
$$
\end{lemma}


\subsection{Reduction of Proposition~\ref{pr-upsilon-4} to exponential
  sums}

We now consider a specific choice of parameters $r_1$, $u_1$, $q_2$ and
$(h_1,v_1,h_2,v_2)$, so that $\Upsilon_4=\Upsilon_4(h_1,v_1,h_2,v_2)$
is a sum with two variables which we write as
$$
\Upsilon_4= \sum_{d \sim \Delta} \Bigl|\sum_n \psi_N(n)
C(n)\Psi(d,n)\Bigr|
$$
where $C(n)$ restricts $n$ to the congruence~(\ref{eq-n-congruence}) and
\begin{equation}\label{psidn}
  \Psi(d,n):=  \Phi_{\ell}(h_1, n,dr_1,q_0,u_1v_1,q_2) 
  \overline{ \Phi_{\ell}(h_2,n, dr_1, q_0,u_1v_2,q_2) }.
\end{equation}
We define $D$ by~(\ref{eq-def-d}), and we first check that this
satisfies the constraints~(\ref{dr}). Indeed, we first have
$$
D=x^{-5\eps}\frac{N}{H^4}=\frac{x^{-9\eps}q_0^4NM^4}{Q^8R^4} \ggcurly
x^{-9\eps-16\varpi}\frac{R^4}{N^3}\ggcurly
x^{1/2-\sigma-16\varpi-4\delta-21\eps}
$$
by~(\ref{slop}) and~(\ref{crop}). Under the
condition~(\ref{eq-condition}), this gives $D\ggcurly 1$ if $\eps>0$
is taken small enough.
\par
Moreover, since $H\gg 1$, we have
$$
D=x^{-5\eps}\frac{N}{H^4}\lessapprox x^{-5\eps}N\lessapprox
x^{-2\eps+\delta}R\leq x^{\delta}R.
$$

We apply the van der Corput technique with respect to the modulus $d$.
Let
\begin{equation}\label{ldf}
L := x^{-\eps} \left\lfloor \frac{N}{\Delta} \right\rfloor.
\end{equation}
Note that from \eqref{dr} and \eqref{crop}, it follows that $L
\ggcurly x^{-\eps}NR^{-1} \geq 1$ for $x$ sufficiently large.
\par
For any $l$ with $1\leq l\leq L$, we have
$$
\sum_n \psi_N(n)C(n) \Psi(d,n)= \sum_n \psi_N(n+dl)C(n+dl)
\Psi(d,n+dl)
$$
and therefore
$$
|\Upsilon_4|\leq \frac{1}{L}\sum_{d \sim \Delta} \sum_{n\ll N} \Bigl|
\sum_{l=1}^L \psi_N(n+dl) C(n+dl)\Psi(d,n+dl)\Bigr|.
$$
\par
By the Cauchy-Schwarz inequality, for some smooth coefficient sequence
$\psi_{\Delta}$ at scale $\Delta$, we have
\begin{equation}\label{eq-upsilon-4}
|\Upsilon_4|^2\leq \frac{N\Delta}{L^2} | \Upsilon_5|
\end{equation}
where
$$
\Upsilon_5:=\sum_{d\sim \Delta} \psi_{\Delta}(d) \sum_n \Bigl| \sum_{l=1}^L
\psi_N(n+dl) C(n+dl)\Psi(d,n+dl)\Bigr|^2.
$$
\par

\begin{lemma}\label{lm-clean-up}
Let 
$$
m=q_0r_1u_1[v_1,v_2]q_2.
$$
There exist residue classes $\alpha\ (m)$ and $\beta\
(m)$, independent of $n$ and $l$, such that for all $n$ and $l$ we
have
$$
\Psi(d,n+dl)=\xi(n,d)e_m\Bigl(\frac{\alpha}{d(n+(\beta+l)d)}\Bigr)
$$
where $|\xi(n,d)|\leq 1$. Moreover we have
$(\alpha,m)=(h_1v_2-h_2v_1,m)$.
\end{lemma}

\begin{proof}
  From the definitions~(\ref{psidn}) and~(\ref{eq-phiell}), if
  $\Psi(d,n)$ does not vanish identically, then we have
\begin{multline*}
  \Psi(d,n+dl)=e_{dr_1}\left( \frac{a(h_1-h_2)}{(n+dl)q_0 u_1v_1q_2}
  \right) e_{q_0 u_1v_1}\left( \frac{b_1h_1}{(n+dl) dr_1 q_2} \right)
  e_{q_0 u_1v_2}\left( -\frac{b_1h_2}{(n+dl) dr_1 q_2} \right)\\
  e_{q_2}\left( \frac{b_2 h_1}{(n+dl+d \ell r_1) dr_1 q_0 u_1v_1} \right)
  e_{q_2}\left(- \frac{b_2 h_2}{(n+dl+d\ell r_1) dr_1 q_0 u_1v_2} \right).
\end{multline*}
By the Chinese Remainder Theorem, the first factor splits into an phase $e_d(\dots)$ that is independent of $l$, and an expression involving $e_{r_1}$, which when combined with the other four factors by another application of the Chinese Remainder Theorem, becomes an expression of
the type
$$
e_m\Bigl(\frac{\alpha }{d(n+ld+\beta d)}\Bigr)
$$
for some residue classes $\alpha$ and $\beta$ modulo $m$ which are
independent of $l$. Furthermore $(\alpha,m)$ is the product of primes
$p$ dividing $m$ such that the product of these four factors is
trivial, which (since $(q_2,q_0u_1[v_1,v_2])=1$) occurs exactly when
$p\mid h_2v_1-h_1v_2$ (recall that $b_1$ and $b_2$ are invertible
residue classes).
\end{proof}

Using this lemma, and the notation introduced there, it follows that
\begin{multline*}
  \Bigl|\sum_{l=1}^L \psi_N(n+dl) C(n+dl) \Psi(d,n+dl)\Bigr|^2 \leq \sum_{1
    \leq l_1,l_2 \leq L} \psi_N(n+dl_1) \psi_N(n+dl_2) C(n+dl_1) C(n+dl_2)\\
\quad\quad\quad\quad\quad\quad\quad\quad\quad\quad\quad\quad
\quad\quad\quad\quad
  e_m\Bigl(\frac{\alpha}{d(n+\beta d+l_1d)}\Bigr)
  e_m\Bigl(-\frac{\alpha}{d(n+\beta d+l_2d)}\Bigr)
  \\
  =\sum_{1 \leq l_1,l_2 \leq L} \psi_N(n+dl_1) \psi_N(n+dl_2)
  e_m\Bigl(\frac{\alpha (l_2-l_1)}{(n+\beta d+l_1d)(n+\beta d+l_2d)}
  \Bigr),
\end{multline*}
and therefore, after shifting $n$ by $dl_1$, writing $l := l_2-l_1$, and splitting $n,d$ into residue classes modulo $q_0$, that
$$ \Upsilon_5 \leq \sum_{n_0, d_0 \in \Z/q_0\Z} C(n_0) \Upsilon_5(n_0,d_0) $$
where
\begin{multline}\label{eq-bound-up5}
  \Upsilon_5(n_0,d_0) := \sumsum_{\substack{|l|\leq L-1\\1\leq l_1
      \leq L}}  \Bigl|\sum_{d=d_0\ (q_0)} \psi_{\Delta}(d)\\
  \times \sum_{n=n_0\ (q_0)} \psi_N(n)
  \psi_N(n+dl)e_m\left(\frac{\alpha l}{(n+\beta d)(n+(\beta+l)d)}
  \right)\Bigr|.
\end{multline}

Note that $m$ is squarefree. Also, as $m$ is the least common multiple of the $x^{\delta+o(1)}$-densely divisible
quantities $r_1$, $q_0 u_1v_1$, $q_0 u_1v_2$, and $q_0 q_2$, Lemma \ref{fq}(ii) implies that $m$ is also
$x^{\delta+o(1)}$-densely divisible.

The contribution of $l=0$ to $\Upsilon_5(n_0,d_0)$ is trivially
\begin{equation}\label{eq-l0}
  \ll \frac{NL\Delta}{q_0^2},
\end{equation}
and this gives a contribution of size
$$
\lessapprox \sqrt{(q_0,\ell)}\frac{N\Delta}{\sqrt{q_0L}}
$$
to $\Upsilon_4$, as can be seen by summing over the $q_0(q_0,\ell)$ permitted residue
classes $(n_0\ (q_0),d_0\ (q_0))$. Using~(\ref{eq-def-d}) we
have
$$
\Delta\lessapprox D=x^{-5\eps}\frac{N}{H^4},
$$
and we see from~(\ref{ldf}) that this contribution is certainly
$$
\lessapprox  (q_0,\ell)x^{-2\eps}\Delta
NH^{-2}q_0
$$
and hence suitable for Proposition~\ref{pr-upsilon-4}.  

Let $\Upsilon_5'(n_0,d_0)$ (resp.\ $\Upsilon'_5$) denote the remaining
contribution to $\Upsilon_5(n_0,d_0)$ (resp. $\Upsilon_5$). 
It will now suffice to show that
\begin{equation}\label{up5-p}
\frac{N\Delta}{L^2} | \Upsilon'_5|\lessapprox \left( (q_0,\ell)x^{-2\eps}\Delta
NH^{-2}q_0 \ \left(h_1v_2-h_2v_1,q_0q_2r_1u_1[v_1,v_2]\right) \right)^2.
\end{equation}

We have
\begin{equation}\label{eq-upsilonp-5}
  \Upsilon'_5(n_0,d_0)=  \sumsum_{\substack{1\leq |l|\leq L-1\\1\leq l_1
      \leq L}} |\Upsilon_6(n_0,d_0)|
\end{equation}
where
\begin{multline}\label{eq-upsilon-6}
  \Upsilon_6(n_0,d_0):=\sum_{d=d_0\ (q_0)} \psi_{\Delta}(d)
  \sum_{n=n_0\ (q_0)} \psi_N(n)\\
  \psi_N(n+dl)e_m\left(\frac{\alpha l}{(n+\beta d)(n+(\beta+l)d)}
  \right).
\end{multline}

For given $l\neq 0$ and $l_1$, the sum $\Upsilon_6(n_0,d_0)$ over $n$
and $d$ in~(\ref{eq-bound-up5}) is essentially an incomplete sum in
two variables of the type treated in
Corollary~\ref{inctrace-q}. However, before we can apply this result,
we must separate the variables $n$ and $d$ in $\psi_{N}(n+dl)$. As in
the previous section, we can do this here using a Taylor expansion.

Let $J\geq 1$ be an integer. Performing a Taylor expansion to order
$J$ we have
$$
\psi_N(n+dl) = \psi\Bigl(\frac{n+dl}{N}\Bigr)= \sum_{j=0}^J
\Bigl(\frac{d}{\Delta}\Bigr)^j\frac{1}{j!} \Bigl(\frac{\Delta l}{N}\Bigr)^j
\psi^{(j)}\Bigl(\frac{n}{N}\Bigr) + O( x^{-\eps J} )
$$
since $dl\ll \Delta L\ll x^{-\eps}N$ by~(\ref{ldf}).  We can absorb
the factor $(\frac{d}{\Delta})^j$ into $\psi_{\Delta}$, and after taking
$J$ large enough depending on $\eps$, we see that we can express $\Upsilon_6(n_0,d_0)$ as
a sum of finitely many sums
$$
\Upsilon'_6(n_0,d_0)= \sum_{d=d_0 (q_0)} \psi_{\Delta}(d) \sum_{n=n_1
  \ (q_0)} \psi'_N(n) e_m\Bigl(\frac{\alpha l}{(n+\beta d)(n+(\beta+l)d)}\Bigr)
$$
for some residue classes $n_1\ (q_0)$,
where $\psi_{\Delta}$ and $\psi'_N$ are coefficient sequences smooth at
scales $\Delta$ and $N$ respectively, possibly different from the previous
ones.

We will prove in Section~\ref{sec-final-exp} the following exponential
sum estimate, using the machinery from Section \ref{deligne-sec}:

\begin{proposition}\label{lode} 
  Let $m$ be a $y$-densely divisible squarefree integer of polynomial
  size for some $y \geq 1$, let $\Delta, N > 0$ be of polynomial size,
  and let $\alpha$, $\beta$, $\gamma_1$, $\gamma_2$, $l\in \Z/m\Z$.
  Let $\psi_{\Delta}, \psi'_N$ be shifted smooth sequences at scale
  $\Delta$, and $N$ respectively.   Then for any divisor $q_0$ of $m$
  and for all residue classes $d_0\ (q_0)$ and $n_0\ (q_0)$, we have
\begin{multline}\label{eq-lode-1}
  \Bigl|\sum_{d=d_0\ (q_0)} \sum_{n=n_0\ (q_0)} \psi_{\Delta}(d) \psi'_N(n)
  e_m\Bigl(
  \frac{\alpha l}{(n+\beta d+\gamma_1)(n+(\beta+l)d+\gamma_2)} \Bigr)\Bigr| \\
  \lessapprox (\alpha l,m) \left(\frac{N}{q_0m^{1/2}}+m^{1/2} \right)
  \left( 1 + \Bigl(\frac{\Delta}{q_0}\Bigr)^{1/2} m^{1/6} y^{1/6} +
    \Bigl(\frac{\Delta}{q_0}\Bigr) m^{-1/2}\right).
\end{multline}
We also have the bound
\begin{multline}\label{eq-lode-2}
  \Bigl|\sum_{d=d_0\ (q_0)} \sum_{n=n_0\ (q_0)} \psi_{\Delta}(d)
  \psi'_N(n) e_m\Bigl(
  \frac{\alpha l}{(n+\beta d+\gamma_1)(n+(\beta+l)d+\gamma_2)} \Bigr)\Bigr| \\
  \lessapprox (\alpha l,m) \left( \frac{N}{q_0m^{1/2}}+ m^{1/2} \right)
  \left( m^{1/2} + \Bigl(\frac{\Delta}{q_0}\Bigr) m^{-1/2}\right).
\end{multline}
\end{proposition}

\begin{remark} Suppose $q_0=1$ for simplicity.
  In practice, the dominant term on the right-hand side of \eqref{eq-lode-2} will be $(\alpha l,m) m^{1/2} \
  \Delta^{1/2} m^{1/6} y^{1/6}$, which in certain regimes improves
  upon the bound of $((\alpha l,m)^{-1/2} m^{1/2}) \ \Delta$ that is
  obtained by completing the sums in the variable $n$ only without
  exploiting any additional cancellation in the variable $d$.
\par
Note that if the phase 
$$
\frac{\alpha l}{(n+\beta
  d+\gamma_1)(n+(\beta+l)d+\gamma_2)}
$$ 
was of the form $f(d)+g(n)$ for some non-constant rational functions
$f$ and $g$, then the two-dimensional sum would factor into the
product of two one-dimensional sums, and then the estimates we claim
would basically follow from the one-dimensional bounds in
Proposition \ref{inc}.  However, no such splitting is available, and
so we are forced to use the genuinely multidimensional theory arising
from Deligne's proof of the Riemann Hypothesis over finite fields.
\end{remark}

Applying Proposition \ref{lode}, we have
$$
\Upsilon'_6(n_0,d_0)\lessapprox (\alpha l,m) \left(m^{1/2} +
  \frac{N/q_0}{m^{1/2}}\right) \left( 1 + (\Delta/q_0)^{1/2} m^{1/6}
  x^{\delta/6} + \frac{\Delta/q_0}{m^{1/2}}\right)
$$
as well as
$$
\Upsilon'_6(n_0,d_0)\lessapprox (\alpha l,m) \left(m^{1/2} +
  \frac{N/q_0}{m^{1/2}}\right) \left(m^{1/2} +
  \frac{\Delta/q_0}{m^{1/2}}\right).
$$ 
Distinguishing the cases $N/q_0 \leq m$ and $N/q_0 > m$, and summing
over the finitely many cases of $\Upsilon'_6(n_0,d_0)$ that give
$\Upsilon_6(n_0,d_0)$, we see that
$$
\Upsilon_6(n_0,d_0)\lessapprox (\alpha l,m) \left\{ m^{1/2} \left( 1 +
    \Bigl(\frac{\Delta}{q_0}\Bigr)^{1/2} m^{1/6} x^{\delta/6} +
    \frac{\Delta/q_0}{m^{1/2}}\right) + \frac{N/q_0}{m^{1/2}}
  \left(m^{1/2} + \frac{\Delta/q_0}{m^{1/2}}\right)\right\}.
$$
Note that $(\alpha l,m)\leq (\alpha,m)(l,m)$ and hence, summing over $l$ and
$l_1$ in~(\ref{eq-upsilonp-5}) (using Lemma~\ref{ram-avg}), we get
$$
\Upsilon'_5(n_0,d_0)\lessapprox (\alpha,m)L^2 \left\{ m^{1/2} +
  \Bigl(\frac{\Delta}{q_0}\Bigr)^{1/2} m^{2/3} x^{\delta/6} +
  \frac{\Delta}{q_0} + \frac{N}{q_0} + \frac{N\Delta}{q_0^2m}\right\}.
$$
Next, summing over the $\leq (q_0,\ell)q_0$ residue classes $(n_0,d_0)$
allowed by the congruence restriction~(\ref{eq-n-congruence}), we get
$$
\Upsilon'_5\lessapprox (q_0,\ell)
(\alpha,m)L^2 \left\{ q_0m^{1/2} + (q_0\Delta)^{1/2} m^{2/3} x^{\delta/6} +
  \Delta + N + \frac{N\Delta}{q_0m}\right\},
$$
and finally by  inserting some additional factors of $q_0$ and $(q_0,\ell)$, we derive
\begin{align*}
\frac{N\Delta}{L^2} | \Upsilon'_5| &\lessapprox (q_0,\ell)(\alpha,m)N\Delta \left\{ q_0m^{1/2} +
    (q_0\Delta)^{1/2} m^{2/3} x^{\delta/6} + \Delta + N +
    \frac{N\Delta}{q_0m}\right\}\\
  &\lessapprox (q_0,\ell)^2(\alpha,m)^2q_0N\Delta \left\{ \Delta^{1/2}
    m^{2/3} x^{\delta/6} + \Delta + N + \frac{N\Delta}{m}\right\}.
\end{align*}
In fact, since $\Delta \lessapprox D \lessapprox N$, we see
that
$$
\frac{N\Delta}{L^2} | \Upsilon'_5| \lessapprox (q_0,\ell)^2(\alpha,m)^2q_0N\Delta \left\{
  \Delta^{1/2} m^{2/3} x^{\delta/6} + N + \frac{N\Delta}{m}\right\}.
$$
\par
We have $m=q_0r_1u_1[v_1,v_2]q_2$ (see Lemma~\ref{lm-clean-up}) and
therefore (using~(\ref{st-2}) and~(\ref{t-bound-2})) we can bound $m$
from above and below by
$$
m\ll q_0\times\frac{R}{\Delta}\times U\times V^2\times
\frac{Q}{q_0}\sim \frac{Q^2RV}{\Delta}\lessapprox
x^{\delta+2\eps}\frac{Q^2RH}{\Delta}
$$
and
$$
m\ggcurly q_0\times \frac{R}{\Delta} \times U\times V
\times\frac{Q}{q_0}\sim \frac{Q^2R}{q_0\Delta},
$$
which leads to
\begin{align*}
\frac{N\Delta}{L^2} | \Upsilon'_5| &\lessapprox (q_0,\ell)^2(\alpha,m)^2q_0^2N\Delta \left\{
    x^{5\delta/6+4\eps/3} \frac{(Q^2RH)^{2/3}}{\Delta^{1/6}} + N
    + \frac{N\Delta^2}{Q^2R}\right\}\\
&  = (q_0,\ell)^2(\alpha,m)^2q_0^2\frac{(N\Delta)^2}{H^4} \left\{
    x^{5\delta/6+2\eps} \frac{H^4(Q^2RH)^{2/3}}{N\Delta^{7/6}} + \frac{H^4}{\Delta} +
    \frac{H^4\Delta}{Q^2R}\right\}
\end{align*}
up to admissible errors. Since
$$
\Delta^{-1}\lessapprox \frac{x^{\delta}}{D} =
x^{\delta+5\eps}\frac{H^4}{N}, \quad\quad \Delta\lessapprox D
= x^{-5\eps}\frac{N}{H^4},
$$
this leads to
$$
\frac{N\Delta}{L^2} | \Upsilon'_5| \lessapprox
(q_0,\ell)^2(\alpha,m)^2q_0^2\frac{(N\Delta)^2}{H^4} \left\{
  x^{2\delta+8\eps} \frac{H^{28/3}Q^{4/3}R^{2/3}}{N^{13/6}} +
  \frac{x^{\delta+5\eps}H^8}{N} + \frac{x^{-5\eps}N}{Q^2R}\right\}
$$
up to admissible errors.  From the assumptions~(\ref{slop})
and~(\ref{crop2}), we have
$$
N\lessapprox x^{1/2}\lessapprox QR,
$$
and thus 
$$
\frac{x^{-5\eps}N}{Q^2R}\lessapprox x^{-5\eps}Q^{-1}\lessapprox
x^{-5\eps}.
$$
On the other hand, from the value of $H$ (see~(\ref{eq-h-bis})) we get
\begin{align*}
 x^{2\delta+8\eps} \frac{H^{28/3}Q^{4/3}R^{2/3}}{N^{13/6}}&
\lessapprox  x^{2\delta+18\eps}\frac{R^{10}Q^{20}}{M^{28/3}N^{13/6}}
\lessapprox  x^{-28/3+2\delta+18\eps}R^{10}Q^{20}N^{43/6}
\\
 \frac{x^{\delta+5\eps}H^8}{N}&\lessapprox
 x^{\delta+13\eps}\frac{R^8Q^{16}}{NM^8}
\lessapprox x^{-8+\delta+13\eps}N^7Q^{16}R^8.
\end{align*}
Using the other conditions $x^{1/2} \lessapprox QR \lessapprox
x^{1/2+2\varpi}$, and
$$
R\ggcurly x^{-3\eps-\delta}N,\quad\quad N\ggcurly x^{1/2-\sigma}
$$
these quantities are in turn bounded respectively by
\begin{align*}
  x^{2\delta+8\eps} \frac{H^{28/3}Q^{4/3}R^{2/3}}{N^{13/6}}& \leq
  x^{2/3+2\delta+40\varpi+18\eps}\frac{N^{43/6}}{R^{10}} \ll
  x^{2/3+12\delta+40\varpi-17/6(1/2-\sigma)+48\eps}
  \\
  \frac{x^{\delta+5\eps}H^8}{N}&\leq
  x^{\delta+32\varpi+13\eps}\frac{N^7}{R^8} \ll
  x^{9\delta+32\varpi+37\eps-(1/2-\sigma)}.
\end{align*}
Thus, by taking $\eps>0$ small enough, we obtain \eqref{up5-p} (and hence
Proposition~\ref{pr-upsilon-4}) provided
$$
\begin{cases}
\tfrac{2}{3}+12\delta+40\varpi-\tfrac{17}{6}(\tfrac{1}{2}-\sigma)<0\\
9\delta+32\varpi-(\tfrac{1}{2}-\sigma)<0.
\end{cases}\quad
\Leftrightarrow \quad\quad
\begin{cases}
\tfrac{160}{3}\varpi+16\delta+\tfrac{34}{9}\sigma<1\\
64\varpi+18\delta+2\sigma<1.
\end{cases}
$$
These are exactly the conditions claimed in
Proposition~\ref{pr-upsilon-4}.

\subsection{Proof of Lemma~\ref{lm-gcd}}

This is a bit more complicated than the corresponding lemmas in
Sections \ref{typeii-sec}-\ref{second-typei} because the quantity
$m=q_0q_2r_1u_1[v_1,v_2]$ depends also on $v_1$ and $v_2$.

We let $w := q_0q_2r_1 u_1$, so that $m = w[v_1,v_2]$ and $w$ is
independent of $(h_1,h_2,v_1,v_2)$ and coprime with $[v_1,v_2]$.

Since $(w,[v_1,v_2])=1$, we have
$$
(h_1v_2-h_2v_1,w[v_1,v_2]) =\sum_{\substack{d\mid h_1v_2-h_2v_1\\d\mid
    w[v_1,v_2]}}\varphi(d)\leq \sum_{d\mid w}d \sum_{\substack{e\mid
    [v_1,v_2]\\ de\mid h_1v_2-h_2v_1}}e,
$$
and therefore
\begin{align*}
  \sumsum_{\substack{(h_1,v_1,h_2,v_2)\\
      h_1v_2\neq h_2v_1}}
  \left(h_1v_2-h_2v_1,q_0q_2r_1u_1[v_1,v_2]\right) &\leq
  \sumsum_{\substack{(h_1,v_1,h_2,v_2)\\
      h_1v_2\neq h_2v_1}} \sum_{d\mid w}d \sum_{\substack{e\mid
      [v_1,v_2]\\ de\mid h_1v_2-h_2v_1}}e\\
  &\leq \sum_{d\mid w}d \sum_{\substack{(d,e)=1\\e\ll V^2\\e\text{ squarefree}}}e
  \sum_{\substack{([v_1,v_2],w)=1\\de\mid h_1v_2-h_2v_1\\e\mid [v_1,v_2]\\
      h_1v_2\neq h_2v_1}}1.
\end{align*}
The variable $d$ is unrelated to the modulus $d$ appearing previously in this section.

Let $d$, $e$ be integers occuring in the outer sums, and
$(h_1,h_2,v_1,v_2)$ satisfying the other summation conditions. Then
$e$ is squarefree, and since $e\mid [v_1,v_2]$ and $e\mid
h_1v_2-h_2v_1$, any prime dividing $e$ must divide one of $(v_1,v_2)$,
$(h_1,v_1)$ or $(h_2,v_2)$ (if it does not divide both $v_1$ and
$v_2$, it is coprime to one of them, and $h_1v_2-h_2v_1=0\ (p)$ gives
one of the other divisibilities). Thus if we factor $e=e_1e_2e_3$
where
$$
e_1:=\prod_{\substack{p\mid e\\p\mid v_1\\p\nmid v_2}}p,\quad
e_2:=\prod_{\substack{p\mid e\\p\nmid v_1\\p\mid v_2}}p,\quad
e_3:=\prod_{\substack{p\mid e\\p\mid (v_1,v_2)}}p,
$$
then these are coprime and we have
$$
e_1\mid h_1,\quad e_2\mid h_2,\quad e_1e_3\mid v_1,\quad
e_2e_3\mid v_2.
$$
We write
$$
h_1=e_1\lambda_1,\quad h_2=e_2\lambda_2,\quad 
v_1=e_1e_3\nu_1,\quad v_2=e_2e_3\nu_2.
$$
Then we get
$$
h_1v_2-h_2v_1=e(\lambda_1\nu_2-\lambda_2\nu_1),
$$
and since $de\mid h_1v_2-h_2v_1$, it follows that $d\mid
\lambda_1\nu_2-\lambda_2\nu_1$. 
\par
Now fix some $e\ll V^2$. For each choice of factorization
$e=e_1e_2e_3$, the number of pairs $(\lambda_1\nu_2,\lambda_2\nu_1)$
that can be associated to this factorization as above for some
quadruple $(h_1,h_2,v_1,v_2)$ is $\ll (HV/e)^2/d$, since each product
$\lambda_1\nu_2$, $\lambda_2\nu_1$ is $\ll HV/e$, and $d$ divides the
difference. By the divisor bound, this gives $\lessapprox (HV)^2/de^2$
for the number of quadruples $(h_1,h_2,v_1,v_2)$. Summing over $d\mid
w$ and $e$, we get a total bound
$$
\lessapprox (HV)^2\tau(w)\sum_{e\ll V^2}e^{-1}\lessapprox H^2V^2,
$$
as desired.

\subsection{Proof of Proposition~\ref{lode}}\label{sec-final-exp}

It remains to establish Proposition \ref{lode}.  We begin with the
special case when $e=1$ and $(\alpha l,m)=1$. For simplicity, we
denote
$$
f(n,d)=\frac{\alpha l}{(n+\beta d+\gamma_1)(n+(\beta+l)d+\gamma_2)}.
$$
By completion of the sum over $n$ (see Lemma \ref{com}(i)), we have
\begin{align}
  \sum_d \sum_n \psi_{\Delta}(d) \psi_N(n) e_m(f(n,d)) &\lessapprox
  \Bigl(\frac{N}{m}+1\Bigr) \sup_{h \in\Z/m\Z} \Bigl|\sum_d
  \psi_{\Delta}(d)\sum_{n \in \Z/m\Z} e_m(f(n,d)+hn)\Bigr|\nonumber
  \\
& = \Bigl(\frac{N}{\sqrt{m}}+\sqrt{m}\Bigr) \sup_{h
    \in\Z/m\Z} \Bigl|\sum_d \psi_{\Delta}(d)K_h(d;m)\Bigr|,\label{eq-complete}
\end{align}
where, for each $h\in \Z/m\Z$, we define
$$
K_h(d;m) := \frac{1}{\sqrt{m}} \sum_{n \in \Z/m\Z} e_m(f(n,d)+hn).
$$

By the first part of Corollary \ref{inctrace-q} (i.e., \eqref{vdctrace-0-q}), we get
\begin{equation}\label{light-toast}
  \left|\sum_d \psi_{\Delta}(d) K_h(d;m)\right| 
\lessapprox m^{1/2} + \Delta m^{-1/2},
\end{equation}
and this combined with~(\ref{eq-complete}) implies the second
bound~(\ref{eq-lode-2}) (in the case $e=1$, $(\alpha l,m)=1$, that is).  Furthermore,
it also implies the first bound~(\ref{eq-lode-1}) for
$\Delta>m^{2/3}y^{-1/3}$.

In addition, from the Chinese Remainder Theorem (Lemma \ref{crt})
and \eqref{klo}, we deduce the pointwise bound
\begin{equation}\label{kqd}
|K_h(d,m)| \lessapprox 1
\end{equation}
which implies the trivial bound
$$
\Bigl|\sum_d \psi_{\Delta}(d) K_h(d;m)\Bigr| \lessapprox 1+\Delta,
$$
which gives~(\ref{eq-lode-1}) for $\Delta\leq m^{1/3}y^{1/3}$. Thus we
can assume that
$$
m^{1/3}y^{1/3}\leq \Delta\leq m^{2/3}y^{-1/3}\leq m.
$$
We can then use the $y$-dense divisibility of $m$ to factor $m=m_1m_2$
where
\begin{gather*}
y^{-2/3}m^{1/3}\leq m_1\leq y^{1/3}m^{1/3}\\
y^{-1/3}m^{2/3}\leq m_2\leq y^{2/3}m^{2/3}.
\end{gather*}
Now the second part of Corollary \ref{inctrace-q} (i.e., \eqref{vdctrace-1-q}) gives
$$ 
\Bigl|\sum_d \psi_{\Delta}(d) K_h(d;m)\Bigr| \lessapprox
\Delta^{1/2}m_1^{1/2} + \Delta^{1/2}m_2^{1/4}\lessapprox
\Delta^{1/2}m^{1/6}y^{1/6},
$$
which together with~(\ref{eq-complete}) gives~(\ref{eq-lode-1}). 
\par
This
finishes the proof of Proposition~\ref{lode} for the special case
$e=1$ and $(\alpha l,m)=1$. The extension to a divisor $e\mid m$ is done
exactly as in the proof of Corollary~\ref{dons} in
Section~\ref{exp-sec}.

We now reduce to the case $(\alpha l,m)=1$. Let
\begin{align*}
  m' &:= m / (\alpha l,m)\\
  y' &:= y (\alpha l,m)\\
  \alpha' &:= \alpha/(\alpha l,m) = \frac{\alpha/(\alpha,m)}{(\alpha
    l,m)/(\alpha,m)},
\end{align*}
where one computes the reciprocal of $(\alpha l,m)/(\alpha,m)$ inside
$\Z/m'\Z$, so that $\alpha'$ is viewed as an element of $\Z/m'\Z$. The integer $m'$ is $y'$-densely
divisible by Lemma~\ref{fq} (ii), and it is also squarefree, and of
polynomial size. We have $(a'l,m')=1$, and furthermore
\begin{multline*}
  \sum_d \sum_n \psi_{\Delta}(d) \psi_N(n) e_m(f(n,d)) = \sum_d \sum_n
  \psi_{\Delta}(d) \psi_N(n) e_{m'}(f'(n,d))\\
  \times \prod_{p|(\alpha l,m)} (1 - \onef_{p | (n+\beta
    d+\gamma_1)(n+(\beta+l)d+\gamma_2)})
\end{multline*}
where
$$
f'(n,d)=\frac{\alpha' l}{(n+\beta d+\gamma_1)(n+(\beta+l)d+\gamma_2)}
$$
(here we use the convention explained at the end of
Section~\ref{ssec-exp-prelim} that leads to $e_p(\alpha x)=1$ if $p$
is prime, $\alpha=0\ (p)$ and $x=+\infty\in\P^1(\Z/p\Z)$).
\par
Denote
$$
g(n,d)= (n+\beta d+\gamma_1)(n+(\beta+l)d+\gamma_2).
$$
Then, expanding the product (as in inclusion-exclusion), we get
$$
\sum_d \sum_n \psi_{\Delta}(d) \psi_N(n) e_m(f(n,d)) =
\sum_{\delta\mid (\alpha l,m)} \mu(\delta) \sumsum_{\substack{d,n\\
    \delta\mid g(n,d)}}\psi_{\Delta}(d)\psi_N(n) e_{m'}(f'(n,d))
$$
(this usage of $\delta$ is unrelated to prior usages of $\delta$ in this section).
Splitting the sum over $n$ and $d$ in residue classes modulo $\delta$,
this sum is then equal to
$$
\sum_{\delta\mid (\alpha l,m)}\mu(\delta)
\sumsum_{\substack{(d_0,n_0)\in(\Z/\delta\Z)^2\\g(n_0,d_0)=0\
    (\delta)}} \sum_{n=n_0\ (n)} \sum_{d=d_0\
  (\delta)}\psi_{\Delta}(d)\psi_N(n) e_{m'}(f'(n,d)).
$$
For each choice of $(n_0,d_0)$, we can apply the case previously
proved of Proposition~\ref{lode} to deduce
$$
\sum_{n=n_0\ (n)} \sum_{d=d_0\ (\delta)}\psi_{\Delta}(d)\psi_N(n)
e_{m'}(f'(n,d)) \lessapprox \Bigl(\sqrt{m'} +
\frac{N}{\delta\sqrt{m'}}\Bigr) \Bigl( 1 +
\frac{\Delta^{1/2}}{\delta^{1/2}} (m'y')^{1/6} +
\frac{\Delta}{\delta \sqrt{m'}}\Bigr)
$$ 
and
$$
\sum_{n=n_0\ (n)} \sum_{d=d_0\ (\delta)}\psi_{\Delta}(d)\psi_N(n)
e_{m'}(f'(n,d)) \lessapprox
\Bigl(\sqrt{m'} +
\frac{N}{\delta\sqrt{m'}}\Bigr) \Bigl( \sqrt{m'}+
\frac{\Delta}{\delta \sqrt{m'}}\Bigr).
$$
Moreover, by the Chinese Remainder Theorem, there are $\lessapprox
\delta$ solutions $(n_0,d_0)\in(\Z/\delta\Z)^2$ of $g(n_0,d_0)=0\
(\delta)$, and there we find
$$
\sum_d \sum_n \psi_{\Delta}(d) \psi_N(n) e_m(f(n,d))
\lessapprox
\sum_{\delta\mid (\alpha l,m)}\delta 
\Bigl(\sqrt{m'} +
\frac{N}{\delta\sqrt{m'}}\Bigr) \Bigl( 1 +
\frac{\Delta^{1/2}}{\delta^{1/2}} (m'y')^{1/6} +
\frac{\Delta}{\delta \sqrt{m'}}\Bigr)
$$
and
$$
\sum_d \sum_n \psi_{\Delta}(d) \psi_N(n) e_m(f(n,d))
\lessapprox
\sum_{\delta\mid (\alpha l,m)}\delta 
\Bigl(\sqrt{m'} +
\frac{N}{\delta\sqrt{m'}}\Bigr) \Bigl( \sqrt{m'}+
\frac{\Delta}{\delta \sqrt{m'}}\Bigr).
$$
It is now elementary to check that these give the bounds of
Proposition~\ref{lode} (note that $m'y'=my$).



\end{document}